\documentclass{amsbook}

\usepackage{tikz}
\usepackage{tikz-cd}
\usetikzlibrary{arrows}
\usepackage{amsmath}
\usepackage{amsfonts}
\usepackage{amssymb}
\usepackage{amsthm} 
\usepackage{stackengine}
\usepackage{mathtools}

\newtheorem{theorem}{Theorem}[chapter]
\newtheorem{lemma}[theorem]{Lemma}
\newtheorem{corollary}[theorem]{Corollary}
\newtheorem{proposition}[theorem]{Proposition}

\theoremstyle{definition}
\newtheorem{definition}[theorem]{Definition}
\newtheorem{example}[theorem]{Example}

\theoremstyle{remark}

\numberwithin{section}{chapter}
\numberwithin{equation}{chapter}

\DeclareMathOperator{\Arg}{Arg}
\DeclareMathOperator{\Aut}{Aut}
\DeclareMathOperator{\Ker}{Ker}
\DeclareMathOperator{\Int}{Int}
\DeclareMathOperator{\Ext}{Ext}

\DeclareMathOperator{\id}{id}
\DeclareMathOperator{\diam}{diam}
\DeclareMathOperator{\inc}{inc}
\DeclareMathOperator{\Real}{Re}
\DeclareMathOperator{\Imaginary}{Im}

\def\2sidelim{%
\lim_{{\substack{\scriptscriptstyle t_1 \leq t \leq t_2 \\%
\scriptscriptstyle  t_2 - t_1 \searrow 0} }}}

\def\sidelimv{%
\lim_{{\substack{\scriptscriptstyle t_1 \leq t_0 \leq t_2 \\%
\scriptscriptstyle  t_2 - t_1 \searrow 0} }}}

\makeindex

\begin{document}

\frontmatter

\title{Loewner Theory on Analytic Universal Covering Maps}


\author{Hiroshi Yanagihara}
\address{Professor Emeritus, Yamaguchi University, Japan \\
Visiting Professor, Indian Institute of Technology Bhubaneswar, 
India}
\email{hiroshi@yamaguchi-u.ac.jp}
\thanks{The author was supported in part 
by JSPS KAKENHI Grant Number 23K03150.}


\date{}

\subjclass[2020]{Primary 30C35, 30F10; Secondary 30F35, 30C45, 30C75.}

\keywords{Loewner chains,
Universal covering maps, 
Subordination principle.
Kernel convergence.
Hyperbolic metrics,
Fuchsian groups,
Univalent functions,
Loewner–Kufarev equation}


\begin{abstract}
We study Loewner chains in $\mathcal{H}_0(\mathbb{D})$ without
assuming univalence of each element. We prove a decomposition:
every chain admits a factorization $f_t=F\circ g_t$, where $F$ is
analytic on $\mathbb{D}(0,r)$ with
$r=\lim_{t \nearrow \sup I} f_t'(0)$, and $\{g_t\}$ is a classical
Loewner chain of univalent functions. Under a mild regularity
assumption on $t \mapsto f_t'(0)$, we derive a partial differential
equation that generalizes the Loewner--Kufarev equation. We then
develop a Loewner theory for chains of universal covering maps. We
characterize such chains in terms of domain families
$\{\Omega_t\}$: continuity and monotonicity of $\{f_t\}$ are
equivalent to kernel continuity and monotonicity of $\{\Omega_t\}$.
We further show that the connectivity
$C(\Omega_t)=\#(\hat{\mathbb{C}}\setminus \Omega_t)$ is a
left-continuous nondecreasing function of $t$. Building on these
results, we formulate a Loewner theory on Fuchsian groups and
obtain evolution equations for deck transformations. As an
application, we study hyperbolic metrics and establish a formula
for the logarithmic derivative of the hyperbolic density along the
chain. Our results provide a unified framework linking classical
Loewner theory, covering maps, and the geometry of hyperbolic
domains.
\end{abstract}

\maketitle

\tableofcontents


\mainmatter
\chapter{Introduction and Main Results}
\label{chapter:introduction}
Let ${\mathbb C}$ denote the complex plane, and let
$\hat{\mathbb{C}}=\mathbb{C} \cup \{ \infty \}$ be the Riemann sphere.
For $c \in \mathbb{C}$ and $r>0$ set
$\mathbb{D}(c,r)=\{ z \in \mathbb{C} : |z-c|<r \}$ and
$\overline{\mathbb{D}}(c,r)=\{ z \in \mathbb{C} : |z-c|\le r \}$.
In particular, we write $\mathbb{D}$ for the unit disc
$\mathbb{D}(0,1)$.
Let $\mathcal{H}(\mathbb{D})$ be the space of analytic functions on
$\mathbb{D}$ endowed with the topology of locally uniform
convergence on $\mathbb{D}$.
Set
$\mathcal{H}_0(\mathbb{D})
=\{ f \in \mathcal{H}(\mathbb{D}) : f(0)=0 \text{ and } f'(0)>0 \}$
and
$\mathfrak{B}
=\{ \omega \in \mathcal{H}_0(\mathbb{D}) : |\omega(z)|\le 1 \}$.
By the Schwarz lemma we also have
$\mathfrak{B}
=\{ \omega \in \mathcal{H}_0(\mathbb{D}) : |\omega(z)|\le |z| \}$.

\section{Background and Motivation}
Let $\Omega$ be a domain in $\mathbb{C}$ with $0 \in \Omega$ such that
$\mathbb{C}\setminus \Omega$ contains at least two points.
Consider the extremal problem
\begin{equation}
\label{eq:extremal_problem}
\sup \{ f'(0) :
f \in \mathcal{H}_0(\mathbb{D}),\ f(\mathbb{D}) \subset \Omega \}.
\end{equation}
If $\Omega$ is simply connected, the Riemann mapping theorem yields a
unique conformal mapping $f_0:\mathbb{D}\to \Omega$ with
$f_0 \in \mathcal{H}_0(\mathbb{D})$.
This $f_0$ solves the extremal problem, that is, for any
$f \in \mathcal{H}_0(\mathbb{D})$ with $f(\mathbb{D}) \subset \Omega$,
we have $f'(0)\le f_0'(0)$, with equality if and only if $f=f_0$.

If we drop simple connectivity, the extremal problem still has a unique
solution.
Let $f_0:S\to \Omega$ be an analytic covering map of a simply
connected Riemann surface $S$ onto $\Omega$.
By the Koebe uniformization theorem, “the single most important
theorem in the whole theory of analytic functions of one variable”
(cf.\ Ahlfors \cite[Chap.~10]{Ahlfors:ConformalInvariants}), we may
assume $S=\mathbb{D}$.
If necessary, after composing with a conformal automorphism of
$\mathbb{D}$, we may normalize $f_0$ so that $f_0(0)=0$ and
$f_0'(0)>0$.
Then for any $f \in \mathcal{H}_0(\mathbb{D})$ with
$f(\mathbb{D}) \subset \Omega$, there exists a unique
$\omega \in \mathfrak{B}$ such that $f=f_0\circ \omega$.
The map $\omega$ is called the lift of $f$ with respect to $f_0$.
For details, see \cite[Chaps.~9–10]{Ahlfors:ConformalInvariants} or
\cite{Springer}.
By Schwarz’s lemma,
\[
f'(0)=f_0'(0)\,\omega'(0)\le f_0'(0),
\]
with equality if and only if $\omega(z)\equiv z$, i.e., $f=f_0$.
Thus $f_0$ is again the unique solution to
\eqref{eq:extremal_problem}.
In this way analytic universal covering maps of $\mathbb{D}$ arise as
a natural generalization of conformal mappings, i.e., univalent
functions.

The theory of univalent functions has a long history and remains an
active field of research.
Likewise, the geometric theory of analytic universal covering maps has
been extensively studied in connection with Fuchsian groups.
Moreover, since the hyperbolic metric on a hyperbolic domain is
obtained by projecting the Poincar\'e metric on $\mathbb{D}$ via any
analytic universal covering map, results on hyperbolic metrics can be
interpreted as theorems on covering maps.

In this article we show that Loewner theory, a powerful method in the
study of univalent functions, is also effective for analytic universal
covering maps.
In 1923, L\"owner \cite{Loewner} discovered that any bounded slit
mapping $f$ of $\mathbb{D}$ admits a parametric representation
satisfying a differential equation now known as the Loewner
differential equation.
The parametric method was subsequently developed by many authors,
notably by Kufarev \cite{Kufarev1,Kufarev2} and Pommerenke
\cite{Pommerenke1965,Pommerenke1}.

In de Branges’s 1985 solution of the coefficient problem for
univalent functions \cite{deBrange}, the Loewner equation played a
central role.
Its use then waned for a time, but in the early 21st century
connections to statistical physics and conformal field theory led to
active study of the stochastic Loewner evolution (SLE).
While SLE primarily concerns conformal mappings of the upper
half-plane, a more unified framework now covers the unit disc, the
upper half-plane, and settings with fixed points away from the origin.
For a comprehensive reference to both the classical theory and its
unified extensions, see Bracci–Contreras–D\'iaz-Madrigal–Vasil'ev
\cite{BCDV}.

\section{Scope: Three Classes of Loewner Chains}
We focus on Loewner chains of analytic functions in $\mathbb{D}$,
introduced by Pommerenke \cite{Pommerenke1965}.
A function $f_0 \in \mathcal{H}(\mathbb{D})$ is subordinate to
$f_1 \in \mathcal{H}(\mathbb{D})$ (written $f_0 \prec f_1$) if there
exists an analytic map $\omega:\mathbb{D}\to \mathbb{D}$ with
$\omega(0)=0$ and $f_0=f_1\circ \omega$.
If $f_0,f_1 \in \mathcal{H}_0(\mathbb{D})$, then
$\omega'(0)=f_0'(0)/f_1'(0)\in (0,1]$, hence $\omega \in \mathfrak{B}$,
and by the identity theorem $\omega$ is uniquely determined by
$f_0,f_1$.

\begin{definition}
Let $I \subset [-\infty,\infty]$ and let $\{f_t\}_{t\in I} \subset
\mathcal{H}_0(\mathbb{D})$.
We say that $\{f_t\}_{t\in I}$ is a \emph{Loewner chain} if
\begin{equation}
\label{eq:subordination-chain}
f_s \prec f_t \quad \text{for all } (s,t)\in I_+^2,
\end{equation}
where $I_+^2=\{(s,t)\in I^2 : s\le t\}$.
For $(s,t)\in I_+^2$ let $\omega_{s,t}\in \mathfrak{B}$ be the unique
map satisfying $f_s=f_t\circ \omega_{s,t}$.
We call $\{\omega_{s,t}\}_{(s,t)\in I_+^2}$ 
the \emph{associated transition family} of $\{f_t\}_{t\in I}$.
By the Schwarz lemma, $f_t'(0)$ is nondecreasing and positive on $I$.
We say that $\{f_t\}_{t\in I}$ is 
\emph{strictly increasing} if $f_t'(0)$ is
strictly increasing in $t\in I$, i.e., $f_s'(0)<f_t'(0)$ whenever
$s<t$.
A Loewner chain $\{f_t\}_{t\in I}$ is called \emph{normalized} if
$f_t'(0)=e^t$, $t\in I$.
We say that $\{f_t\}_{t\in I}$ is 
\emph{continuous} if the map
$I \ni t \mapsto f_t \in \mathcal{H}(\mathbb{D})$ is continuous, 
that is, for each $t_0 \in I$, 
$f_t \to f_{t_0}$ locally uniformly on $\mathbb{D}$ as
$t\to t_0$ in $I$, equivalently, $f(z,t)$ is continuous on
$\mathbb{D}\times I$.
Here, we follow the convention $f(z,t):=f_t(z)$.
\end{definition}

In many texts and papers each $f_t$ is assumed univalent on
$\mathbb{D}$ in the definition of a Loewner chain.
We do not make this assumption.
Without it, Pommerenke \cite{Pommerenke1965} showed that if
$\{f_t\}_{t\in I}$ is a normalized Loewner chain, then for almost
every $t$ in the interior of $I$ the family $\{f_t\}_{t\in I}$
satisfies the Loewner–Kufarev PDE, which generalizes the classical
Loewner equation.

Since we are primarily concerned with continuous 
Loewner chains, we henceforth
assume $I$ is connected; that is, $I$ is an interval in
$[-\infty,\infty]$.
We consider three classes:
\begin{itemize}
\item[{\rm (I)}]
Loewner chains $\{f_t\}_{t\in I}$ with each
$f_t \in \mathcal{H}_0(\mathbb{D})$ univalent.
\item[{\rm (II)}]
Loewner chains $\{f_t\}_{t\in I}$ with each $f_t$ the universal
covering map of $\mathbb{D}$ onto
$\Omega_t:=f_t(\mathbb{D})$.
\item[{\rm (III)}]
All Loewner chains $\{f_t\}_{t\in I}$ without additional assumptions.
\end{itemize}

\section{Main Results for General Chains}
In the first half of the paper
(Chapters \ref{chapter:BasicEstimates}–%
\ref{chapter:SchlichtSubordination}),
we focus on class (III).
Chapters \ref{chapter:BasicEstimates}–\ref{chapter:Solution} treat
properties shared by classes (I) and (III), while
Chapter \ref{chapter:SchlichtSubordination} highlights differences.
Class (III) goes back to Pommerenke \cite{Pommerenke1965}, who did not
assume connectedness of $I$ and mainly studied normalized 
Loewner chains.
A decade later, Pommerenke \cite{Pommerenke1} introduced class (I) and
undertook a detailed study, especially for normalized 
Loewner chains of
univalent functions.

In the latter half
(Chapters \ref{chapter:KernelAndCovering}–%
\ref{chapter:HyperbolicMetrics}),
we study class (II), a geometrically natural generalization of class
(I).

To develop the theory for class (III), Chapter~2 introduces basic
estimates for transition families.
We then prove a key criterion: a Loewner chain
$\{f_t\}_{t\in I}$ is continuous if and only if
$a(t):=f_t'(0)>0$ is continuous on $I$.
The main result is the following decomposition theorem.

\begin{theorem}[Decomposition Theorem]
\label{theorem:FirstDecompositionTheorem}
Let $I \subset [-\infty,\infty)$ be a right-open interval with
$\beta=\sup I \notin I$, and let $\{f_t\}_{t\in I}$ 
be a Loewner chain with $a(t)=f_t'(0)$.
Let $a(\beta)=\lim_{t \nearrow \beta} a(t) \in (0,\infty]$.
\begin{itemize}
\item[{\rm (i)}]
The locally uniform limit $f_\beta=\lim_{t \nearrow \beta} f_t$ 
exists if and only if $a(\beta)<\infty$.
In this case there exist a unique analytic function
$F:\mathbb{D}(0,a(\beta))\to \mathbb{C}$ with $F(0)=0$ and
$F'(0)-1=0$, and a Loewner chain
$\{g_t\}_{t\in I\cup\{\beta\}}$ with
$\bigcup_{t\in I} g_t(\mathbb{D})
= g_\beta(\mathbb{D})=\mathbb{D}(0,a(\beta))$ such that
$f_t=F\circ g_t$ for $t\in I\cup\{\beta\}$.
Furthermore, if $\{f_t\}$ is continuous, each $g_t$ is univalent on
$\mathbb{D}$ for $t\in I\cup\{\beta\}$.
\item[{\rm (ii)}]
If $\{f_t\}$ is continuous and $a(\beta)=\infty$, then there exist a
unique entire function $F:\mathbb{C}\to \mathbb{C}$ with $F(0)=0$ and
$F'(0)-1=0$, and a Loewner chain $\{g_t\}_{t\in I}$ of univalent
functions with $\bigcup_{t\in I} g_t(\mathbb{D})=\mathbb{C}$ such that
$f_t=F\circ g_t$ for $t\in I$.
\end{itemize}
In both cases {\rm (i)} and {\rm (ii)}, the Loewner chains 
$\{f_t\}_{t \in I}$ and $\{g_t\}_{t \in I}$ share the same 
transition family.
\end{theorem}

A similar representation for normalized Loewner chains 
was already studied by
Pommerenke; see \cite[Satz~5]{Pommerenke1965}.

In Chapter~3, without assuming normalization, we show that a strictly
increasing and continuous Loewner chain $\{f_t\}_{t\in I}$ and its
transition family $\{\omega_{s,t}\}_{(s,t)\in I_+^2}$ satisfy,
respectively, a partial differential equation and an ordinary
differential equation with respect to $a(t):=f_t'(0)$.
We denote the partial derivative of $k(z,t)$ with respect to $a(t)$
by
\begin{align*}
\frac{\partial k}{\partial a}(z,t)
:=&\ \2sidelim \frac{k(z,t_2)-k(z,t_1)}{a(t_2)-a(t_1)}
= \lim_{\tau \to t} \frac{k(z,\tau)-k(z,t)}{a(\tau)-a(t)}.
\end{align*}
Let $\mu_a$ denote the Lebesgue–Stieltjes measure associated with
$a(t)$.

\begin{theorem}
\label{thm:IntLoewner-DEs}
Let $\{f_t\}_{t\in I}$ be a strictly increasing continuous Loewner
chain with transition family $\{\omega(\cdot,s,t)\}_{(s,t)\in I_+^2}$
and $a(t)=f_t'(0)$.
Then there exists a $G_\delta$-set $N \subset I$ with $\mu_a(N)=0$
such that for all $z \in \mathbb{D}$ and $t \in I\setminus N$ the
limit
\begin{align}
\label{eq:Int-differentail-wrt-a}
P(z,t)
=\ \2sidelim
\frac{\frac{\omega(z,t_1,t_2)}{z}-1}{\frac{a_{t_1}}{a_{t_2}}-1},
\quad z \in \mathbb{D}
\end{align}
exists, and the convergence is locally uniform on $\mathbb{D}$ for
each fixed $t \in I\setminus N$.
Define $P(z,t):=1$ for $(z,t)\in \mathbb{D}\times N$.
Then $P$ is Borel measurable on $\mathbb{D}\times I$, analytic in $z$,
and satisfies $\Real P(z,t)>0$ and $P(0,t)=1$.
Furthermore,
\begin{align}
\label{eq:Int-generating-ODE}
\frac{\partial \omega}{\partial a}(z,t)
=&\ -\frac{z}{a(t)}\,P(z,t),\quad t \in I\setminus N,\\
\label{eq:Int-Loewner-PDE-transition}
\frac{\partial \omega}{\partial a}(z,t,t_0)
=&\ \frac{z P(z,t)}{a(t)}\,\omega'(z,t,t_0),\quad
t \in (I \cap [-\infty,t_0))\setminus N,\\
\label{eq:Int-Loewner-ODE}
\frac{\partial \omega}{\partial a}(z,t_0,t)
=&\ -\frac{\omega(z,t_0,t)}{a(t)}\,
P(\omega(z,t_0,t),t),\quad
t \in (I \cap (t_0,\infty])\setminus N,\\
\label{eq:Int-Loewner-PDE}
\frac{\partial f}{\partial a}(z,t)
=&\ \frac{z}{a(t)}\,P(z,t)\,f'(z,t),\quad t \in I\setminus N.
\end{align}
Here, $\omega'(z,t_1,t_2)$ and $f'(z,t)$ denote derivatives with
respect to $z$.
\end{theorem}

These differential equations generalize the classical Loewner–Kufarev
equations.
Let $\psi:I\to \mathbb{R}$ be strictly increasing and continuous.
Let $\mu_\psi$ and ${\mathcal F}_\psi$ denote the Lebesgue–Stieltjes
measure and the associated $\sigma$-algebra on $I$ with respect to
$\psi$, respectively.
Then $(I,{\mathcal F}_\psi,\mu_\psi)$ is a complete measure space, and
$\mathcal{B}(I)\subset {\mathcal F}_\psi$, where $\mathcal{B}(I)$
denotes the Borel $\sigma$-algebra.
A family $\{P(\cdot,t)\}_{t\in I}\subset \mathcal{H}(\mathbb{D})$ is a
Herglotz family if $\Real P(z,t)>0$ on $\mathbb{D}$ and $P(0,t)=1$ for
each $t \in I$.
Given a $\sigma$-algebra $\mathcal{F}$ on $I$, we say
$\{P(\cdot,t)\}_{t\in I}$ is $\mathcal{F}$-measurable if, for each
fixed $z \in \mathbb{D}$, the map $t \mapsto P(z,t)$ is
$\mathcal{F}$-measurable.

In Chapter~\ref{chapter:Solution} we show that
\eqref{eq:Int-Loewner-ODE} has a unique solution.
Let $a(t)$ be strictly increasing, continuous, and positive on $I$,
and let $\{P(\cdot,t)\}_{t\in I}$ be a
$\mathcal{F}_a$-measurable Herglotz family.
Then, for each fixed $s \in I$, the ODE
\[
\frac{dw}{da}(t)=-\frac{1}{a(t)}\,w\,P(w,t),\quad
t \in I \cap [s,\infty),
\]
with initial condition $w(s)=z \in \mathbb{D}$, has a unique solution
on $I \cap [s,\infty)$.
Writing this solution as $\omega_{s,t}(z)$, the family
$\{\omega_{s,t}\}_{(s,t)\in I_+^2}$ forms a transition family.
We also give a representation formula for $g_t(z)$ in
Theorem~\ref{theorem:FirstDecompositionTheorem}.

In Chapter~\ref{chapter:SchlichtSubordination} 
we study properties that
hold for class (I) but need not hold for class (III).
We say $f \in \mathcal{H}_0(\mathbb{D})$ is 
\emph{maximal (in the sense of continuous Loewner chains)} 
if there is no continuous Loewner chain
$\{f_t\}_{0\le t \le \varepsilon}$ with $f_0=f$ and
$f_\varepsilon'(0)>f'(0)$ for some $\varepsilon>0$.
Pommerenke \cite{Pommerenke1} proved that for any univalent
$f \in \mathcal{H}_0(\mathbb{D})$ there exists a continuous 
Loewner chain
$\{f_t\}_{0\le t<\infty}$ of univalent functions with $f_0=f$ and
$\lim_{t\to \infty} f_t'(0)=\infty$; thus univalent functions in
$\mathcal{H}_0(\mathbb{D})$ are never maximal.

\begin{theorem}
\label{thm:no-nontangentaial-limit}
If $f \in \mathcal{H}_0(\mathbb{D})$ has 
nontangential boundary values almost nowhere,
(i.e,  for almost every $\zeta \in \partial \mathbb{D}$,
the nontangential limit of $f$ does not exist at $\zeta$),
then $f$ is maximal in
the sense of continuous Loewner chains.
\end{theorem}

\section{Kernel Convergence and Universal Coverings} 
Loewner chains of analytic universal covering maps, which are our main
concern in the latter half, provide a natural geometric generalization
of chains of univalent functions.
We record two properties shared by univalent functions and universal
covering maps.

For $j=0,1$, let $\Omega_j$ be a hyperbolic domain in $\mathbb{C}$
and let $f_j:\mathbb{D}\to \Omega_j$ be a universal covering map with
$f_0(0)=f_1(0)$.
If $\Omega_0$ and $\Omega_1$ are simply connected, then since $f_0$
and $f_1$ are univalent it is easy to see that
\begin{equation}
\label{inclusion-subordination}
\Omega_0 \subset \Omega_1
\quad \text{if and only if} \quad f_0 \prec f_1 .
\end{equation}

The same equivalence holds in general.
Indeed, if $f_0 \prec f_1$ then $\Omega_0 \subset \Omega_1$ is clear.
Conversely, if $\Omega_0 \subset \Omega_1$, then by the lifting lemma
for covering maps (see \cite[Lemma~97.1]{Munkres}) the map
$f_0:\mathbb{D}\to \Omega_0(\subset \Omega_1)$ lifts to a unique
continuous map $\omega:\mathbb{D}\to \mathbb{D}$ with
$f_0=f_1\circ \omega$ and $\omega(0)=0$,
\[
\begin{tikzcd}[column sep=4em]
& \mathbb{D} \arrow[d,"f_1"] \\
\mathbb{D} \arrow[r,"f_0"'] \arrow[ru,"\omega"]
& \Omega_0 \overset{\inc}{\hookrightarrow} \Omega_1 \hspace{6ex}
\end{tikzcd}
\]
Since $f_0$ and $f_1$ are analytic and locally univalent, $\omega$ is
analytic on $\mathbb{D}$, hence $f_0 \prec f_1$.

Next, let $a \in \mathbb{C}$ and let $\{\Omega_n\}_{n=1}^\infty$ be a
sequence of hyperbolic domains with $a \in \Omega_n$ for each $n$.
Let $\Ker(a,\{\Omega_n\})$ be the set consisting of $a$ and all points
$w$ for which there exist a domain $H$ and $N \in \mathbb{N}$ with
$a,w \in H \subset \Omega_n$ for all $n \ge N$.
By definition, $\Ker(a,\{\Omega_n\})$ is a domain containing $a$, or
else $\{a\}$.
We say $\{\Omega_n\}$ converges to a domain $\Omega$ in the sense of
kernel with respect to $a$ if
$\Ker(a,\{\Omega_{n_k}\})=\Omega$ for every subsequence
$\{\Omega_{n_k}\}$.
Let $f$ and $f_n \in \mathcal{H}(\mathbb{D})$ be the universal
covering maps of $\mathbb{D}$ onto $\Omega$ and $\Omega_n$ with
$f(0)=f_n(0)=a$ and $f'(0),f_n'(0)>0$.

If $\Omega$ and each $\Omega_n$ are simply connected, the
Carath\'eodory kernel convergence theorem yields the equivalence
\begin{align}
\label{kernel-loc-unif-conv}
& \Omega_n \to \Omega \text{ in the sense of kernel w.r.t.\ } a \\
& \text{ if and only if } f_n \to f \text{ locally uniformly on }
\mathbb{D}. \nonumber
\end{align}
When $\Omega$ and $\Omega_n$ are not necessarily simply connected, the
Carath\'eodory theorem does not apply.
However, Hejhal’s generalization \cite{Hejhal} implies that the
equivalence \eqref{kernel-loc-unif-conv} still holds.
In Chapter~\ref{chapter:KernelAndCovering} we recall Pommerenke’s
criterion for kernel convergence and slightly generalize Hejhal’s
theorem.

Using \eqref{inclusion-subordination} and
\eqref{kernel-loc-unif-conv}, Pommerenke \cite{Pommerenke1} developed
his theory of Loewner chains of univalent functions in 1975.
In the same spirit we obtain the following.

\begin{theorem}
\label{theorem:f_t_and_Omega_t}
Let $\{\Omega_t\}_{t\in I}$ be a family of hyperbolic domains in
$\mathbb{C}$ with $0 \in \Omega_t$ for each $t \in I$.
For each $t$ let $f_t:\mathbb{D}\to \Omega_t$ be the universal
covering map with $f_t(0)=0$ and $f_t'(0)>0$.
Then $\{f_t\}_{t\in I}$ is a continuous Loewner chain of universal
covering maps if and only if $\{\Omega_t\}_{t\in I}$ is nondecreasing
and continuous.
\end{theorem}

Here,``$\{\Omega_t\}_{t\in I}$ is continuous'' means that for every
$t_0 \in I$ and every sequence $\{t_n\} \subset I$ with
$t_0 \ne t_n \to t_0$, one has $\Omega_{t_n} \to \Omega_{t_0}$ in the
sense of kernel with respect to $0$.

It is natural to expect that Pommerenke’s theory extends to 
Loewner chains of
universal covering maps.
For instance, $\{f_t\}$ and its transition family satisfy the
generalized Loewner–Kufarev equations
\eqref{eq:Int-generating-ODE}–\eqref{eq:Int-Loewner-PDE}.
Nonetheless, phenomena arise that do not occur in the univalent case.

For a domain $G \subset \hat{\mathbb{C}}$ write
$C(G) \in \mathbb{N}\cup\{\infty\}$ for the number of connected
components of $\hat{\mathbb{C}}\setminus G$.
Let $\Omega,\Omega_n$ be hyperbolic domains and assume
$\Omega_n \to \Omega$ in the sense of kernel.
Suppose that  each $\Omega_n$ is simply connected.
Clearly, we have $C(\Omega_n)=1$. 
One can show that $\Omega$ is also simply
connected and $\Omega \ne \hat{\mathbb{C}}$, 
so $C(\Omega)=1= \lim_{n\to \infty} C(\Omega_n)$.
In general, however, only the lower semicontinuity
$C(\Omega) \le \liminf_{n\to \infty} C(\Omega_n)$ holds (see
Theorem~\ref{ineq:lower-semi-continuity}
and Example~\ref{example:strip-minus-sgments}).

Concerning the limiting behavior of the image domains of a 
Loewner chain, we have the following example.
\begin{example}
\label{eg:CantorSet}
Let $E_0$ be the closed line segment joining $1$ and $2$ in
$\mathbb{C}$.
For $0<t\le 1$, obtain $E_t$ from $E_0$ by removing the concentric
open subsegment of length $3^{-1}t$.
Then $E_1$ consists of two closed segments.
For $1<t\le 2$, obtain $E_t$ from each of those two segments by
removing a concentric open subsegment of length $3^{-2}(t-1)$.
Continuing indefinitely yields $\{E_t\}_{t\ge 0}$ and
$E_\infty=\bigcap_{t\ge 0} E_t$, which is a translate of the Cantor
ternary set.
Let $f_t$ be the universal covering maps 
of $\mathbb{D}$ onto
$\Omega_t:=\mathbb{C}\setminus E_t$ with $f_t(0)=0$ and $f_t'(0)>0$.
Then $\{\Omega_t\}_{t\in [0,\infty]}$ is strictly increasing and
continuous in the kernel sense, and hence
$\{f_t\}_{t\in [0,\infty]}$ is a strictly increasing continuous
Loewner chain of universal covering maps.
\end{example}

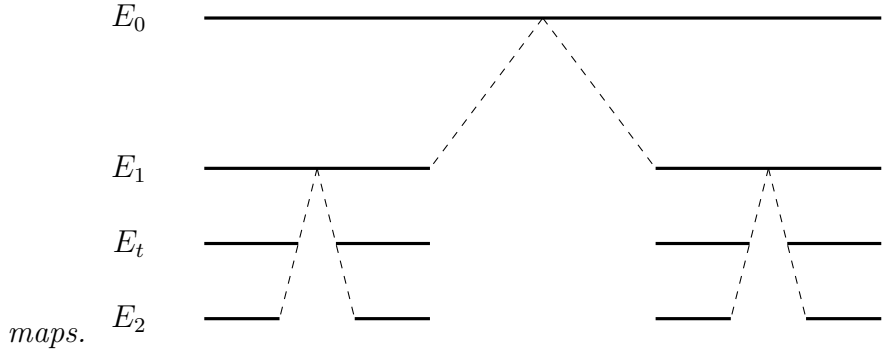
\begin{center}
\begin{tikzpicture}
\draw[very thick] (0,0)--(9,0);
\node at (-1,0) {$E_0$};
\draw[very thick] (0,-2)--(3,-2);
\draw[very thick] (6,-2)--(9,-2);
\node at (-1,-2) {$E_1$}; 
\draw[dashed] (4.5,0)--(3,-2);
\draw[dashed] (4.5,0)--(6,-2);
\draw[very thick] (0,-4)--(1,-4);
\draw[very thick] (2,-4)--(3,-4);
\draw[very thick] (6,-4)--(7,-4);
\draw[very thick] (8,-4)--(9,-4);
\node at (-1,-4) {$E_2$};
\draw[dashed] (1.5,-2)--(1,-4);
\draw[dashed] (1.5,-2)--(2,-4);
\draw[dashed] (7.5,-2)--(7,-4);
\draw[dashed] (7.5,-2)--(8,-4);
\draw[very thick] (0,-3)--(1.25,-3);
\draw[very thick] (1.75,-3)--(3,-3);
\draw[very thick] (7.75,-3)--(9,-3);
\draw[very thick] (6,-3)--(7.25,-3);
\node at (-1,-3) {$E_t$};
\end{tikzpicture} 
\end{center}

Let $\{\Omega_t\}_{t\in I}$ be a continuous nondecreasing family of
domains in $\hat{\mathbb{C}}$, and set
$E_t=\hat{\mathbb{C}}\setminus \Omega_t$.
Example~\ref{eg:CantorSet} shows that if $C$ is a connected component
of $E_{t_0}$, then the sets $C\cap E_t$ for $t \ge t_0$ shrink and may
split into many pieces as $t$ increases.
In Chapter~\ref{chapter:Connectivity} we prove the following.

\begin{theorem}
\label{thm:nonvanishing}
Let $\{\Omega_t\}_{t\in I}$ be a continuous nondecreasing family of
domains in $\hat{\mathbb{C}}$, and let $E_t=\hat{\mathbb{C}}\setminus
\Omega_t$.
If $C$ is a connected component of
$\hat{\mathbb{C}}\setminus \Omega_{t_0}$ for some $t_0 \in I$, then
\[
C \cap \bigcap_{t \in I} E_t \ne \emptyset .
\]
In particular, $C(\Omega_t)$ is left-continuous and nondecreasing in
$t \in I$.
\end{theorem}

Note that this theorem has a trivial counterpart 
for Loewner chains of universal covering maps.

A subset $A \subset \hat{\mathbb{C}}$ is totally disconnected if each
connected component of $A$ is a single point.
We say that 
a domain $\Omega \subset \hat{\mathbb{C}}$ is 
\emph{maximal in the sense of kernel} 
if there is no continuous family $\{\Omega_t\}_{0\le t <
\varepsilon}$ with $\Omega_0=\Omega$ and $\Omega \subsetneq
\Omega_\varepsilon$.
\begin{corollary}
If the complement $\hat{\mathbb{C}}\setminus \Omega$ is totally
disconnected, then $\Omega$ is maximal in the sense of kernel.
\end{corollary}
Hence $\Omega_\infty$ in Example~\ref{eg:CantorSet} is maximal.
We also give an example of a maximal domain whose complement is not
totally disconnected.

To prove the lower semicontinuity of connectivity and
Theorem~\ref{thm:nonvanishing}, we introduce a simple topological
separation lemma.
Let $\alpha:\partial \mathbb{D}\to \hat{\mathbb{C}}$ be a simple
closed curve.
By the Jordan curve theorem,
$\hat{\mathbb{C}}\setminus \alpha(\partial \mathbb{D})$ consists of
exactly two domains $D_1$ and $D_2$ with
$\partial D_1=\partial D_2=\alpha(\partial \mathbb{D})$.
We say that $\alpha$ separates sets $B_1$ and $B_2$ if they lie in
different components of
$\hat{\mathbb{C}}\setminus \alpha(\partial \mathbb{D})$.
\begin{lemma}
\label{lemma:separation}
Let $\Omega \subset \hat{\mathbb{C}}$, and let $C$ be a connected
component and $F$ a nonempty closed subset of
$\hat{\mathbb{C}}\setminus \Omega$ with $C \cap F=\emptyset$.
Then there exists a simple closed curve in $\Omega$ that separates $C$
and $F$.
\end{lemma}
See Newman \cite[Theorem~3.3, Chapter~VI]{Newman} for a proof in the
case where $F$ is also a component of
$\hat{\mathbb{C}}\setminus \Omega$.
Although natural and elementary, we are not aware of a reference in
full generality, so we provide a short proof in
Appendix~\ref{chapter:separation}.

\section{Loewner Theory on Fuchsian Groups and Applications}
In Chapter~\ref{chapter:DeckTransformation} we study Loewner theory on
Fuchsian groups.
Let $\{\Omega_t\}_{t\in I}$ be a continuous nondecreasing family of
hyperbolic domains in $\mathbb{C}$ with $0 \in \Omega_t$, and let
$\{f_t\}_{t\in I} $ be the
corresponding Loewner chain of universal covering maps.
For each $t$ let $\Gamma_t$ be the covering transformation group of
$f_t:\mathbb{D}\to \Omega_t$, i.e., 
$\Gamma_t = 
\{ \varphi \in \Aut(\mathbb{D}) : f_t \circ \varphi = f_t \}$.
Here $\Aut(\mathbb{D})$ denotes the group of all automorphisms of
$\mathbb{D}$.

We introduce a family of 
mappings $\{\sigma_{s,t}\}_{(s,t)\in I_+^2}$ with
$\sigma_{s,t}:\Gamma_s \to \Gamma_t$ 
satisfying the semigroup relation
\begin{equation}
\sigma_{t_1,t_2}\circ \sigma_{t_0,t_1}=\sigma_{t_0,t_2}
\end{equation}
for $t_0\le t_1 \le t_2$.

\begin{theorem}
\label{thm:int_the_injective_homomorphsim}
Let $\{f_t\}_{t\in I}$ be a continuous strictly increasing chain of
universal covering maps. 
For $(s,t)\in I_+^2$, the map 
$\sigma_{s,t}:\Gamma_s \to \Gamma_t$ is
an injective homomorphism and satisfies
$\omega_{s,t}\circ \varphi=\sigma_{s,t}(\varphi)\circ \omega_{s,t}$.
\end{theorem}

\begin{theorem}
\label{theorem:int_LoewnerPDEforDeckTransformation}
Let $\{f_t\}_{t\in I}$ be a continuous strictly increasing chain of
universal covering maps, with Herglotz family
$\{P(\cdot,t)\}_{t\in I}$ and a $G_\delta$-set $N \subset I$ with
$\mu_a(N)=0$.
Let $t_0 \in I$ and $\varphi \in \Gamma_{t_0}$, and set
$\varphi_t:=\sigma_{t_0,t}(\varphi)\in \Gamma_t$ for
$t \in I \cap [t_0,\infty)$.
Then the map $t \mapsto \varphi_t \in \Aut(\mathbb{D})$ is continuous
on $I \cap [t_0,\infty)$ and
\begin{align}
\label{eq:int_LoewnerPDEforDeckTransformation}
\frac{\partial \varphi}{\partial a(t)}(z,t)
=&\ \frac{1}{a(t)}\Big\{ z P(z,t)\,\varphi'(z,t)
- \varphi(z,t)\,P(\varphi(z,t),t) \Big\}
\end{align}
for $t \in (I \cap [t_0,\infty))\setminus N$, where
$\varphi(z,t):=\varphi_t(z)$.
If in addition $a(t)$ is locally absolutely continuous on $I$ and
$\dot a(t):=\frac{da}{dt}(t)>0$ a.e., then for each fixed
$z \in \mathbb{D}$ the map $t \mapsto \varphi_t(z)$ is absolutely
continuous and
\begin{equation}
\label{eq:int_LoewnerPDEforDeckTransformation_by_time}
\frac{\partial \varphi}{\partial t}(z,t)
=\frac{\dot a(t)}{a(t)}\Big\{ z P(z,t)\,\varphi'(z,t)
- \varphi(z,t)\,P(\varphi(z,t),t) \Big\}
\end{equation}
holds a.e.\ on $I \cap [t_0,\infty)$.
\end{theorem}

In Chapter~\ref{chapter:HyperbolicMetrics} we apply Loewner theory for
universal covering maps to hyperbolic metrics and derive a formula for
the logarithmic derivative of the hyperbolic density.

\chapter{Transition Families and Loewner Chains}
\label{chapter:BasicEstimates}
\section{Preliminaries on Transition Families}
Recall that $\mathfrak{B}$ is the class of all holomorphic mappings 
$\omega : \mathbb{D} \rightarrow \mathbb{D}$  
with $\omega(0)=0$ and $\omega'(0) > 0$. 
\begin{definition}
For an interval $I \subset [- \infty, \infty ]$ let
$I_+^2 = \{ (s,t) : s,t \in I \text{ with } s \leq t  \}$.
Let $\{ \omega_{s,t} \}_{(s,t) \in I_+^2} $ 
be a family of functions in $\mathfrak{B}$. 
We say that $\{ \omega_{s,t} \}_{(s,t) \in I_+^2} $ 
is a \textit{transition family}
if 
\begin{equation}
\label{eq:stu-transitions}
  \omega_{t,t} =  \id_{\mathbb{D}} \quad \text{and} \quad
  \omega_{t_1,t_2} \circ \omega_{t_0,t_1} =  \omega_{t_0,t_2}
\end{equation}
for all $t, t_0, t_1,t_2 \in I$ with 
$t_0 \leq t_1 \leq t_2$.
Here, $\id_{\mathbb{D}}$ is the identity mapping on $\mathbb{D}$.
\end{definition}
Let  $\{\omega_{s,t} \}_{(s,t) \in I_+^2} $ be a transition family. 
Define 
\[
  a_{s,t} = \omega_{s,t}'(0) > 0 , 
  \qquad (s, t) \in I_+^2 .
\]
By the Schwarz lemma, we have
$a_{s,t} \leq 1$ for all $(s,t) \in I_+^2$,
and from (\ref{eq:stu-transitions}) it follows
\begin{equation}
\label{eq:stu-transition-derivative}
   a_{t,t} = 1 \quad \text{and} \quad a_{t_1,t_2} a_{t_0,t_1} =  a_{t_0,t_2}
\end{equation}
for all $t , t_0,t_1,t_2 \in I$ with  
$t_0 \leq t_1 \leq t_2$.
Thus, for each fixed $s$, the function $a_{s,t}$ 
is nonincreasing in $t \in I \cap [s, \infty ]$, 
and for each fixed $t$, it is nondecreasing 
in $s \in I \cap [ - \infty ,t ]$.
Fix $t_0 \in I$ and $c>0$, and define
\begin{equation}
\label{eq:defi-of-at}
    a(t) = \begin{cases} 
	    c a_{t , t_0},  \quad & t \in I \cap [-\infty , t_0] ,\\
	    \frac{c}{a_{t_0 ,t}},  \quad & t \in I \cap  (t_0, \infty ] .
	   \end{cases} 
\end{equation}
Then $a(t)$ is nondecreasing, such that $a(t_0) = c$, and satisfies
\begin{equation}
\label{eq:relation-between-a_s-and-a_st}
  a_{s,t} = \frac{a(s)}{a(t)}, \quad (s,t) \in I_+^2 .
\end{equation}
Conversely, 
if $a(t)$, $t \in I$, is a positive and nondecreasing function,
then $a_{s,t}$ defined by
(\ref{eq:relation-between-a_s-and-a_st})
satisfies (\ref{eq:stu-transition-derivative}).

Let $\{f_t \}_{t \in I}$ be a Loewner chain.
Then for each $(s,t) \in I_+^2$,
there exists a function $\omega_{s,t} \in \mathfrak{B}$
satisfying $f_s = f_t \circ \omega_{s,t}$.
Since $f_t '(0)>0$, $f_t$ is univalent in a neighborhood
of the origin, and hence $\omega_{s,t}$ is uniquely determined.
It is easy to see that
$\{ \omega_{s,t} \}_{(s,t) \in I_+^2}$ forms a transition family.
We call $\{ \omega_{s,t} \}_{(s,t) \in I_+^2}$ 
the transition family associated with $\{f_t \}_{t \in I}$.
In this case we have $a_{s,t} = a(s)/a(t)$, 
where we put $a(t) = f_t'(0)> 0$, $t \in I$.

We begin our discussion the following 
fundamental inequalities.
\begin{lemma}
\label{lemma:fundamental_inequalities}
Let $\omega \in \mathfrak{B}$ with $\alpha = \omega' (0) > 0$.
Then, for $z \in \mathbb{D}$, the following inequalities hold:
\begin{align}
  &
 \left| 
     \displaystyle 
    \frac{\omega(z)}{z} - \alpha 
  \right|
  \leq 
   |z|
     \left|   1 - \alpha \frac{\omega(z)}{z} \right| ,
\label{ineq:1st-inequality-by-Schwarz-Pick}
\\
  &
 \left| 
 \displaystyle \frac{\omega (z)}{z} 
 -
 \frac{ \alpha  (1-|z|^2) }
 {1- \alpha^2 |z|^2}
 \right|
 \leq
 \frac{\left( 1 - \alpha^2  \right) |z|  }
 {1- \alpha^2 |z|^2} ,
\label{ineq:2nd-inequality-in-Euclidean-form}
\\[2ex]
 & \left|
   \omega (z)
   -
   z
   \right|
   \leq
   (1- \alpha) \frac{|z|(1+|z|) }
   {\left( 1 - \alpha |z| \right) } ,
\label{ineq:distance-between-omega/z-and-1}
\\[2ex]
 & 
  |z| \frac{\alpha -|z|}{1 -  \alpha  |z|}
  \leq
  \left|
   \omega(z)
   \right|
   \leq
   |z| \frac{|z|+ \alpha }{1 +  \alpha  |z|} .
\label{ineq:omega_growth}
\end{align} 
\end{lemma}
\begin{proof}
We may assume $0< \alpha < 1$, since the case 
$\alpha = 1$ implies $\omega = \id_{\mathbb{D}}$ 
by the Schwarz lemma, and all the inequalities 
(\ref{ineq:1st-inequality-by-Schwarz-Pick})-(\ref{ineq:omega_growth})
 hold trivially.
Applying the Schwarz-Pick inequality 
to the analytic function $g(z) = \omega(z)/z$ in $\mathbb{D}$,
we obtain
\begin{equation}
\label{ineq:Schwarz-Pick}
   \left| \frac{g(z)-\alpha}{1-\alpha g(z)} \right| \leq |z|,
 \quad z \in \mathbb{D} .
\end{equation}
This inequality is equivalent to 
(\ref{ineq:1st-inequality-by-Schwarz-Pick}).
A direct computation from (\ref{ineq:Schwarz-Pick}) 
yields 
\begin{equation}
\label{ineq:Schwarz-Pick-euclidean-form}
  \left| g(z) - \frac{\alpha(1-|z|^2)}{1-\alpha^2 |z|^2} \right|
 \leq
 \frac{(1-\alpha^2)|z|}{1-\alpha^2|z|^2 },
\end{equation}
which in turn is equivalent
to (\ref{ineq:2nd-inequality-in-Euclidean-form}).
Next, observe that 
\[
   1
   -
   \frac{\alpha (1-|z|^2) }{1- \alpha^2 |z|^2} 
   =
   \frac{(1- \alpha) 
   \left( 
    1  + \alpha |z|^2 
   \right)}
   {1- \alpha^2 |z|^2} .   
\]
Hence using the triangle inequality, we have
\begin{align*}
 \left| g(z) - 1 \right|
  & \leq
  \left| g(z)
 - \frac{\alpha (1-|z|^2) }{1- \alpha^2 |z|^2} 
  \right|
 +
 \left| 
  \frac{\alpha (1-|z|^2)}{1- \alpha^2 |z|^2 }
   -
  1
  \right|
\\
 &
  \leq
  \frac{ (1- \alpha^2) |z| }{1- \alpha^2 |z|^2 }
  +
 \frac{(1-\alpha)(1+ \alpha |z|^2 )}
   { 1- \alpha^2 |z|^2 } 
  \leq
  \frac{(1- \alpha) (1+ |z| )}{1 - \alpha |z| } ,
\end{align*}
which proves (\ref{ineq:distance-between-omega/z-and-1}).
On the other hand, 
from (\ref{ineq:Schwarz-Pick-euclidean-form}), we obtain
\begin{align*}
 \frac{\alpha -|z|}{1 - \alpha |z|} 
 = & \,  \frac{\alpha (1-|z|^2)}{1-\alpha^2 |z|^2} 
 -   \frac{(1-\alpha^2)|z|}{1-\alpha^2|z|^2 }
\\
 \leq & \,  |g(z)|
 \leq 
 \frac{\alpha (1-|z|^2)}{1-\alpha^2 |z|^2} 
 +   \frac{(1-\alpha^2)|z|}{1-\alpha^2|z|^2 }
 = \frac{|z|+ \alpha}{1+\alpha |z|} ,
\end{align*}
which implies (\ref{ineq:omega_growth}).
\end{proof}

\begin{proposition}
\label{proposition:fundamental_inequalities}
Let $\{ \omega_{s,t} \}_{(s,t) \in I_+^2} $ be 
a transition family, and define $a_{s,t} = \omega_{s,t}'(0)$
for ${(s,t) \in I_+^2}$.
Then for $(s,t) \in I_+^2$ and 
for $t_0,t_1,t_2 \in I$ with $t_0 <  t_1 < t_2 $, 
the following inequalities hold:
\begin{align}
 & \left|
   \omega_{s,t}(z)
   -
   z
   \right|
   \leq
   (1- a_{s,t}) \frac{|z|(1+|z|) }
   {\left( 1 - a_{s,t} |z| \right) } ,
\label{ineq:distance-between-omega-and-id}
\\[2ex]
 & \left|
   \omega_{t_0,t_2}(z)
   -
   \omega_{t_0,t_1}(z)
   \right|
   \leq
   (1- a_{t_1,t_2}) \frac{|z|(1+|z|) }
   {\left( 1 - a_{t_1,t_2}  |z| \right) } .
\label{ineq:omega_isLipschitz_wrt_a}
\end{align}  
\end{proposition}
\begin{proof}
Inequality (\ref{ineq:distance-between-omega-and-id})
follows directly from (\ref{ineq:distance-between-omega/z-and-1}).
By replacing $s$ and $t$ by $t_1$ and $t_2$ respectively in 
(\ref{ineq:distance-between-omega-and-id}),
and substituting $\omega_{t_0,t_1}(z)$ for $z$,  
we obtain
\[
 \left| 
 \omega_{t_1,t_2} (\omega_{t_0,t_1}(z)) - \omega_{t_0,t_1}(z)
 \right|
 \leq
   (1-a_{t_1,t_2}) 
   \frac{|\omega_{t_0,t_1}(z)|(1+|\omega_{t_0,t_1}(z)|) }
   { \left( 1 - a_{t_1,t_2} |\omega_{t_0,t_1}(z)| \right) } . 
\]
Since $\omega_{t_1,t_2} (\omega_{t_0,t_1}(z)) = \omega_{t_0,t_2} (z)$,
(\ref{ineq:omega_isLipschitz_wrt_a}) easily follows
from the above inequality 
and $|\omega_{t_0,t_1}(z)| \leq |z|$.
\end{proof}

\section{Continuity and Univalence of Transition Families}
As simple applications of 
Proposition~\ref{proposition:fundamental_inequalities}
we give criteria for the continuities of 
a transition family and a Loewner chain. 
\begin{definition}
Let $\{ \omega_{s,t} \}_{(s,t) \in I_+^2}$ be a transition family.
We say that $\{ \omega_{s,t} \}_{(s,t) \in I_+^2}$ 
is \textit{continuous}
if the mapping
$I_+^2 \ni (s,t) \mapsto \omega_{s,t} \in \mathcal{H}(\mathbb{D})$ 
is continuous on $I_+^2$; that is,
$\omega_{s,t} \rightarrow \omega_{s_0,t_0}$ 
locally uniformly on $\mathbb{D}$
as $(s,t) \rightarrow (s_0,t_0)$ in $I_+^2$ at every $(s_0,t_0) \in I_+^2$.
Also we say that $\{ \omega_{s,t} \}_{(s,t) \in I_+^2}$ is 
\textit{strictly monotone}
if $a_{st} < 1$ for $(s,t) \in I_+^2$ with $s \not= t$.
This is equivalent to
that $a(t)$ is strictly increasing, where $a(t)$ is defined by 
(\ref{eq:defi-of-at}). 
\end{definition}

\begin{theorem}
\label{theorem:continuity-of-transition-family}
Let $\{ \omega_{s,t} \}_{(s,t) \in I_+^2}$ be a transition family
with $a_{s,t} = \omega_{s,t}'(0)$, $(s,t) \in I_+^2$
and let $a(t)$, $t \in I$ be defined by
(\ref{eq:defi-of-at}) for some $c>0$.
Then the following five conditions are equivalent:
\begin{itemize}
 \item[{\rm (i)}]
The function $a(t)$ is continuous on $I$.
 \item[{\rm (ii)}]
For all $t_0 \in I$,
$a_{s,t} \rightarrow a_{t_0,t_0} = 1$ 
as $t-s \searrow 0$ with $s \leq t_0 \leq t$.
 \item[{\rm (iii)}]
The mapping
$I_+^2 \ni (s,t) \mapsto a_{s,t} \in (0,1]$ 
is continuous on $I_+^2$.
 \item[{\rm (iv)}]
The mapping
$I_+^2 \ni (s,t) \mapsto \omega_{s,t} \in \mathcal{H}(\mathbb{D})$ 
is continuous on $I_+^2$.
 \item[{\rm (v)}]
The mapping
$\mathbb{D} \times I_+^2 \ni (z,s,t) \mapsto \omega_{s,t}(z) \in \mathbb{D}$ 
is continuous on $\mathbb{D} \times I_+^2$.
\end{itemize} 
\end{theorem}
\begin{proof}
The equivalence of (iv) and (v) is elementary, 
and its proof is omitted.
By (\ref{eq:relation-between-a_s-and-a_st}) 
it is clear that (i) implies (iii), and that (iii) implies (ii).
Assume (ii). 
Then by (\ref{eq:relation-between-a_s-and-a_st}) 
\begin{align*}
  \lim_{s \nearrow t} a(s) 
 = & \, 
    \lim_{s \nearrow t} a(t) a_{s,t}
  = a(t) \lim_{s \nearrow t} a_{s,t} = a(t) 
\\
  \lim_{t \searrow s} a(t)
 = & \,
  \lim_{t \searrow t} \frac{a(s)}{a_{s,t}} 
 =  a(s) \lim_{t \searrow s} \frac{1}{a_{s,t}} = a(s) .
\end{align*}
Therefore the function $a(t)$ is left-continuous and right-continuous 
on $I$ and (i) holds.

From
\[
 a_{s,t} =   \omega_{s,t}'(0) 
 = 
 \frac{1}{2 \pi i} 
 \int_{|z|=r} \frac{\omega_{s,t}(z)}{z^2} \, dz , \quad 0 <  r < 1 ,
\]
it follows that (iv) implies (iii).

It remains to show that (iii) implies (iv).
To see this let $(s_0,t_0) \in I_+^2$.
We show that 
$\{ \omega_{s,t}\}_{(s,t) \in I_+^2}$ is continuous at $(s_0,t_0)$.
First we consider the case that $s_0=t_0$.
Since  by 
(\ref{ineq:distance-between-omega-and-id})
we have for $(s,t) \in I_+^2$
\[
   |\omega_{s,t}(z)-\omega_{t_0,t_0}(z)|
 =
   |\omega_{s,t}(z)-z| 
 \leq
  (1-a_{s,t})\frac{|z|(1+|z|)}{1-|z|} , 
\]
it is clear $\omega_{s,t}(z) \rightarrow \omega_{t_0,t_0}(z)=z$
locally uniformly on $\mathbb{D}$
as $(s,t) \rightarrow (t_0,t_0)$ in $I_+^2$.

Finally we consider the case that $s_0 < t_0$.
Here we introduce the notation which is used throughout the article;
\[
  x \vee y = \max \{ x, y \}
 \quad 
 \text{and} 
 \quad
 x \wedge y = \min \{ x, y \}
 \quad
 \text{for} \; x, y \in \mathbb{R} .
\] 
By making use of the inequality 
for $\omega \in \mathfrak{B}$ and $|z_0|,|z_1| \leq r < 1$
\begin{equation}
  |\omega(z_1) - \omega (z_0)| 
 \leq |z_1-z_0| \int_0^1 |\omega'((1-t)z_0+ t z_1)| \, dt 
 \leq
 \frac{|z_1-z_0|}{1-r^2} 
\end{equation}
and (\ref{ineq:omega_isLipschitz_wrt_a}) we have 
for $(s,t) \in I_+^2$ with $s<t_0$ and $s_0<t$
and $|z| \leq r$
\begin{align*}
  &
 |\omega_{s,t}(z) - \omega_{s_0,t_0} (z)|
\\
 \leq &
 |\omega_{s,t}(z) - \omega_{s_0,t} (z)|
+ |\omega_{s_0,t}(z) - \omega_{s_0,t_0} (z)|
\\
 = &
 |\omega_{s\wedge s_0,t}(z) - \omega_{s \vee s_0,t} (z)|
+ |\omega_{s_0,t\vee t_0}(z) - \omega_{s_0,t \wedge t_0} (z)|
\\
 = &
 |\omega_{s \vee s_0, t}(\omega_{s\wedge s_0,s \vee s_0}(z)) - \omega_{s \vee s_0,t} (z)|
+
 |\omega_{t \wedge t_0,t \vee t_0} (\omega_{s_0,t \wedge t_0}(z)) 
 - \omega_{s_0,t \wedge t_0} (z)|
\\
 \leq &
  \frac{|\omega_{s\wedge s_0,s \vee s_0}(z) - z|}{1-r^2}
 +
 (1-a_{t \wedge t_0,t \vee t_0}) 
 \frac{|z|(1+|z|)}
 {1-a_{s_0,t \wedge t_0}|z|}
\\
 \leq &
  (1-a_{s\wedge s_0,s \vee s_0}) \frac{|z|(1+|z|)}{(1-r^2)(1-|z|)}
 +
 (1-a_{t \wedge t_0,t \vee t_0}) 
 \frac{|z|(1+|z|)}{(1-|z|)} .
\end{align*}
This implies 
$\omega_{s,t}(z) \rightarrow \omega_{s_0,t_0}(z)$ 
locally uniformly in $\mathbb{D}$
as $(s,t) \rightarrow (s_0,t_0)$ in $I_+^2$.
\end{proof}

By slightly generalizing 
the original proof in \cite{Pommerenke1965},
we show that 
if $\{ \omega_{s,t} \}_{(s,t) \in I_+^2}$ is continuous, then
each $\omega_{s,t}$ is univalent in $\mathbb{D}$.
We require the following lemma due to Landau and Dieudonn\'{e}.
We include a proof here for the reader's convenience.
For alternative proofs
see \cite[Theorem 10.1]{Heins} and \cite[Theorem VI.10]{Tsuji}.

\begin{lemma}[Landau-Dieudonn\'{e}]
\label{lemma:Landau}
Let $\omega \in \mathfrak{B}$ 
with $\omega'(0)= \alpha \in (0,1)$.
Then 
$\omega$ is univalent in $\mathbb{D}(0, \rho)$,
where $\rho = \rho( \alpha ) = \alpha /(1+ \sqrt{1-\alpha^2})$.
Furthermore 
$\mathbb{D}(0, \rho^2) \subset \omega (\mathbb{D}(0, \rho))$.
\end{lemma}
Note that $\lim_{\alpha \nearrow 1} \rho( \alpha ) = 1$.
\begin{proof}
Let $g(z) = \frac{\omega(z)}{z}$. 
Then  $g$ is analytic in $\mathbb{D}$ with $g(0) = \alpha \in (0,1)$
and satisfies $|g(z)| \leq 1$ in $\mathbb{D}$.
We have by (\ref{ineq:omega_growth}) that 
for $|z| < \alpha$
\begin{align*}
  \left| \frac{zg'(z)}{g(z)} \right|
 \leq
  \frac{|z|(1-|g(z)|^2) } {(1-|z|^2)|g(z)|} 
  \leq
  \frac{(1-\alpha^2)|z|}{(\alpha-|z|)(1-\alpha |z|)} .
\end{align*}
It follows from this that
$\left| \frac{zg'(z)}{g(z)} \right| < 1$ 
for $|z| < \rho (\alpha ) = \frac{\alpha}{1+\sqrt{1-\alpha^2}}$
Note that $\rho (\alpha ) <\alpha)$.
Since
$\frac{z\omega'(z)}{\omega(z)} = 1 + \frac{zg'(z)}{g(z)}$,
we have
\[
 \Real \left\{ \frac{z\omega'(z)}{\omega(z)} \right\}
 = 1 + \Real \left\{ \frac{z g'(z)}{g(z)} \right\}
 \geq
 1- \left| \frac{z g '(z)}{g(z)} \right|
 > 0 \quad
 \text{for} \quad |z| < \rho (\alpha ) .
\]
This implies that $\omega$ is starlike univalent 
in $\mathbb{D}(0, \rho (\alpha))$.

In view of the inequality~(\ref{ineq:omega_growth}),  
let us define the function 
\[
  k(x) = x \frac{\alpha - x}{1 - \alpha x}, 
  \quad 0 \leq x \leq \alpha .
\]  
It is easy to verify that 
\[
  \max_{0 \leq x \leq \alpha} k(x) 
 = k(\rho (\alpha)) 
 = \rho (\alpha)^2 .
\]
Combining this with the fact that $\omega$ is starlike univalent 
in $\mathbb{D}(0, \rho (\alpha))$,
it follows that $\omega ( \mathbb{D}(0, \rho (\alpha)))$ 
contains the disk $\mathbb{D}(0,\rho(\alpha)^2)$.
\end{proof}

\begin{theorem} 
\label{thm:omega-is-univalent-if-it-is-continuous}
Let $\{ \omega_{s,t} \}_{(s,t) \in I_+^2}$ be a transition family.
If $\{ \omega_{s,t} \}_{(s,t) \in I_+^2}$ is continuous,
then
each $\omega_{s,t} $ is univalent in 
$\mathbb{D}$ for every $(s,t) \in I_+^2$.
\end{theorem}
\begin{proof}
Fix $(s_0,t_0) \in I_+^2$ and consider $\omega_{s_0,t_0}$.
We may assume $s_0< t_0$, since the univalence is trivial
when $s_0=t_0$. 

For $r \in (0,1)$ take $\alpha \in (0,1)$ with $\rho( \alpha ) > r$.
Since $a_{s,t}$ is continuous on $I_+^2$ and $a_{t,t} = 1$ for $t \in I$,
for each $s \in [s_0,t_0]$ 
there exists an interval $I(s)$ containing  $s$,
which is open in the subspace topology of $[s_0,t_0]$
and  satisfies $a_{u,v} > \alpha$ for all $(u,v) \in I(s)_+^2$.
Consider the open covering 
\[
  [s_0,t_0] 
 \subset 
 \bigcup_{s_0 \leq s \leq t_0} I(s) .
\]
Since $[s_0,t_0]$ is a compact metric space, 
by the Lebesgue number theorem
we can find $\delta > 0$ such that
any subset of $[s_0,t_0]$ with diameter less than $\delta $
is contained in some $I(s)$.
Choose a partition
$s_0 < s_1< \cdots < s_n = t_0$ 
such that $\max_{1 \leq k \leq n} s_k-s_{k-1} < \delta$.
Then 
$a_{s_{k-1},s_k} \geq \alpha$, $k=1, \ldots , n$.
Hence by the Landau-Dieudonn\'{e} lemma 
each $\omega_{s_{k-1},s_k}$ is univalent in $\mathbb{D}(0,r)$.
Since each 
$\omega_{s_{k-1},s_k}$
maps $\mathbb{D}(0,r))$ into itself,
it follows that the composition
\[
 \omega_{s_0,t_0} 
 = \omega_{s_{n-1},s_n} \circ \cdots 
  \circ \omega_{s_1,s_2} \circ \omega_{s_0,s_1}
\]
is also univalent in $\mathbb{D}(0,r)$.
Since $r \in (0,1)$ was chosen arbitrarily, we conclude that 
$\omega_{s_0,t_0}$ is univalent in $\mathbb{D}$.
\end{proof}

\section{Continuity of Loewner Chains}
For a Loewner chain $\{ f_t \}_{t \in I}$,
it is easy to see that
$\{ f_t \}_{t \in I}$ is continuous
if and only if the function $f(z,t):= f_t(z)$ is continuous
in both variables $z\in \mathbb{D}$ and $t \in I$.
\begin{theorem}
\label{thm:omega-is-univalent-if-ft-is-continuous}
Let $\{ f_t \}_{t \in I} \subset \mathcal{H}_0 (\mathbb{D})$ 
be a Loewner chain 
with $a(t) = f_t'(0)$, $t \in I$.
Then $\{ f_t \}_{t \in I}$
is continuous 
if and only if 
the function $a(t)$ is continuous on $I$.
Furthermore, in this case, the associated transition family
$\{ \omega_{s,t}\}_{(s,t) \in I_+^2}$ is also continuous,
and each $\omega_{s,t}$ is univalent in $\mathbb{D}$
for all $(s,t) \in I_+^2$.
\end{theorem}
\begin{proof}
The latter statement follows directly from
Theorem \ref{theorem:continuity-of-transition-family} and
\ref{thm:omega-is-univalent-if-it-is-continuous}.

Take $r \in (0,1)$ and 
consider 
\[
    |a(t) -a(t_0)| 
   = |f_t'(0) -f_{t_0}'(0) | 
    \leq 
    \frac{1}{2 \pi} 
    \int_{|z|=r} 
    \frac{| f_t(z) - f_{t_0}(z)|}{|z|^2} \, |dz| .
\]
If
$\{ f_t \}_{t \in I}$
is continuous at $t_0 \in I$, 
then $f_t (z) \rightarrow f_{t_0} (z)$ 
uniformly on $\partial \mathbb{D}(0,r)$ as 
$t \rightarrow t_0$ in $I$.
Hence $a(t) \rightarrow a(t_0)$ and 
$a(t)$ is continuous at $t_0 \in I$.

To prove the converse 
let $t_0 \in I$.
If $t_0 < \sup I$, choose $t^* \in I$ with $t_0 < t^*$.  
If $t_0 = \sup I$, let $t^* = t_0$.  
For each fixed $r \in (0,1)$, it suffices to show 
that $|f_{t_2}(z) -f_{t_1}(z)| \rightarrow 0$ uniformly 
on $\overline{\mathbb{D}}(0,r)$
as $t_2 - t_1 \searrow 0$ with 
$t_1 \leq t_0 \leq t_2 \leq t^*$.
Let 
\[
  M (r) = \max_{|z| \leq r } |f_{t^*}(z)|, \quad 0 \leq r<1.
\]
Then for any $t \in I$ with $t \leq t^*$,
using $|\omega_{t,t^*}(z)| \leq |z|$, we obtain
\begin{align*}
  \max_{|z|\leq 2^{-1} (1+r)}|f_t(z)| 
= & \, \max_{|z|\leq 2^{-1} (1+r)}|f_{t^*}(\omega_{t,t^*}(z))| 
\\
\leq & \, \max_{|z|\leq 2^{-1} (1+r)}|f_{t^*}(z)| 
 = M \left(2^{-1}(1+r) \right) .
\end{align*}
Therefore, for $|z| \leq r$, we have 
\[
 |f_t'(z)| 
 \leq
    \frac{1}{2 \pi} 
    \int_{|\zeta|= 2^{-1}(1+r)} 
    \frac{| f_t(\zeta)|}{|\zeta - z|^2} \, |d \zeta|
  \leq
  2M \left(2^{-1}(1+r) \right) \frac{1+r}{(1-r)^2} .
\]
Using inequality (\ref{ineq:distance-between-omega-and-id}), 
we get
\begin{align}
\label{ineq:Lipshctz-cont-f_t}
 |f_{t_2}(z) - f_{t_1}(z) |
 &=
 |f_{t_2}(z) - f_{t_2}(\omega_{t_1,t_2}(z)) |
\\
 &=
 \left| \int_{\omega_{t_1,t_2}(z)}^z f_{t_2}'(\zeta ) \, d \zeta \right|
\nonumber \\
 &\leq 2 M \left(2^{-1}(1+r) \right) 
 \frac{1+r}{(1-r)^2} |z-\omega_{t_1,t_2}(z)|
\nonumber \\
 &\leq
\frac{ 2 (a(t_2)-a(t_1)) M \left(2^{-1}(1+r) \right) 
 r(1+r)^2}{a(t_2)(1-r)^3}
 \; \rightarrow \; 0 
\nonumber 
\end{align}
as $t_2 - t_1 \searrow 0$
with 
$t_1 \leq t_0 \leq t_2 \leq t^*$.
\end{proof}

For later use, we provide the following estimate.
\begin{lemma}
\label{lemma:Lipshictz_estimte_for_derivative_of_f}
Let $\{ f_t \}_{t \in I} \subset \mathcal{H}_0 (\mathbb{D})$ 
be a Loewner chain 
and define $a(t) = f_t'(0)$ for $t \in I$.
Let $t^* \in I$, and for each $r \in [0,1)$, define
$M(r) = \max_{|z|\leq r} |f_{t^*}(z)|$.
Then, for $t_1,t_2 \in I$ with $t_1 \leq t_2 \leq t^*$
and $|z| \leq r < 1$,
we have
\[
 \left| f_{t_2}^{(m)}(z) - f_{t_1}^{(m)}(z)\right|
 \leq
 \frac{m!2^{m+7} M((3+r)/4) }{(1-r)^{m+4}}
 \frac{a(t_2)-a(t_1)}{a(t_2)} .
\] 
\end{lemma}
\begin{proof}
Define an analytic function $g$ on $\mathbb{D}$ by
$g(z) = f_{t_2}(z) - f_{t_1}(z)$.
Applying inequality (\ref{ineq:Lipshctz-cont-f_t}) 
with $|z| \leq 2^{-1}(1+r)$
we obtain
\[
   |g(z)|
 \leq
 \frac{2^6 M((3+r)/4)}{(1-r)^3} \frac{a(t_2)-a(t_1)}{a(t_2)},
 \quad |z| \leq \frac{1+r}{2} .
\]
By the Cauchy integral formula, it follows that
for $|z| \leq r$
\begin{align*}
  \left| f_{t_2}^{(m)}(z) - f_{t_1}^{(m)}(z)\right|
 = & \, |g^{(m)}(z)| 
\\
 = & \, \left|  
 \frac{m!}{2 \pi i } 
 \int_{|\zeta|= 2^{-1}(1+r)}
 \frac{g(\zeta)}{(\zeta - z)^{m+1}}
 \, d \zeta \right|
\\
  \leq & \, 
 m! \frac{1+r}{2} \frac{1}{\left( \frac{1+r}{2}-r \right)^{m+1}}
 \frac{2^6 M((3+r)/4)}{(1-r)^3} \frac{a(t_2)-a(t_1)}{a(t_2)}
\\
  < & \, 
 \frac{m! 2^{m+7} M((3+r)/4)}{(1-r)^{m+4}} \frac{a(t_2)-a(t_1)}{a(t_2)}.
\end{align*}
\end{proof}

Let $I \subset [-\infty , \infty ]$ be a right-open interval 
with $\beta = \sup I \not\in I $ 
and $\{ \omega_{s,t} \}_{(s,t) \in I_+^2} $ 
be a transition family on $I$.
We will show that, 
for any fixed $s \in I$,
the limit 
\[
  \omega_{s, \beta} 
 = \lim_{t \nearrow \beta} \omega_{s,t}
\]
exists in $\mathcal{H}(\mathbb{D})$.
Following the argument of Pommerenke \cite{Pommerenke1960},
we will give a necessary and sufficient condition
under which the extended family 
$\{ \omega_{s,t} \}_{(s,t) \in (I\cup \{ \beta \})_+^2}$
forms a transition family.
As an application,
we derive the decomposition theorem for Loewner chains.
To this end, We require a lemma concerning 
the  inverse of a univalent function 
$f \in \mathcal{H}_0(\mathbb{D})$,
as well as the Vitali-Porter convergence theorem.
For proofs and further details see 
\cite[Corollary 7.5]{Burckel}
or 
\cite[Chap. 7]{Remmert}.

\begin{lemma}
\label{lemma:estimate-for-inverse-of-univalent-functions}
Let $h:\mathbb{D} \rightarrow \mathbb{C}$ be 
a univalent analytic function
with $h(0) =0$ and $h'(0)=a > 0$.
Then
$\mathbb{D}(0, a/4 ) \subset h(\mathbb{D})$
and 
\[
    |ah^{-1}(w) - w| 
 \leq
 \frac{16|w|^2}{a- 4|w|}, \quad
 |w| < \frac{a}{4} .
\]
\end{lemma}
\begin{proof}
By the Koebe one-quarter theorem 
we have $\mathbb{D}(0,a/4) \subset f(\mathbb{D})$.
Let 
\[
  h^{-1}(w) = a^{-1}w + \sum_{n=2}^\infty b_nw^n, 
 \quad |w| < a/4
\].
Then, for $R< a/4$,
\[
   |b_n| 
   = 
  \left| 
  \frac{1}{2\pi i} 
  \int _{|w|=R} \frac{f^{-1}(w)}{w^{n+1}} \, dw 
  \right|
 \leq \frac{1}{R^n}, \quad n \geq 2.
\]
Thus 
\[
   |ah^{-1}(w)- w|
 \leq
 \sum_{n=2}^\infty a (|w|/R)^n 
 = \frac{a|w|^2}{R(R-|w|)} , \quad |w| < \frac{a}{4} .
\]
Letting $R \nearrow a/4$ 
we obtain the desired inequality.
\end{proof}

\begin{lemma}[the Vitali-Porter convergence theorem I]
\label{lemma:Vitali_for_sequnece}
Let $\{ g_n \}_{n=1}^\infty$ be a  sequence 
in $\mathcal{H}( \mathbb{D})$ 
that is locally uniformly bounded in $\mathbb{D}$. 
Suppose that $\lim_{n \rightarrow \infty} g_n(z)$ exists
for all $z$ in a subset $A$ of $\mathbb{D}$  
which has at least an accumulation point in $\mathbb{D}$.
Then $\{ g_n \}_{n=1}^\infty$ converges 
locally uniformly on $\mathbb{D}$ to an analytic function.
\end{lemma}

In the next chapter, we use the following form of 
the Vitali-Porter convergence theorem.  
We provide a proof for completeness, 
although it is essentially 
the same as that of the version stated above.

\begin{lemma}[the Vitali-Porter convergence theorem II]
\label{lemma:Vitali_for_family}
Let $\Lambda$ be a metric space with a distance function $d$,
and $\{ g_\lambda \}_{\lambda \in \Lambda}$ 
be a  family of analytic functions on a domain $D \subset \mathbb{C}$,
indexed by $\Lambda$.
Let $\lambda_0 \in \Lambda$, and let $A$ be a subset of $D$
that has at least one accumulation point in $D$.
Suppose that $\{ g_\lambda \}_{\lambda \in \Lambda}$ 
is locally uniformly bounded in $D$, and that
$\lim_{\lambda \to \lambda_0}g_\lambda (z)$ 
exists for every $z \in A$.
Then there exists an analytic function 
$g$ on $D$ such that
$g_\lambda \to g$ locally uniformly on $D$ 
as $\lambda \to \lambda_0$.
\end{lemma}
\begin{proof}
Suppose, for the sake of contradiction, 
that the conclusion does not hold. 
Then there exists $\varepsilon >0$, 
a compact set $K \subset D$,
sequences $\{ \lambda_j^{(1)} \}_{j=1}^\infty$, 
$\{ \lambda_j^{(2)} \}_{j=1}^\infty \subset \Lambda$,
and $\{ z_j \}_{j=1}^\infty \subset K$ sucht that
\begin{align}
\label{eq:lambda_j_tends_to_lambda_0}
   & d( \lambda_j^{(1)}, \lambda_0 ) \to 0, \quad 
  d( \lambda_j^{(2)}, \lambda_0 ) \to 0 \quad \text{as } j \to \infty, 
\\
\label{eq:modulus_is_geq_epsilon}
  & \left| 
 g_{\lambda_j^{(1)}} (z_j) - g_{\lambda_j^{(2)}} (z_j) 
  \right| 
 \geq \varepsilon .
\end{align}
Since $K$ is compact, we may assume 
(by passing to a subsequence if necessary)
that $z_j \to z_0$ as $j \to \infty$
for some $z_0 \in K$.
As the family $\{ g_\lambda \}_{\lambda \in \Lambda}$ 
is locally uniformly bounded on $D$, 
both sequences $\{ g_{\lambda_j^{(1)}} \}_{j=1}^\infty$ and
$\{ g_{\lambda_j^{(2)}} \}_{j=1}^\infty$ are normal families.
Hence, by Montel's theorem, 
we may also assume that 
$g_{\lambda_j^{(1)}} \to g_1$ and  
$g_{\lambda_j^{(2)}} \to g_2$ locally uniformly on $D$ 
as $j \to \infty$,
 for some analytic functions $g_1$ and $g_2$ in $D$.
Since 
$\lim_{\lambda \to \lambda_0} g_\lambda (z)$ 
exists for every $z \in A$
and 
the convergence in 
(\ref{eq:lambda_j_tends_to_lambda_0}) holds,
it follows that for every $z \in A$,
\[
   g_1(z) 
 = \lim_{j \to \infty } g_{\lambda_j^{(1)}} (z)
 = \lim_{\lambda \to \lambda_0} g_\lambda (z)
 = \lim_{j \to \infty } g_{\lambda_j^{(2)}} (z)
 = g_2(z) .
\]
Therefore, by the identity theorem for analytic functions,
we conlude that $g_1=g_2$ on $D$.
However, taking the limit 
in (\ref{eq:modulus_is_geq_epsilon})
as $j \to \infty$, 
we obtain $|g_1 (z_0)-g_2(z_0)| \geq \varepsilon$,
which contradicts the fact $g_=g_2$.
This completes the proof.
\end{proof}

\section{Extension of Transition Families}
We are now in a position to describe 
the extendability of a transition family to its right endpoint, 
based on the limiting behavior of the associated 
derivative functions.

\begin{theorem}
\label{thm:extendability-of-transiton-family}
Let $I \subset [-\infty , \infty )$ be a right-open interval 
with $\beta = \sup I \not\in I $, 
and let $\{ \omega_{s,t} \}_{(s,t) \in I_+^2}$ 
be a transition family on $I$ with 
$a_{s,t} = \omega_{s,t}'(0)$, $(s,t) \in I_+^2$.
Define 
$a_{s, \beta} = \lim_{t \nearrow \beta} a_{s,t} \in [0,1]$ 
for $s \in I$.
Then, for each $s \in I$,
the locally uniform limit 
\[
  \omega_{s, \beta } = \lim_{t \nearrow \beta} \omega_{s,t}
\]
exists on $\mathbb{D}$,
and the following statements hold:
\begin{itemize}
 \item[{\rm (i)}]
If $a_{s_0,\beta} > 0$ 
for some $s_0 \in I$,
then $a_{s,\beta} >0$ 
and 
$\omega_{s,\beta } \in \mathfrak{B}$ for all $s \in I$,
and the extended family 
$\{ \omega_{s,t} \}_{(s,t) \in (I\cup \{ \beta\})_+^2 }$ 
forms a transition family on 
$I \cup \{  \beta \}$.
Here, we set $\omega_{\beta,\beta} = \id_{\mathbb{D}}$ 
and $a_{\beta , \beta} = 1$. 
 \item[{\rm (ii)}]
If $a_{s_0,\beta} = 0$ for some $s_0 \in I$,  
then $a_{s,\beta} = 0$ and $\omega_{s,\beta} = 0$ 
for all $s \in I$.
Furthermore, if $\{ \omega_{s,t} \}_{(s,t) \in I_+^2} $
is continuous, then for any fixed $t_0 \in I$ and $c>0$, 
the locally uniform limit 
\[
  g_t 
  = 
  \lim_{\tau \nearrow  \beta} 
  \frac{c}{a_{t_0,\tau }} \omega_{t,\tau}
\]
exists
and is univalent on $\mathbb{D}$ for all $t \in I$.
The family $\{ g_t \}_{t \in I}$ forms a continuous Loewner chain of
univalent functions, 
has  $\{ \omega_{s,t} \}_{(s,t) \in I_+^2 }$
as its associated transition family, and satisfies
$g_t'(0) \rightarrow \infty$ as $t \nearrow \beta$.
\end{itemize}
\end{theorem}
By a similar argument, one can prove an analogue of the above theorem 
for transition families defined on \textit{left}-open intervals 
$I$ with $\alpha := \inf I \not\in I$.
We omit the statement for brevity and to avoid unnecessary complications. 

\begin{proof}[Proof of Theorem \ref{thm:extendability-of-transiton-family}]
(i) Suppose $a_{s_0,\beta} > 0$
for some $s_0 \in I$.
Then, by (\ref{eq:stu-transition-derivative}),
we have 
$a_{s,\beta} = \lim_{t \nearrow \beta} a_{s,t} > 0$
for all $s \in I$.

Fix $s \in I$ arbitrarily. 
Since the family 
$\{ \omega_{s,t} \}_{s\leq t \in I}$ 
is uniformly bounded  on $\mathbb{D}$,
there exists a sequence $\{t_n \}_{n=1}^\infty \subset I$
with $s < t_1 <t_2 < \cdots < t_n \nearrow \beta $ 
such that
$\omega_{s,t_n} \rightarrow \varphi$ locally uniformly on $\mathbb{D}$ 
for some analytic function $\varphi$.

We show that $\omega_{s,t} \rightarrow \varphi$ 
locally uniformly on $\mathbb{D}$ as
$t \nearrow \beta$.
Fix $r \in (0,1)$ and $\varepsilon > 0$.
Choose $N \in \mathbb{N}$
such that 
for all $n \geq N$ and $|z| \leq r$
\[
   | \omega_{s,t_n}(z) - \varphi (z)| < \frac{\varepsilon}{2}
\]
and 
\[
   1 - \frac{\varepsilon (1-r)}{2r(1+r)} < a_{t,u} 
 = \frac{a_{s,u}}{a_{s,t}}
\]
for $t_N \leq t \leq u < \beta$.
Then, for $t \in (t_N , \beta )$ and $|z| \leq r$, 
we have, 
by (\ref{ineq:distance-between-omega-and-id}) and 
the inequality $|\omega_{s,t_N}(z)| \leq |z|$,
\begin{align*}
   |\omega_{s,t}(z) - \varphi (z)|
 \leq &
 |\omega_{t_N,t}(\omega_{s,t_N}(z)) - \omega_{s,t_N}(z) | 
 + |\omega_{s,t_N}(z) - \varphi (z)|
\\
 \leq &
 (1- a_{t_N,t}) 
 \frac{|\omega_{s,t_N}(z)|(1+|\omega_{s,t_N}(z)|)}
 {1-a_{t_N,t}|\omega_{s,t_N}(z)|}
 + |\omega_{s,t_N}(z) - \varphi (z)|
\\
 \leq &
 (1- a_{t_N,t}) 
 \frac{r(1+r)}{1-r}
 + \frac{\varepsilon}{2} < \varepsilon .
\end{align*}
Thus $\omega_{s,t} \rightarrow \varphi $ 
locally uniformly on $\mathbb{D}$ as
$t \nearrow \beta$. 
We henceforth denote the limit $\varphi$ by $\omega_{s,\beta}$.

We now verify 
that $\omega_{s, \beta} \in \mathfrak{B}$.
Indeed, since
$|\omega_{s,\beta}(z)| = \lim_{t \nearrow \beta}|\omega_{s,t}(z)| \leq 1$,
$\omega_{s,\beta}(0) = \lim_{t \nearrow \beta} \omega_{s,t}(0) =0$
and 
$\omega_{s, \beta}'(0) = \lim_{t \nearrow \beta} \omega_{s,t}'(0) 
= a_{s,\beta} > 0$,
it follows that $\omega_{s,\beta} \in \mathfrak{B}$.

Next taking the limit $u \nearrow \beta$ 
in the identity 
\[
   \omega_{s,u}(z) = \omega_{t,u} (\omega_{s,t}(z)) \quad
  \text{ for $z \in \mathbb{D} $ and $s \leq t \leq u < \beta$}, 
\]
we obtain
\[
  \omega_{s,\beta}(z) = \omega_{t,\beta} (\omega_{s,t}(z)) 
 \quad \text{ for $z \in \mathbb{D} $ and $s \leq t < \beta$}.
\]
Hence, the extended family 
$\{ \omega_{s,t} \}_{(s,t) \in (I\cup \{ \beta \})_+^2 }$ 
is a transition family on $I \cup \{  \beta \}$.

(ii) Now suppose $a_{s_0,\beta} = \lim_{t \nearrow \beta} a_{s_0,t} = 0$
for some $s_0 \in I$.
Then by (\ref{eq:stu-transition-derivative})
we have 
$a_{s,\beta} = \lim_{t \nearrow \beta} a_{s,t} = 0$
for all $s \in I$.

Fix $s \in I$. We show that $\omega_{s,t} \rightarrow 0$ 
locally uniformly on $\mathbb{D}$ as $t \nearrow \beta$.
Let $r \in (0,1)$ and 
$\rho$ be the unique solution to 
\[
   \frac{r+\rho}{1+\rho r} = \frac{1+r}{2},
\]
which gives 
\[
  0 < \rho = \frac{1-r}{2-r(1+r)} < \frac{1}{2}. 
\]
Choose a sequence $\{ t_n \}_{n=0}^\infty \subset I$
with $s = t_0  <t_1<t_2 < \cdots < t_n < \cdots$
such that
\[
  a_{t_{n-1},t_n} \leq \rho, \quad n \in \mathbb{N}.
\]
Note that $t_n \nearrow \beta$ as $n \rightarrow \infty$.
Indeed, if $\beta_0 := \sup_{n \in \mathbb{N}} t_n < \beta$,
then for all $n \in \mathbb{N}$, we have
\[
  0 
 < a_{s,\beta_0} 
 \leq a_{t_0,t_n}
 = a_{t_{n-1},t_n} \cdots a_{t_0,t_1}
 \leq \rho^n ,
\]
which is a contradiction.
From (\ref{ineq:omega_growth}), we get for $|z| < r$
and $1 \leq k \leq n$:
\[
  |\omega_{t_{k-1},t_k}(z)| 
 \leq 
 |z| \frac{|z|+a_{t_{k-1},t_k}}{1+a_{t_{k-1},t_k}|z|}
 \leq
 |z| \frac{r + \rho}{1+\rho r}
 = |z| \frac{1+r}{2}, 
 \quad |z| \leq r .
\]
Therefore, by repeated application,
we obtain for $|z| \leq r$ and $t \geq t_n $ 
\begin{align*}
 |\omega_{s,t}(z)|
 =
 |\omega_{t_n,t}(\omega_{t_0,t_n}(z))|
 \leq & 
 |\omega_{t_0,t_n}(z)|
\\
 \leq &
 |\omega_{t_{n-1},t_n}(\omega_{t_0,t_{n-1}}(z))|
\\
 \leq &
 \frac{1+r}{2} |\omega_{t_0,t_{n-1}}(z)|
\\
 \leq & \left( \frac{1+r}{2} \right)^n |\omega_{t_0,t_0}(z)| 
 = \left( \frac{1+r}{2} \right)^n |z| .
\end{align*}
This shows that
$\omega_{s,t} (z) \rightarrow  0$ locally uniformly on $\mathbb{D}$
as $t \nearrow \beta$.


We now consider the case where 
$\{\omega_{s,t} \}_{(s,t) \in I_+^2}$ is continuous.
In this case, by Theorem 
\ref{thm:omega-is-univalent-if-it-is-continuous},
each $\omega_{s,t}$ is univalent in $\mathbb{D}$ 
for all $(s,t) \in I_+^2$.

Fix $t_0 \in I$ and $c > 0$ be fixed, and
define $a(t)$ by (\ref{eq:defi-of-at}).
Note that 
\[
  \omega_{s,t}'(0) = a_{s,t} = \frac{a(s)}{a(t)}
  \quad \text{and} \quad  \lim_{t \nearrow \beta} a(t) 
  =  \lim_{t \nearrow \beta} \frac{c}{a_{t_0,t}} = \infty .
\]
Fix $\tau \in I$ and consider the family 
$\{ a(t) \omega_{\tau,t}\}_{\tau \leq t < \beta}$.
By the growth theorem for univalent functions,
we have
\[
   a(t) |\omega_{\tau,t}(z)|
 \leq
  \frac{a(\tau )|z|}{(1-|z|)^2}, \quad 
 z \in \mathbb{D}.
\]
This implies that the family 
$\{ a(t) \omega_{\tau ,t}\}_{\tau  \leq t < \beta}$ 
is locally uniformly bounded, 
and hence forms a normal family.
Thus there exists a sequence 
$\{ \tau_n \}_{n=1}^\infty \subset I$ such that
$\tau < \tau_1< \cdots < \tau_n \nearrow \beta$ and 
a locally uniform limit 
\begin{equation}
\label{eq:convergence_of_g_tau}
   g_\tau (z) 
 := 
 \lim_{n \rightarrow \infty} a(\tau_n) \omega_{\tau ,\tau_n}(z) , 
 \quad 
 z \in \mathbb{D} 
\end{equation}
exists.

For $t \in I$ with 
$t < \tau$, we have
\[
  a(\tau_n) \omega_{t,\tau_n}(z)
  = a(\tau_n) \omega_{\tau,\tau_n} (\omega_{t,\tau}(z))
 \rightarrow g_\tau (\omega_{t,\tau}(z)), \quad
 n \rightarrow \infty . 
\]
Therefore, for $z \in \mathbb{D}$,
the limit 
$g_t(z) :=  \lim_{n \rightarrow \infty} a(\tau_n) \omega_{s,\tau_n}(z)$
exists and satisfies
$g_t(z) = g_\tau ( \omega_{t,\tau}(z))$.
Since the convergence in 
(\ref{eq:convergence_of_g_tau})
is locally uniform on $\mathbb{D}$
and $|\omega_{t,\tau}(z)| \leq |z|$,
it follows that the convergence
of the limit  
$g_t =  \lim_{n \rightarrow \infty} a(\tau_n) \omega_{s,\tau_n}$
is also locally uniform on $\mathbb{D}$.

Let $t \in I$ with $t  > \tau $.
Choose $n_0 \in \mathbb{N}$ such that $\tau_{n_0} > t$.  
Then we also have
\[
  g_\tau (z) = 
 \lim_{n_0 \leq n \rightarrow \infty} 
  a(\tau_n) \omega_{\tau ,\tau_n}(z)
 =
 \lim_{n_0 \leq n \rightarrow \infty} 
  a(\tau_n) \omega_{t,\tau_n} (\omega_{\tau ,t}(z)) .
\]
Hence, for each $\zeta \in \omega_{\tau ,t}(\mathbb{D})$,
the limit 
$\lim_{n_0 \leq n \rightarrow \infty } 
a(\tau_n) \omega_{t,\tau_n} ( \zeta )$
exists.
Since $\omega_{\tau ,t}(\mathbb{D})$ is 
a nonempty subdomain of $\mathbb{D}$,
and the family $\{ a(\tau_n) \omega_{t,\tau_n} \}_{n \geq n_0}$ 
is locally uniformly bounded on $\mathbb{D}$,
the Vitali-Porter convergence theorem implies that
the locally uniform limit 
$g_t(z) :=  \lim_{n_0 \leq n \rightarrow \infty} 
a(\tau_n) \omega_{t,\tau_n}(z)$
exists on $\mathbb{D}$.

At this point, it is straightforward to verify that  
the family $\{ g_t \}_{t \in I}$ forms a Loewner chain  
that shares the same associated transition family 
as $\{ f_t \}_{t \in I}$.
To see this, let $(s,t) \in I_+^2$.
Taking the limit as $n \rightarrow \infty $ 
in the identity
\[
   a(t_n) \omega_{s,t_n}(z) = a(t_n) \omega_{t,t_n}(\omega_{s,t}(z))
 \quad \text{for} \quad  t_n > t ,
\]
we obtain $g_s = g_t \circ \omega_{s,t}$, as required.

Since 
$g_t '(0) =  a( t ) > 0$
and $a(\tau_n)\omega_{t,\tau_n}$ is univalent in $\mathbb{D}$, 
Hurwitz's theorem implies that $g_t$ is univalent on $\mathbb{D}$. 
In particular,  
since $a(\tau) \rightarrow \infty$ as $\tau \nearrow \beta$,
it follows from 
Lemma~\ref{lemma:estimate-for-inverse-of-univalent-functions}
that
\[
  a(\tau ) \omega_{t,\tau} (z) = a(\tau) g_\tau^{-1} (g_t(z)) 
 \rightarrow g_t(z)
\]
locally uniformly on $\mathbb{D}$ as $\tau \nearrow \beta$.
\end{proof}

\section{Decomposition Theorem for Loewner Chains}
We now present the proof of the decomposition theorem for Loewner chains,  
as stated in Chapter~\ref{chapter:introduction}.

\begin{proof}[Proof of Theorem~\ref{theorem:FirstDecompositionTheorem}]
Let $\{ \omega_{s,t}\}_{(s,t) \in I_+^2}$ be the transition family 
associated with $\{ f_t \}_{t \in I}$.

(i) 
If the locally uniform limit 
$f_\beta = \lim_{t \nearrow \beta} f_t $ exists,
then it is clear that 
$a(\beta) = \lim_{t \nearrow \beta} f_t'(0) = f_\beta'(0) < \infty$.
Conversely if $a(\beta) < \infty$, then 
by Theorem \ref{thm:extendability-of-transiton-family}
$\{ \omega_{s,t}\}_{(s,t) \in I_+^2}$ has the extension
$\{ \omega_{s,t} \}_{(s,t) \in (I \cup \{ \beta \})_+^2}$ with
$\omega_{s,\beta} = \lim_{t \nearrow \beta} \omega_{s,t}$, $s \in I$.

Next, we aim to establish the existence of 
the locally uniform limit $f_\beta := \lim_{t \nearrow \beta} f_t$.
To this end, let us suppose for the moment 
that such a limit exists  
and satisfies the relation $f_t = f_\beta \circ \omega_{t,\beta}$.  
Then it would follow that 
$ f_\beta = f_t \circ \omega_{t,\beta}^{-1}$.
We now refine this heuristic idea into 
a rigorous argument as follows.

Since 
$\omega_{t, \beta} \in \mathfrak{B}$ 
and $\omega_{t, \beta}'(0) = a_{t,\beta} = \frac{a(t)}{a(\beta)}$,
by Lemma \ref{lemma:Landau}
the function $\omega_{t, \beta} $ is univalent in
$\mathbb{D}(0, \rho (a_{t,\beta} ))$
and 
$\mathbb{D}(0, \rho (a_{t,\beta} )^2) \subset 
\omega_{t, \beta} (\mathbb{D}(0, \rho (a_{t,\beta} )))$.
For each $t \in I$, 
let $\omega_{t, \beta}^{-1}$
denote the inverse 
of the restriction 
$\omega_{t, \beta}|_{\mathbb{D}(0, \rho (a_{t,\beta} )))}$
such that $\omega_{t, \beta}^{-1}$ is defined on 
on $\mathbb{D}(0, \rho (a_{t,\beta} )^2)$.
From 
the identity
$\omega_{s,\beta} = \omega_{t,\beta} \circ \omega_{s,t}$ 
for $(s,t) \in I_+^2$,  
it follows that 
$f_s \circ \omega_{s,\beta}^{-1}$ coincides 
with $f_t \circ \omega_{t,\beta}^{-1}$
on a neighborhood of the origin.
Hence by the identity theorem for analytic functions, 
$f_s \circ \omega_{s,\beta}^{-1}$ 
coincides with $f_t \circ \omega_{t,\beta}^{-1}$
on $\mathbb{D}(0,r(a_{s,\beta}))$.
Note that 
$a_{t,\beta}$ is nondecreasing in $t$ and 
$a_{t,\beta} \nearrow 1$ as $t \nearrow \beta$, 
and that 
$\rho (\alpha )$ is strictly increasing on $(0,1)$
and $\rho(\alpha ) \nearrow 1$ as $\alpha \nearrow 1$. 
Therefore there exists  a unique analytic function 
$f_\beta : \mathbb{D} \rightarrow \mathbb{C}$ 
such that for all $t \in I$
\[
  f_\beta = f_t \circ \omega_{t,\beta}^{-1} 
 \quad \text{on} \quad 
  \mathbb{D} (0, \rho(a_{t, \beta })^2).
\]
Thus, again by the identity theorem for analytic functions 
we have $f_t = f_\beta \circ \omega_{t, \beta}$ on $\mathbb{D}$.

Once the existence of the function $f_\beta$ 
has been established,  
the identity $f_t = f_\beta \circ \omega_{t,\beta}$,  
together with the fact that 
$\omega_{t,\beta} \rightarrow \id_{\mathbb{D}}$ 
locally uniformly on $\mathbb{D}$ as $t \nearrow \beta$,  
implies that  
$f_t \rightarrow  f_\beta$
locally uniformly on $\mathbb{D}$ as $t \nearrow \beta$.
It is clear that the extended family 
$\{ f_t \}_{t \in I \cup \{\beta \} }$
is also a Loewner chain
with the associated transition family  
$\{ \omega_{s,t} \}_{(s,t) \in (I \cup \{ \beta\})}$.

Let
\[
   g_t = a(\beta) \omega_{t, \beta} \text{ for } t \in I 
  \quad \text{and} \quad g_\beta  = a(\beta) \id_{\mathbb{D}} .
\]
Then it is easy to see that 
$\{ g_t \}_{t \in I \cup \{ \beta \}}$
forms a Loewner chain with 
$\{ \omega_{s,t} \}_{(s,t) \in (I \cup \{ \beta\})}$ 
as the associated transition family,
and that $g_t'(0)=f_t'(0)$ for $t \in I$.

Since 
$a_{t,\beta} \rightarrow 1$ as $t \nearrow \beta$,
it follows from Proposition~\ref{proposition:fundamental_inequalities}
that $g_t \rightarrow g_\beta$ 
locally uniformly on $\mathbb{D}$ as $t \nearrow \beta$.
Define
\[
   F(w) 
 = f_\beta \left( \frac{w}{a(\beta)} \right), 
   \quad |w| < a (\beta )
\]
Then we have 
$F \circ g_t = f_\beta \circ \omega_{t,\beta} = f_t$,
as required.

To prove the uniqueness, suppose there exists an analytic function 
$\tilde{F} : \mathbb{D}(0, a(\beta))\rightarrow \mathbb{C}$
and a Loewner chain $\{ \tilde{g}_t \}_{t \in I \cup \{ \beta \} }$ 
satisfying
$\tilde{F}(0) = \tilde{F}'(0)-1=0$, 
$\tilde{g}_\beta (\mathbb{D}) = \mathbb{D}(0, a(\beta))$
and $F \circ g_t = \tilde{F} \circ \tilde{g}_t$ 
for all $t \in I \cup \{ \beta \}$.
Then we have
\[
   \tilde{F} \circ \tilde{g}_t \circ \omega_{s,t} 
 = F \circ g_t \circ \omega_{s,t} = f_t \circ \omega_{s,t} = f_s
 = \tilde{F} \circ \tilde{g}_s .
\]
Since $\tilde{F} (0) =0$ and $\tilde{F}'(0) = 1$, 
it follows that
the Loewner chain 
$\{ \tilde{g}_t \}_{t \in I \cup \{ \beta \}}$ shares 
the same associated transition family
$\{ \omega_{s,t} \}_{(s,t) \in (I\cup\{\beta\})_+^2}$.
Since $\tilde{g}_\beta(\mathbb{D}) = \mathbb{D}(0, a(\beta))$,
$\tilde{g}_\beta(0)=0$ and 
$g_\beta'(0) = a(\beta)$,
the Schwarz implies that 
$\tilde{g}_\beta (z) = a(\beta) z = g_\beta (z)$.
Hence 
$\tilde{g}_t(z) = \tilde{g}_\beta (\omega_{t,\beta}(z)) 
= a(\beta )\omega_{t,\beta}(z) = g_t (z)$, $t \in I$,
and therefore $\tilde{F} = F$.

If $\{ f_t \}_{t \in I}$ is continuous,
then the function $a(t)$ is positive and 
continuous on $I \cup \{ \beta \}$,
so $a_{s,t} = a(s)/a(t)$ is continuous on  $(I \cup \{ \beta \})_+^2$.
By Theorem \ref{theorem:continuity-of-transition-family},
the transition family 
$\{ \omega_{s,t}\}_{(s,t) \in (I\cup \{ \beta \})_+^2}$ is continuous.
Hence, by 
Theorem \ref{thm:omega-is-univalent-if-it-is-continuous},
each $\omega_{t,\beta}$ and thus each 
$g_t = a( \beta ) \omega_{t,\beta}$,
is univalent on $\mathbb{D}$.

(ii) Suppose that 
$\{f_t \}_{t \in I}$ is continuous and $a(\beta) = \infty$.
Then, for each $s \in I$,
we have $a_{s,\beta} = \lim_{t \nearrow \beta} \frac{a(s)}{a(t)} = 0$.
Fix $t_0 \in I$ arbitrarily and set $c=a(t_0)$.
Then, for $\tau \in I$ with 
$\tau \geq t_0$, we have 
$\frac{c}{a_{t_0,\tau}} = a( \tau )$.

Applying Theorem~\ref{thm:extendability-of-transiton-family}~(ii),
the locally uniform limit 
\[
  g_t 
 = \lim_{\tau \nearrow \beta} 
 \frac{c}{a_{t_0,\tau}}  \omega_{t,\tau}
 = 
 \lim_{\tau \nearrow \beta} 
 a(\tau) \omega_{t,\tau}
\]
exists and is univalent on $\mathbb{D}$ for all $t \in I$.
Moreover, the family $\{ g_t \}_{t \in I}$ 
forms a Loewner chain 
with $\{ \omega_{s,t}\}_{(s,t) \in I_+^2}$ 
as the associated transition family. 
Note that  
$g_t'(0)= a(t)=f_t'(0)$, $t \in I$.

Now consider the family
\[
  \{ f_t \circ g_t^{-1} \}_{t \in I}.
\]
Each function $f_t \circ g_t^{-1}$
is defined on the domain $g_t( \mathbb{D})$,
and the family $\{ g_t( \mathbb{D}) \}_{t \in I}$ 
is nondecreasing in $t$.
For $(s,t) \in I_+^2$, 
we have, on $g_s(\mathbb{D})$
\[
  f_s \circ g_s^{-1} 
 = f_t \circ \omega_{s,t} \circ (g_t \circ \omega_{s,t})^{-1}
 =f_t \circ g_t^{-1} .
\]
By the Koebe one-quarter theorem we have 
$\mathbb{D}(0, a(t)/4) \subset g_t(\mathbb{D}) $.
Combining this and $\lim_{t \nearrow \beta} g_t'(0) = a(\beta) = \infty$
it follows that 
$\bigcup_{t \in I} g_t( \mathbb{D}) = \mathbb{C}$. 
Therefore,
the family 
$\{ f_t \circ g_t^{-1} \}_{t \in I}$ 
defines a unique entire function 
satisfying $F(0) = F'(0)-1= 0$,
and such that 
\[
   F(w) = f_t \circ g_t^{-1}(w) , \quad w \in g_t(\mathbb{D})  
\]
for all $t \in I$.
Thus, we conclude that 
$f_t = F \circ g_t $, as required.

Finally, 
to prove the uniqueness, 
suppose that there exists an entire function 
$\tilde{F} : \mathbb{C} \rightarrow \mathbb{C}$
satisfying $\tilde{F}(0) = \tilde{F}'(0)-1=0$,
and a Loewner chain $\{ \tilde{g}_t \}_{t \in I}$ 
of univalent functions 
such that 
$\tilde{F} \circ \tilde{g}_t = f_t$, $t \in I$.

Then, as before,   
the family $\{ \tilde{g}_t \}_{t \in I}$ shares
the same transition family 
$\{ \omega_{s,t}\}_{(s,t) \in I_+^2}$ 
with $\{ g_t \}_{t \in I}$.
In particular, we have
$\omega_{s,t} = g_t^{-1} \circ g_s 
= \tilde{g}_t^{-1} \circ \tilde{g}_s$ for all $(s,t) \in I_+^2$.
Since $\tilde{F}(0)= \tilde{F}'(0)-1=0$,
we have $\tilde{g}_t'(0) = f_t'(0) = g_t'(0) = a(t)$, $t \in I$.
Therefore, applying 
Lemma~\ref{lemma:estimate-for-inverse-of-univalent-functions}
we obtain 
\[
 g_t (z) 
 = \lim_{\tau \nearrow \beta} a(\tau) \omega_{t,\tau}(z) 
 =
 \lim_{\tau \nearrow \beta}  a(\tau) \tilde{g}_\tau^{-1} (\tilde{g}_t(z))
 = \tilde{g}_t(z).
\]
This implies $\tilde{F} = F$.
\end{proof}

\section{The Loewner Range}
\begin{definition}
Let $\{ f_t \}_{t \in I}$ be a Loewner chain on a right-open interval $I$ 
with $\beta = \sup I \not\in I$.
Then the domain $\Omega_\beta = \bigcup_{t \in I} f_t(\mathbb{D}) $
 is called the \textit{Loewner range} of $\{ f_t \}_{t \in I}$.
\end{definition}

We study the relation between $\Omega_\beta$ and $a(\beta)$.
In the case where each $f_t$, we have the following equivalence.
\begin{proposition}
\label{proposition:Omega_beta_coincides_withC}
Let $\{ f_t \}_{t \in I}$ be a Loewner chain
of univalent functions on a right-open interval $I$ 
with $\beta = \sup I \not\in I$,
and let $\Omega_\beta$ denote 
the Loewner range of $\{ f_t \}_{t \in I}$.
Then 
$a(\beta) = \lim_{t \nearrow \beta} f_t'(0) = \infty$ 
if and only if $\Omega_\beta = \mathbb{C}$.
\end{proposition}
\begin{proof}
For each $t \in I$, define 
\[
  \rho_1 (t) 
 = \sup \{ r > 0 : \mathbb{D}(0,r) \subset f_t( \mathbb{D})\} .
\]
Then the proposition easily follows from 
the inequality
\[
   \frac{a(t)}{4} \leq \rho_1 (t) \leq a(t),
\] 
which is is a consequence
of the Koebe one-quarter theorem 
and the Schwarz lemma.
\end{proof}

In the general case, we have the following result.
Here, we temporarily use the notion of universal covering maps
which is systematically treated in 
Chapter~ \ref{chapter:KernelAndCovering}.

\begin{theorem}
Let $\{ f_t \}_{t \in I}$ be 
a Loewner chain on a right-open interval $I$
with 
$\beta = \sup I \not\in I$,
and let 
$\Omega_\beta$ denote the Loewner range of 
$\{ f_t \}_{t \in I}$.
If ${}^\# (\mathbb{C} \backslash \Omega_\beta ) \geq 2$,
then $a(\beta) < \infty$.
\end{theorem}
\begin{proof}
Take  two distinct points 
$w_0,w_1 \in (\mathbb{C} \backslash \Omega_\beta)$,
and let 
$F: \mathbb{D} \rightarrow \mathbb{C} \backslash \{ w_0,w_1\}$ 
be the unique universal covering map with $F(0)=0$ and $F'(0)> 0$.
Then for each $t \in I$,
the map 
$f_t : \mathbb{D} \rightarrow \mathbb{C} \backslash \{ w_0,w_1\}$ 
admits a unique analytic lift 
$\phi_t: \mathbb{D} \rightarrow \mathbb{D}$
such that $F  \circ \phi_t = f_t$ and $\phi_t (0)=0$.
By the Schwarz lemma, we have 
\[
  a(t) = f_t'(0) \leq F'(\phi_t(0)) \phi'(t) \leq F'(0) ,
 \quad t \in I, 
\]
and hence $a(\beta) \leq F'(0) < \infty$.
\end{proof}

From the above theorem, 
it follows that if $a (\beta) = \infty$,
either $\Omega_\beta = \mathbb{C}$ or 
$\Omega_\beta = \mathbb{C} \backslash \{ w _0 \}$ for 
some $w_0 \in \mathbb{C} \backslash \{ 0 \}$.
We now provide examples of Loewner chains
corresponding to both cases.

Let $\{f_t \}_{t \in I}$
be a Loewner chain of univalent function
on right-open interval $I$
with $a(\beta) = \infty$.
Then by Proposition \ref{proposition:Omega_beta_coincides_withC}
the Loewner range of $\{f_t \}_{t \in I}$
clearly coincides with $\mathbb{C}$. 
This implies the Loewner chain $\{ \exp (f_t )  -1 \}_{t \in I}$
has $\mathbb{C} \backslash \{-1\}$ as its Loewner range.

The converse of the above theorem  does not hold.
Indeed, let  
$F : \mathbb{D} \rightarrow \mathbb{C}$ be 
any surjective analytic map satisfying
$F(0)=0$ and $F'(0) > 0$,
and define $f_t (z) = F(tz)$ for $0<t<1$.
Then the family 
$\{ f_t \}_{0 < t <1}$ forms a strictly increasing 
and continuous Loewner 
chain with $\lim_{t \nearrow 1} f_t '(t) = F'(0) < \infty $, 
and its Loewner range coincides with  
$F(\mathbb{D}) = \mathbb{C}$.

\chapter{Loewner--Kufarev Differential Equations}
\label{chapter:Loewner-Kufarev}

Let $\{ f_t \}_{t \in I}$ be 
a  Loewner chain on an interval $I$, 
which is strictly increasing and continuous.
Define $a_t = f_t'(0)$ for $t \in I$, 
and let $\{ \omega_{s,t} \}_{(s,t) \in I_+^2}$ denote
the associated transition family.

Intuitively, the inequality~(\ref{ineq:Lipshctz-cont-f_t})
suggests that for a fixed $z \in \mathbb{D}$,
the function $f_t(z)$, viewed as a function of $t \in I$, 
is Lipschitz continuous with respect to $a(t)$,
which is strictly increasing and continuous.
Moreover, 
using the identity $a_{t_1,t_2} = a_{t_0,t_2}/a_{t_0,t_1}$
together with the estimate (\ref{ineq:omega_isLipschitz_wrt_a}),
we find that for fixed $z \in \mathbb{D}$ and $t_0 \in I$, 
the function $\omega_{t_0,t}(z)$, defined for 
$t \in I \cap [t_0, \infty )$, 
is also Lipschitz continuous with respect to $a_{t_0,t}$.

Note that Lipschitz continuity implies
absolute continuity and almost everywhere differentiability.

At this point, two strategies are available:
one is to reparameterize the family, 
and the other is to retain the original parameter and 
differentiate directly with respect to $a(t)$. 
Following the former approach, Pommerenke~\cite{Pommerenke1965} 
derived a differential equation by 
reparameterizing so that $a(t)=e^t$, 
and then differentiating with respect to $t$.
See~\cite{Rosenblum-Rovnyak} 
for an alternative reparameterization.

However, in practice, it is seldom possible to compute
$a(t)=f_t'(0)$ explicitly. 
Therefore, we adopt the second approach:
we investigate the partial derivatives of $f_t(z)$ and
$\omega_{t_0,t}(z)$ with respect to $a(t)$ and $a_{t_0,t}$, 
respectively.

\section{Preliminaries on Lebesgue--Stieltjes Measures}
Here we summarize basic results related to
the Fundamental Theorem of Calculus 
with respect to a strictly  increasing and continuous function.  
For terminology and further details, see~\cite{Folland}.

Let $\alpha, \beta \in \mathbb{R}$ with
$\alpha < \beta$,  
and let $\psi :[\alpha, \beta] \rightarrow \mathbb{R}$
be a continuous and strictly increasing function. 
Let $\mathcal{E} ([\alpha , \beta ])$ be the collection of 
all intervals of the form
$(a,b]$ or $[\alpha , a]$ or $\emptyset$,
where $\alpha \leq a < b \leq \beta$.
Then the class $\mathcal{A} ([\alpha , \beta ])$, consisting of 
all finite disjoint unions of elements of 
$\mathcal{E}([\alpha , \beta ])$,
forms an algebra; that is, 
if $E, F \in \mathcal{A}([\alpha , \beta ])$, 
then $E \cup F \in \mathcal{A}([\alpha , \beta ])$,
and if $E \in \mathcal{A}([\alpha , \beta ])$, 
then 
$E^c := [\alpha, \beta ] \backslash E 
\in \mathcal{A}([\alpha , \beta ])$.

If $E_j \in \mathcal{E}([\alpha , \beta ])$ 
for $j=1,\ldots , n$ are disjoint intervals 
with $a_j = \inf E_j $ and $b_j = \sup E_j $,
define
\begin{equation}
\label{eq:mu_on_A}
   \tilde{\mu}_\psi \left( \bigcup_{j=1}^n E_j \right)
   = \sum_{j=1}^n \left( \psi (b_j) - \psi (a_j ) \right) , 
\end{equation}
and set $\tilde{\mu}_\psi ( \emptyset ) = 0$.

Although an element of $\mathcal{A} ([\alpha , \beta ])$ 
can be expressed in more than one way 
as a disjoint union of intervals 
in $\mathcal{E} ([\alpha , \beta ])$,
the set function 
$\tilde{\mu}_\psi: \mathcal{A}([\alpha , \beta ]) 
\rightarrow [0, \infty ) $ 
is well defined and constitutes a premeasure on the algebra 
$\mathcal{A}([\alpha , \beta ])$.

Let $\mathcal{P}([\alpha, \beta ])$ denote the collection 
of all subsets of $[\alpha ,  \beta ]$.
Define an outer measure 
$\mu_\psi^* : \mathcal{P} ([\alpha, \beta]) \rightarrow [0, \infty)$
by 
\begin{equation}
\label{def_eq:mu_psi_outer_measure}
   \mu_\psi^* (A) 
   = \inf \left\{ \sum_{j=1}^\infty \tilde{\mu}_\psi (E_j) : E_j 
  \in \mathcal{A}([\alpha, \beta]) , \quad  
   A \subset \bigcup_{j=1}^\infty E_j \right\} . 
\end{equation}

Let $\mathcal{M}_\psi ([\alpha , \beta ])$ be the collection of 
all sets $E$ that satisfy the Carath\'{e}odory condition:
\begin{equation}
\label{eq:Caratheodory_condition}
 \mu_\psi^*(A) 
 \geq
 \mu_\psi^*(A \cap E)  + \mu_\psi^*(A \backslash  E)  
\quad \text{for all} \quad A \in \mathcal{P} ([\alpha, \beta]). 
\end{equation}
Then $\mathcal{M}_\psi ([\alpha , \beta ])$ forms 
a $\sigma$-algebra
containing $\mathcal{A}([\alpha , \beta ])$,
and the restriction 
$\mu_\psi := \mu_\psi^* |_{\mathcal{M}_\psi([\alpha , \beta ])}$
defines a measure on $\mathcal{M}_\psi ([\alpha , \beta ])$.

The measure space
$([\alpha, \beta], \mathcal{M}_\psi ([\alpha , \beta ]) , \mu_\psi)$
is complete.
The measure $\mu_\psi$ is called the Lebesgue--Stieltjes measure 
associated with $\psi$,  
and each set $E$ in $\mathcal{M}_\psi ([\alpha , \beta ])$ is 
referred to as $\mu_\psi$-measurable.

Since the smallest $\sigma$-algebra 
containing $\mathcal{A}([\alpha , \beta ])$ 
coincides with the Borel $\sigma$-algebra,
it follows that
$\mathcal{B}([\alpha, \beta]) 
\subset \mathcal{M}_\psi ([\alpha , \beta ])$,
and 
\[
  \mu_\psi^* (E) = \mu_\psi (E)= \tilde{\mu}_\psi (E)
  \quad \text{for all} \quad E \in \mathcal{A}([\alpha , \beta ]).
\]

In particular, 
the Lebesgue--Stieltjes measure associated with
the function $\psi (t) = t$ for $\alpha \leq t \leq \beta$ 
is simply the Lebesgue measure.
We denote its corresponding outer measure, 
$\sigma$-algebra and measure by 
$\mu_1^*$, $\mathcal{M}_1([\alpha , \beta ])$ and $\mu_1$, 
respectively.

Since $\psi$ is a homeomorphism from $[\alpha , \beta]$
onto $[\psi (\alpha), \psi (\beta)]$, 
the following lemma follows directly from 
(\ref{eq:mu_on_A}) and (\ref{def_eq:mu_psi_outer_measure}).
\begin{lemma}
\label{lemma:image_of_psi_outer_measure_is_the_Lebesgue_outer_measure}
Suppose that  $\psi $ is strictly increasing and continuous 
on $[\alpha, \beta]$. 
Then 
\[
   \mu_\psi^* (A) = \mu_1^* (\psi (A))
\]
for every subset $A$ of $[\alpha, \beta]$,
where $\mu_1^*$ is the Lebesgue outer measure 
on $[\psi (\alpha), \psi (\beta)]$.
\end{lemma}

The following proposition is 
an immediate consequence of 
Lemma~\ref{lemma:image_of_psi_outer_measure_is_the_Lebesgue_outer_measure}.
\begin{proposition}
\label{proposition:image_of_psi_measure_is_the_Lebesgue_measure}
Suppose that  $\psi $ is strictly increasing and continuous 
on $[\alpha, \beta]$. 
Then 
a subset $E \subset [\alpha, \beta]$ 
is $\mu_\psi$-measurable (resp. Borel measurable) 
if and only if $\psi(E)$ is Lebesgue 
measurable (resp. Borel measurable) 
in $[\psi (\alpha), \psi (\beta)]$.
In either case,  
we have 
\[
  \mu_\psi (E) = \mu_1 ( \psi (E)) .	    
\]
Furthermore, a function 
$h:[\alpha, \beta] \rightarrow [- \infty, \infty]$
is $\mu_\psi$-integrable if and only if
$h \circ \psi^{-1}$ is 
Lebesgue integrable, 
and for any $\mu_\psi$-measurable set $A \subset [\alpha. \beta ]$
\[
 \int_A h(t) \, d \mu_\psi (t) 
 =
 \int_{\psi (A)} 
 h(\psi^{-1}(s)) \, d \mu_1 (s) .
\]
\end{proposition}
\begin{proof}
Since $\psi :[\alpha, \beta ] \rightarrow [\psi(\alpha), \psi(\beta) ]$
is a bijection, 
it follows from the Carath\'{e}odory condition
\eqref{eq:Caratheodory_condition} and 
Lemma~\ref{lemma:image_of_psi_outer_measure_is_the_Lebesgue_outer_measure}
that for any $E \in \mathcal{P}([\alpha, \beta ])$,
we have
\begin{align*}
  & E \in \mathcal{M}_\psi([\alpha, \beta]) \\
  \Longleftrightarrow \quad &
  \mu_\psi^*(A) \geq \mu_\psi^*(A \cap E) + \mu_\psi^*(A \setminus E)
  \quad \text{for all } A \in \mathcal{P}([\alpha, \beta]) \\
  \Longleftrightarrow \quad &
  \mu_1^*(\psi(A)) \geq \mu_1^*(\psi(A) \cap \psi(E)) + 
  \mu_1^*(\psi(A) \setminus \psi(E)) 
  \quad \text{for all } A \in \mathcal{P}([\alpha, \beta]) \\
  \Longleftrightarrow \quad &
  \mu_1^*(B) \geq \mu_1^*(B \cap \psi(E)) + \mu_1^*(B \setminus \psi(E))
  \quad \text{for all } B \in \mathcal{P}([\psi(\alpha), \psi(\beta)]) \\
  \Longleftrightarrow \quad &
  \psi(E) \in \mathcal{M}_1([\psi(\alpha), \psi(\beta)]).
\end{align*}

Next, let $\mathcal{S}$ be the collection 
of all subsets $A  \subset [\alpha, \beta]$
such that $\psi(A) \in \mathcal{B}([\psi(\alpha), \psi(\beta)])$.
Since $\psi$ is a bijection, it is easy to verify that
$\mathcal{S}$ is a $\sigma$-algebra on $[\alpha, \beta]$.
Moreover, because $\psi$ is an open map,
$\mathcal{S}$ contains all open subsets of 
$[\alpha, \beta]$.
Hence,
we have $\mathcal{B}([\alpha, \beta]) \subset \mathcal{S}$,
which implies that
if $A \in \mathcal{B}([\alpha, \beta])$,
then $\psi(A) \in \mathcal{B}([\psi (\alpha), \psi (\beta) ])$.

Similarly, since $\psi$ is a homeomorphism,
we also have that 
if $B \in \mathcal{B}([\psi (\alpha), \psi (\beta)])$,
then $\psi^{-1}(B) \in \mathcal{B}([\alpha, \beta ])$.

The second statement follows immediately from 
the first by the definition of the integral 
with respect to a measure.
\end{proof}

\section{The Fundamental Theorem of Calculus %
for Lebesgue--Stieltjes Measures}
For a function $u: [\alpha, \beta] \rightarrow \mathbb{R}$, 
the upper and lower $\psi$-derivatives of $u$ at $t$
are defined, respectively, by 
\[
  D_\psi^+ u(t) 
   = 
   \limsup_{\substack{\scriptscriptstyle t_1 \leq t \leq t_2 \\%
\scriptscriptstyle  t_2 - t_1 \searrow 0}} 
   \frac{u(t_2)-u(t_1)}{\psi(t_2)-\psi(t_1)},
\quad
   D_\psi^- u (t) 
   = \liminf_{\substack{\scriptscriptstyle t_1 \leq t \leq t_2 \\
\scriptscriptstyle  t_2 - t_1 \searrow 0}}
      \frac{u(t_2)-u(t_1)}{\psi(t_2)-\psi(t_1)} .
\]
It is a straightforward to verify that
\[
   D_\psi^+ u(t) 
   = \limsup_{\Delta t \rightarrow 0} 
   \frac{u(t+\Delta t)-u(t)}{\psi(t+\Delta t)-\psi(t)},
\quad
   D_\psi^- u(t)  
   = \liminf_{\Delta t \rightarrow 0} 
   \frac{u(t+\Delta t)-u(t)}{\psi(t+\Delta t)-\psi(t)}.
\]
If the upper and lower $\psi$-derivatives of $u$ at $t$
are both finite and equal,
we say that $u$ is $\psi$-differentiable at $t$. 
Their common value is denoted by $D_\psi u (t)$ and
is called the $\psi$-derivative of $u$ at $t$.
Clearly, $u$ is 
$\psi$-differentiable at $t$ if and only if 
the limit 
\[
\lim_{\Delta t \rightarrow 0} 
\frac{u(t+\Delta t)-u(t)}{\psi(t+\Delta t)-\psi(t)}
\]
exists. 
In this case the limit agrees with $D_\psi u(t)$.

We denote the usual derivative 
(i.e., with respect to the identity function on $I$)
by $D$.
In particular,
if both $u$ and $\psi$ are  differentiable at $t$
and $D \psi (t) \, (= \psi'(t)) \not= 0$, then
\begin{equation}
   D_\psi u (t)
   =
   \frac{D u (t)}{D \psi (t)} . 
\end{equation}

We say that a function 
$u : [\alpha, \beta] \rightarrow \mathbb{R}$ 
is absolutely $\psi$-continuous if, for each $\varepsilon > 0$,
there exists $\delta > 0$ such that
for any finite collection of disjoint intervals 
$(\alpha_1 , \beta_1), \ldots , (\alpha_n, \beta_n)$ 
in $[\alpha, \beta]$,
the implication
\begin{equation}
   \sum_{k=1}^n (\psi(\beta_k)-\psi(\alpha_k)) < \delta
    \quad \Longrightarrow \quad
    \sum_{k=1}^n |u(\beta_k)- u(\alpha_k)| < \varepsilon 
\end{equation}
holds.
Note that any absolutely $\psi$-continuous function is 
necessarily continuous. 
For a complex-valued function 
$h=u+iv: [\alpha, \beta] \rightarrow \mathbb{C}$
we say that $h$ is absolutely $\psi$-continuous 
(or $\psi$-differentiable) 
if both $u$ and $v$ are $\psi$-absolutely continuous 
(or $\psi$-differentiable, respectively).

\begin{proposition}[The Fundamental Theorem of Calculus 
for Lebesgue--Stieltjes Measures]
\label{proposition:fundamental-thm-of-calculus}
Suppose that a function $h$ is absolutely $\psi$-continuous 
on  $[\alpha, \beta]$.
Then, for $\mu_\psi$-almost every $t \in [\alpha, \beta]$,
the function $h$ is $\psi$-differentiable at $t$,
and $D_\psi h $ is $\mu_\psi$-integrable.
Furthermore, we have
\begin{equation}
\label{eq:fundamental_thorem_of_calculus_for_psi_measure}
   h(t) - h( \alpha ) = \int_{[\alpha ,t]} D_\psi h (\tau ) \, 
   d \mu_\psi (\tau ) , \quad 
 t \in [\alpha, \beta ] . 
\end{equation}
Conversely, if $k$ is a $\mu_\psi$-integrable function 
on $[\alpha, \beta]$ and 
\begin{equation}
   h(t) := \int_{[\alpha , t]} k (\tau ) \, d \mu_\psi (\tau ) , \quad 
 t \in [\alpha , \beta ] ,
\end{equation}
then $h$ is absolutely $\psi$-continuous on $[\alpha, \beta]$,
and $D_\psi h (t) = k(t)$ for $\mu_\psi$-almost every 
$t \in [\alpha, \beta]$.
\end{proposition}
\begin{proof}
Put $\tilde{\alpha}=\psi( \alpha )$ and 
$\tilde{\beta}= \psi( \beta )$.
Then, by definition, 
$h \circ \psi^{-1}$ is  absolutely continuous 
on $[\tilde{\alpha}, \tilde{\beta}]$ in the usual sense.
Therefore, there exists a set 
$\tilde{N} \subset [\tilde{\alpha}, \tilde{\beta}]$ 
of Lebesgue measure zero (which may be empty)
such that, for every 
$s \in [\tilde{\alpha}, \tilde{\beta}] \setminus \tilde{N}$, 
$h \circ \psi^{-1}$ is differentiable at $s$; that is,
the limit 
\[
    D (h \circ \psi^{-1})(s)
   := \lim_{\substack{\scriptscriptstyle s_1 \leq s \leq s_2 \\
\scriptscriptstyle  s_2 - s_1 \searrow 0}}
     \frac{h(\psi^{-1}(s_2)) - h(\psi^{-1}(s_1))}{s_2 - s_1} 
\]
exists.
Replacing $\tilde{N}$ with a larger set if necessary, 
we may assume 
that $\tilde{N}$ is 
a $G_\delta$ set.
Define $N = \psi^{-1}(\tilde{N})$.
Then, by 
Proposition~\ref{proposition:image_of_psi_measure_is_the_Lebesgue_measure},
we have $\mu_\psi (N) = \mu_1 (\tilde{N}) = 0$, 
and it is easy to verify that
\[
  D (h \circ \psi^{-1})( \psi (t)) = D_\psi h (t),
 \quad t \in [\alpha,\beta ] \setminus N .
\]
Note that both $D_\psi h $ and $D (h \circ \psi^{-1} )$ are 
Borel measurable on  $[\alpha, \beta] \backslash N$
and $[\tilde{\alpha}, \tilde{\beta}] \backslash \tilde{N}$,
respectively.
Since $h \circ \psi^{-1}$ is absolutely continuous,
$D (h \circ \psi^{-1})$ is integrable with respect to
the Lebesgue measure.
It follows from 
Proposition~\ref{proposition:image_of_psi_measure_is_the_Lebesgue_measure}
that $D_\psi h$ is integrable with respect to $\mu_\psi$ and 
\[
 \int_{A} D_\psi \, d \mu_\psi 
 = 
 \int_{\psi (A)} D (h \circ \psi^{-1}) \, d \mu_1 
\]
for any $\mu_\psi$-measurable set $A \subset [\alpha , \beta ]$.
Applying the fundamental theorem of calculus 
to an absolutely continuous function $h \circ \psi^{-1}$,
we obtain, for $s \in [\tilde{\alpha}, \tilde{\beta}]$,
\[
   h \circ \psi^{-1} (s) - h \circ \psi^{-1} (\tilde{\alpha})
 =
 \int_{[\tilde{\alpha},s]} D(h \circ \psi^{-1}) \, d \mu_1
 =
 \int_{[\alpha ,\psi^{-1}(s)]} (D_\psi h)  \, d \mu_\psi ,
\]
which implies \eqref{eq:fundamental_thorem_of_calculus_for_psi_measure}.
The converse follows similarly from the corresponding part 
of the fundamental theorem of calculus.
\end{proof}

\section{Loewner--Kufarev Equations for Transition Families}
Let $\{ \omega_{s,t}\}_{(s,t) \in I_+^2}$ be a transition family. 
We write $\omega(z,s,t)$ instead of $\omega_{s,t}(z)$
for $(s,t) \in I_+^2$ and $z \in \mathbb{D}$,
and define
\begin{align}
   \frac{\partial \omega}{\partial \psi (t)} (z,t) 
   = & \, \2sidelim 
   \frac{\omega(z,t_1,t_2)-z}{\psi (t_2)- \psi (t_1)} , \quad
    t \in I,
\\
    \frac{\partial \omega }{\partial \psi (t)} (z,t, t_0) 
    = & \, \2sidelim 
    \frac{\omega(z,t_2,t_0) - \omega(z,t_1,t_0)}{\psi (t_2)- \psi (t_1)}, 
   \quad
    \text{for }  t< t_0 , 
\\
    \frac{\partial \omega }{\partial \psi (t)} (z,t_0,t)   
    = & \, \2sidelim 
    \frac{\omega(z,t_0,t_2) - \omega(z,t_0,t_1)}{\psi (t_2)- \psi (t_1)},
  \quad
    \text{for } t > t_0 .
\end{align}

\begin{theorem}
\label{theorem:a-derivative-of-omega}
Let $\{ \omega(\cdot, s,t ) \}_{(s,t) \in I_+^2}$ be 
a strictly monotone and continuous transition family
and let $a(t)$, $t \in I$ be a strictly increasing 
and positive function 
defined by (\ref{eq:defi-of-at}) for some $c>0$.
Then there exists a $G_\delta$ set $N (\subset I )$ of 
$\mu_a$-measure zero and a  
Herglotz family $\{ P(\cdot , t) \}_{t \in I}$ 
such that $P(z,t)$ is  Borel measurable on $\mathbb{D} \times I$,
and such that
for each 
$t \in I \backslash N$,
\begin{align}
\label{eq:differentail-wrt-a}
    \2sidelim 
    \frac{\frac{\omega(z,t_1,t_2)}{z} - 1}{\frac{a(t_1)}{a(t_2)} -1 }
    = 
    - \frac{a(t)}{z} 
    \frac{\partial \omega}{\partial a(t)} (z,t) 
   = P(z,t), \; z \in \mathbb{D}
\end{align}
and the convergence is locally uniform on $\mathbb{D}$.
Furthermore, 
for each fixed $t_0 \in I$ and $z \in \mathbb{D}$,
\begin{align}
\label{eq:first-a-derivative}
    \frac{\partial \omega }{\partial a(t)} (z,t, t_0) 
    = & \, 
    \frac{zP(z,t)}{a(t)}
    \omega'(z,t,t_0)
\end{align}
for $t \in (I \cap (-\infty, t_0))\backslash N$  and
\begin{align}
\label{eq:second-a-derivative}
    \frac{\partial \omega }{\partial a(t)} (z,t_0,t)   
    = & \, - \frac{\omega(z,t_0,t)}{a(t)} P(\omega(z,t_0,t),t)
\end{align}
for $t \in (I \cap (t_0, \infty )) \setminus N$.
In either case, the convergence is locally uniform on $\mathbb{D}$.
\end{theorem}
\begin{proof}
Without loss of generality, we may assume that
$I = [\alpha, \beta]$ with $- \infty <\alpha < \beta < \infty $.

\noindent\textbf{Step 1.}
First we show that there exists 
a $G_\delta$ set $N \subset [\alpha, \beta]$ of
$\mu_a$ measure zero such that 
for every $t \in I \setminus N$,
the limit
\[
  \frac{\partial \omega}{\partial a(t)} (z, \alpha , t)
 = 
 \2sidelim 
 \frac{\omega (z,\alpha , t_2)- \omega (z,\alpha , t_1)}{a(t_2)-a(t_1)},
 \quad z \in \mathbb{D}
\]
exists, and the convergence is locally uniform on $\mathbb{D}$.

To see this, choose a sequence of distinct points 
$\{ z_j\}_{j=1}^\infty \subset \mathbb{D}$ with $z_j \to 0$.
By~(\ref{ineq:omega_isLipschitz_wrt_a}), 
we have, 
for $\alpha \leq t_1 \leq t_2 \leq \beta$,
\begin{equation}
\label{ineq:Lipschitz-continuity-and-boundedeness}
   | \omega(z_j,\alpha ,t_2) - \omega(z_j, \alpha ,t_1)| 
  \leq 
   \frac{a(t_2)-a(t_1)}{a(t_2)} \frac{|z_j|(1+|z_j|)}{1-|z_j|} . 
\end{equation}
It follows that for each fixed $j \in \mathbb{N}$, 
as a function of $t$, $\omega(z_j, \alpha , t)$ is 
Lipschitz continuous with respect to $a(t)$,
and hence absolutely $a$-continuous on $[\alpha, \beta ]$.
Thus, there exists a set $N_j \subset I$
of $\mu_a$-measure zero such that for 
$t \in I \backslash N_j$
the limit 
\begin{equation}
\label{eq:limit_exists_at_z_j}
    \2sidelim 
    \frac{\omega(z_j, \alpha ,t_2) - \omega(z_j,\alpha ,t_1)}
    {a(t_2)-a(t_1)}  
\end{equation}
exists.
Now choose a $G_\delta$-set $N$ of $\mu_a$-measure zero
such that $\bigcup_{j=1}^\infty N_j  \subset N$.
Fix $t \in I \setminus N$.
For  $t_1,t_2 \in I$ with 
$t_1 \leq t \leq t_2$
and $t_1 < t_2$, 
define
\[
    g (z, t_1,t_2)
 =  \frac{ \omega(z,\alpha ,t_2) - \omega(z, \alpha ,t_1) }
    {a(t_2)-a(t_1)}, 
 \quad z \in \mathbb{D} .
\]
Then the family of analytic functions 
\[
 \mathcal{G} :=  \{ g(\cdot ,t_1,t_2) : t_1 , t_2 \in I  
 \text{ with $t_1 \leq t \leq t_2$ and $t_1 < t_2$} \} 
\] 
is locally uniformly bounded in $\mathbb{D}$.
Moreover, by~(\ref{eq:limit_exists_at_z_j}), 
the limit $\2sidelim g(z_j, t_1,t_2)$ exists 
for each $j \in \mathbb{N}$.
Therefore, by the Vitali--Porter convergence theorem
(Lemma~\ref{lemma:Vitali_for_family}),
for each fixed $t \in I \backslash N $, the limit
\[
  \2sidelim 
    \frac{ \omega(z, \alpha ,t_2) - \omega(z, \alpha ,t_1) }{a(t_2)-a(t_1)}
 =  \2sidelim g(z, t_1,t_2),
  \quad z \in \mathbb{D}
\]
exists, and the convergence is locally uniform on $\mathbb{D}$.

\noindent\textbf{Step 2.} 
Next, we show that for each $t \in I \backslash N$,
the limit
\begin{equation}
\label{eq:partailly_differentiable}
  \frac{\partial \omega }{\partial a(t)} (z, t)
 :=
 \2sidelim \frac{\omega (z, t_1 , t_2) - z}{a(t_2)-a(t_1)}, 
  \quad z \in \mathbb{D}  
\end{equation}
exists, and the convergence is locally uniform on $\mathbb{D}$.

We claim that for $t \in I \backslash N$ and $z \in \mathbb{D}$, 
\begin{equation}
\label{eq:lim_exists_in_a_subdomain}
    \2sidelim 
    \frac{\omega(\omega(z,\alpha ,t), t_1,t_2) - \omega(z, \alpha ,t)}
{a(t_2)-a(t_1)} 
    =
   \2sidelim 
    \frac{\omega(z, \alpha,t_2) - \omega(z, \alpha ,t_1)}{a(t_2)-a(t_1)} . 
\end{equation}

To see this, note that 
$\omega'(z,t_1,t_2) \rightarrow 1$ locally uniformly on $\mathbb{D}$ 
as $t_2-t_1 \searrow 0$ with $t_1 \leq t \leq t_2$.
Then, by (\ref{ineq:Lipschitz-continuity-and-boundedeness}) 
we have
\begin{align}
& \,
   \omega(\omega(z,\alpha ,t), t_1,t_2) - \omega(z, \alpha ,t)
   - \omega(z, \alpha ,t_2) + \omega(z, \alpha ,t_1)
\\
   = & \,
  \omega(\omega(z, \alpha ,t), t_1,t_2) - \omega(\omega(z, \alpha ,t_1), t_1,t_2)
   - \omega(z,\alpha ,t) + \omega(z,\alpha ,t_1) 
\nonumber
\\
   = & \,
   \left( \omega(z,\alpha,t) - \omega(z,\alpha,t_1) \right)
  \int_0^1 \left\{ \omega'(\gamma (\theta ), t_1,t_2) -1 \right\} \, 
 d \theta
\nonumber
\\
  = & \,  \text{\rm o}(a(t)-a(t_1)) =  \text{\rm o}(a(t_2)-a(t_1)) ,
\nonumber
\end{align}
where $\gamma (\theta )$ denotes 
the line segment joining
$\omega(z,\alpha ,t_1)$ and $\omega(z,\alpha,t)$, i.e., 
\[
    \gamma ( \theta )
 = (1-\theta ) \omega(z,\alpha ,t_1) + \theta \omega(z,\alpha,t),
 \quad
 0 \leq \theta \leq 1 .
\]

From (\ref{eq:lim_exists_in_a_subdomain}), 
it follows that the limit
\[
   \2sidelim 
    \frac{\omega(\zeta , t_1,t_2) - \zeta}{a(t_2) - a(t_1)}
\]
exists for every $\zeta \in \omega(\mathbb{D}, \alpha , t)$.
Since $\omega(\mathbb{D}, \alpha , t)$ is a nonempty domain,
it follows from (\ref{ineq:Lipschitz-continuity-and-boundedeness}) 
and the Vitali--Porter convergence theorem, as in Step~1,
that the above limit exists for all $\zeta \in \mathbb{D}$
and the convergence is locally uniform on $\mathbb{D}$.

\noindent\textbf{Step 3.} 
For $t \in I \setminus N$ and $z \in \mathbb{D}$, define
\begin{align}
\label{eq:generating-function-for-transition-family}
  P(z, t) 
  := & \,  
  \2sidelim 
    \frac{\frac{\omega(z,t_1,t_2)}{z} - 1}{\frac{a(t_1)}{a(t_2)} -1 }
\\
   = & \,  
   - \2sidelim 
    \frac{a(t_2)}{z} 
    \frac{\omega(z,t_1,t_2) - z}{a(t_2) - a(t_1) }
  = - \frac{a(t)}{z} \frac{\partial \omega}{\partial a(t)} (z, t) ,
\nonumber
\end{align}
and set $P(z,t) = 1$ for $t \in N$ and $z \in \mathbb{D}$.
Then $P(z,t)$ is Borel measurable on $\mathbb{D} \times I$,
and $P(0,t) = 1$, 
since $\omega (z,t_1,t_2) = \frac{a(t_1)}{a(t_2)}z+ \cdots$.

We show that $\Real P(z, t) > 0$ for $(z, t) \in \mathbb{D} \times I$.
To see this, let $(t_1,t_2) \in I_+^2$ with 
$t_1 \not= t_2$ and $z \in \mathbb{D}$. 
Define
\begin{align}
  \Phi (z,t_1,t_2) 
   = & \, \frac{ \frac{a(t_1)}{a(t_2)}- \frac{\omega(z,t_1,t_2)}{z} }
    {1 - \frac{a(t_1)}{a(t_2)} \frac{\omega(z,t_1,t_2)}{z} },
\\
  A(z,t_1,t_2) 
  = & \,
 1 - \frac{\frac{\omega(z,t_1,t_2)}{z} -1}{\frac{a(t_1)}{a(t_2)} -1} ,
\\
 B(z,t_1,t_2) 
 = & \, 1 + \frac{a(t_1)}{a(t_2)}
 \frac{\frac{\omega(z,t_1,t_2)}{z} -1}{\frac{a(t_1)}{a(t_2)}-1}.
\end{align}
Then we have
\[
   \Phi (z,t_1,t_2) 
 =
    - \frac{\frac{a(t_1)}{a(t_2)} -1 - \frac{\omega(z,t_1,t_2)}{z} +1 }
    {\frac{a(t_1)}{a(t_2)}-1 +  \frac{a(t_1)}{a(t_2)} 
   \left(\frac{\omega(z,t_1,t_2)}{z} -1 \right) }  
 =
 - \frac{A(z,t_1,t_2)}{B(z,t_1,t_2)} .
\]
By Schwarz's lemma, it follows that
\begin{equation}
\label{ineq:|A|-is-less-than-or-eq-|B|}
  |\Phi (z,t_1,t_2)| \leq |z| 
  \; \text{and} \; |A(z,t_1,t_2)| \leq |z| |B(z,t_1,t_2)|
  , \quad z \in \mathbb{D} . 
\end{equation}
For $t \in I \setminus N$, we have
\[
 \2sidelim A(z,t_1,t_2) = 1- P(z,t) ,
 \quad
 \2sidelim B(z,t_1,t_2) = 1+ P(z,t) 
\]
and both limits  converge locally uniform on $\mathbb{D}$.

Since $P(0,t) = 1$,
$1+P(z,t) \not= 0$ on a neighborhood $U$ of the origin.
Therefore, for each $z \in U$, the limit
$\2sidelim \Phi (z,t_1,t_2)$ exists 
and equals to $-(1-P(z,t))/(1+P(z,t))$.
Moreover, 
the family 
$\{ \Phi(\cdot ,t_1,t_2) \}_{\alpha \leq t_1 < t_2 \leq \beta }$
is locally uniformly bounded.
Hence, by the Vitali--Porter convergence theorem, 
\[
   \Phi (z,t)
  :=  \2sidelim \Phi (z,t_1,t_2) 
 =  - \2sidelim \frac{A(z,t_1,t_2)}{B(z,t_1,t_2)} 
 =  - \frac{1-P(z,t)}{1+P(z,t)} 
\]
exists for every $z \in \mathbb{D}$
and the convergence is locally uniform on 
$\mathbb{D}$.
Clearly, we have $|\Phi(z,t)| \leq |z|$.
It follows that
$\Real P(z,t) = \Real \left( \frac{1+\Phi(z,t)}{1-\Phi(z,t)} \right) > 0$
in $\mathbb{D}$.


\noindent\textbf{Step 4.} 
We show (\ref{eq:first-a-derivative}) 
and (\ref{eq:second-a-derivative}).
Let $t, t_0,t_1,t_2 \in I$ with $t_1 \leq t \leq  t_2 < t_0$
and $t_2 - t_1>0$.
Put
\[
 \gamma( \lambda ) = (1-\lambda )z + \lambda \omega(z,t_1,t_2 ) ,
 \quad 0 \leq \lambda \leq 1.
\]
Then for $t \in I \setminus N$
letting $t_2 -t_1 \searrow 0$ with $t_1 \leq t \leq t_2$,
we have by (\ref{eq:differentail-wrt-a})
\begin{align*}
 \frac{\omega(z,t_2,t_0)-\omega(z,t_1,t_0)}{a(t_2)-a(t_1)}
 = & \,
 \frac{\omega(z,t_2,t_0)-\omega(\omega(z,t_1,t_2), t_2, t_0)}{a(t_2)-a(t_1)}
\\
 = & \, - \frac{\omega(z,t_1,t_2)-z}{a(t_2)-a(t_1)}
 \int_0^1 \omega'(\gamma(\lambda ),t_2,t_0) \, d \lambda 
\\
 \rightarrow & \, 
 - \frac{\partial \omega}{\partial a(t)} (z,t) \omega'(z,t,t_0) 
 = \frac{z}{a(t)} P(z,t) \omega'(z,t,t_0) .
\end{align*}
This implies (\ref{eq:first-a-derivative}).

Let $t_0 \in I$ and $t \in I \setminus N$ with $t_0 < t$.
Since the convergence of (\ref{eq:differentail-wrt-a})
is locally uniform on $\mathbb{D}$, 
letting $t_2 -t_1 \searrow 0$ with $t_1\leq t \leq t_2$
we have
\begin{align*}
  \frac{\omega(z,t_0,t_2)- \omega(z,t_0,t_1)}{a(t_2)-a(t_1)}
 = & \, 
  \frac{\omega(\omega(z,t_0,t_1),t_1,t_2)- \omega(z,t_0,t_1)}{a(t_2)-a(t_1)}
\\
 \rightarrow & \,
 \frac{\partial \omega}{\partial a(t)}(\omega(z,t_0,t),t)
\\
 = & \, 
 - \frac{1}{a(t)} \omega(z,t_0,t) P(\omega(z,t_0,t),t),
\end{align*}
which shows (\ref{eq:second-a-derivative}).
By the Vitali--Porter convergence theorem 
the limits on the left hand sides of 
(\ref{eq:first-a-derivative}) and (\ref{eq:second-a-derivative})
converge locally uniformly on $\mathbb{D}$.
\end{proof}

\section{Loewner--Kufarev Equations for Loewner Chains}
\begin{theorem}
\label{thm:Loewner-DEs}
Let $\{ f_t \}_{t \in I}$ be a strictly  increasing 
and continuous Loewner chain with 
$a(t) = f_t'(0)$, $t \in I$.
Then there exist a $G_\delta$ set $N (\subset I)$ with
$\mu_a(N) = 0$ and 
a Borel measurable Herglotz family $\{ P(\cdot , t) \}_{t \in I}$
such that 
\begin{align}
\label{eq:Loewner-PDE}
  \frac{\partial f}{\partial a(t)}  (z,t)
  =& \, \frac{z}{a(t)} P(z,t) f'(z, t) , \quad
  z \in \mathbb{D} \text{ and } t \in I \backslash N.
\end{align}
\end{theorem}
\begin{proof}
Let $\{ \omega (\cdot ,s,t) \}_{(s,t) \in I_+^2}$
be the associated transition family to $\{ f_t \}_{t \in I}$, 
and let $N$ and $\{ P(\cdot , t) \}_{t \in I}$ as in 
Theorem \ref{theorem:a-derivative-of-omega}.
Notice that since $\{ f_t \}_{t\in I}$ is continuous,
$f_{\tau}' \rightarrow f_t'$ locally uniformly on $\mathbb{D}$
as $\tau \rightarrow t$.

Let $t \in I \setminus N$.
Then
by letting $t_2 -t_1 \searrow 0$ 
with $t_1 \leq t \leq t_2$ 
\begin{align}
\label{eq:differential_wrt_a}
  & \, 
 \frac{f(z,t_2) -f(z,t_1)}{a(t_2)-a(t_1)}
\\
  =&
  \frac{f(z,t_2) -f(\omega(z,t_1,t_2),t_2)}{a(t_2)-a(t_1)}
\nonumber \\
  =&
  - \frac{\omega(z,t_1,t_2)-z}{a(t_2)-a(t_1)} 
 \int_0^1 f'((1-\lambda )z  + \lambda \omega(z,t_1,t_2) , t_2 ) 
 \, d\lambda  
\nonumber \\
 \rightarrow&
 - \frac{\partial \omega}{\partial a(t)} (z,t) f'(z,t) 
 =
    \frac{z}{a(t)}
    P(z,t)
    f'(z,t) .
\nonumber 
\end{align}
\end{proof}

\begin{corollary}
\label{corollary:the_convergence_of_a_derivative}
For $m \in \mathbb{N}$, $t \in I \backslash N$ and $z \in \mathbb{D}$
the limit
\[
  \frac{\partial}{ \partial a}
  \left( \frac{\partial^m f}{\partial z^m} \right)
 (z,t) 
 : =   
 \2sidelim 
 \frac{\frac{\partial^m f}{\partial z^m} (z , t_2) 
 - \frac{\partial^m f}{\partial z^m} (z , t_1)}{a(t_2)-a(t_1)} 
\]
exists and the convergence is locally uniform on $\mathbb{D}$.
Moreover we have 
\[
 \frac{\partial}{ \partial a}
 \left( \frac{\partial^m f}{\partial z^m}\right)
 (z,t) 
 =
 \frac{\partial^m}{\partial z^m} 
 \left( \frac{\partial f}{\partial a}\right)
 (z, t)  .
\]
\end{corollary}
\begin{proof}
Choose $r$ and $\rho$ such that $0< r < \rho < 1$.
Since the limit in the left-hand side of 
\eqref{eq:differential_wrt_a}
converges locally uniformly on $\mathbb{D}$,
by the Cauchy formula we have for $|z|\leq r $ that 
\begin{align*}
 \frac{\partial}{ \partial a}
  \left( \frac{\partial^m f}{\partial z^m} \right)
 (z,t) 
 = & \,  
 \2sidelim 
 \frac{\frac{\partial^m f}{\partial z^m} (z , t_2) 
 - \frac{\partial^m f}{\partial z^m} (z , t_1)}{a(t_2)-a(t_1)} 
\\
  = & \,  
 \2sidelim 
 \frac{m!}{2 \pi i } \int_{|\zeta|=\rho} \frac{1}{(\zeta -z)^m}
 \frac{f(\zeta , t_2) - f(\zeta , t_1)}{a(t_2)-a(t_1)} 
 \, d \zeta
\\
 = & \,  
 \frac{m!}{2 \pi i } \int_{|\zeta|=\rho} 
 \frac{\frac{\partial f}{\partial a}(\zeta,t)}{(\zeta -z)^m} 
 \, d \zeta ,
\end{align*}
and that the convergence is iniform on $\overline{\mathbb{D}}(0,r)$

While for fixed $t \in I \backslash N$,
$\frac{\partial f}{\partial a}(z,t)$ is analytic in $\mathbb{D}$,
 we have
by the Cauchy formula
\[
  \frac{\partial^m}{\partial z^m} 
 \left(\frac{\partial f}{\partial a}\right) (z,t)
 =
 \frac{m!}{2 \pi i } \int_{|\zeta|=r} 
 \frac{\frac{\partial f}{\partial a}(\zeta,t)}{(\zeta -z)^m} 
 \, d \zeta, 
\]
as required.
\end{proof}

We note that
Theorem \ref{thm:IntLoewner-DEs}
easily follows 
from Theorems \ref{theorem:a-derivative-of-omega} and 
\ref{thm:Loewner-DEs}.

\section{The Case of Absolutely Continuous $a(t)$}
We now consider how the differential equations 
in Theorems 3.4 and 3.5 are formulated when the function 
$a(t)$ is absolutely continuous.
To this end, we need a few lemmas.

\begin{lemma}
\label{lemma:growth_lemma}
Let $\psi: [\alpha , \beta ]  \to \mathbb{R}$ 
be a strictly increasing and continuous function,
$p \geq 0$, and $E \subset [\alpha , \beta ]$.
If 
\[
   D^- \psi (t) 
  := \liminf_{\substack{\scriptscriptstyle t_1 \leq t \leq t_2 \\
\scriptscriptstyle  t_2 - t_1 \searrow 0}} 
 \frac{\psi (t_2) - \psi(t_1)}{t_2-t_1} \leq p,
 \quad \text{for every } t \in E, 
\]
then $\mu_1^* (\psi (E)) \leq p \mu_1^*(E)$.
\end{lemma}
For a proof, see \cite[p.~207]{Natanson} or
\cite[Lemma~7.1]{BBT}.

\begin{lemma}[Banach--Zarecki]
Let $\psi: [\alpha , \beta ]  \to \mathbb{R}$ 
be a strictly increasing and continuous function. 
Then $\psi$ is absolutely continuous if and only if
it satisfies the Lusin (N) condition:
$\mu_1^*( \psi (A)) = 0$ whenever 
$A \subset [\alpha, \beta]$ and $\mu_1^*(A) = 0$.
\end{lemma}
Although the Banach--Zarecki theorem 
is typically stated for continuous functions of bounded variation, 
we provide a proof here 
in the case where $\psi$ is strictly increasing 
and continuous, 
since the argument becomes significantly simpler in this setting.
\begin{proof}
Suppose that $\psi$ 
is  absolutely continuous.
Let $A \subset [\alpha, \beta]$ with $\mu_1^*(A) = 0$.
We show that $\mu_1^*( \psi (A)) = 0$. 
Without loss of generality, we may assume 
$A \subset ( \alpha , \beta )$.

Let $\varepsilon > 0$ and choose $\delta > 0$ 
such that for any $\{ (\alpha_k, \beta_k) \}$ is a finite or countable
collection of disjoint open intervals in $[\alpha , \beta ]$
with $\sum_{k} (\beta_k - \alpha_k) < \delta$,
we have 
$\sum_{k} \big(\psi( \beta_k) - \psi (\alpha_k)\big) < \varepsilon$.
Choose an open set $G$ such that $A \subset G$ and 
$\mu_1(G) < \delta$.
Write $G = \bigcup_{k} (\alpha_k, \beta_k) $ as a union 
of the connected components of $G$.
Then, since $\sum_{k} (\beta_k - \alpha_k) = \mu_1(G) < \delta$,
we have $\sum_{k} (\psi( \beta_k) - \psi (\alpha_k)) < \varepsilon$.

Since the intervals 
$\{ \big( \psi (\alpha_k), \psi( \beta_k) \big) \}_{k}$ 
are disjoint, it follows that
\[
 \mu_1^*(\psi (A))
 \leq
 \mu_1^*(\psi (G))
 = \mu_1^* 
 \left( \bigcup_{k} \psi ( (\alpha_k, \beta_k))  \right)
 = \sum_{k}  \psi( \beta_k) - \psi (\alpha_k)) < \varepsilon .
\] 

Now suppose that $\psi$ is not absolutely continuous.
We show that there exists $A \subset [\alpha, \beta]$ 
with $\mu_1^*(A) = 0$ but $\mu_1^*( \psi (A)) > 0$.

Since $\psi$ is not absolutely continuous,
there exists $\varepsilon_0 > 0$ such that
for any $\delta > 0$, there exists 
finite or countable collection of
disjoint intervals $\{ (\alpha_k, \beta_k) \}$
in $[\alpha, \beta]$
such that $\sum_{k}  ( \beta_k - \alpha_k) < \delta$ and 
$\sum_{k}  \big( \psi( \beta_k) - \psi (\alpha_k) \big) 
\geq  \varepsilon_0$.
Choose a sequence of positive numbers
$\{ \delta_i \}_{i=1}^\infty$ 
such that $\sum_{i=1}^\infty \delta_i < \infty$.
For each $i \in \mathbb{N}$, choose 
a finite or countable family 
of disjoint intervals 
$\{ (\alpha_k^{(i)}, \beta_k^{(i)} ) \}_{k=1}^{n_i}$ 
in $[\alpha , \beta ]$
such that 
\begin{align*}
  \sum_{k=1}^{n_i} (\beta_k^{(i)} - \alpha_k^{(i)}) < \delta_i
 \quad \text{and} \quad
  \sum_{k=1}^{n_i}  
 \left( \psi( \beta_k^{(i)}) - \psi (\alpha_k^{(i)}) \right) 
 \geq  \varepsilon_0 .
\end{align*}
Define 
\begin{align*}
   E_i =  \bigcup_{k=1}^{n_i} (\alpha_k^{(i)}, \beta_k^{(i)} ), 
 \quad i \in \mathbb{N} 
\quad \text{and} \quad
  A =   \bigcap_{n=1}^{\infty} \bigcup_{i=n}^{\infty} E_i .
\end{align*}
Then,
\[
   \mu_1(A) 
 = \lim_{n \to \infty} 
 \mu_1 \left( \bigcup_{i=n}^{\infty} E_i \right)
 \leq 
 \lim_{n \to \infty} \sum_{i=n}^\infty \mu_1 \left(  E_i \right) 
 \leq 
 \lim_{n \to \infty} \sum_{i=n}^\infty \delta_i = 0 .
\]
Since $\psi$ is a bijection, we have
$\psi (A) = \bigcap_{n=1}^{\infty} \bigcup_{i=n}^{\infty} \psi (E_i)$
and
$\mu_1 (\psi (E_i)) 
= \sum_{k=1}^{n_i} 
\left( \psi( \beta_k^{(i)}) - \psi (\alpha_k^{(i)}) 
\right) 
 \geq  \varepsilon_0$.
Hence,
\[
   \mu_1( \psi (A)) 
 = \lim_{n \to \infty} 
 \mu_1 \left( \bigcup_{i=n}^{\infty} \psi (E_i) \right)
 \geq 
 \liminf_{n \to \infty} \mu_1 \left(  \psi (E_n ) \right) 
 \geq \varepsilon_0.
\]
\end{proof}

\begin{lemma}[Zarecki]
\label{lemma:Zarecki}
Let $\psi $ be a strictly increasing and continuous function 
on $[\alpha, \beta]$.
Define 
\begin{align*}
 E_\infty 
 = & \, \left\{ t \in [\alpha, \beta] : 
 D^- \psi (t) 
 := \liminf_{\substack{\scriptscriptstyle t_1 \leq t \leq t_2 \\
\scriptscriptstyle  t_2 - t_1 \searrow 0}} 
 \frac{\psi (t_2) - \psi(t_1)}{t_2-t_1} 
 = \infty \right\},
\\
 E_0
 = & \, \left\{ t \in [\alpha, \beta] : 
 D^+ \psi (t) 
 := \limsup_{\substack{\scriptscriptstyle t_1 \leq t \leq t_2 \\
\scriptscriptstyle  t_2 - t_1 \searrow 0}} 
 \frac{\psi (t_2) - \psi(t_1)}{t_2-t_1} 
 = 0 \right\}.
\end{align*}
Then, $\psi$ is absolutely continuous on $[\alpha , \beta]$
if and only if $\mu_1( \psi (E_\infty ) ) = 0$.
Moreover,
$\psi^{-1}$ is absolutely continuous 
on $[\psi(\alpha ) , \psi( \beta )]$ if and only if
$\mu_1( E_0 ) = 0$.
\end{lemma}
For references, see \cite[Chap.IX, Exercises 12 and 13]{Natanson}
and \cite[Exercises 3.33 and 3.45]{Leoni}.
For completeness we provide a proof below.
\begin{proof}
Suppose that $\psi$ is absolutely continuous.
Then, by the Banach--Zarecki theorem, 
it satisfies the Lusin (N) condition.
Since $\psi$ is nondecreasing, 
it is differentiable $\mu_1$-a.e.
Hence, $\mu_1(E_\infty) = 0$, and by the Lusin (N) condition
we also have $\mu_1(\psi(E_\infty)) = 0$.

Now suppose that $\mu_1(\psi(E_\infty)) = 0$.
Let $A \subset [\alpha, \beta ]$ with $\mu_1^* (A) = 0$.
For each $n \in \mathbb{N}$, define
\[
  A_n = A \cap \{ t \in [\alpha, \beta ] : 
 n-1 \leq D^- \psi (t) < n \} .
\]
Then we have the decomposition
$A = (A \cap E_\infty ) \bigcup_{n=1}^\infty A_n $,
and thus
$\psi (A) = \psi (A \cap E_\infty )  \bigcup_{n=1}^\infty \psi(A_n) $.
Then by Lemma~\ref{lemma:growth_lemma}
it follows that
$\mu_1^* (\psi(A_n)) \leq n \mu_1^* (A_n) \leq n \mu_1^* (A) = 0$.
Moreover, 
since $\mu_1 (\psi (A \cap E_\infty )) \leq \mu_1(E_\infty) = 0$,
we obtain
\[
 \mu_1^* ( \psi(A))
 \leq
 \mu_1 (\psi (A \cap E_\infty )) + \sum_{n=1}^\infty \mu_1^*(\psi (A_n))
 = 0 .
\]
Hence $\psi$ satisfies the Lusin (N) condition,
and by the Banach--Zarecki theorem, 
$\psi$ is absolutely continuous.

Finally, for $t \in [\alpha , \beta ]$ 
and $s = \psi (t) \in [\psi (\alpha), \psi(\beta)]$,
it is straightforward to verify that
$D^- \psi^{-1} (s) = \infty$ if and only if
$D^+ \psi (t) = 0$.
Therefore, the second statement follows immediately from the first.
\end{proof}

\begin{theorem}
\label{thm:abs-cont-Loewner-DEs}
Let $\{ \omega_{s,t} \}_{(s,t) \in I_+^2}$ be 
a strictly monotone and continuous transition family,
and let $a(t)$, $t \in I$, be a strictly  increasing 
and positive function 
defined by (\ref{eq:defi-of-at}) for some $c>0$.
If $a(t)$ is locally absolutely continuous on $I$ and
$\dot{a}(t) := \frac{d a}{dt}(t) > 0$ for 
$\mu_1$-almost every $t \in I$, 
then there exist a $G_\delta$ set $E \subset I$ with 
$\mu_1(E) = 0$, and 
a Borel measurable Herglotz family $\{ P(\cdot , t) \}_{t \in I}$ 
such that for $z \in \mathbb{D}$
\begin{align}
\label{eq:abs-cont-generator-wrt-a}
   \frac{\partial \omega}{\partial t}(z,t)
  =& \, - \frac{\dot{a}(t)}{a(t)} zP(z,t), \; 
 t \in I \backslash E ,
\\
\label{eq:abs-cont-downward-PDE-for-omega}
  \frac{\partial \omega }{\partial t } 
   (z,t, t_0)
  =& \, \frac{\dot{a}(t)}{a(t)} z P(z,t) \omega'(z, t,t_0) , \;
  t \in (I \cap (-\infty, t_0)) \backslash E ,
\\
\label{eq:abs-cont-upward-ODE-for-omega}
   \frac{\partial \omega }{\partial t } (z,t_0, t)
  =& \, 
  - \frac{\dot{a}(t) }{a(t)} \omega (z,t_0, t) P(\omega (z,t_0, t),t) , \;
  t \in (I \cap (t_0, \infty )) \backslash E .
\end{align}
In particular, if $\{ \omega( \cdot , s,t) \}_{(s,t) \in I_+^2}$
is associated with a strictly increasing 
and continuous Loewner chain $\{ f_t \}_{t \in I}$ 
satisfying $a(t) = f_t'(0)$,
then 
\begin{align}
\label{eq:abs-cont-Loewner-PDE}
   \frac{\partial f}{\partial t} (z,t)
  =& \, \frac{\dot{a}(t)}{a(t)} z P(z,t) f'(z, t) , \;
  t \in I \setminus E .
\end{align}
\end{theorem}
Note that in the special case that $a(t) =e^t$, $t \in I$,
equations 
\eqref{eq:abs-cont-upward-ODE-for-omega} and
\eqref{eq:abs-cont-Loewner-PDE} reduce to the classical 
Loewner--Kufarev ordinary and partial differential equations,
respectively.

\begin{proof}
It suffices to show the theorem in the case 
where $I = [\alpha, \beta ]$ 
with  $-\infty < \alpha <  \beta < \infty$.
Take a $G_\delta$ set $N \subset [\alpha, \beta ]$ with $\mu_a(N) = 0$ 
as in Theorem \ref{theorem:a-derivative-of-omega}.
Then by 
Proposition \ref{proposition:image_of_psi_measure_is_the_Lebesgue_measure}
we have $\mu_1 (a(N)) = \mu_a (N) = 0$.
From Lemma \ref{lemma:Zarecki}
it follows that $a^{-1}$ is absolutely continuous
and hence $a^{-1}$ has the Lusin (N) property, i.e.,
$a^{-1}$ maps a $\mu_1$-null set to a $\mu_1$-null set. 
Therefore $\mu_1(N) = \mu_1 (a^{-1}(a(N))) = 0$.

Let $E_0 (\subset [\alpha, \beta ])$ 
denote the set of all points $t \in I$ at which 
$a$ is not differentiable.
Since $a(t)$ is absolutely continuous on $[\alpha, \beta ]$, 
we have $\mu_1 (E_0) =0$.
Also, let $E_1 (\subset [\alpha, \beta ])$ be 
the set of all points $t \in I$ at which 
$a$ is differentiable but $\dot{a} (t) =  0$.
By assumption, $\mu_1(E_1) =0$.
Now take a $G_\delta$ set $E_2$ with $E_0 \cup E_1  \subset E_2$ 
and $\mu_1(E_2) = 0$.
Set $E = N \cup E_2$. 
Then $E$ is a $G_\delta$ set with 
$\mu_1 (E) = 0$.
By (\ref{eq:differentail-wrt-a}),
for $t \in I \backslash E$ and $z \in \mathbb{D}$, 
we have 
\begin{align*}
 \frac{\partial \omega}{\partial t}(z,t) 
 = & 
 \2sidelim 
  \frac{\omega(z,t_1,t_2) - z}{t_2-t_1}
\\
 = &
 \2sidelim 
  \frac{\omega(z,t_1,t_2) - z}{a(t_2)-a(t_1)} \cdot 
 \2sidelim 
  \frac{a(t_2)-a(t_1)}{t_2-t_1}
\\
 = &
 \frac{\partial \omega}{\partial t} (z,t) \cdot \frac{da}{dt} (t) 
  =  
  - \frac{\dot{a}(t)}{a(t)} z P(z,t).
\end{align*}

Similarly 
(\ref{eq:abs-cont-downward-PDE-for-omega}),%
\eqref{eq:abs-cont-upward-ODE-for-omega} and
(\ref{eq:abs-cont-Loewner-PDE})
follow from 
(\ref{eq:first-a-derivative}),
(\ref{eq:second-a-derivative}) 
and (\ref{eq:Loewner-PDE}) respectively.
\end{proof}

\chapter{Solutions to Loewner--Kufarev %
Ordinary Differential Equations}
\label{chapter:Solution}
\section{Differential Inequalities with Respect to $a(t)$}
Let $I$ be an interval, and 
$a(t)$ be a strictly  increasing, 
positive, and  continuous function on $I$.
Let $\mathcal{M}_a (I)$ denote the $\sigma$-algebra,
and $\mu_a$ the Lebesgue--Stieltjes measure
associated with the function $a(t)$.

In this chapter, for a given 
$\mathcal{M}_a (I)$-measurable
Herglotz family $\{ P (\cdot , t ) \}_{t \in I}$ 
we consider the ordinary differential equation
(\ref{eq:second-a-derivative}). 
Specifically, 
for each fixed $t_0 \in I$ and $z \in \mathbb{D}$,
we study the equation 
\begin{equation}
   D_a w (t)
   = -\frac{w(t)}{a(t)}P(w(t),t), 
   \quad t \in I \cap [t_0, \infty )
\end{equation}
subject to the initial condition
\begin{equation}
  w(t_0) = \, z .
\end{equation}
Before solving the differential equation, 
we collect some preparatory results
and establish several auxiliary lemmas. 
We then construct a solution using 
the method of successive approximations.

Let $I_0$ be a compact subinterval of $I$, 
and let $c_1,c_2 \in \mathbb{C}$. 
Suppose that  $u$, $v$  are 
absolutely $a$-continuous functions  on $I_0$.
Then 
both $c_1 u + c_2 v$ and  $uv$ are also 
absolutely $a$-continuous on $I_0$,
and $D_a (c_1 u + c_2 v)(t)
= c_1D_a u (t) + c_2 D_a v (t)$,
$D_a (uv) (t)
=  D_a u(t) \cdot v(t) + u(t) \cdot D_a v (t)$ 
hold $\mu_a$-a.e.
Furthermore if $h$ is a $C^1$ function defined 
on an interval containing $u(I_0)$,
then $h \circ u$ is absolutely $a$-continuous on $I_0$
and 
\begin{equation}
\label{eq:chain_rule_wrp_to_a}
   D_a (h \circ u)(t) =  Dh (u (t))D_a u (t)
\end{equation}
holds $\mu_a$-a.e.
In particular, applying 
this to the function $h(s) = \left( \log \frac{s}{a(t_0)} \right)^n $
and $u(t)=a(t)$,
we obtain from Proposition \ref{proposition:fundamental-thm-of-calculus}
that for $n \in \mathbb{N}$ and
$t  \in I \cap [t_0,\infty ) $,
\begin{equation}
\label{eq:log-to-n-th-power}
 \left( \log \frac{a(t)}{a(t_0)} \right)^n
    =
    n \int_{[t_0,t]} \frac{1}{a(\tau )} 
    \left( \log \frac{a(\tau )}{a(t_0)} \right)^{n-1} \, d \mu_a(\tau ) .
\end{equation}

We require differential inequalities with respect to $a(t)$.
\begin{lemma}
\label{lemma:differentail-inequalities}
Let $u$ be an absolutely $a$-continuous function 
on $[\alpha , \beta ] \subset I$ 
with $-\infty < \alpha < \beta < \infty$. 
Suppose that for some positive constant $M$, 
$u$ satisfies
\begin{equation}
\label{ineq:diffrential-inequality}
   \left| D_a u(t) \right|
   \leq \frac{M}{a(t)} |u(t)| \quad \text{$\mu_a$-a.e.} 
\end{equation}
Then for $t \in [\alpha , \beta ]$
\begin{align*}
& \,
 |u(\alpha )| \left( \frac{a(\alpha )}{a(t)} \right)^M 
 \leq
 |u(t)|
 \leq
 |u( \beta )| \left( \frac{a(\beta )}{a(t)} \right)^M ,
\\ 
& \,
 |u(\beta )| \left( \frac{a(t) }{a(\beta )} \right)^M 
 \leq
 |u(t)|
 \leq
  |u( \alpha )| \left( \frac{a(t)}{a(\alpha )} \right)^M .
\end{align*}
\end{lemma}
\begin{proof}
Since $u$ is absolutely $a$-continuous on $[\alpha, \beta]$,
so is $|u|$.
Thus, for $\mu_a$ almost every $t \in I$
\begin{align*}
   \left| D_a |u| (t)\right|
 = & \, \left| \2sidelim \frac{|u(t_2)|-|u(t_1)|}{a(t_2)-a(t_1)}
 \right|
\\
 \leq & \, 
 \left| \2sidelim \frac{u(t_2)-u(t_1)}{a(t_2)-a(t_1)} \right|
 = \left| D_a u(t) \right| .
\end{align*}
From this and (\ref{ineq:diffrential-inequality}),
it follows that for $\mu_a$ almost every $t \in I$ 
\begin{align*}
 D_a \{ |u(t)| a(t)^{M} \} 
 \geq - \left| D_a u(t) \right|a(t)^{M} + M |u(t)| a(t)^{M-1} 
 \geq 0 , 
\\
 D_a \{ |u(t)| a(t)^{-M} \} 
\leq \left| D_a u(t) \right|a(t)^{-M} - M |u(t)| a(t)^{-M-1}
 \leq 0 .
\end{align*} 
Thus
$|u(t)|a(t)^M$ is nondecreasing, 
and $|u(t)|a(t)^{-M}$ is nonincreasing
in $t$.
Hence, for $t \in [\alpha, \beta]$
\begin{align*}
& \, |u(\alpha)| a(\alpha )^M 
 \leq |u(t)|a(t)^M \leq |u(\beta)| a(\beta )^M ,
\\
& \, |u(\alpha)| a(\alpha )^{-M} 
\geq |u(t)|a(t)^{-M} 
\geq |u(\beta)| a(\beta)^{-M}
\end{align*}  
as required.
\end{proof}

\section{Analytic Estimates and Measurability Inputs}
We now state some useful estimates for analytic functions 
with positive real part.
For details, see \cite[\S 2.1]{Pommerenke1}.
\begin{lemma}
\label{lemma:positive-real-part}
Let $p \in \mathcal{H}(\mathbb{D})$ with $\Real P(z) > 0$ and $p(0)=1$.
Then, for $z \in \mathbb{D}$,  
the following inequalities hold:
\begin{align*}
 \left| p(z)-1 \right| 
 \leq \frac{2|z|}{1-|z|}, \quad 
 \frac{1-|z|}{1+|z|} \leq  |p(z)| \leq \frac{1+|z|}{1-|z|}, \quad 
  |p'(z)| \leq \frac{2}{(1-|z|)^2} .
\end{align*}
\end{lemma}

\begin{lemma}
\label{lemma:measurability-of-w(z,t)}
Let $\mathcal{F}$ be a $\sigma$-algebra on $I$ that contains
the Borel $\sigma$-algebra $\mathcal{B}(I) $ on $I$.
Let $\{ P( \cdot , t) \}_{t \in I}$ be an
${\mathcal F}$-measurable Herglotz family, and let 
$w(z,t)$ be a function on $\mathbb{D} \times I$ 
such that $w(z,t)$ is continuous in $t$ 
for each fixed $z \in \mathbb{D}$
and is analytic in $z$ for each fixed $t \in I$.
Then 
$P(w(z,t),t)$ is ${\mathcal F}$-measurable in $t$ 
for each fixed $z \in \mathbb{D}$,
and analytic in $z$ for each fixed $t \in I$.
\end{lemma}
\begin{proof}
For each fixed $t \in I$, it is clear that 
$P(w(z,t),t)$ is analytic in $z$.
For each $k \in \mathbb{N}$, take a sequence of 
disjoint Borel subsets $\{ S_j^{(k)} \}_{j=1}^{N_k} $ of $\mathbb{D}$
such that $\mathbb{D} = \cup_{j=1}^{N_k} S_j^{(k)}$ and  
$\diam (S_j^{(k)} ) := \sup_{w,z \in S_j^{(k)}} |w-z| 
\leq \frac{1}{k}$.
For each $k \in \mathbb{N}$ 
and $1 \leq j \leq N_k$,
choose $\zeta_j^{(k)} \in S_j^{(k)}$ arbitrarily, 
and define
$P_k(z,t)$, $(z,t) \in \mathbb{D} \times I$ by 
\[
   P_k(z,t) = P(\zeta_j^{(k)} , t) \quad \text{for } z \in S_j^{(k)} .
\]
Then $P_k(z,t) \to P(z,t)$ as $k \to \infty$ 
for all $(z,t) \in \mathbb{D} \times I$.
Therefore $P_k(w(z,t) , t) \to P(w(z,t),t)$ 
as $k \rightarrow \infty$.
Thus, to show the lemma, it suffices to show that
$P_k(w(z,t),t)$ is $\mathcal{F}$-measurable in $t$ 
for each fixed $z \in \mathbb{D}$.

For an open set $V \subset \mathbb{C}$,
we have
\[
  \{ (z,t) \in \mathbb{D} \times I : P_k(z,t) \in V \}
 = \bigcup_{j=1}^{N_k} 
 S_j^{(k)} 
 \times 
 \{ t \in I : P(\zeta_j^{(k)},t) \in V \} .
\]
Thus, for each fixed $z \in \mathbb{D}$, 
we obtain 
\begin{align}
\label{eq:mathcalF-measurable}
 & \{ t \in I : P_k(w(z,t), t) \in V \}
\\
 = & 
 \bigcup_{j=1}^{N_k} 
 \{t \in I : w(z,t) \in S_j^{(k)} \}
 \cap \{ t \in I : P(\zeta_j^{(k)},t) \in V \} .  
\nonumber
\end{align}
Since $\{t \in I : w(z,t) \in S_j^{(k)} \} 
\in \mathcal{B}(I) \subset \mathcal{F}$
and $\{ t \in I : P(\zeta_j^{(k)},t) \in V \} \in \mathcal{F}$,
the set on the right-hand side of (\ref{eq:mathcalF-measurable}) 
is clearly $\mathcal{F}$-measurable.
\end{proof}

\section{Existence and Uniqueness of Solutions}
\begin{theorem}
\label{thm:solution_to_LoewnerODE}
Let $I$ be an interval, 
$a(t)$ a strictly increasing, positive and continuous function
on $I$, 
and 
$\{ P(\cdot , t) \}_{t \in I}$ 
an $\mathcal{M}_a (I)$-measurable Herglotz family.
Then, for each fixed $s \in I$ and $z \in \mathbb{D}$, 
there exists a unique, locally absolutely 
$a$-continuous function 
$w: I \cap [s, \infty ) \to \mathbb{D}$ 
satisfying the differential equation
\begin{equation}
\label{eq:Loewner-ODE-2}
   D_a w (t)
   = 
   - \frac{w(t)}{a(t)} P(w(t),t), 
 \quad \text{$\mu_a$-a.e.} 
\end{equation}
with the initial condition $w(s)=z$.
Furthermore, 
for $(s,t) \in I_+^2$ and $z \in \mathbb{D}$, 
let 
$\omega_{s, t} (z)$ denote the unique solution to
(\ref{eq:Loewner-ODE-2}) 
with $\omega_{s,s}(z) = z$.
Then the family 
$\{ \omega_{s,t}\}_{(s,t) \in I_+^2}$ is a transition family
satisfying $\omega_{s,t}'(0) = a(s)/a(t)$, $(s,t) \in I_+^2$.
In particular, $\{ \omega_{s,t}\}_{(s,t) \in I_+^2}$ is continuous,
and each $\omega_{s, t} (z)$ is univalent in $\mathbb{D}$.
\end{theorem}
The following proof is a straightforward generalization of 
Theorem 6.3 in Pommerenke \cite{Pommerenke1}.

\begin{proof}
\noindent\textbf{Step 1.}  
Let $z \in \mathbb{D}$, and let $t_0, t_1 \in I$ with $t_0 < t_1$.
Suppose that 
$w: [t_0,t_1] \to \mathbb{D}$ be a function with $w(t_0)=z$.

We claim that $w(t)$ is an absolutely $a$-continuous solution
to (\ref{eq:Loewner-ODE-2}) on $[t_0,t_1]$ 
if and only if it is a continuous solution 
to the integral equation
\begin{equation}
\label{eq:equivalent-integral-equation}
  w(t) 
 = 
 z \exp 
 \left[ - \int_{[t_0,t]} 
 \frac{1}{a(\tau)} P(w(\tau), \tau) \, d \mu_a (\tau ) \right] 
\end{equation}
on $[t_0,t_1]$.

Suppose that $w: [t_0,t_1] \to \mathbb{D}$ is continuous.
Let $\rho = \max_{t_0 \leq t \leq t_1 } |w(t)| \in [0,1)$, and
set $M = \frac{1+\rho}{1-\rho}$.
Then, by Lemma~\ref{lemma:positive-real-part},
we have $|P(w(\tau), \tau)| \leq M$ on $[t_0,t_1]$.
Thus the function $a(\tau)^{-1} P(w(\tau), \tau)$
is $\mu_a$-integrable. 
It follows from (\ref{eq:chain_rule_wrp_to_a})
and Proposition \ref{proposition:fundamental-thm-of-calculus}
that
$w(t)$ is absolutely $a$-continuous 
and satisfies (\ref{eq:Loewner-ODE-2}) with
$w(t_0)=z$.

Now, suppose that $w(t)$ is  an absolutely $a$-continuous solution
to (\ref{eq:Loewner-ODE-2}) on $[t_0,t_1]$ with $w(t_0)=z$.
Assume first that $z \not= 0$.
Set $\rho = \max_{t_0 \leq t \leq t_1 } |w(t)| \in [0,1)$ and
$M = \frac{1+\rho}{1-\rho}$.
Then, by Lemma \ref{lemma:positive-real-part}, we have 
$|D_a w(t)| \leq M a(t)^{-1} |w(t)| $.
Hence, by Lemma  \ref{lemma:differentail-inequalities},
we obtain  
$|w(t)| \geq |z|(a(t_0)/a(t_1))^M > 0$.
Choose $\zeta \in \mathbb{C}$ such that $z = e^{\zeta}$.
Then there exists a single-valued 
branch $\log w(t)$ 
satisfying $\log w(t_0) = \zeta$,
i.e., $u(t) = \log w(t)$ is the unique continuous function 
on $[t_0,t_1]$ such that $w(t) = e^{u(t)}$ and $u(t_0) = \zeta$.
It is easy to see that  
$\log w(t)$ is absolutely $a$-continuous on $[t_0,t_1]$,
and 
\[
    D_a ( \log w(t)) 
 = \frac{1}{w(t)} D_aw(t)
 = - \frac{1}{a(t)} P(w(t),t) \quad \text{$\mu_a$-a.e.}
\]
Integration then yields
\[
   \log w(t) - \log w(t_0)
  = - \int_{[t_0,t]} 
 \frac{1}{a( \tau )} P(w(\tau), \tau ) \, d \mu_a ( \tau ) ,
\]
which is equivalent to (\ref{eq:equivalent-integral-equation}).
Moreover, taking real parts of both sides shows 
that $|w(t)|$ is strictly decreasing in $t$.

In the case that $z=0$,
it follows from Lemma \ref{lemma:differentail-inequalities}
that $w(t)= 0$ on $[t_0,t_1]$, which clearly solve
(\ref{eq:equivalent-integral-equation})
with $w(t_0)=0$.

\noindent\textbf{Step 2.}
We construct a sequence of functions that approximate the solution
to (\ref{eq:equivalent-integral-equation}). 
Define $w_1 (z, t) \equiv z$ 
for $(z,t) \in \mathbb{D} \times [t_0,t_1]$. 
Suppose that inductively we have 
a function $w_n(z,t)$ satisfying:
\begin{itemize}
 \item[{\rm (a)}]
For each fixed $z \in \mathbb{D}$, 
$w_{n}(z,t)$ is continuous in 
$t \in [t_0, t_1]$
with $w_n(z,t_0)= z$. 
 \item[{\rm (b)}]
For $z \in \mathbb{D}$ and $t \in [t_0, t_1]$,
we have $|w_n(z,t)| \leq |z|$.
 \item[{\rm (c)}]
For each fixed $t \in [t_0 , t_1]$, 
$w_{n}(z,t)$ is analytic in $z \in \mathbb{D}$.
\end{itemize}
Then, by Lemma \ref{lemma:measurability-of-w(z,t)},
for each fixed $z \in \mathbb{D}$,
the function 
$P(w_n(z,t),t)$ is $\mathcal{M}_a (I)$-measurable in $[t_0, t_1]$.
By (b) and Lemma \ref{lemma:positive-real-part}, we have 
\[ 
   \left| 
  \frac{1}{a(t)} P(w_n(z,t ), t)
  \right|
  \leq
  \frac{1}{a(t_0)} \frac{1+|z|}{1-|z|}, 
 \quad (z,t ) \in \mathbb{D} \times [t_0, t_1] .
\]
Hence, $P(w_n(z,t),t)$ is $\mu_a$-integrable on $[t_0,t_1]$.
Therefore, we define
\[
  w_{n+1}(z,t) 
 = 
 z \exp \left[ - \int_{[t_0,t]}
 \frac{1}{a(\tau)} P(w_n(z, \tau ), \tau) \, d \mu_a (\tau ) \right] ,
 \quad t \in [t_0, t_1]. 
\] 
Clearly $w_{n+1}$ satisfies (a).
Since  
$\Real P(z,t) > 0$ and $a(t) > 0$, 
it also satisfies (b).
Let $q(z,t) = P(w_n(z,t),t)$. 
Then $q(z,t)$ is analytic in $z$ with
$\Real q(z,t) > 0$ and $q(0,t)=1$.
By Lemma \ref{lemma:positive-real-part},
$|q'(z,t)| \leq \frac{2}{(1-r)^2}$ 
for $|z| \leq r < 1$ and $t \in I \cap [t_0, \infty ) $.
Fix $t \in I$ with $t > t_0$, 
and define 
\[
  h(z) = \int_{[t_0.t]} \frac{1}{a(\tau )} q(z,\tau ) \, d \mu_a( \tau ).
\]
Then, by the Lebesgue dominated convergence theorem,
\begin{align*}
 \frac{h(z+\Delta z)-h(z)}{\Delta z }
 = & \, 
 \int_{[t_0.t]} \frac{1}{a(\tau )} 
 \left\{ \int_0^1 q'(z + \theta \Delta z , \tau )\, d \theta \right\} 
 \, d \mu_a( \tau )
\\ 
 \rightarrow & \, 
 \int_{[t_0.t]} \frac{1}{a(\tau )} q'(z,\tau  ) \, d \mu_a( \tau ),
\end{align*}
as $\Delta z \rightarrow 0$ with 
$z, z + \Delta z \in \overline{\mathbb{D}}(0,r)$. 
Hence $h(z)$ is analytic in $z$,
and so $w_{n+1}(z,t)$ satisfies (c).

\noindent\textbf{Step 3.}
We show that $\{ w_n(z,t)\}_{n=1}^\infty$ converges locally uniformly on 
$\mathbb{D} \times [t_0,t_1]$, and that 
the limit function $w(z,t) := \lim_{n \rightarrow \infty } w_n(z,t)$ 
satisfies the integral equation (\ref{eq:equivalent-integral-equation}).
Note that from the locally uniform convergence it follows that
$w(z,t)$ satisfies conditions (a), (b) and (c).

For $\Real a \geq 0$ and $\Real b \geq 0$,
we have
\begin{align}
  |e^{-b}-e^{-a}| 
  = & \, 
  \left| (b-a) \int_0^1 e^{-(1-\lambda )a-\lambda b } \, d \lambda 
  \right| 
\\ 
 \leq & \, 
  |b-a| \int_0^1 e^{-(1-\lambda ) \Real a - \lambda \Real b } \, d \lambda 
  \leq  |b-a|.
\nonumber
\end{align}
Using this
and Lemma~\ref{lemma:positive-real-part}, for $n\geq 2$,
we obtain
\begin{align*}
 & |w_{n+1}(z,t)-w_n(z,t)| 
\\
 \leq&
 |z| 
  \int_{[t_0,t]} \frac{1}{a( \tau )} 
   \left|  P(w_n(z,\tau), \tau ) - P(w_{n-1}(z,\tau), \tau ) 
  \right|
  \, d \mu_a (\tau ) 
\\
 \leq&
 \frac{2|z|}{(1-|z|)^2}
 \int_{[t_0,t]} \frac{1}{a( \tau )} 
   \left| w_n(z,\tau) - w_{n-1}(z,\tau)
  \right|
  \, d \mu_a (\tau ) .
\end{align*}
Similarly, using 
(\ref{eq:log-to-n-th-power}), we have
\begin{align*}
 & |w_2(z,t)-w_1(z,t)| 
\\
 =&  |z| 
 \left| 
   \exp \left[ - \int_{[t_0,t]} \frac{1}{a( \tau )} 
   P(z, \tau ) \, d \mu_a (\tau ) \right] 
   -
   1 
 \right|
\\
 \leq&
 |z| \left|
  \int_{[t_0,t]} \frac{1}{a( \tau )} 
     P(z, \tau )
    \, d \mu_a (\tau ) 
 \right|
\\
 \leq&
  \frac{|z|(1+|z|)}{1-|z|}
  \int_{[t_0,t]} \frac{1}{a( \tau )} 
    \, d \mu_a( \tau ) 
 = \frac{|z|(1+|z|)}{1-|z|} \log \frac{a(t)}{a(t_0)} .
\end{align*}
Hence by induction and (\ref{eq:log-to-n-th-power}), we obtain
\[
  |w_{n+1}(z,t)-w_n(z,t)|
  \leq
  \frac{|z|(1+|z|)}{n!(1-|z|)}
  \left( 
  \frac{2|z|}{(1-|z|)^2} \right)^{n-1}
  \left( \log \frac{a(t)}{a(t_0)} \right)^n .
\]
Thus $\{ w_n(z,t)\}_{n=1}^\infty$ converges locally uniformly on 
$\mathbb{D} \times [t_0,t_1]$. 
Define $w(z,t) := \lim_{n \rightarrow \infty } w_n(z,t)$. 
It is easy to verify that $w(z,t)$ satisfies
conditions (a), (b) and (c).
Since
\begin{align*}
&    \left|  
   \int_{[t_0,t]} 
  \frac{P(w_n(z,\tau), \tau ) }{a( \tau )} \, d \mu_a (\tau ) 
   - 
   \int_{[t_0,t]} 
 \frac{P(w(z,\tau), \tau ) }{a( \tau )} \, d \mu_a (\tau ) 
  \right|
\\
&  \leq
 \frac{2}{(1-|z|)^2}
  \frac{|w_n(z,\tau)- w(z,\tau)| }{a( \tau )} \, d \mu_a (\tau ) 
 \to 0 \quad \text{as } n \to \infty ,
\end{align*}
we have
\begin{align*}
   w(z,t) 
 & = \lim_{n \to \infty} w_n(z,t)
\\
 & = z \exp \left[- \int_{[t_0,t]} 
 \frac{P(w_n(z,\tau), \tau ) }{a( \tau )} \, d \mu_a (\tau ) \right]
\\
 & = z \exp \left[- \int_{[t_0,t]} 
 \frac{P(w(z,\tau), \tau ) }{a( \tau )} \, d \mu_a (\tau ) \right] .
\end{align*}
Thus $w(z,t)$ satisfies the integral equation 
(\ref{eq:equivalent-integral-equation}),
and hence it solves the differential equation (\ref{eq:Loewner-ODE-2}).

\noindent\textbf{Step 4.}
We show the uniqueness of the solution $w(z,t)$
and its univalence in $z$.
Assume that 
another function 
$\tilde{w}(z, t)$ is absolutely $a$-continuous in $t \in [t_0,t_1]$,
and satisfies $\tilde{w}(z, t_0) =z$ and the differential equation 
(\ref{eq:Loewner-ODE-2}).
Since $\tilde{w}(z, \cdot )$ satisfies 
(\ref{eq:equivalent-integral-equation}),
we have 
$|\tilde{w}(z,t)| \leq |z|$ holds for $t \in I \cap [t_0, t_1]$.

By Lemma \ref{lemma:positive-real-part}, 
$|\{ z P(z,t) \}'| \leq |P(z,t)| + |zP'(z,t)| \leq \frac{2}{(1-|z|)^2}$,
so for $z, \tilde{z} \in \overline{\mathbb{D}}(0,r)$, 
$|zP(z,t)-\tilde{z}P(\tilde{z},t)| \leq \frac{2|z-\tilde{z}|}{(1-r)^2}$.
In particular, for each $z \in \mathbb{D}$, we obtain
\[
   \left| D_a \left( w(z,t) - \tilde{w}(z,t) \right) \right|
 \leq
  \frac{2\left| w(z,t) - \tilde{w}(z,t) \right|}{a(t)(1-r)^2}.
\]
Since $w(z,t_0)-\tilde{w}(z,t_0) =z-z=0$,
Lemma \ref{lemma:differentail-inequalities} implies
$w(z,t) = \tilde{w}(z,t)$ for all $t \in I \cap [t_0, t_1]$.
Thus the solution is unique.

Similarly, since  for $z_1 , z_2 \in \overline{\mathbb{D}}(0,r)$, 
\[
  |D_a(w(z_1,t) - w(z_2,t))| 
 \leq \frac{2|w(z_1,t) - w(z_2,t)|}{a(t)(1-r)^2} ,
\]
we have, by $w(z_1,t_0)-w(z_2,t_0) = z_1-z_2 $ 
and Lemma \ref{lemma:differentail-inequalities}, 
\[
   |w(z_1,t) - w(z_2,t)| \geq |z_1-z_2| 
 \left( \frac{a(t_0)}{a(t)} \right)^{\frac{2}{(1-r)^2}} .
\]
Thus $w(z,t)$ is univalent in $z$.


\noindent\textbf{Step 5.}
Now we write $\omega(z,t_0,t)$ instead of $w(z,t)$
for  $(t_0 ,t) \in I_+^2$ and $z \in \mathbb{D}$.
We show that $\{ \omega (\cdot, s,t) \}_{(s,t)\in I_+^2}$ 
forms a transition family 
with $\omega' (0, s,t) = \frac{a(s)}{a(t)}$
for all $(s,t) \in I_+^2$.
Note that, 
by the initial condition, 
we clearly have $\omega (z,s,s) \equiv z$
for all $s \in I$.

Let $(t_0, t_1) \in I_+^2$.
As a function of $t \in I \cap [t_1, \infty)$, 
$\omega(z,t_0,t)$ and $\omega(\omega(z,t_0,t_1),t_1,t)$
satisfy the same equation (\ref{eq:Loewner-ODE-2})
with the same initial condition 
$\omega (z,t_0,t_1) = \omega(\omega(z,t_0,t_1),t_1,t_1)$.
Therefore, by uniqueness, we have
\[
  \omega(z,t_0,t) = \omega(\omega(z,t_0,t_1),t_1,t),
\]
and thus $\{ \omega (\cdot, s,t) \}_{(s,t)\in I_+^2}$ 
forms a transition family.

Next, by $\omega(0,t_0,t)=0$ and $P(0,t)=1$, we have 
\begin{align*}
  \omega'(0,t_0,t) 
 = \lim_{z \rightarrow 0} \frac{\omega(z,t_0,t)}{z}
 =& \lim_{z \rightarrow 0} 
 \exp \left[ - \int_{[t_0,t]} \frac{1}{a(\tau)}
 P(\omega (z,t_0,\tau ),\tau)  \, d \mu_a( \tau ) \right]
\\
 =& \exp \left[ 
 - \int_{[t_0,t]} \frac{1}{a(\tau)} \, d \mu_a ( \tau ) \right]
 = \frac{a(t_0)}{a(t)}.
\end{align*}
\end{proof}

\section{Integral Representation of the Limit Chain}
Suppose that 
$I$ is a right-open interval 
with $\beta = \sup I (\not\in I)$.
In \S \ref{chapter:BasicEstimates}, we saw
that if $\{ \omega_{s,t}\}_{(s,t) \in I_+^2}$ is a continuous 
transition family,
the locally uniform limit 
$g_s(z) := \lim_{t \nearrow \beta} a(t) \omega(z,s,t)$
exists on $\mathbb{D}$, and 
the family $\{ g_s \}_{s \in I}$ forms a Loewner chain 
whose associated transition family is 
$\{ \omega_{s,t}\}_{(s,t) \in I_+^2}$.
If $a(t)$ is strictly increasing and continuous,
then an integral representation of $g_s$ can be obtained.

By (\ref{eq:equivalent-integral-equation}),
we have
\[
    a(t)\omega(z,s,t)
    =
    a(s) z \exp 
    \left[ 
    \int_{[s,t]}
    \frac{1}{a(\tau)} 
   \{ 1 - P(\omega (z,s,\tau ),\tau)\}
    \, d \mu_a ( \tau )
    \right] .
\]
Since $\frac{a(t)}{a(s)} \omega( \cdot ,s,t) $ is univalent,
the growth theorem implies
\[
    |\omega (z,s,t )| \leq \frac{a(s)|z|}{a(t)(1-|z|)^2} .
\]
Combining this with Lemma~\ref{lemma:positive-real-part}
and the fact that $|\omega(z,s,t)|\leq |z|$, 
we obtain for $|z| \leq r$
\begin{align*}
  \frac{1}{a(t)}|1- P(\omega(z,s,t),t)| 
 \leq&
  \frac{1}{a(t)}
 \frac{2|\omega(z,s,t)|}{(1-|\omega(z,s,t)|)}
\\
 \leq&
  \frac{2a(s) r}{a(t)^2 (1-r)^3}.
\end{align*}
Since $\frac{1}{a(t)^2}$ is $\mu_a$-integrable on $[t_0, \beta )$,
we obtain the expression 
\begin{align}
    g_s (z)
    = \, & 
    \lim_{t \nearrow  \beta } a(t) \omega(z,s, t)
\\
    = \, &
    a(s) z \exp 
    \left[ 
    \int_{[s, \beta )}
    \frac{1}{a(\tau)} \{1- P(\omega(z,s,\tau),\tau) \}
    \, d \mu_a( \tau)
    \right] .
\nonumber
\end{align}

\chapter{Schlicht Subordination and Connecting Chain}
\label{chapter:SchlichtSubordination}
\section{Classes, Expanding Behavior, Boundedness %
and a Counterexample}
Let $I \subset [-\infty, \infty]$ be an interval.
We begin by illustrating, with examples, 
some simple differences between 
(I) the class of Loewner chains of univalent functions on $I$
and (III) the class of general Loewner chains on $I$.

Let $\{ f_t \}_{t \in I}$ be a Loewner chain.
Then the family $\{ f_t (\mathbb{D}) \}_{t \in I}$ 
of domains in $\mathbb{C}$ 
is nondecreasing; that is,
$f_s (\mathbb{D}) \subset f_t (\mathbb{D}) $ for $(s,t) \in I_+^2$.
Note that the function $a(t) := f_t'(0)$, $t \in I$,
is also nondecreasing, 
and recall that $\{ f_t \}_{t \in I}$ is said to be 
strictly increasing if $a(t)$ is strictly increasing.  

\begin{definition}
A Loewner chain $\{ f_t \}_{t \in I}$ is said to be  
\emph{strictly expanding} if 
the associated 
family of domains 
$\{ f_t(\mathbb{D} ) \}_{t\in I}$ 
is strictly increasing; that is, 
\begin{equation}
f_{t_1}(\mathbb{D}) \subsetneq f_{t_2}(\mathbb{D}) \quad
\text{whenever $t_1,t_2 \in I$ with $t_1<t_2$.}
\end{equation} 
\end{definition}

By the uniqueness part of the Schwarz lemma, 
if a Loewner chain $\{ f_t \}_{t \in I}$ is strictly expanding, 
then the function $a(t) = f_t '(0)$, $t \in I$, is strictly increasing.
When $\{ f_t \}_{t \in I}$ consists of univalent functions,
the converse is also true.
However, if the functions $f_t$ are not required to be univalent,
the converse fails.
We present a simple counterexample.
\begin{example}
\label{eg:increasing_not_expanding}
For $t>0$, let $g_t$ be the conformal mapping of
$\mathbb{D}$ onto the rectangle 
$\{ w \in \mathbb{C}: |\Real w|<1 , |\Imaginary w | < t \;\}$,
with $g_t(0)=0$ and $g_t'(0)>0$.
Define $f_t = e^{g_t}$, $t \in I$.
Then it is easy to see that 
the function $f_t ' (0) (=g_t'(0))$ is strictly increasing in $t$
and so is  the family $\{ f_t \}_{t > 0}$ by definition. 
However, we have
$f_t(\mathbb{D}) = \{ w \in \mathbb{C}:  e^{-1} < |w|  < e \}$
for any $t > \pi $.
Thus, the family $\{ f_t \}_{t > 0 }$ is not strictly expanding.
\end{example}

A Loewner chain $\{ f_t\}_{t \in I}$
is bounded  on 
$I \cap [- \infty ,t_0] \times \overline{\mathbb{D}}(0,r)$
for every $t_0 \in I$ and $r \in (0,1)$.  
Indeed, let $\{ \omega_{s,t}\}_{(s,t) \in I_+^2}$ be 
the associated transition family.
Then, for $t  \in I$ with $t \leq t_0 $ and $r \in (0,1)$,
we have, by the Schwarz lemma,
\begin{equation}
\label{ineq:bound_for_maximum_on_subdisk}
    \max_{z \in \overline{\mathbb{D}}(0,r)} |f_t(z)| 
   =  \max_{z \in \overline{\mathbb{D}}(0,r)} |f_{t_0}(\omega_{t,t_0}(z))| 
   \leq \max_{\zeta \in \overline{\mathbb{D}}(0,r)} |f_{t_0}(\zeta)|.  
\end{equation}

Suppose that 
$\{ f_t \}_{t \in I}$ 
is a Loewner chain of univalent functions. 
Then, by the growth theorem for univalent analytic functions, we have
\[
    |f_t(z)| \leq \frac{f_t' (0) |z|}{(1-|z|)^2} , 
    \quad z  \in \mathbb{D}.
\]
Therefore, for any $M >0 $, 
the class of Loewner chains 
$\{ f_t \}_{t \in I}$ of univalent functions defined on some interval $I$
satisfying $\sup_{t \in I} f_t'(0) \leq M$
is uniformly bounded on $I \times \overline{\mathbb{D}}(0,r)$ 
for every $r \in (0,1)$.

Contrary to the case of univalent functions,
there are no local upper bounds for
the class of all Loewner chains  
$\{ f_t\}_{t \in I}$ satisfying $\sup_{t \in I} f_t'(0) \leq M$.  
For example, let 
\[
   f_n(z,t) = \frac{1}{n} \left\{ e^{ne^t z } - 1 \right\},
   \quad (z,t) \in \mathbb{D} \times (-\infty , \infty ),
   \quad n \in \mathbb{N}.
\]
Then it is easy to see that $\{ f_n ( \cdot ,t ) \}_{n=1}^\infty$ 
is a sequence of normalized Loewner chains, 
and that for any $r \in (0,1)$ and $t_0 \in \mathbb{R}$, 
\[
   \max_{|z| \leq r, \, t \leq t_0 } |f_n(z,t)|
   = \frac{1}{n} \left\{ (e^{e^{t_0} r })^n  - 1 \right\}
   \rightarrow \infty , \quad
   n \rightarrow \infty .
\]  
Later, 
we shall give a family of 
Loewner chains $\{ g_n (\cdot , t) \}_{0<t<\infty} $,
$n \in \mathbb{N}$, consisting of universal covering maps 
on $I = (0, \infty) $ with $g_n '(0,t)=t$, $0<t<\infty$,
which is not uniformly bounded
on  $(0 , t_0] \times \overline{\mathbb{D}}(0,r)$ 
for any $t_0 > 0$ and $r \in (0,1)$.
See Example \ref{example:universal-covering-onto-outside-of-a-disc}.

\section{Schlicht Subordination and Continuous Connection}
Next we consider the question of when there exists a 
continuous Loewner chain connecting two given analytic functions 
in $\mathbb{D}$.
\begin{definition}
A function $f \in \mathcal{H} (\mathbb{D})$ is said to be 
\emph{schlicht subordinate} 
to a function $g \in \mathcal{H} (\mathbb{D})$ 
if there exists a univalent analytic map
$\omega : \mathbb{D} \rightarrow \mathbb{D}$
with $\omega(0)=0$ such that $f= g \circ \omega$. 
We say that $f \in \mathcal{H}_0 (\mathbb{D})$ can be 
\emph{continuously connected to} $g \in \mathcal{H}_0 (\mathbb{D})$
by a Loewner chain if there exists a continuous 
Loewner chain 
$\{ f_t \}_{\alpha \leq t \leq \beta}$ 
satisfying $f_\alpha = f$ and $f_\beta = g$.
\end{definition}

The following result is known;
See Pommerenke \cite[\S 4 Folgerung 1]{Pommerenke1965}.
For completeness and for later applications, 
we provide a proof.
We make use of the
Carath\'{e}odory kernel convergence theorem. 
For proofs and details, see  
\cite[Chapter 6]{Ahlfors:ConformalInvariants},
\cite[Chapter 3]{Duren:univ-functions},
\cite[Chapter 1]{Pommerenke1}
or Chapter \ref{chapter:KernelAndCovering} of this article,
where we a generalization of the theorem is also given. 

\begin{theorem}[Pommerenke]
\label{theorem:schlicht-subordination-and-continuous-chain}
A function $f \in \mathcal{H}_0 (\mathbb{D})$ can be 
continuously connected to 
a function $g \in \mathcal{H}_0 (\mathbb{D})$ 
by a Loewner chain 
if and only if $f$ is schlicht subordinate to $g$.
\end{theorem}

\begin{proof}
The necessity easily follows from 
Theorem \ref{thm:omega-is-univalent-if-ft-is-continuous}.
Conversely, let $\varphi \in \mathfrak{B}$ be 
the unique univalent mapping with 
$f = g \circ \varphi$.
We may assume $\varphi'(0) = f'(0)/g'(0) \in (0,1)$, 
since otherwise $f$ coincides with $g$.

Take a sequence $\{ r_n \}_{n=1}^\infty$ with 
$0< r_1 < \cdots < r_n \nearrow  1$ and
set $\varphi_n(z) = \varphi(r_n z)$, $n \in \mathbb{N}$.
Then for each $n \in \mathbb{N}$,
$\varphi_n( \overline{\mathbb{D}}) \subset \mathbb{D}$.
Choose $z_n \in \partial \mathbb{D}$ such that 
$|\varphi_n(z_n) | = \max_{z \in \overline{\mathbb{D}}} |\varphi_n(z)|$.
Let $\gamma_n : [0,1] \rightarrow \overline{\mathbb{D}}$ 
be the curve
consisting of radial line segment from 
$\varphi_n(z_n)/|\varphi_n(z_n)| \in \partial \mathbb{D}$
to $\varphi_n(z_n)$, and the boundary curve 
$\varphi_n( \partial \mathbb{D})$ from and to $\varphi_n(z_n)$.
By reparametrizing if necessary, we may assume
that $\gamma_n$ is defined on $[0,1]$ and injective on $[0,1)$.
For each fixed $0 < t \leq 1 $,  
let $\varphi_n (z , t)$, $z \in \mathbb{D}$, be 
the unique conformal mapping of $\mathbb{D}$
onto the simply connected domain 
$\mathbb{D} \backslash \{ \gamma_n(s): 0 \leq s \leq 1-t \}$,
and set $\varphi_n (z,0) = \varphi_n (z)$, $z \in \mathbb{D}$.
Note that $\varphi_n (z,1) \equiv z$, $z \in \mathbb{D}$.
Then for fixed $n \in \mathbb{N}$,
since the family of simply connected domains 
$\{ \varphi_n(\mathbb{D},t) \}_{0 \leq t \leq 1}$ 
is strictly increasing and continuous 
in the sense of kernel convergence
with respect to the origin,
$\{ \varphi_n( \cdot , t) \}_{0 \leq t \leq 1 }$ is  
a Loewner chain. 
By reparametrizing, we may assume
$\varphi_n ' (0,t) = t$, $r_n \alpha \leq t \leq 1$,
where $\alpha : = \varphi'(0) \in (0,1)$. 
Then the family $\{ \varphi_n(\cdot ,t) \}_{\alpha \leq t \leq 1}$
is a continuous and strictly increasing 
Loewner chain of univalent functions.
Furthermore,
since $|\varphi_n(z,t)| \leq 1$,  
the sequence $\{ \varphi_n (z,t) \}_{n=1}^\infty $
of functions of two variables 
$(z,t) \in \mathbb{D} \times [\alpha ,1]$ 
is uniformly bounded on $\mathbb{D} \times [\alpha ,1]$.

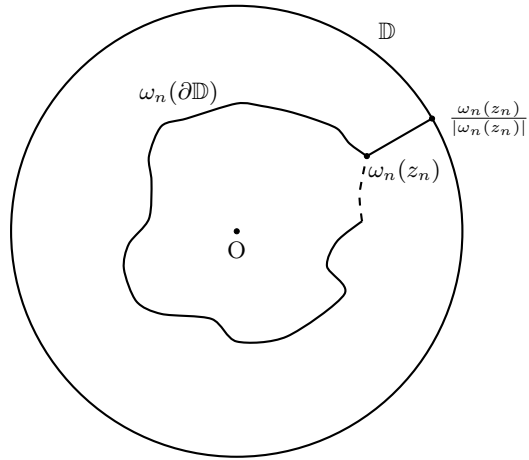
\begin{figure}[h]
\begin{tikzpicture}[scale=1/3]
\draw[thick] (0,0) circle (9);
\draw[fill] (0,0) circle (0.1);
\node[below] at (0,0) {\scriptsize $\rm{O}$}; 
\draw[thick, dashed] plot [smooth] 
coordinates {(5,0.4) (4.9,1.5) (5.19,3)};
\draw[thick] plot [smooth] 
coordinates {(5.19,3) (4.5,3.5) 
(4,4.1) (3,4.5) (2,4.8)
(1,5) (0,5.1) (-2,4.5) (-3,4.2) (-3.5,3)
(-3.4,2) (-3.5,0.5) (-4.3,-0.5) (-4.5,-1.5) (-4,-2.8) 
(-3,-3.3) (-1,-3.5) (0,-4.4) (2,-4.2) (4.3,-2.5) (3.6,-1.4)
(4,-0.4) (5,0.4)}; 
\draw[fill] (5.19,3) circle (0.1);
\draw[fill] (7.79,4.5) circle (0.1);
\draw[thick] (5.19,3)--(7.79,4.5);
\node at (6.7,2.4) {\scriptsize $\varphi_n(z_n)$};
\node[right]  at (8,4.5) 
{\scriptsize $\frac{\varphi_n(z_n)}{|\varphi_n(z_n)|}$};
\node at (-2.3,5.5) {\scriptsize $\varphi_n(\partial \mathbb{D})$};
\node at (6,8) {\scriptsize $\mathbb{D}$};
\end{tikzpicture}  
\end{figure}

For any fixed $r \in (0,1)$
we show that 
the sequence $\{ \varphi_n (z,t) \}_{n=1}^\infty $
is equicontinuous on 
$\overline{\mathbb{D}}(0,r) \times [\alpha ,1]$. 
Let $\{\omega_n (\cdot ,s,t) \}_{\alpha\leq s\leq t \leq 1 }$ be the
associated transition family of 
$\{ \varphi_n(\cdot ,t) \}_{\alpha \leq t \leq 1}$.
Since  $|\varphi_n ' (z,t)| \leq (1-|z|^2)^{-1}$ 
and $|\varphi_n  (z,t)| \leq |z|$, 
we have for $s\leq t$ and $z_0,z_1 \in \overline{\mathbb{D}}(0,r)$
\begin{align*}
  | \varphi_n(z_1,t) - \varphi_n(z_0,s)| 
  =& 
    | \varphi_n(z_1,t) - \varphi_n(\omega_n (z_0,s,t),t)| 
\\
  =&
  \left| \int_{\omega_n(z_0,s,t)}^{z_1} \varphi_n ' (\zeta ,t) \, d \zeta  \right| 
\\
  \leq & 
  \frac{|z_1-  \varphi_n(z_0,s,t)|}{1-r^2} 
\\
  \leq & 
  \frac{|z_1- z_0| + |z_0 - \omega_n(z_0,s,t)|}{1-r^2} .
\end{align*}
By (\ref{ineq:distance-between-omega/z-and-1})
\[
  |z_0 - \omega_n(z_0,s,t)|
   \leq 
  \left( 1 - \frac{\varphi_n '(0,s)}{\varphi_n '(0,t)} \right) \frac{r(1+r)}{1-r} 
  = 
  \left( 1 - \frac{s}{t} \right) \frac{r(1+r)}{1-r} .
\]
Combining these inequalities,
we obtain
\begin{equation}
\label{ineq:diff_bet_omega_s_and_t}
 | \varphi_n(z_1,t) - \varphi_n(z_0,s)| 
 \leq 
 \frac{|z_1- z_0| }{1-r^2} 
 +  \left( 1 - \frac{s}{t} \right) \frac{r}{(1-r)^2} .
\end{equation}
Therefore,
the sequence $\{ \varphi_n (z,t) \}_{n=1}^\infty $
is equicontinuous on 
$\overline{\mathbb{D}}(0,r) \times [\alpha ,1]$.

Applying the Arzel\`{a}--Ascoli theorem  to 
the sequence $\{ \varphi_n (z,t) \}_{n=1}^\infty $,
we obtain a subsequence $\{ \varphi_{n_k} (z,t) \}_{k=1}^\infty$
that converges to a function 
$\varphi(z,t)$, $(z,t) \in \mathbb{D} \times [\alpha ,1]$ 
uniformly on compact subsets 
$\overline{\mathbb{D}}(0,r) \times [\alpha,1]$ 
for every $r \in (0,1)$.

For each fixed $t \in [\alpha ,1]$, 
as a function of $z \in \mathbb{D}$,
$\varphi(z,t) $ is a locally uniform limit of
the sequence of univalent functions 
$\{ \varphi_{n_k}(z,t)\}_{k=1}^\infty$ in $\mathbb{D}$ and 
satisfies
$\varphi'(0,t) = \lim_{k \rightarrow \infty} \varphi_{n_k}'(0,t) 
= t \neq 0$.
Hence, by Hurwitz's theorem, 
$\varphi(z,t)$ is univalent analytic in $\mathbb{D}$  
with $\varphi'(0,t)=t$, $\alpha \leq t \leq 1$.

Next, we show $\varphi (\cdot ,s) \prec \varphi (\cdot ,t)$ 
for each fixed $s$ and $t$ with $\alpha \leq s \leq t \leq 1 $.
Since $|\omega_{n_k}(z,s,t)| \leq 1$,
by passing to a further subsequence if necessary,
we may assume that $\{ \omega_{n_k}(z,s,t) \}_{k=1}^\infty$ 
converges locally uniformly on $\mathbb{D}$ 
to a function $\omega_{s,t}$ in $\mathfrak{B}$.
Therefore,
\[
 \varphi (z ,s)
 = \lim_{k \rightarrow \infty}
 \varphi_{n_k} (z,s )
 = \lim_{k \rightarrow \infty}
 \varphi_{n_k} (\omega_{n_k}(z,s,t),t )
 = \varphi (\omega_{s,t} (z),t) ,
\]
as required.

We have shown that 
$\{ \varphi (z,t) \}_{\alpha \leq t \leq 1}$
is a Loewner chain.
Since  $\varphi'(0,t) = t$, $\alpha \leq t \leq 1$,
$\{ \varphi (z,t) \}_{\alpha \leq t \leq 1}$ is continuous.
Furthermore, by (\ref{ineq:diff_bet_omega_s_and_t}),
we have
\begin{align*}
 | \varphi_{n_k}(z,\alpha ) - \varphi(r_{n_k} z)|
 = & \, 
 | \varphi_{n_k}(z,\alpha ) - \varphi_{n_k}(z,r_{n_k} \alpha )|
\\
 \leq & \, 
 \left( 
 1 - \frac{\varphi_{n_k}'(0, r_{n_k} \alpha )}{\varphi_{n_k}'(0, \alpha )}
 \right)
 \frac{|z|}{(1-|z|)^2} 
 =
 \frac{(1-r_{n_k})|z|}{(1-|z|)^2} \rightarrow 0,
\end{align*}
as $k \rightarrow \infty$.
Thus,
\[
  \varphi (z,\alpha ) 
 = \lim_{k \rightarrow \infty}\varphi_{n_k}(z, \alpha ) 
 = \lim_{k \rightarrow \infty}\varphi_{n_k}(z, r_{n_k} \alpha )
 = \lim_{k \rightarrow \infty}\varphi(r_{n_k}z)
 = \varphi (z) .
\]
Also we have 
\[
  \varphi(z,1) = \lim_{k \rightarrow \infty} \varphi_{n_k}(z,1) = z.
\]
Let  $f_t(z) = g ( \phi (z,t))$, $t \in [\alpha , 1]$.
Then $\{ f_t \}_{\alpha \leq t \leq 1}$ 
is a continuous Loewner chain connecting 
$f_\alpha = g \circ \varphi = f$ 
and $f_1 = g \circ \id_{\mathbb{D}} = g$.
\end{proof}

The following result is an immediate consequence of 
the above theorem,  
and the proof is therefore omitted.

\begin{corollary}
\label{corollary:univalent_f_and_g}
Let $f, g \in \mathcal{H}_0 (\mathbb{D})$ be univalent functions.
Then $f$ can be 
continuously connected to $g$ 
by a Loewner chain 
if and only if $f (\mathbb{D}) \subset g(\mathbb{D})$.
\end{corollary}

\section{Embedding Univalent Maps into Loewner Chains}
\begin{definition}
Let $f \in \mathcal{H}_0 (\mathbb{D})$.  
We say that $f$ is 
\emph{maximal in the sense of continuous Loewner chains}
if there exists no continuous Loewner chain 
$\{ f_t \}_{0 \leq t \leq \varepsilon}$ 
for any $\varepsilon > 0$ satisfying $f_0 = f$
and $f'(0) < f_\varepsilon'(0)$.
\end{definition}

Pommerenke~(\cite[Theorem 6.1]{Pommerenke1}) 
proved that for any univalent $f \in \mathcal{H}_0 (\mathbb{D})$
can be embedded in a normalized Loewner chain 
$\{ f_t \}_{\alpha \leq t < \infty}$ of univalent functions 
such that $f_\alpha = f$ and $\alpha = \log f'(0)$. 
Therefore, no univalent function $f$ can be maximal. 
As a simple application of 
Corollary~\ref{corollary:univalent_f_and_g},
we give a proof of Pommerenke's result.  
Here, normalization of the Loewner chain is not required,  
since a reparametrization can be made if necessary.

\begin{theorem}[Pommrenke]
\label{theorem:univalent_function_can_be_emmbedded_in_L-chain}
Let $f \in \mathcal{H}_0 (\mathbb{D})$ be a univalent function.
Then there exists a continuous 
Loewner chain 
$\{ f_t \}_{0 \leq t < \infty}$ of univalent functions 
such that
$f_0=f$ and $\lim_{t \to \infty}  f_t'(0) = \infty$.
\end{theorem}
\begin{proof}
Let $\Omega = f(\mathbb{D})$, $E = \hat{\mathbb{C}} \setminus \Omega$, 
and $r_0 = \sup \{ r > 0 : \mathbb{D}(0,r) \subset \Omega \}$. 
Define 
$E_r = E \cap \left( \hat{\mathbb{C}} \setminus \mathbb{D}(0,r)\right)$ 
for $r > 0$.
Let $C_r$ denote the component of $E_r$ containing $\infty$,
and let $D_r$ be the component of $\hat{\mathbb{C}} \setminus C_r$
containing $0$.
Then, for $0 < r_1 < r_2 < \infty$,
we successively have
\[
 E_{r_1} \supset E_{r_2}
 \quad \Longrightarrow \quad
 C_{r_1} \supset C_{r_2}
 \quad \Longrightarrow \quad
 D_{r_1} \subset D_{r_2}.
\]

We show that for each $r >0$, 
the domain $D_r$ is simply connected.
To prove this, choose a Jordan curve 
$\gamma :[0,1] \to D_r$ arbitrarily,
and let $U$ be bounded domains enclosed by $\gamma$. 
It suffices to verify $U \subset D_r$.

Since $\gamma ([0,1]) \cap C_r = \emptyset$ and $C_r$ is connected,
we have either $C_r \subset U$ or 
$C_r \subset \hat{\mathbb{C}} \setminus \overline{U}$.
Since $\infty \in C_r$, it follows that 
$C_r \subset \hat{\mathbb{C}} \setminus \overline{U}$.
Hence, $U \subset \hat{\mathbb{C}} \setminus C_r$.
Since $\partial U = \gamma ([0,1]) \subset D_r$,
$U \cap D_r \neq \emptyset$.
Therefore, we conclude $U \subset D_r$.

It is easy to see that $D_r = \Omega$ for $r \leq r_0$,
and that $\mathbb{D}(0,r) \subset D_r$ for $r > 0$.
Let $\Omega_0 = \Omega$ and $\Omega_n = D_{r_0+n}$ 
for $n \in \mathbb{N}$.
Then $\{ \Omega_n \}_{n=0}^\infty$ is a noncecreasing sequence
of simply connected proper domains in $\mathbb{C}$ such that
$\bigcup_{n=0}^\infty \Omega_n = \mathbb{C}$.
Applying Corollary~\ref{corollary:univalent_f_and_g},
for each $n \in \mathbb{N}$, 
one can construct a continuous 
Loewner chain $\{ f_t \}_{n-1 \leq t \leq n}$ of univalent functions
such that $f_{n-1}$ and $f_n$ are the unique conformal maps
in $\mathcal{H}(\mathbb{D})$ 
from $\mathbb{D}$ onto $\Omega_{n-1}$ and $\Omega_n$, respectively.
Since $\Omega_n = D_{r_0+n} \supset \mathbb{D}(0,r_0+n)$,
we have $f_n'(0) \geq r_0+n$, 
and hence $\lim_{n \to \infty} f_n'(0) = \infty$.
Therefore, the concatenated family $\{ f_t \}_{0 \leq t < \infty}$  
is the desired one.
\end{proof}

\section{Maximality and Boundary Behavior}
We now consider a condition that ensures the maximality 
of $f \in \mathcal{H}_0(\mathbb{D})$.
We recall that by Fatou's theorem (\cite[Theorem 2.2]{Duren:Hp-spaces}), 
a bounded analytic function in $\mathbb{D}$ has
nontangential boundary values almost everywhere
with respect to the Lebesgue measure on $\partial \mathbb{D}$.
For $\omega \in \mathfrak{B}$, we denote
the nontangential boundary value of $\omega$ 
at $\zeta \in \partial \mathbb{D}$
by $\omega (\zeta )$.

Let us recall that an analytic function $f$ in $\mathbb{D}$
is said to be an inner function if
$|f(z)| \leq 1$ in $\mathbb{D}$ and $|f(\zeta)| = 1$
for almost every $\zeta \in \partial \mathbb{D}$.

\begin{lemma}
\label{Lemma:boundary_values_for_omega}
Let $\omega \in \mathfrak{B}$ be univalent
in $\mathbb{D}$ with $|\omega'(0)| < 1$.
Then the Lebesgue measure 
of the set 
$\{ \zeta \in \partial \mathbb{D} : |\omega (\zeta )| < 1 \}$
of $\partial \mathbb{D}$ is positive. 
\end{lemma}
\begin{proof}
Suppose, on the contrary, that $ |\omega (\zeta )| =1$
for almost every $\zeta \in \partial \mathbb{D}$.
Then $\omega$ is an inner function 
and hence by the Frostman theorem (see \cite[Theorem 2.6.4]{Garnett}), 
for all $c \in \mathbb{D}$, 
except possibly for a set $E \subset \mathbb{D}$ of capacity zero,   
the function 
\[
   B_c(z) = \frac{\omega(z)- c}{1-\overline{c} \omega (z)} ,
 \quad z \in \mathbb{D}
\]
is a Blaschke product. 
Fix $c \in \mathbb{D} \setminus E$
Since 
$\omega$ is univalent
by 
Theorem \ref{theorem:schlicht-subordination-and-continuous-chain},
$B_c$ must be a Blaschke product of order one.
Therefore, both $B_c$ and $\omega$ 
are linear fractional transformations preserving $\mathbb{D}$.
Hence, by $\omega (0) = 0$ and $\omega'(0) > 0$,
we must have $\omega (z) \equiv z$,
which contradicts $|\omega'(0)| < 1$.
\end{proof}

Now we prove that if $f$ has nontangential boundary value 
almost nowhere, then  $f$ is maximal 
in the sense of continuous Loewner chains.
\begin{proof}[Proof of Theorem \ref{thm:no-nontangentaial-limit}]
Suppose, on the contrary, that $\{ f_t \}_{0 \leq t \leq \varepsilon}$ 
is a continuous Loewner chain 
with $f_0 = f$ and $f_0'(0) < f_\varepsilon '(0)$.
Take $\omega \in \mathfrak{B}$ such that $f = f_\varepsilon \circ \omega$.
Then, since $\{ f_t \}_{0 \leq t \leq \varepsilon}$ 
is continuous,
$\omega$ is univalent 
by Theorem \ref{theorem:schlicht-subordination-and-continuous-chain}.
Furthermore we have
$\omega'(0) = \frac{f'(0)}{f_\varepsilon '(0)} < 1$.

By Lemma \ref{Lemma:boundary_values_for_omega}
the Lebesgue measure of the set 
$A : = \{ \zeta \in \partial \mathbb{D} : |\omega (\zeta )| < 1 \}$
is positive.
For each $\zeta \in A$, we have 
\[
  f (z) 
 = f_\varepsilon (\omega( z)) 
 \rightarrow f_\varepsilon (\omega( \zeta))
\]
as $z \rightarrow \zeta$ nontangentially in $\mathbb{D}$,
which clearly contradicts the assumption.
\end{proof}

In 1930, by using a probabilistic argument, 
Littlewood proved the existence of 
an analytic function $f$ on $\mathbb{D}$ which has radial limits
almost nowhere on $\partial \mathbb{D}$ 
(for details see Theorem A.5 in \cite{Duren:Hp-spaces}).
It then follows that $f$ has nontangential limits
almost nowhere on $\partial \mathbb{D}$.
Therefore, by Theorem \ref{thm:no-nontangentaial-limit},
$f$ is maximal in the sense of continuous 
Loewner chains.

In 1962, MacLane \cite{MacLane} constructed, 
by making use of Runge's theorem,
an analytic function $F$ in $\mathbb{D}$ satisfying
\[
\liminf_{r \nearrow 1} |F(r \zeta )| = 0 , \quad      
\limsup_{r \nearrow 1} |F(r \zeta )| = + \infty  
 \quad \text{for all } \zeta \in \partial \mathbb{D}. 
\]
Then $F$ is also maximal in the sense of continuous 
Loewner chains.

For concrete examples of functions 
that have nontangential limits almost nowhere,
see the following proposition
and \cite{Nicholls} for further details of its proof.

\begin{proposition}
Let $f$ be a universal covering map of $\mathbb{D}$
onto a domain $\Omega$ in $\mathbb{C}$.
Suppose that $\Omega$ does not admit a Green's function.
Then $f$ has nontangential limits
almost nowhere on $\partial \mathbb{D}$.
\end{proposition}
\begin{proof}
Let $\Gamma$  be the covering transformation group of $f$.
A point $\zeta \in \partial \mathbb{D}$ is said to be
a conical limit point of $\Gamma$ if, for every
$z \in \mathbb{D}$, there exists a sequence
$\{ \gamma_n \}_{n=1}^\infty$ such that 
$\gamma_n(z) \to \zeta$ in some Stolz domain at $\zeta$
as $n \to \infty$.
It is easy to verify that  
if $\zeta$ is a conical limit point of $\Gamma$,
then $f$ does not have nontangential boundary value at $\zeta$.

Suppose that $\Omega$ does not admit Green's function.
Then, almost every point $\zeta \in \partial \mathbb{D}$
is a conical limit point of $\Gamma$.
Hence, $f$ has nontangential limits
almost nowhere on $\partial \mathbb{D}$.
\end{proof}

\chapter{Kernel Convergence of Domains } 
\label{chapter:KernelAndCovering}
\section{Foundations and Classical Results on Kernel Convergence}
Let $X$ be topological space.
A continuous map $\alpha :[0,1] \rightarrow X$
is called a path from the initial point $\alpha(0)$
to the final point $\alpha (1)$.
We say that $X$ is path-connected if for any
$x$, $y \in X$, there exists a path
from $x$ to $y$.
For $x_0 \in X$, a path
$\gamma : [0,1] \rightarrow X$ is called a loop based at $x_0$
if $\gamma(0) = \gamma (1) =x_0$.  
We define the loop $e_{x_0}: [0,1] \rightarrow X$ by
\[
   e_{x_0}(t)=x_0, \quad t \in [0,1]. 
\]
For paths $\alpha$ and $\beta$ with
the same initial and final points,
$\alpha$ is said to be (path) homotopic to
$\beta$ (denoted $\alpha \sim \beta$ in short) 
if there exists a continuous map
$F:[0,1] \times [0,1] \rightarrow X$ such that
\begin{align}
\label{eq:initial_and_final_paths}
 & F(t,0) = \alpha (t)  \quad \text{and} \quad F(t,1) = \beta (t), 
 \quad t \in [0,1] \\
\label{eq:initial_and_final_points}
 & F(0,u) = \alpha (0) = \beta (0)  \quad \text{and} \quad 
 F(1,u) = \alpha (1) = \beta (1), \quad u \in [0,1] . 
\end{align}
The map $F$ is called a path homotopy from $\alpha$ to $\beta$. 
It is not difficult to see that $\sim$ is an equivalence relation.
For details, see for example, \cite{Massey}.

A path-connected topological space $X$ 
is said to be simply connected,
if, whenever paths $\alpha$ and $\beta$ 
share the same initial and final points,
$\alpha$ is homotopic to $\beta$.
A path-connected topological space $X$ is simply connected
if and only if,  
for any point $x_0 \in X$ and any loop $\gamma$ 
based at $x_0$,
the loop $\gamma$ is null-homotopic; that is,
there exists a path homotopy 
from $\gamma$ to $e_{x_0}$ (i.e.,  $\gamma \sim e_{x_0}$).

For a domain $\Omega$ in $\hat{\mathbb{C}}$,
we have the following useful criterion:
$\Omega$ is simply connected if and only if 
$\hat{\mathbb{C}} \setminus \Omega$ 
is either empty or connected.

Let $w_0 \in \hat{\mathbb{C}}$ and 
$\{ \Omega_n \}_{n=1}^\infty$ be a sequence of domains in 
$\hat{\mathbb{C}}$.
Suppose that $w_0 \in \Omega_n$ for all sufficiently large $n$.
The kernel 
of $\{ \Omega_n \}_{n=1}^\infty$ with respect to the reference point 
$w_0$ is defined as the set consisting of $w_0$
together with all points $w$ such that
there exists a domain $H$ and $N \in \mathbb{N}$
with
\begin{align}
   w_0,w \in H \subset \Omega_n \quad \text{for all } n \geq N.
\nonumber
\end{align}
We denote the kernel by $\ker(w_0, \{ \Omega_n \}_{n=1}^\infty )$. 
Then $\ker(w_0, \{ \Omega_n \}_{n=1}^\infty )$ is either a domain 
containing $w_0$,  or coincides with the singleton set $\{w_0 \}$.

A sequence $\{ \Omega_n \}_{n=1}^\infty$ of domains 
is said to converge
to a domain $\Omega$ with respect to $w_0$ in the sense of kernel
if $\ker(w_0, \{ \Omega_{n_k} \}_{k=1}^\infty ) = \Omega$ 
for every subsequence $\{ \Omega_{n_k} \}_{k=1}^\infty$ of 
$\{ \Omega_n \}_{n=1}^\infty$.
Also, $\{ \Omega_n \}_{n=1}^\infty$ is said to degenerate to $\{ w_0 \}$
if $\ker(w_0, \{ \Omega_{n_k} \}_{k=1}^\infty ) = \{ w_0 \}$ for
every subsequence $\{ \Omega_{n_k} \}_{k=1}^\infty$.

Since $\hat{\mathbb{C}}$ is compact, it is easy to see that
$\{ \Omega_n \}_{n=1}^\infty$ converges to $\hat{\mathbb{C}}$
if and only if $\Omega_n = \hat{\mathbb{C}}$ 
for all sufficiently  large $n$.

Suppose temporarily that $w_0 \in \mathbb{C}$
and each $\Omega_n$ is a simply connected 
domain properly contained in $\mathbb{C}$.
Then there exists a unique conformal map 
$f_n$ of $\mathbb{D}$ onto $\Omega_n$  
with $f_n(0)=w_0$ and $f_n'(0)>0$.
The Carath\'{e}odory convergence theorem states the following:
\begin{itemize}
 \item[{\rm 1.}]
If $\{ \Omega_n \}_{n=1}^\infty$ converges to a domain $\Omega$ 
with $\Omega \subsetneq  \mathbb{C}$ 
in the sense of kernel with respect to $w_0$, 
then $\Omega$ is simply connected, 
and $\{f_n \}_{n=1}^\infty$  converges locally uniformly on $\mathbb{D}$
to the unique conformal map $f$ of $\mathbb{D}$ onto $\Omega$ with 
$f(0)=w_0$ and $f'(0)>0$. 
If $\{ \Omega_n \}_{n=1}^\infty$ degenerates to $\{ w_0 \}$,
then  $f_n \rightarrow w_0$ locally uniformly on $\mathbb{D}$.
 \item[{\rm 2.}]
If $\{ f_n \}_{n=1}^\infty$ converges locally uniformly on $\mathbb{D}$ 
to a nonconstant function $f$, then $f$ is analytic and univalent 
in $\mathbb{D}$, 
and $\{ \Omega_n \}_{n=1}^\infty$ converges to $f(\mathbb{D})$ 
in the sense of kernel with respect to $w_0$.
If $f_n \rightarrow w_0 $ locally uniformly on $\mathbb{D}$,
then  $\{ \Omega_n \}_{n=1}^\infty$ degenerates to $\{ w_0 \}$.
\end{itemize}

By replacing conformal maps with universal covering maps, 
Hejhal (\cite{Hejhal})
was able to generalize the Carath\'{e}odry theorem 
to the case 
where each $\Omega_n$ is not necessarily simply connected.

In the next chapter, we will endeavor to generalize  
Hejhal's theorem 
and provide a detailed proof.
Before doing so, we will study the basic properties of 
the notion of kernel convergence,
particularly in the case where 
$\Omega_n$, $n \in \mathbb{N}$, are multiply connected.
We begin with an equivalent condition for the kernel convergence 
that was introduced 
by Pommerenke (\cite[Problem~3, p.31]{Pommerenke1}
and \cite[\S~1.8]{Pommerenke2}).
For the reader's convenience, we will also provide 
a proof of this equivalence.

For $z,w \in \mathbb{C}$ 
we define the chordal distance between them by 
\begin{align*}
  d^*(z,w) 
  = & \, 
  \frac{|z-w|}{\sqrt{1+|z|^2}\sqrt{1+|w|^2}},
\\
  d^*(z, \infty) = & \, d^*(\infty, z) 
  = 
  \frac{1}{\sqrt{1+|z|^2}} , \quad 
 d^*(\infty, \infty ) = 0 .
\end{align*}
For $z \in \hat{\mathbb{C}}$ and nonempty sets
$E, F \subset \hat{\mathbb{C}} $, 
we define
\[
    d^*(z,E) 
  = 
  \inf_{w \in E} d^*(z, w ) \quad
  \text{and} \quad 
    d^*(E,F) 
  = 
  \inf_{\zeta \in E, \; w \in F} d^*(\zeta, w ).
\] 
We set
\[
   \mathbb{D}^* (z,r) = \{ w \in \hat{\mathbb{C}} : d^*(z,w) < r \},
   \quad z \in \hat{\mathbb{C}} \quad \text{and} \quad r > 0 . 
\]

\section{Equivalent Characterizations and Degeneracy Criteria}
We denote the set of interior and exterior points of a subset $A$
of a topological space by $\Int A$ and $\Ext A$, respectively.
We also denote the complementary set of $A$ by $A^c : = X \setminus A$.
In this chapter we frequently use the following simple lemma 
without mention.

\begin{lemma}
Let $A$ and $C$ be nonempty subsets of the topological space $X$.
Suppose that $C$  is connected,
$C \cap A \not= \emptyset$ 
and $C \backslash A (=C \cap A^c) \neq \emptyset$.
Then $C  \cap \partial A \not= \emptyset$.
\end{lemma}
\begin{proof}
Suppose $C  \cap \partial A = \emptyset$.
Then, since $X$ can be written as a disjoint union
$X = \Int A \cup \partial A \cup \Ext A$,
we have $C \subset \Int A \cup \Ext A$; that is,
$C$ is contained in the union of the two open sets $\Int A$
and $\Ext A$.
Furthermore, 
\begin{align*}
\emptyset \neq & \, 
C \cap A \subset C \cap (\Int A \cup \partial A) = C \cap \Int A , 
\\
\emptyset \neq & \, 
C \backslash A  \subset C \cap (\Ext A \cup \partial A) = C \cap \Ext A . 
\end{align*}
This contradicts the assumption that $C$ is connected. 
\end{proof}

\begin{theorem}
\label{thm:equivalent-definition-of-kernel-convergence}
Let 
$\{ \Omega_n \}_{n=1}^\infty$ be 
a sequence of domains in $\hat{\mathbb{C}}$, 
and let $\Omega$ a domain in $\hat{\mathbb{C}}$. 
Then the following three conditions 
are equivalent:
\begin{itemize}
 \item[{\rm (i)}]
Both of the following hold:
\begin{itemize}
\item[{\rm (a)}]
For every compact subset $K$ of $\Omega$,
there exists $N \in \mathbb{N}$ such that 
$K \subset \Omega_n$ for all $n \geq N$.
 \item[{\rm (b)}]
For every $c \in \partial \Omega$, 
$d^* (c , \partial \Omega_n)  \rightarrow 0$ 
as $n \rightarrow \infty$. 
\end{itemize}
 \item[{\rm (ii)}]
For every $w_0 \in \Omega$,  
$\{ \Omega_n \}_{n=1}^\infty $ converges to $\Omega$ 
in the sense of kernel with respect to $w_0$.
 \item[{\rm (iii)}]
For some $w_0 \in \Omega$,
$\{ \Omega_n \}_{n=1}^\infty $ converges to $\Omega$ 
in the sense of kernel with respect to $w_0$.
\end{itemize}
\end{theorem}
\begin{proof}
Assuming (i) we show (ii).
Take $w_0 \in \Omega$ arbitrarily.
We claim that 
$\Omega \subset \ker ( w_0, \{ \Omega_n \}_{n=1}^\infty )$. 
To see this, let $w \in \Omega$, and
choose a polygonal line $\ell$ connecting $w_0$ and $w$ 
contained in $\Omega$. 
Let  $\delta_1 \in (0,  d^* (\ell , \partial \Omega ) )$,
and put $H = \bigcup_{\zeta \in \ell } \mathbb{D}^*(\zeta, \delta_1 )$.
Then $H$ is a domain  with 
$w_0,w \in H \subset \overline{H} \subset \Omega$. 
Since $\overline{H}$ is compact, by condition (a) we have 
that $\overline{H} \subset \Omega_n$
for all sufficiently large $n$.
Hence $w \in \ker ( w_0, \{ \Omega_n \}_{n=1}^\infty )$,
and we obtain $\Omega \subset \ker ( w_0, \{ \Omega_n \}_{n=1}^\infty )$.

Next we show that
$\ker ( w_0, \{ \Omega_n \}_{n=1}^\infty ) \subset \Omega$.
Suppose, on the contrary, that
\[
  \ker ( w_0, \{ \Omega_n \}_{n=1}^\infty ) 
\setminus \Omega \neq \emptyset .
\]
By (a), as shown above, 
$\Omega \subset \ker ( w_0, \{ \Omega_n \}_{n=1}^\infty )$. 
Since $\ker ( w_0, \{ \Omega_n \}_{n=1}^\infty )$ is connected
and intersects both $\Omega$ and $\hat{\mathbb{C}} \setminus \Omega$, 
there exists 
$c \in \ker ( w_0, \{ \Omega_n \}_{n=1}^\infty ) \cap \partial \Omega$.
Then, since $c \in \ker ( w_0, \{ \Omega_n \}_{n=1}^\infty )$,
some neighborhood of $c$ is contained in $\Omega_n$ 
for all sufficiently large $n$.
This contradicts (b). 
Therefore, $\ker ( w_0, \{ \Omega_n \}_{n=1}^\infty ) 
\setminus \Omega = \emptyset$, and hence
$\ker ( w_0, \{ \Omega_n \}_{n=1}^\infty ) = \Omega$.

If 
$\{ \Omega_n \}_{n=1}^\infty$ and $\Omega$ satisfy
(a) and (b), then any subsequence $\{ \Omega_{n_k}\}_{k=1}^\infty$
and $\Omega$ also satisfy (a) and (b).
Therefore, the above argument can be applied for any
subsequence $\{ \Omega_{n_k}\}_{k=1}^\infty$,
and we obtain 
$\ker ( w_0, \{ \Omega_{n_k} \}_{k=1}^\infty ) = \Omega $.
Thus, (ii) holds.

It is clear that (ii) implies (iii).
Now assume (iii) holds for some $w_0 \in \Omega$.
Let $K$ be a compact subset of 
$\Omega = \ker ( w_0, \{ \Omega_n \}_{n=1}^\infty )$.
For each $w \in K$ there exists a domain $H_w$ contained in $\Omega$
and $n(w) \in \mathbb{N}$ such that 
$w_0,w \in H_w \subset \Omega_n$ for all $n \geq n(w)$.
Since $K \subset \bigcup_{w \in K} H_w$ is an open covering of 
the compact set $K$, we can choose $w_1, \ldots, w_j$ such that
$K \subset H_{w_1} \cup \cdots \cup H_{w_j}$.
Therefore, we have 
$K \subset H_{w_1} \cup \cdots \cup H_{w_j} \subset \Omega_n$  
for all $n \geq \max \{ n(w_1), \ldots , n(w_j) \}$,
and hence (a) holds.

To show (b), suppose, on the contrary, that 
there exists $c \in \partial \Omega$, $\delta_2 > 0$
and a subsequence $\{ \Omega_{n_k} \}_{k=1}^\infty$ such that
$d^* (c , \partial \Omega_{n_k}) \geq \delta_2$ 
for all $k \in \mathbb{N}$.
Since $\mathbb{D}^*(c, \delta_2)$ is connected
and $\mathbb{D}^*(c, \delta_2) \subset \Omega_{n_k} \cup 
\left( \hat{\mathbb{C}} \backslash \overline{\Omega_{n_k}}\right)$,
we have either
$\mathbb{D}^*(c, \delta_2) \subset \Omega_{n_k}$
or 
$\mathbb{D}^* (c, \delta_2) \subset 
\hat{\mathbb{C}} \backslash \overline{\Omega_{n_k}}$.
Since $c \in \partial \Omega$, we can take 
$c^* \in \mathbb{D}^*(c,\delta_2) \cap \Omega 
=\mathbb{D}^*(c,\delta_2) \cap \ker ( w_0, \{ \Omega_n \}_{n=1}^\infty )$
and a domain $H$ with $w_0,c^* \subset H \subset \Omega_n$ 
for all sufficiently large $n$.
Hence we obtain $\mathbb{D}^*(c, \delta_2) \subset \Omega_{n_k}$ 
for all sufficiently large $k$.
Since $c^* \in H \cap \mathbb{D}^*(c,\delta_2)$, 
the union $H \cup \mathbb{D}^*(c,\delta_2)$ is a domain containing
$c$ and $w_0$. 
Therefore 
\[
  c \in H \cup \mathbb{D}^*(c,\delta_2) 
 \subset \ker ( w_0, \{ \Omega_{n_k}\}_{k=1}^\infty ) = \Omega, 
\]
which contradicts $c \in \partial \Omega$.
\end{proof}

By modifying the above proof, we obtain 
the following criterion.
\begin{theorem}
Let $w_0 \in \hat{\mathbb{C}}$ and  $\{ \Omega_n \}_{n=1}^\infty$ 
be a sequence of domains in $\hat{\mathbb{C}}$ 
such that $w_0 \in \Omega_n$ for all sufficiently large $n$.
Then $\{ \Omega_n \}_{n=1}^\infty$  degenerates to $\{w_0\}$
if and only if 
\begin{itemize}
 \item[{\rm (c)}]
$d^* (w_0, \partial \Omega_n ) \rightarrow 0$ as $n \rightarrow \infty$
\end{itemize}
holds.
\end{theorem}

In defining kernel convergence, 
we can omit the reference point $w_0$ 
by applying 
Theorem~\ref{thm:equivalent-definition-of-kernel-convergence}.

\begin{definition}
\label{def:kernel-convergence-without-reference-point}
We say that a sequence of domains $\{ \Omega_n \}_{n=1}^\infty$ 
\emph{converges to} a domain $\Omega$
\emph{in the sense of kernel} (denoted simply by 
$\Omega_n \rightarrow \Omega$ as $n \rightarrow \infty$),
if both {\rm (a)} and {\rm (b)} hold.
\end{definition}

We make an important remark.
The above definition does not guarantee the uniqueness of 
the limit domain $\Omega$.
\begin{example}
\label{example:strip-minus-sgments}
For $n \in \mathbb{N}$, let $\Omega_n$ be the domain 
obtained from the strip 
$\{ w \in \mathbb{C} : | \Imaginary w | < 1 \}$
by removing the line segments 
$\{ k +is : | s |  \leq 1 - n^{-1} \}$ for each $k \in \mathbb{Z}$.
For $k \in \mathbb{Z}$, 
define
$D_k = \{ w \in \mathbb{C}: k < \Real w < k+1, \; | \Imaginary w | < 1 \}$.
Then for each $k \in \mathbb{Z}$,
we have $\Omega_n \rightarrow D_k$
as $n \rightarrow \infty$. \\[12pt]
\begin{tikzpicture}
\node at (0.5,0.6) {$\Omega_n$};
\draw[very thick] (-1.3,-1) -- (3.5,-1);
\draw[very thick] (-1.3,1) -- (3.5,1);
\foreach \n in {-1,...,3}
{\draw (\n,-0.5) node[below=0.6cm] {$\n$} --(\n,0.5);}
\end{tikzpicture}
\hfill
\begin{tikzpicture}
\draw[very thick] (-1.3,-1) -- (3.5,-1);
\draw[very thick] (-1.3,1) -- (3.5,1);
\foreach \n in {-1,...,3}
{\draw (\n,-1) node[below=0.1cm] {$\n$} 
node[above=1cm, right] {$D_{\n}$} --(\n,1);}
\end{tikzpicture}
\end{example}

\begin{proposition}
\label{prop:quasi-uniqueness-of-limit-domain}
If $\Omega_n \rightarrow \Omega$
and $\Omega_n \rightarrow \Omega'$ as $n \rightarrow \infty$
in the sense of kernel,
then either $\Omega = \Omega'$ or 
$\Omega \cap \Omega' = \emptyset$ holds.
\end{proposition}
\begin{proof}
First we show that $\Omega \cap \partial  \Omega' = \emptyset$.
To this end, suppose that $c \in \Omega \cap \partial  \Omega'$.
Take $r > 0$ with $\overline{\mathbb{D}}^*(c,r) \subset \Omega$.
Then by (a), there exists $N \in \mathbb{N}$ 
such that 
$\overline{\mathbb{D}}^*(c,r) \subset \Omega_n$
for all $n \geq N$ .
Hence $d^*(c, \partial \Omega_n ) \geq r$ for $n \geq N$.
On the other hand, since $c \in \partial \Omega'$, 
(b) implies $d^*(c, \partial \Omega_n ) \rightarrow 0$, 
which is a contradiction. 

If $\Omega \cap \Omega' = \emptyset$, there is nothing to prove.
Suppose that $\Omega \cap \Omega' \not= \emptyset$.
Then, since $\Omega \cap \partial  \Omega' = \emptyset$
and $\Omega$ is connected,
either $\Omega \subset \Omega'$
or $\Omega \subset \hat{\mathbb{C}} \backslash \overline{\Omega'}$ 
holds.
From $\Omega \cap \Omega' \not= \emptyset$, it follows that 
$\Omega \subset \Omega'$.
Similarly, by replacing $\Omega$ with $\Omega'$, 
we obtain $\Omega' \subset \Omega$, 
and hence $\Omega = \Omega'$.
\end{proof}

\section{Set Operations and Kernel Convergence}
When $\{ \Omega_n \}_{n=1}^\infty$ is monotone, 
one can easily prove the following.
\begin{theorem}
\label{theorem:monotone-covergence-thm-for-domains}
Let 
$\{ \Omega_n \}_{n=1}^\infty$ be a sequence of domains 
in $\hat{\mathbb{C}}$.
\begin{itemize}
 \item[{\rm (i)}]
If $\Omega_1 \subset \Omega_2 \subset \cdots$,
then $\Omega_n \rightarrow \bigcup_{k=1}^\infty \Omega_k$ 
as $n \rightarrow \infty$. 
 \item[{\rm (ii)}]
If $\Omega_1 \supset \Omega_2 \supset \cdots$ and
$\Int \left( \bigcap_{k=1}^\infty \Omega_k \right) \not= \emptyset$,
then $\Omega_n \rightarrow D$ as $n \rightarrow \infty$ 
for every connected component $D$ of 
$\Int \left( \bigcap_{k=1}^\infty \Omega_k \right)$.
\end{itemize}
\end{theorem}

Next we will investigate the relationship between the operations 
of taking intersections of sets and kernel convergence.
\begin{theorem}
\label{thm:continuity-preseved-under-intersection}
Let $k \in \mathbb{N}$, and
for each $i =1, \ldots , k$
let $\{ \Omega_n^{(i)} \}_{n=1}^\infty$ be 
a sequence of domains in $\hat{\mathbb{C}}$ 
which converges to a domain
$\Omega^{(i)}$.
Suppose that 
$w_0 \in \Omega^{(i)} \cap \Omega_n^{(i)}$  
for all $i=1, \ldots ,k$ and $n \in \mathbb{N}$.
Let $D$ and $D_n$ be the components of 
$\Omega^{(1)} \cap \cdots \cap \Omega^{(k)}$
and $\Omega_n^{(1)} \cap \cdots \cap \Omega_n^{(k)}$ 
containing $w_0$, respectively. 
Then $D_n \to D$ as $n \rightarrow \infty$.
\end{theorem}
\begin{proof}
First we show $\{ D_n \}_{n=1}^\infty$ and $D$ satisfy condition (a).
Let $K$ be a compact subset of $D$. 
Since $D$ is a domain containing $w_0$,
we can take a compact and connected set $\tilde{K}$ 
with $K \cup \{ w_0 \} \subset \tilde{K} \subset D$.
Then for each $i=1, \ldots , k$ 
there exist $N_i \in \mathbb{N}$ such that 
$\tilde{K} \subset \Omega_n^{(i)}$ for all $n \geq N_i$.
Since $\tilde{K}$ is connected and contains $w_0$, 
we have 
\[
   K \subset \tilde{K} 
 \subset 
 D_n , \quad n \geq \max\{ N_1, \ldots, N_k \}.
\]

Next, we show $\{ D_n \}_{n=1}^\infty$ and $D$ satisfy the condition (b).
Let $c \in \partial D$.
Since 
\[
  \partial D 
 \subset \partial (\Omega^{(1)} \cap \cdots \cap \Omega^{(k)})
 \subset \partial \Omega^{(1)} \cap \cdots \cap \partial \Omega^{(k)} , 
\]
we have $c \in \partial \Omega^{(i)}$ for some $i$.
Then for any $\varepsilon > 0$, there exists $N \in \mathbb{N}$ 
such that $d^*(c, \partial \Omega_n^{(i)}) < \varepsilon$ 
for $n \geq N$.
Thus
$\mathbb{D}^*(c,\varepsilon ) \setminus \Omega_n^{(i)} \neq \emptyset$,
and hence
$\mathbb{D}^*(c,\varepsilon ) \setminus D_n \not= \emptyset$
for all $n \geq N$.
Also, since $c \in \partial D$, there exists 
$\tilde{c} \in D \cap \mathbb{D}^*(c,\varepsilon )$.
As shown above we can take $N' \in \mathbb{N}$ such that 
$\tilde{c} \in D_n$ for all $n \geq N'$.
Therefore, for all $n \geq N'' := \max\{ N, N' \}$,
we have both 
$\mathbb{D}^*(c,\varepsilon ) \cap D_n \not= \emptyset$
and  $\mathbb{D}^*(c,\varepsilon ) \backslash  D_n \not= \emptyset$.
Since $\mathbb{D}^*(c,\varepsilon )$ is connected,
this implies 
$\mathbb{D}^*(c,\varepsilon ) \cap \partial D_n \not= \emptyset$
which is equivalent to $d^*(c, \partial D_n) < \varepsilon$. 
Therefore we obtain $d^*(c, \partial D_n) \rightarrow 0$ 
as $n \rightarrow \infty$.
\end{proof}

In the above theorem
we cannot replace the intersection symbol with the union symbol. 
\begin{example}
Let $\{ \Omega_n \}_{n=1}^\infty$ and $\{ D_k \}_{k \in \mathbb{Z}}$ 
as in Example \ref{example:strip-minus-sgments}.
Let $\Omega'= \Omega_n' = \mathbb{D}(0, \rho )$ with $0< \rho <1$.
Then although $\Omega_n \rightarrow D_0$ and 
$\Omega_n' \rightarrow \Omega'$ we have
$\Omega_n \cup \Omega_n' \rightarrow 
D_{-1} \cup D_0 \cup \Omega' \supsetneq  D_0 \cup \Omega'$ 
as $n \rightarrow \infty$.
\end{example}

In the following, we consider subsets in $\mathbb{C}$.
To avoid confusion, 
we will temporarily denote the boundary of a set $E$
in $\mathbb{C}$ 
by $\partial E$,
and its boundary in $\hat{\mathbb{C}}$ 
by $\partial^* E$.
Note that if $E$ is bounded in $\mathbb{C}$,
then $\partial^* E$ coincides with $\partial E$. 
When $E$ is unbounded, 
we have 
$\partial^* E = \partial E \cup \{ \infty \}$.

We also adopt the standard notation
\[
   d(z,A) = \inf \{ |z-w| : w \in A \}
\]
to denote the distance from the point $z$
to the nonempty set $A \subset \mathbb{C}$.

\begin{lemma}
Let $\{ \Omega_n \}_{n=1}^\infty$ be a sequence  of domains 
in $\mathbb{C}$.
Then $\{ \Omega_n \}_{n=1}^\infty$ converges to 
a domain $\Omega$ in $\mathbb{C}$ in the sense of kernel, 
if and only if both of the following conditions hold: 
\begin{itemize}
 \item[{\rm (a')}]
For every subset $K$ of $\Omega$ that is 
compact with respect to the topology of $\mathbb{C}$, 
there exists $N \in \mathbb{N}$ such that
$K \subset \Omega_n$ for all $n \geq N$.
 \item[{\rm (b')}]
For every $c \in \partial \Omega$, 
we have 
$d (c , \partial \Omega_n)  \rightarrow 0$ as $n \rightarrow \infty$.
\end{itemize}
\end{lemma}
\begin{proof}
Let $K \subset \Omega$. 
Then $K$ is compact in $\mathbb{C}$ 
if and only if $K$ is compact in $\hat{\mathbb{C}}$.
Hence, (a) and (a') are equivalent.

Assume (b), i.e.,
$d^*(c, \partial^*\Omega_n) \rightarrow 0$ 
for all $c \in \partial^* \Omega$.
Let $w \in \partial \Omega$.
Then, since $w \in \partial^* \Omega$,
$d^*(w, \partial^*\Omega_n) \rightarrow 0$.
Choose $N \in \mathbb{N}$ 
such that $d^*(w, \partial^* \Omega_n) < d^*(w,\infty)$ for $n\geq N$.
Then $d^*(w, \partial \Omega_n) = d^*(w, \partial^* \Omega_n)$ 
for $n\geq N$.
Therefore, $d^*(w, \partial \Omega_n) \rightarrow 0$. 
This implies $d(w,\partial \Omega_n) \rightarrow 0$,
and hence (b') holds.

Assuming (a') and (b') we show (b).
It suffices to see that 
$d^*(\infty, \partial^* \Omega_n) \rightarrow 0$ as $n \rightarrow \infty$,
when $\infty \in \partial^* \Omega$.
Let $\varepsilon > 0$.
Since $\Omega$ is unbounded, 
for any $R > 0$ with $(1+R^2)^{-1/2} < \varepsilon$,
there exists $w \in \Omega $ with $|w| > R$.
By (a') there is $N \in \mathbb{N}$ such that
$w \in \Omega_n$ for $n \geq N$. 
Combining this with $\infty \not\in \Omega_n$,
and noting that
$\mathbb{D}^* (\infty, (1+R^2)^{-1/2}) 
= \{ z \in \mathbb{C} : |z| > R\} \cup \{ \infty \}$ 
is connected, 
there exists $w_n \in \partial^* \Omega_n$ 
with $|w_n| > R$ for $n \geq N$. 
We note that $w_n$ may possibly coincide with $\infty$.
Thus 
\[
 d^*(\infty, \partial^* \Omega_n) 
 \leq d^*(\infty, w_n) 
 \leq (1+R^2)^{-1/2} < \varepsilon
\]
for  $n \geq N$. 
Hence $d^*(\infty, \partial^* \Omega_n) \rightarrow 0$ 
as $n \rightarrow \infty$.
\end{proof}

Let $w_0 \in \mathbb{C}$ and  $\{ \Omega_n \}_{n=1}^\infty$ 
be a sequence of domains in $\mathbb{C}$
with $w_0 \in \Omega_n$ for all sufficiently large $n$.
Then it is easy to see that
$\{ \Omega_n \}_{n=1}^\infty$  degenerates to $\{w_0\}$
if and only if 
$d (w_0, \partial \Omega_n ) \rightarrow 0$ as $n \rightarrow \infty$.
However, it may be beneficial to consider the following example.
\begin{example}
Let $\Omega_n = \{ w \in \mathbb{C} : \Real w > - n^{-1} \}$ 
for  $n \in \mathbb{N}$.
Then, for any $i \eta$, $\eta \in \mathbb{R}$,
$\{ \Omega_n \}_{n=1}^\infty$ degenerates to $\{ i \eta \}$
and furthermore, $\{ \Omega_n \}_{n=1}^\infty$ 
converges also to the 
right half plane in the sense of kernel. 
\end{example}

\section{Limits of Image Domains under Analytic Maps}
At this point, we present a proposition 
that may be regarded as a variant of Hurwitz's theorem 
(see Hille \cite[Theorem~14.3.4]{Hille} 
and Goluzin \cite[Theorem V.5.1]{Goluzin}),
which also constitutes a part of the kernel convergence theorem. 
The result follows directly from Rouche's theorem.
\begin{proposition}
\label{prop:first-part-of-the-Hejhal-theorem}
Let 
$D$ be a domain in $\mathbb{C}$ 
and $\{ f_n \}_{n=1}^\infty$
a sequence of analytic functions in $D$
which converges to an analytic function $f$ locally uniformly on $D$.
If $f$ is non-constant, then 
$\{ f_n (D) \}_{n=1}^\infty$ and $f(D)$ satisfy condition {\rm (a')}
and $f(D) \subset \ker (w_0, \{f_n(D)\}_{n=1}^\infty )$ 
for all $w_0 \in f(D)$. 
\end{proposition}
\begin{proof}
Let $w^* \in f(D)$.
Choose $z^* \in D$ with $f(z^* ) = w^*$.
Since $f$ is non-constant,
there exists 
$k \in \mathbb{N}$ and $c \in \mathbb{C} \backslash \{ 0 \}$ 
such that $f(z) = c(z - z^*)^k + \cdots $, 
$|z-z|^* < d:= d (z^* , \partial D)$.
Define the analytic function $g$ by 
\[
   f(z) -w^* = c(z - z^*)^k (1+g(z)), \quad |z-z^*| < d .
\]
Since $g(0)=0$, there exist $r  > 0$
such that $|g(z)| < \frac{1}{2}$ for $|z-z^*| \leq r$.
Note that $f-w^*$ has exactly $k$ zeros 
in $\overline{\mathbb{D}}(z^*,r)$, counted with multiplicity.
Put
$\rho = \rho(w^*) = 4^{-1} |c|r^k$.
Then, for $|z-z^*|=r$, we have 
\[
   |f(z)-w^*| 
 = |c||z-z^*|^k|1+g(z)|
 \geq
 \frac{|c|r^k}{2} = 2 \rho .
\]
Choose $N = N(w^*) \in \mathbb{N}$
such that $|f_n(z)-f(z)| < \rho$
on $\partial \mathbb{D}(z^*,r)$ for all $n \geq N$. 
Then, for all $w$ with $|w-w^*| < \rho$
and $z$ with $|z-z^*|=r$
we have
\[
 |f_n(z)-w -(f(z)-w^*)|
 \leq
 |f_n(z)-f(z)| + |w-w^*|
 < 
 \rho + \rho 
 = 2 \rho 
 \leq 
 |f(z)-w^*| .
\] 
Thus, by Rouch\'{e}'s theorem, 
$f_n - w$ has exactly $k$ zeros in 
$\mathbb{D}(z^*,r)$, counted with multiplicity.
In particular, 
$\mathbb{D}(w^*, \rho (w^*) ) \subset f_n (D)$ 
for $n \geq N(w^*)$.

To see (a'), let $K$ be a compact subset of $f(D)$.
Consider the open covering 
$K \subset \bigcup_{w^* \in K} \mathbb{D}(w^*,\rho(w^*))$.
One can choose $w_1^*, \ldots , w_p^*$
such that 
$K \subset \mathbb{D}(w_1^*,\rho(w_1^*)) 
\cup \cdots \cup \mathbb{D}(w_p^*,\rho(w_p^*))$.
Then for $n \geq \max \{ N(w_1^*), \ldots , N(w_p^*)\}$
we obtain $K \subset f_n(D)$.

Similar to the first part of the proof of 
Theorem \ref{thm:equivalent-definition-of-kernel-convergence},
it is easy to see that (a') implies 
$f(D) \subset \ker (w_0, \{f_n(D)\}_{n=1}^\infty )$ 
for all $w_0 \in f(D)$.
\end{proof}

As shown by the following example,
the reverse inclusion 
$\ker (w_0, \{f_n(D)\}_{n=1}^\infty ) \subset f(D)$ 
does not necessarily hold.
To ensure this inclusion, 
one must impose additional conditions 
on the functions beyond analyticity. 
This observation motivates the introduction of covering maps, 
which will be discussed in the next chapter.

\begin{example}
\label{eg:discontinuity}
Let $D_n$ be the domain obtained from the rectangle 
$\{ w \in \mathbb{C}: |\Real w|<1 , |\Imaginary w | < 2\pi \;\}$ 
by removing the two line segments 
$\{ t \pm \pi i/6 : -1+ n^{-1} \leq t \leq 1 \}$.
Then 
$D_n \rightarrow D 
:=\{ w \in \mathbb{C}: |\Real w|<1 , |\Imaginary w | < \pi/6 \;\}$ 
in the sense of kernel.
Let $g_n$ and $g$ be the conformal maps of $\mathbb{D}$ onto 
$D_n$ and $D$ normalized by 
$g_n(0)=g(0)=1$, $g_n'(0) > 0$ and $g'(0)>0$, 
respectively.
Then by the Carath\'{e}odry kernel convergence theorem
$g_n \rightarrow g$ locally uniformly in $\mathbb{D}$ and hence
$f_n := e^{g_n} \rightarrow f:=e^g$ 
locally uniformly on $\mathbb{D}$.
Each $f_n(\mathbb{D})$, $n \in \mathbb{N}$, 
and $\ker (1, \{f_n(D)\}_{n=1}^\infty )$
coincide with the annulus 
$\{ \zeta \in \mathbb{C}: e^{-1} < |\zeta | < e  \}$.
However, the image $f(\mathbb{D})$ is
$\{\zeta \in \mathbb{C}: e^{-1} < |\zeta | < e,  \;
 |\Arg \zeta | < \pi/6  \}$, 
which is a proper subset of the annulus.
\end{example}

\chapter{Kernel Convergence %
and Locally Uniform Convergence of Covering Maps}
\label{chapter:relation_between_kernel_convergence_and_locally_uniform_convergence}
\section{Preliminaries on Covering Maps}
We begin by recalling the notion of 
a covering Riemann surface. 
For further details on this topic, 
see references \cite{Ahlfors:ConformalInvariants}, 
\cite{Springer}, \cite{Forster} or \cite{Ahlfors-and-Sario}.

Let $R$ and $\tilde{R}$ be Riemann surfaces.
An analytic surjection  
$p: \tilde{R} \rightarrow R$ 
is called a \textit{covering map} 
if for each $x \in R$ there exists a connected (open) 
neighborhood $V$ of $x$ 
such that for every connected component $\tilde{V}$ of $p^{-1}(V) $, 
the restriction $p|_{\tilde{V}}$ is a conformal map
of $\tilde{V}$ onto $V$.

The Riemann surface $\tilde{R}$ is called 
a \textit{covering surface} of the \textit{base surface} $R$,
and $V$ is  called an \textit{evenly covered neighborhood} of $x$.
For any $\tilde{x} \in p^{-1}(\{ x \})$, by definition,
there exists a unique component 
$\tilde{V}$ of $p^{-1}(V)$ that contains $\tilde{x}$. 
This component $\tilde{V}$ is called 
the \textit{slice} of $V$ containing $\tilde{x}$.

When $\tilde{R}$ is simply connected,
$p$ and $\tilde{R}$ are called the \textit{universal covering map}
and the \textit{universal covering surface}, respectively.

Every Riemann surface has a universal covering surface.
The Koebe uniformization theorem states
that a every simply connected Riemann surface is 
conformally equivalent to 
either the unit disk $\mathbb{D}$, the complex plane $\mathbb{C}$ 
or the Riemann sphere $\hat{\mathbb{C}}$.
A Riemann surface is called \textit{hyperbolic} 
if its universal covering surface
is conformally equivalent to $\mathbb{D}$.

Let $\Omega$ be a domain in $\hat{\mathbb{C}}$.
Then $\Omega$ is a Riemann surface by definition
and 
$\Omega$ is hyperbolic if and only if 
${}^\# (\hat{\mathbb{C}} \backslash \Omega ) \geq 3 $; that is,  
the complement of $\Omega$ has at least three points.

When $\Omega$ is a hyperbolic domain in $\mathbb{C}$,
for any $z_0 \in \mathbb{D}$, $w_0 \in \Omega$ and $\theta \in \mathbb{R}$ 
there exists a unique analytic covering map
$f: \mathbb{D} \rightarrow \Omega$ satisfying
$f(z_0) =w_0$ and $\arg f'(z_0) = \theta$.

We say that an analytic function $f$ in a domain 
$D \subset \mathbb{C}$ is a \textit{covering map}
if $f: D \rightarrow f(D) (\subset \mathbb{C})$
is a covering map.
By definition, if $f$ is univalent in $D$, 
then $f$ is a covering map.

Let $p: \tilde{R} \rightarrow R$ is a covering map,
and let $h$ be a continuous map of a Riemann surface $X$ into $R$.
A continuous map $\tilde{h} : X \rightarrow \tilde{R}$
is called a \textit{lift} of $h$ if $h = p \circ \tilde{h}$.

We now collect some basic facts about lifts of maps.
For proofs, see, for example, \cite[Chapter V]{Massey}.

\begin{lemma}
\label{lemma:uniqueness_of_lift}
Let $x_0 \in X$ and set $a_0 = h(x_0)$.
Let $\tilde{a}_0 \in p^{-1} ( \{ a_0 \})$. 
Suppose that $\tilde{h} : X \to \tilde{R}$ is
a lift of $h$ 
satisfying $\tilde{h}(x_0) = \tilde{a}_0$.
Then $\tilde{h}$ is unique.
\end{lemma}

\begin{lemma}[Path Lifting Lemma]
\label{lemma:PathLftingLemma}
Let $p: \tilde{R} \rightarrow R$ be a covering map,
and let $\tilde{a}_0 \in \tilde{R}$ and $a_0 \in R$ satisfy
$p(\tilde{a}_0) = a_0$.
Then for any path $\alpha : [0,1] \rightarrow R$ 
with initial point $a_0$, 
there exists a unique path $\tilde{\alpha}:[0,1] \rightarrow \tilde{R}$
with initial point $\tilde{a}_0$ 
such that
$p \circ \tilde{\alpha} = \alpha$.
\end{lemma}

\begin{lemma}[Homotopy Lifting Lemma]
\label{lemma:HomotopyLiftingLemma}
Let $p: \tilde{R} \rightarrow R$ be a covering map,
and let $a_0 , a_1 \in R$ and $\tilde{a}_0 \in \tilde{R}$ satisfy
$p(\tilde{a}_0) = a_0$.
Let $\alpha, \beta : [0,1] \rightarrow R$ be 
paths with $\alpha (0) = \beta (0) = a_0$ 
and $\alpha (1) = \beta (1) = a_1$,
and let $\tilde{\alpha}$, $\tilde{\beta}$ be the unique lifts 
of $\alpha$ and $\beta$ starting at$\tilde{a}_0$, respectively.
Suppose that $\alpha$ is path homotopic to $\beta$,
and that 
$F: [0,1] \times [0,1] \rightarrow R$ 
is a path homotopy from $\alpha$ to $\beta$, i.e.,
$F$ is a continuous map satisfying
(\ref{eq:initial_and_final_paths})
and (\ref{eq:initial_and_final_points}).
Then there exists a unique lift 
$\tilde{F} : [0,1] \times [0,1] \rightarrow \tilde{R}$ of
$F$ satisfying $\tilde{F}(0,0) = \tilde{a}_0$. 
Moreover,
$\tilde{\alpha}$ and  $\tilde{\beta}$
have the same final point,
and $\tilde{F}$ is a path homotopy from 
$\tilde{\alpha}$ to $\tilde{\beta}$.
\end{lemma}

\section{Subordination and Domain Inclusion for Universal Coverings}
We now combine the previous lemmas to obtain 
a fundamental lifting result:
any continuous map from 
a simply connected Riemann surface into 
a base surface admits a lift to the covering surface. 
Moreover, if the map is injective, then so is its lift.
\begin{proposition}
\label{proposition:a_map_of_simply_connected_domain_is_subordinate_to_covering_map}
Let $p: \tilde{R} \rightarrow R$
be an analytic covering map of a Riemann surface $\tilde{R}$
onto a Riemann surface $R$.
Suppose $X$ is a simply connected Riemann surface and
$h: X \rightarrow R$ is analytic.
Then, for any $x_0 \in X$ and $\tilde{a}_0 \in \tilde{R}$ with
$p(\tilde{a}_0) = h(x_0)$, there exists  a unique analytic map
$\tilde{h}: X \rightarrow \tilde{R}$ 
such that 
$p \circ \tilde{h} = h$ and $\tilde{h}(x_0) = \tilde{a}_0$. 
Furthermore, if $h$ is injective,
then $\tilde{h}$ is also injective, and
the restriction $p|_{\tilde{h}(X)} : \tilde{h}(X) \rightarrow h(X)$ 
is a conformal map, that is, an analytic bijection.
\end{proposition}
\begin{proof}
The uniqueness part follows from 
Lemma \ref{lemma:uniqueness_of_lift}.

For later use, we describe the construction of the map $\tilde{h}$.
For details, see \cite[Theorem~4.17]{Forster},
\cite[Theorem~V.5.1]{Massey}  or \cite[Lemma~79.1]{Munkres}. 

Take $x \in X$ arbitrarily. 
Let $\alpha : [0,1] \rightarrow X$ be a path
from the initial point $x_0$ to the end point $x$.
Let $\tilde{\alpha} : [0,1] \rightarrow \tilde{R}$
be the lifted path of $h \circ \alpha$ 
starting at $\tilde{a}_0$.
We claim that the final point $\tilde{a}:= \tilde{\alpha}(1)$ 
does not depend on the choice of $\alpha$.
To see this, let $\beta : [0,1] \rightarrow X$ be another path
from the initial point $x_0$ to the end point $x$,
and let $\tilde{\beta}$ be the lifted of 
$h \circ \beta$ starting at $\tilde{a}_0$.
Since  $X$ is simply connected,
$\alpha$  is homotopic to  $\beta$, and hence
$h \circ \alpha$  is homotopic to  $h \circ \beta$.
By the homotopy lifting lemma 
(Lemma \ref{lemma:HomotopyLiftingLemma}) 
$\tilde{\alpha}$ is homotopic to  $\tilde{\beta}$.
In particular $\tilde{\alpha}(1)= \tilde{\beta}(1)$.

Define $\tilde{h}: X \rightarrow \tilde{R}$ by
$\tilde{h}(x) = \tilde{a}$.
Then, clearly, $\tilde{h}$  satisfies  $p \circ \tilde{h} = h$
with $\tilde{h}(x_0) = \tilde{a}_0$.
Since $p$ is a local homeomorphism, 
it follows  easily that $\tilde{h}$ is continuous.
Moreover, since $h$ is analytic and $p$ is locally conformal,
$\tilde{h}$ is also analytic.

Assume that $h$ is injective. 
Then $h : X \to h(X)$ is conformal, and since $p \circ \tilde{h} = h$,
it follows that $\tilde{h}$ is also injective.
Thus $\tilde{h}: X \rightarrow \tilde{h}(X)$ is conformal,
and hence 
$p|_{\tilde{h}(X)} = h \circ \tilde{h}^{-1}: 
\tilde{h}(X) \rightarrow h(X)$
is a conformal map.
\end{proof}

At this point, we state a preparatory lemma.
Since the proof is straightforward, we omit the details.
\begin{lemma}
\label{lemma:boundary_value_of_conformal_map}
Let $R_1$ and $R_2$ be Riemann surfaces,
and let $f: D_1 \to D_2$ be a homeomorphism of
a domain $D_1$ in $R_1$ onto a domain $D_2$ in $R_2$.
Suppose that $f$ can be continuously extended to a point 
$c \in \partial D_1$.
Then $f(c) := \lim_{D_1 \ni \zeta \to c} f(\zeta) \in \partial D_2$.  
\end{lemma}

We now show that every simply connected domain 
in the base surface is an evenly covered neighborhood 
of each of its points.

\begin{proposition}
\label{proposition:simply-conncted-domain-is-evnbh}
Let $p: \tilde{R} \rightarrow R$
be an analytic covering map of a Riemann surface $\tilde{R}$
onto a Riemann surface $R$.
Let $a \in R$, and let
$D$ be a simply connected domain in $R$ containing $a$.
For each $\tilde{a} \in p^{-1}(\{ a \})$,
let $\tilde{h}_{\tilde{a}}: D \rightarrow \tilde{R}$
denote the unique lift of 
the inclusion map $\inc_D : D \rightarrow R$
satisfying $\tilde{h}_{\tilde{a}}(a) = \tilde{a}$.
Then $D$ is an evenly covered neighborhood of $a$, 
and for each $\tilde{a} \in p^{-1}({ a })$, 
the image $\tilde{h}_{\tilde{a}}(D)$ is the slice of 
$D$ containing $\tilde{a}$. 
Moreover, the two maps 
$p|_{\tilde{h}_{\tilde{a}}(D)} : \tilde{h}_{\tilde{a}}(D) \to D$
and
$\tilde{h}_{\tilde{a}} : D \to \tilde{h}_{\tilde{a}}(D)$
are conformal and inverses of each other.
In addition,
\[
   p^{-1}(D) 
 = 
 \bigcup_{\tilde{a} \in p^{-1}(\{ a \})} 
 \tilde{h}_{\tilde{a}} (D)
\]
gives the decomposition of $p^{-1}(D)$ into 
its connected components.
\end{proposition}

\begin{proof}
Fix $\tilde{b} \in p^{-1}(D)$ arbitrarily, and 
let $\tilde{D}_{\tilde{b}}$ 
be the connected component of $p^{-1}(D)$ containing $\tilde{b}$.
To prove the proposition, it is enough to show the following:
\begin{itemize}
 \item[{\rm (i)}]
There exists
$\tilde{a} \in p^{-1}(\{ a \})$ such that
$\tilde{a} \in \tilde{D}_{\tilde{b}} = \tilde{h}_{\tilde{a}} (D)$.
 \item[{\rm (ii)}]
If $\tilde{a}_1, \tilde{a}_2 \in p^{-1}(\{ a \})$ satisfy
$\tilde{h}_{\tilde{a}_1} (D) \cap \tilde{h}_{\tilde{a}_2} (D) \neq \emptyset$,
then $\tilde{a}_1 = \tilde{a}_2$.
\end{itemize}
We show (i).
Let $b = p(\tilde{b})$, and 
choose a path $\beta: [0,1] \rightarrow D$ 
from $b$ to $a$.
Let $\tilde{\beta} [0,1] \rightarrow \tilde{R} $ 
be the lift of $\beta$ from $\tilde{b}$, and 
set $\tilde{a} = \tilde{\beta}(1)$.
Since $p \circ \tilde{\beta} = \beta $,
we have that $\tilde{\beta}([0,1])$ is a connected subset
of $p^{-1}(D)$.
Therefore, by $\tilde{a}, \tilde{b} \in \tilde{\beta}([0,1])$,
we obtain $\tilde{a} \in \tilde{D}_{\tilde{b}}$.

By 
Proposition~\ref{proposition:a_map_of_simply_connected_domain_is%
_subordinate_to_covering_map} 
there exists the unique lift of $\tilde{h}_{\tilde{a}}$
of the inclusion map on $D$ satisfying 
$\tilde{h}_{\tilde{a}} (a) = \tilde{a}$
such that
$\tilde{h}_{\tilde{a}} : D \to \tilde{h}_{\tilde{a}} (D)$
and 
$p|_{\tilde{h}_{\tilde{a}} (D)} \rightarrow D$
are conformal.
In particular, since
$p|_{\tilde{h}_{\tilde{a}} (D)} \circ \tilde{h}_{\tilde{a}} 
= \inc_D$, 
the mappings $p|_{\tilde{h}_{\tilde{a}} (D)}$ and 
$\tilde{h}_{\tilde{a}}$ are inverses of each other.
Note that 
$\tilde{h}_{\tilde{a}} (D)$ is a connected subset of
$p^{-1}(D)$ and satisfies 
$\tilde{b} \in \tilde{h}_{\tilde{a}} (D)$.
Therefore we obtain
$\tilde{h}_{\tilde{a}} (D) \subset \tilde{D}_{\tilde{b}}$.

To see the reverse inclusion, 
suppose that 
$\tilde{D}_{\tilde{b}} \setminus \tilde{h}_{\tilde{a}} (D) 
\neq \emptyset$.
Then there exists 
$\tilde{c} \in \tilde{D}_{\tilde{b}} \cap %
\partial \tilde{h}_{\tilde{a}} (D)$.
Since $\tilde{c} \in \tilde{D}_{\tilde{b}}$,
we obtain $p( \tilde{c}) \in D$.
On the other hand, 
since $\tilde{c} \in \partial \tilde{h}_{\tilde{a}} (D)$,
the lemma implies 
$p(\tilde{c}) \in \partial p ( \tilde{h}_{\tilde{a}} (D)) = \partial D$,
which is a contradiction.

Now we show (ii).
Suppose 
$\tilde{c} \in 
\tilde{h}_{\tilde{a}_1} (D) \cap \tilde{h}_{\tilde{a}_2} (D)$
for some $\tilde{a}_1 , \tilde{a}_2 \in p^{-1}(\{ a \})$.
Take a path 
$\tilde{\alpha}_1: [0,1] \rightarrow \tilde{h}_{\tilde{a}_1} (D)$
from $\tilde{a}_1$ to $\tilde{c}$ 
and a path 
$\tilde{\alpha}_2: [0,1] \rightarrow \tilde{h}_{\tilde{a}_2} (D)$
from $\tilde{c}$ to $\tilde{a}_2$. 
Then the product $\tilde{\alpha}_1 * \tilde{\alpha}_2$
defined by
\[
 \tilde{\alpha}_1 * \tilde{\alpha}_2 (t)
 = \begin{cases}
   \tilde{\alpha}_1 (2t) \quad &  
    \text{if} \quad 0 \leq t \leq \frac{1}{2} , \\
    \tilde{\alpha}_2 (2t-1) \quad &
    \text{if} \quad \frac{1}{2} < t \leq 1 ,
   \end{cases}
\]
is a path from
$\tilde{a}_1$  to $\tilde{a}_2$.
Then $p(\tilde{\alpha}_1 * \tilde{\alpha}_2) $ is a path
from $p(\tilde{a}_1) =a$ to $p(\tilde{a}_2) =a$.
Hence it is a loop in $D$ based at $a$.
Since $D$ is simply connected,
the loop $p(\tilde{\alpha}_1*\tilde{\alpha}_2) $ is null-homotopic.
Hence the lifted path $\tilde{\alpha}_1*\tilde{\alpha}_2$ 
of $p(\tilde{\alpha}_1*\tilde{\alpha}_2) $ is also a loop,
and in particular, we obtain $\tilde{a}_1 = \tilde{a}_2$.
\end{proof}

The following is a straightforward application 
of Proposition 
\ref{proposition:a_map_of_simply_connected_domain_is_subordinate_to_covering_map}.
\begin{theorem}
\label{theorem:subordination-and-inclusion}
Let $f \in \mathcal{H}_0(\mathbb{D})$ and let
$g \in \mathcal{H}_0(\mathbb{D})$ be a universal covering.
Suppose that $f(0)=g(0)$.
Then $f \prec g$ if and only if $f(\mathbb{D}) \subset g(\mathbb{D})$.
\end{theorem}

\section{Locally Uniform Limits of Covering Maps %
and Kernel Convergence of Images}
We now recall a growth estimate for
analytic functions in $\mathbb{D}$ that omit
the values $0$ and $1$. 
\begin{lemma}
\label{lemma:avoid0and1}
There exists a constant $K > 0$ such that
for any analytic function 
$g: \mathbb{D} \rightarrow \mathbb{C} \backslash \{ 0,1 \}$, 
the following inequality holds:
\[
   \log |g(z)| 
 \leq 
 \left( K + \log^+ |g(0)| \right) \frac{1+|z|}{1-|z|}, \quad
 z \in \mathbb{D} .
\] 
\end{lemma}
Here, $\log^+ y := \max \{ \log y , 0 \} $ for $y>0$.
For a proof with $k=7$, 
see \cite[Theorem~1-13]{Ahlfors:ConformalInvariants}. 
For more precise estimate with $K=\pi$, we refer the reader to
\cite{Hayman1947} and \cite {Jenkins1955}.

\begin{theorem}
\label{thm:2ndPart}
Let $D$ be a hyperbolic domain in $\mathbb{C}$, 
and let $\{ f_n \}_{n=1}^\infty$ be a sequence 
of analytic covering maps of $D$.
Suppose that $\{f_n \}_{n=1}^\infty$ converges 
locally uniformly on $D$ to a nonconstant analytic function $f$.
Then $f$ is also a covering map, and 
$f_n(D) \rightarrow f(D)$ 
as $n \rightarrow \infty$ in the sense of kernel.
\end{theorem}
\begin{proof}
We divide the proof into several steps.

\noindent\textbf{Step 1.} 
Let $z_0 \in D$, and let $V$ be a simply connected domain 
such that $w_0 := f(z_0) \in V$ and
$V \subset f_n(D)$ for all $n \in \mathbb{N}$.
We show that there exists a univalent analytic function 
$\varphi : V \rightarrow D$ satisfying
$f (\varphi(w)) \equiv w$ on $V$ and $\varphi (w_0) = z_0$.
Once this is established, it follows that $V \subset f(D)$.

For $n \in \mathbb{N}$, let $w_n = f_n (z_0)$.
Since $w_n \to f(z_0) = w_0 \in V$,
we may assume without loss of generality that
$w_n \in V$ for all $n \in \mathbb{N}$. 
Applying Proposition~\ref{proposition:simply-conncted-domain-is-evnbh}
to the covering map $f_n : D \rightarrow f_n(D)$,
there exists a subdomain $\tilde{V}_n$ of $D$ 
and a conformal map $\varphi_n : V \to \tilde{V}_n$
such that $f_n \circ \varphi_n (w) = w$ on $V$ 
and $\varphi_n (w_n) = z_0 $.
Note that 
the restriction $f_n|_{\tilde{V}_n}$ is a conformal map
of $\tilde{V}_n$ onto $V$, and that 
$(f_n|_{\tilde{V}_n})^{-1} = \varphi_n$.

We  claim that the family $\{ \varphi_n \}_{n=1}^\infty$ 
is locally uniformly bounded on $V$
and thus forms a normal family.
Indeed, since $D$ is hyperbolic, we can choose
distinct points $z_1, z_2 \in \mathbb{C} \setminus D$.
Let $h : \mathbb{D} \rightarrow V$ be a conformal map
with $h(0) = w_0 \in V$, 
and put $\zeta_n = h^{-1}(w_n)$ for $n \in \mathbb{N}$.
Then the function
\[
 H_n (\zeta ) 
 = 
 \frac{ \varphi_n \left( 
 h \left( \frac{\zeta + \zeta_n}{1+\overline{\zeta_n} \zeta} 
    \right) 
 \right) - z_1}{z_2 - z_1 }, \quad
 \zeta \in \mathbb{D},
\]
omits the values $0$ and $1$, 
and satisfies $H_n(0) = (z_0 - z_1)/(z_2-z_1)$.
Thus by Lemma~\ref{lemma:avoid0and1}
\[
 \log \left| \frac{ \varphi_n \left( 
 h \left( \frac{\zeta + \zeta_n}{1+\overline{\zeta_n} \zeta} 
    \right) 
 \right) - z_1}{z_2 - z_1 } \right| 
 \leq
 \left( K + \log^+ \left| \frac{z_0 - z_1}{z_2 -z_1} \right| \right)
 \frac{1+|\zeta|}{1-|\zeta|} .
\]
By replacing 
$\zeta$ with $\frac{\zeta - \zeta_n}{1-\overline{\zeta_n} \zeta}$
and using the inequality  
$\left| \frac{\zeta - \zeta_n}{1-\overline{\zeta_n} \zeta} \right| 
\leq \frac{|\zeta| + |\zeta_n|}{1+|\zeta_n| |\zeta|} $,
we obtain
\begin{align*}
 \log \left| \frac{ \varphi_n \left( 
 h \left( \zeta \right) 
 \right) - z_1}{z_2 - z_1 } \right| 
 \leq & \, 
 \left( K + \log^+ \left| \frac{z_0 - z_1}{z_2 -z_1} \right| \right) 
 \frac{1+ \frac{|\zeta| + |\zeta_n|}{1+|\zeta_n| |\zeta|}}
 {1-\frac{|\zeta| + |\zeta_n|}{1+|\zeta_n| |\zeta|}} \\
 = & \, 
 \left( K + \log^+ \left| \frac{z_0 - z_1}{z_2 -z_1} \right| \right) 
 \frac{(1+|\zeta_n|)(1+|\zeta|)}{(1-|\zeta_n|)(1-|\zeta|)} . 
\end{align*}
Since $\zeta_n \rightarrow h^{-1}(w_0) = 0$, 
it follows that the family $\{ \varphi_n \}_{n=1}^\infty$ is 
locally uniformly bounded on $V$.

Choose a subsequence $\{ \varphi_{n_k} \}_{k=1}^\infty$ 
such that $\{ \varphi_{n_k} \}_{k=1}^\infty$ converges 
locally uniformly on $V$ to an analytic function  $\varphi$.
By the identity $\varphi_{n_k}(w_{n_k})= z_0$,
we obtain
$\varphi(w_0) = \lim_{k \to \infty} \varphi_{n_k}(w_{n_k}) = z_0$.
By Hurwitz's theorem 
we have either $\varphi$ is univalent on $V$ or $\varphi \equiv z_0$.
Since $f_{n_k}(\varphi_{n_k}(w)) = w$ on $V$, we have
\[
    f'(z_0) \varphi_{n_k}'(w_{n_k}) =
   f'(\varphi_{n_k}(w_{n_k})) \varphi_{n_k}'(w_{n_k}) 
 =  1 .
\]
Letting $k \to \infty$, we obtain
$f'(z_0) \varphi'(w_0) = 1$.
Thus, $\varphi'(w_0) \neq 0$, and hence 
$\varphi$ is univalent on $V$.

We now show that $\varphi(V) = \tilde{V} \subset D$.
Since $\varphi_{n_k}(V) \subset D$,
we have $\tilde{V} =  \varphi(V) \subset \overline{D}$.
Assume for contradiction, that
$\varphi(V) \cap \partial D \neq \emptyset$.
Then there exists $w^* \in V$ and $z^* \in \partial D$
such that $z^* = \varphi (w^*)$.
Since $\varphi_{n_k} \rightarrow \varphi$ locally uniformly on $V$,
it follows from Proposition~\ref{prop:first-part-of-the-Hejhal-theorem}
that 
$z^* \in \varphi_{n_k} (V) \subset D$ for all sufficiently large $k$.
This contradicts the assumption that $z^* \in \partial D$.
Therefore, $\varphi(V) \cap \partial D = \emptyset$,
and thus $\varphi(V) \subset D$.

\noindent\textbf{Step 2.} 
We show that $f_n (D) \to f(D)$ as $n \to \infty$
in the sense of kernel.
By Proposition \ref{prop:first-part-of-the-Hejhal-theorem},
the condition (a') is satisfied,
so it remains to verify the condition (b').

Suppose, to the contrary, that (b') does not hold.
Then there exist a point $c \in \partial f(D)$, 
a constant $\varepsilon > 0$
and a subsequence $\{ f_{n_k} (D) \}_{k=1}^\infty$ such that
$d (c , \partial f_{n_k} (D) ) \geq \varepsilon$ for all $k$.
Since $\mathbb{D}(c, \varepsilon)$ is connected,
this implies that either
$\mathbb{D}(c, \varepsilon) \subset f_{n_k} (D) $ or
$\mathbb{D}(c, \varepsilon) \subset \mathbb{C} 
\setminus \overline{f_{n_k} (D) }$ holds for each $k$.

On the other hand, since $c \in \partial f(D)$, 
there exists a point $a^* \in D$ 
$c^* :=f(a^*) \in f(D) \cap \mathbb{D}(c, \varepsilon)$.
By Proposition~\ref{prop:first-part-of-the-Hejhal-theorem},
we have that $c^* \in f_{n_k} (D) $ for all sufficiently large $k$.
Hence, 
$f(a^*) \in \mathbb{D}(c, \varepsilon) \subset f_{n_k} (D) $
for all sufficiently large $k$.
It then follows from  Step 1 that
$\mathbb{D}(c, \varepsilon) \subset f(D)$, which contradicts
the assumption that $c \in \partial f(D)$.

\noindent\textbf{Step 3.} 
We show that the mapping $f: D \to f(D)$ is a covering map.
To this end, it suffices to prove that 
for every $w_0 \in f(D)$, there exists
an evenly covered neighborhood of $w_0$. 

Choose a simply connected domain $V$ with $w_0 \in V$ 
such that $\overline{V}$ is compact and contained in $f(D)$.
Let $U$ be a connected component of $f^{-1}(V)$.
Choose $z^* \in U$ arbitrarily and set $w^*=f(z^*)$.
Note that $U$ is the largest connected subset of $f^{-1}(V)$
that contains $z^*$. 
Since $f_n (D) \to f(D)$ in the sense of kernel, 
condition (a') ensures that 
$\overline{V} \subset f_n (D)$ for all sufficiently
large $n \in \mathbb{N}$.
Therefore, by Step 1, there exists a univalent function 
$\varphi : V \rightarrow D$
such that $\varphi(f(z^*)) = z^*$ and $f(\varphi(w))\equiv w$ on $V$.
Since $\varphi (V)$ is connected, contained in $f^{-1}(V)$ and 
contains $z^*$,
we have $\varphi (V) \subset U$.

To prove the reverse inclusion, suppose to the contrary
that  $U \setminus \varphi (V) \not= \emptyset$. 
Then there exists a point $z' \in U \cap \partial \varphi (V)$.
Since $z' \in U$, it follows that $f(z') \in V$.
On the other hand, since the restriction $f|_{\varphi(V)}$ 
is a conformal map of $\varphi(V)$ onto $V$, 
and $z' \in \partial \varphi (V)$, 
Lemma \ref{lemma:boundary_value_of_conformal_map}
implies that $f(z') \in \partial V$, which is a contradiction.

We have shown that 
for any connected component $U$ of $f^{-1}(V)$, there exists
a conformal map $\varphi : V \rightarrow U$ such that
$f|_U = \varphi^{-1}$.
Thus, $V$ is an evenly covered neighborhood of $w_0$. 
\end{proof}

\begin{corollary}
\label{corollary:inverse_maps_on_evenly_covered_nbh_converges}
Let $D$, $\{ f_n \}_{n=1}^\infty$ and $f$ be as in 
Theorem~\ref{thm:2ndPart}.
Let $a \in D$, and let $V $ be a simply connected 
bounded domain such that
$f(a) \in V \subset \overline{V} \subset f(D)$.
Then there exists $N \in \mathbb{N}$ such that 
$f_n(a) \in V \subset f_n(D)$ for $n \geq N$.
Moreover, for $n \geq N$, 
let $\varphi_n = f_n^{-1}$ on $V$ with $\varphi_n(f_n(a))=a$.
Then $\varphi_n \rightarrow \varphi$ locally uniformly on $V$,
where $\varphi = f^{-1}$ on $V$ with $\varphi(f(a))=a$.
\end{corollary}
\begin{proof}
The existence of $N$ and the functions $\varphi_n$ 
follows from Proposition \ref{prop:first-part-of-the-Hejhal-theorem}
and Proposition \ref{proposition:simply-conncted-domain-is-evnbh},
respectively.

From Step 1 in the proof of the theorem
it follows that every subsequence of
$\{ \varphi_n \}_{n=1}^\infty$ 
has a further subsequence that 
converges locally  uniformly on $V$ 
to $\varphi = (f|_{\tilde{V}})^{-1}$.
Therefore, the entire sequence $\{ \varphi_n \}_{n=1}^\infty$ 
converges locally  uniformly on $V$ to  $\varphi$.
\end{proof}

\section{Degenerate Limits and Necessity of Normality}
\begin{theorem}
\label{thm:2ndPart-degenerate-case}
Let $w_0 \in \mathbb{C}$, and let $D$ be a hyperbolic domain 
in $\mathbb{C}$, 
Let $\{ f_n \}_{n=1}^\infty$ be a sequence of 
analytic covering maps of $D$
with $w_0 \in f_n(D)$ for all $n \in \mathbb{N}$. 
Suppose $\{f_n \}_{n=1}^\infty$ converges locally uniformly on $D$
to the constant function $w_0$.
Then 
$d(w_0, \partial f_n(D)) \rightarrow 0$ as $n \rightarrow \infty$, i.e.,
the family $\{ f_n(D) \}_{n=1}^\infty $ degenerates to 
the singleton $\{ w_0 \}$.
\end{theorem}
\begin{proof}
Suppose that $f_n \rightarrow w_0$ locally uniformly on $D$.
Assume, for contradiction, that $d(w_0, \partial f_n(D)) \not\to 0$.
Then there exist $\varepsilon> 0$ and a subsequence 
$\{ f_{n_k}\}_{k=1}^\infty$ such that
$d(w_0, \partial f_{n_k}(D)) \geq 2 \varepsilon$.
Since  $w_0 \in f_n (\mathbb{D})$,
this implies  
$\overline{\mathbb{D}}(w_0, \varepsilon) 
\subset \mathbb{D}(w_0, \varepsilon) \subset f_{n_k}(D)$ 
for $k \in \mathbb{N}$.

Choose $a \in D$ arbitrarily.
Since $f_{n_k} \rightarrow w_0$, there exists $k_0 \in \mathbb{N}$
such that $f_{n_k} (a) \in \mathbb{D}(w_0, \varepsilon)$ 
for all $k \geq k_0$.
Hence, by 
Corollary~\ref{corollary:inverse_maps_on_evenly_covered_nbh_converges},
there exists 
a univalent analytic function 
$\varphi_{n_k}: \mathbb{D}(w_0, \varepsilon) \rightarrow D$
such that $\varphi_{n_k}(f_{n_k}(a)) = a$ 
and $f_{n_k} (\varphi_{n_k} (w)) \equiv w$ 
on $\mathbb{D}(w_0, \varepsilon)$.
In particular, we have
\begin{equation}
\label{eq:diff_f_and_varph}
 f_{n_k}'(a) \varphi_{n_k}'(f_{n_k}(a)) = 1 . 
\end{equation}
Since $D$ is hyperbolic and $f_{n_k}(a) \rightarrow w_0$,
$\{ \varphi_{n_k} \}_{k=1}^\infty$ 
forms a  normal family as in Step 2 of 
the proof of Theorem \ref{thm:2ndPart}.
Consequently, there exists $M>0$ 
such that $|\varphi_{n_k}'(f_{n_k}(a))| \leq M$ 
for all $k \in \mathbb{N}$.
Thus, by (\ref{eq:diff_f_and_varph}), we obtain
$|f_{n_k}'(a)| \geq \frac{1}{M}$ for all $k \in \mathbb{N}$.
However, since 
$f_{n_k} \rightarrow w_0$,
it follows that $f_{n_k}'(a) \rightarrow 0$,
which is a contradiction.
\end{proof}

The following theorem is not new (see Hejhal \cite{Hejhal}).
Nevertheless, for the sake of completeness, we include a proof here.
\begin{theorem}
\label{thm:3rd-part}
Let $\{ \Omega_n \}_{n=1}^\infty$ be a sequence of hyperbolic domains 
in $\mathbb{C}$ 
that converges to a hyperbolic domain $\Omega$ in $\mathbb{C}$
in the sense of kernel.
Let $f$ and $f_n$ $(\text{for each } n \in \mathbb{N})$ 
be analytic universal covering maps
of $\mathbb{D}$ onto $\Omega$ and $\Omega_n$, respectively. 
Let $\{ a_n \}_{n=1}^\infty $ 
be a sequence in $\mathbb{D}$ that converges to 
a point $a \in \mathbb{D}$.
Suppose that $f_n(a_n) \rightarrow f(a)$ 
and $\arg f_n'(a_n) \rightarrow \arg f'(a)$. 
Then 
$f_n \rightarrow f$ as $n \rightarrow \infty$ locally uniformly 
on $\mathbb{D}$.
\end{theorem}
\begin{proof}
Choose $r > 0$ such that $\overline{\mathbb{D}}(f(a),r) \subset \Omega$.
Then, there exists $n_0 \in \mathbb{N}$ such that
$f_n (a_n ) \in \mathbb{D}(f(a),r) \subset 
\overline{\mathbb{D}}(f(a),r) \subset \Omega_n$ 
for all $n \geq n_0$.
By Proposition~\ref{proposition:simply-conncted-domain-is-evnbh},
there exists 
a univalent function 
$\varphi_n :\mathbb{D}(f(a),r) \rightarrow \mathbb{D}$ 
such that $f_n (\varphi_n (w)) \equiv w$ on $\mathbb{D}(f(a),r)$
and $\varphi_n (f_n (a_n)) = a_n$.
By the Schwarz-Pick lemma, we have 
\[
  |\varphi_n '(f_n(a_n))| 
 \leq 
 \frac{r(1-|a_n|^2 )}{r^2-|f_n(a_n)-f(a)|^2}.
\]
Since $a_n \rightarrow a$ and $f_n(a_n) \rightarrow f(a)$,
there exists a constant $M > 0$ such that 
$|\varphi_n '(f_n(a_n))|  \leq M$ for all $n \in \mathbb{N}$.
Since $f_n (\varphi_n (w)) \equiv w$,
it follows that 
$|f_n'(a_n)|= \frac{1}{|\varphi_n'(f_n(a_n))|} \geq \frac{1}{M}$.

Since $\Omega$ is hyperbolic, we can choose 
distinct points $w_1, w_2 \in \partial \Omega$.
Because $\Omega_n \rightarrow \Omega$ in the sense of kernel,
there exist sequences 
$\{ w_1^{(n)} \}_{n=1}^\infty$ and $\{ w_2^{(n)} \}_{n=1}^\infty$ 
with $w_1^{(n)} , w_2^{(n)} \in \partial \Omega_n$
such that
$w_1^{(n)} \rightarrow w_1$ and $w_2^{(n)} \rightarrow w_2$
as $n \rightarrow \infty$.
We may assume that $w_1^{(n)} \not= w_2^{(n)}$ for all $n \in \mathbb{N}$.
Since each $f_n$ omits both $w_1^{(n)}$ and $w_2^{(n)}$, 
and $a_n \rightarrow a$, $f_n(a_n ) \rightarrow f(a)$, 
it follows from Lemma~\ref{lemma:avoid0and1},
as in the proof of Theorem \ref{thm:2ndPart},  
that the sequence $\{ f_n \}_{n=1}^\infty$ forms a normal family.

Let $\{ f_{n_k} \}_{k=1}^\infty$ 
be a subsequence  of $\{ f_n \}_{n=1}^\infty$ 
that converges to some analytic function $g$ 
locally uniformly on $\mathbb{D}$.
Then we have 
\begin{align*}
  g(a) 
 = & \, \lim_{k \rightarrow \infty }f_{n_k}(a)
 = \lim_{k \rightarrow \infty } f_{n_k}(a_{n_k}) = f(a) ,
\\
  |g'(a)| 
 = & \, \lim_{k \rightarrow \infty } |f_{n_k}'(a)|
 = \lim_{k \rightarrow \infty } |f_{n_k}'(a_{n_k})| \geq \frac{1}{M} .
\end{align*}
In particular, this implies that $g$ is nonconstant.
Therefore, by Theorem \ref{thm:2ndPart}, 
$g$ is a covering and 
$\Omega_{n_k} = f_{n_k} (\mathbb{D}) \rightarrow g(\mathbb{D}) $
in the sense of kernel.

Since $\Omega_{n_k} \rightarrow \Omega$ and
$f(a) = g(a) \in \Omega \cap g(\mathbb{D})$,
it follows from Proposition \ref{prop:quasi-uniqueness-of-limit-domain}
that $g(\mathbb{D}) = \Omega$. 
Therefore, $g$ is a universal covering map of $\mathbb{D}$
onto $\Omega$ with $g(a) = f(a)$,
satisfying
\[
 \arg g'(a) 
 = \lim_{k \rightarrow \infty } \arg f_{n_k}'(a) 
 = \lim_{k \rightarrow \infty } \arg f_{n_k}'(a_{n_k})
 = \arg f'(a) . 
\]
By the uniqueness theorem for universal covering maps, 
we conclude that $g$ coincides with $f$.

We have shown that 
$\{ f_n \}_{n=1}^\infty$ forms a normal family, 
and that every convergent subsequence of $\{ f_n \}_{n=1}^\infty$
converges to $f$ locally uniformly on $\mathbb{D}$.
Therefore, the original sequence $\{ f_n \}_{n=1}^\infty$
converges to $f$ locally uniformly on $\mathbb{D}$.
\end{proof}

In the above theorem we cannot drop the assumption that 
$\Omega$ is hyperbolic.
For example, define
\[
  f_n(z) 
  = \frac{1}{n} e^{(\log n ) \frac{1+z}{1-z} } -1  ,
  \quad z \in \mathbb{D} .
\]
Then $f_n$ is the unique analytic covering map
of $\mathbb{D}$ onto 
$\Omega_n =\mathbb{C} \backslash \overline{\mathbb{D}}(-1,n^{-1}) $
with $f_n(0) =0$ and $f_n'(0)> 0$.
The sequence $\{ \Omega_n \}_{n=1}^\infty$ converges, 
in the sense of kernel,
to  the non-hyperbolic domain $\mathbb{C} \backslash \{ -1 \}$ .
Note that for $x \in (-1,1)$,
we have $f_n(x) = n^{\frac{2x}{1-x}}$.
It is easy to verify that 
\[
  \lim_{n \rightarrow \infty} f_n(x) 
  = 
  \begin{cases}
  \infty & \text{if } 0<x<1, \\
  0 & \text{if } -1<x<0 .
 \end{cases}
\]

\begin{theorem}
\label{thm:3rdPart-degenerate-case}
Let $w_0 \in \mathbb{C}$, 
and let $\{ \Omega_n \}_{n=1}^\infty$ 
be a sequence of hyperbolic domains 
in $\mathbb{C}$ with $w_0 \in \Omega_n$ for all $n \in \mathbb{N}$.
For each $n \in \mathbb{N}$, 
let $f_n$ be an analytic universal 
covering map of $\mathbb{D}$ onto $\Omega_n$ with $f_n(0) = w_0$.
Suppose that $\{ \Omega_n \}_{n=1}^\infty$
degenerates to $\{ w_0 \}$ in the sense of kernel,
and that the sequence $\{ f_n \}_{n=1}^\infty$ forms a normal family. 
Then $\{ f_n \}_{n=1}^\infty$ converges locally uniformly 
on $\mathbb{D}$ to the constant function 
$w_0$ as $n \rightarrow \infty$.
\end{theorem}
\begin{proof}
Assume that $\{ f_n \}_{n=1}^\infty$ does not converge to $w_0$ 
locally uniformly on $\mathbb{D}$.
Then there exist 
$\varepsilon > 0$, $r \in (0,1)$,
a subsequence $\{ f_{n_k} \}_{k=1}^\infty$ and 
a sequence $\{z_k \}_{k=1}^\infty \subset \overline{\mathbb{D}} (0,r)$
such that $|f_{n_k}(z_k) - w_0| \geq \varepsilon$ 
for all $k \in \mathbb{N}$.
By passing to a further subsequence if necessary, 
we may assume
that $z_k \to z_0 $ 
and $f_{n_k} \rightarrow f$ locally uniformly on $\mathbb{D}$
for some point $z_0 \in \overline{\mathbb{D}} (0,r)$
and some analytic function $f$ on $\mathbb{D}$.
Then $f(z_0) \not= w_0$.
Since $f_n(0) = w_0$ for $n \in \mathbb{N}$, 
we have $f(0)=w_0$.
Therefore, $f$ is nonconstant.
By Theorem~\ref{thm:2ndPart}
we have that 
the function $f$ is a covering map 
and $\Omega_{n_k} \rightarrow f(\mathbb{D})$
in the sense of kernel.

Choose $\delta > 0$ such that 
$\overline{\mathbb{D}}(w_0, \delta ) \subset f(\mathbb{D})$.
Then condition (a') implies 
that $\overline{\mathbb{D}}(w_0, \delta ) \subset \Omega_{n_k}$ for 
all sufficiently large $k$,  
which contradicts the assumption that
$\{ \Omega_n \}_{n=1}^\infty$
degenerates to $\{ w_0 \}$ in the sense of kernel.
\end{proof}

In the above theorem, we cannot drop the assumption that 
$\{ f_n \}_{n=1}^\infty$ forms a normal family.
For example, let 
\[
  f_n(z) = \frac{1}{n^2}e^{(\log n ) \frac{1+z}{1-z} } - \frac{1}{n},
  \quad z \in \mathbb{D} 
\]
for $n \geq 2$.
Then $f_n$ is the unique analytic covering map
of $\mathbb{D}$ onto 
$\Omega_n := \mathbb{C} \setminus \overline{\mathbb{D}}(-n^{-1},n^{-2}) $
with $f_n(0) =0$ and $f_n'(0)> 0$.
Since $d(0, \partial f_n(\mathbb{D})) = \frac{1}{n} \to 0$,
the sequence 
$\{ \Omega_n \}_{n=1}^\infty$ degenerates to $\{ 0 \}$.
However, 
$\{ f_n \}_{n=1}^\infty$ does not converge to $0$,
since
\[
  \lim_{n \rightarrow \infty} f_n(x) 
  =
  \begin{cases}
  \infty & \text{if} \quad 3^{-1}<x<1 , \\
  0 & \text{if} \quad -1<x<3^{-1}.
  \end{cases}
\] 
Furthermore, this implies that 
the the sequence $\{ f_n \}_{n=1}^\infty$ does not form 
a normal family.

\section{One-Parameter Families and Continuity in the Kernel Sense}
To conclude this chapter, we define
the continuity of $\{ \Omega_t \}_{t\in I}$ at 
a point $t_0 \in I$,
and provide a characterization in terms of sequences.

\begin{definition}
Let $I \subset [-\infty, \infty ]$ be an interval, 
and let $\{ \Omega_t \}_{t \in I}$ 
be a family of domains in $\hat{\mathbb{C}}$.
We say that $\{ \Omega_t \}_{t \in I}$ is 
\emph{continuous at $t_0 \in I$ (in the sense of kernel)}
if the following two conditions are satisfied:
\begin{itemize}
 \item[$(a^*)$]
For every compact subset $K$ of $\Omega_{t_0}$ 
there exists $\delta > 0$ such that
$K \subset \Omega_t$ for all  $t \in I$ with $0< |t-t_0| < \delta$.
\item[$(b^*)$] 
For every $c \in \partial \Omega_{t_0}$, 
we have $d^* (c, \partial \Omega_t ) \rightarrow 0$ 
as $I \backslash \{t_0 \} \ni t \rightarrow t_0$.
\end{itemize} 
If $\{ \Omega_t \}_{t \in I}$ is continuous at every $t_0 \in I$,
we simply say that it is 
\emph{continuous (in the sense of kernel)}.
\end{definition}
It is easy to see that
$\{ \Omega_t \}_{t \in I}$ is 
continuous at $t_0 \in I$
if and only if 
$\Omega_{t_n} \rightarrow \Omega_{t_0}$ as $n \rightarrow \infty$
in the sense of kernel 
for every
sequence $\{t_n \}_{n=1}^\infty \subset I$ 
with $t_n \neq t_0$ and $t_n \rightarrow t_0$.

Here we summarize the results concerning the relationship 
between 
a one-parameter family of hyperbolic domains 
and the corresponding family of universal covering maps. 
The following theorem, which generalizes Theorem 
\ref{theorem:f_t_and_Omega_t},
follows directly from  
Theorems~\ref{theorem:subordination-and-inclusion}, 
\ref{thm:2ndPart} and~\ref{thm:3rd-part}.
\begin{theorem}
Let $\{ \Omega_t \}_{t \in I}$ be a family of hyperbolic domains
in $\mathbb{C}$ with $0 \in \Omega_t$ for all $t \in I$. 
For each $t$, let $f_t : \mathbb{D} \rightarrow \Omega_t$
be the universal covering map normalized by 
$f_t(0)=0$ and $f_t'(0) > 0$.
Then the following assertions hold:
\begin{itemize}
 \item[{\rm (i)}]
$\{ f_t \}_{t \in I}$ is a Loewner chain
if and only if
$\{ \Omega_t \}_{t \in I}$ is nondecreasing; that is, 
$\Omega_s \subset \Omega_t$ whenever $s,t \in I$ with $s \leq t$,
 \item[{\rm (ii)}]
$\{ f_t \}_{t \in I}$ is a strictly increasing Loewner chain
if and only if
$\{ \Omega_t \}_{t \in I}$ is strictly increasing, i.e., 
$\Omega_s \subsetneq \Omega_t$ whenever $s,t \in I$ with $s < t$,
 \item[{\rm (iii)}]
$\{ f_t \}_{t \in I}$ is continuous
if and only if
$\{ \Omega_t \}_{t \in I}$ is continuous in the sense of kernel.
\end{itemize}
\end{theorem}

\chapter{Kernel convergence and  connectivity of domains}
\label{chapter:Connectivity}
The \emph{connectivity} of a domain  $\Omega$ 
in $\mathbb{C}$ (or in $\hat{\mathbb{C}}$)
is defined as the number of connected components of 
$\hat{\mathbb{C}} \setminus \Omega$. 
We denote this number by $C(\Omega)$.
Following a common convention we write $C(\Omega) = \infty$ 
when the number is not finite, that is, 
we ignore the distinction between countable 
and uncountable cardinalities 
and simply set $C(\Omega)=\infty$ 
whenever the number is not finite.

In this chapter, we first show 
that if $\Omega_n \to \Omega$ as $n \to \infty$,
then 
$\liminf_{n \to \infty} C(\Omega_n) \geq C( \Omega)$, that is,
the connectivity of domains is lower semicontinuous with respect to 
kernel convergence. 
After establishing a few auxiliary results, 
we prove a fundamental result 
(see Theorem~\ref{thm:nonvanishing-of-components})
concerning continuous and nondecreasing families of domains 
in $\hat{\mathbb{C}}$.
It is noteworthy that these results have natural analogues 
in the context of universal covering maps,
provided that all domains are hyperbolic.

\section{Lower Semicontinuity of Connectivity under Kernel Convergence}
\begin{theorem}
\label{thm:lower-semi-continuity}
Let 
$\{ \Omega_n \}_{n=1}^\infty$ be a sequence of domains in 
$\hat{\mathbb{C}}$, and
suppose that $\{ \Omega_n \}_{n=1}^\infty$ 
converges to a domain $\Omega$ 
in $\hat{\mathbb{C}}$ in the sense of kernel.
Then,
\begin{equation}
\label{ineq:lower-semi-continuity}
    \liminf_{n \rightarrow \infty} C(\Omega_n) \geq C(\Omega ). 
\end{equation}
 \end{theorem}
\begin{proof}
If $C(\Omega ) = 0$, then 
the inequality (\ref{ineq:lower-semi-continuity})
clearly holds.
If $C(\Omega ) = 1$, then 
$\hat{\mathbb{C}}\setminus \Omega \neq \emptyset$ and hence
$\partial \Omega$ is not empty.
Choose an arbitrary point $c \in \partial \Omega$.
By
Definition~\ref{def:kernel-convergence-without-reference-point},
we have $d^*(c ,\partial \Omega_n) \rightarrow 0$ 
as $n \rightarrow \infty$.
This implies 
$\partial \Omega_n \neq \emptyset$ for all sufficiently large $n$,
and therefore
$C(\Omega_n) \geq 1$ for all sufficiently large $n$.

Assume $C(\Omega) \geq 2$.
If $C(\Omega) = \infty$, 
choose $k \in \mathbb{N}$ with $k\geq 2$ arbitrarily; 
otherwise, let $k=C(\Omega)$.
Let $E_1, \ldots , E_k$ be 
$k$ distinct components of $\hat{\mathbb{C}} \setminus \Omega$. 
For each $i = 1,\ldots, k$, by Lemma~\ref{lemma:separation},
there exists a simple closed curve 
$\gamma_i :[0,1] \rightarrow \Omega$ 
that separates $E_i$ and $\bigcup_{j \neq i }E_j $.

Let $D_i$ be the component of 
$\hat{\mathbb{C}} \setminus \gamma_i([0,1])$
that contains $E_i$,
and let $D_i'$ be the other component.
Then $E_i$ is a connected set 
satisfying $E_i \subset D_i \cap \bigcap_{j \neq i}D_j'$,
for each $i=1, \ldots , k$.
Let $V_i$ denote the connected component 
of $D_i \cap \bigcap_{j  \neq i}D_j'$ that contains $E_i$.
Clearly, $E_i \subset V_i$ for each $i$, 
and the sets $V_1, \ldots , V_k$ are mutually disjoint.

For each $i=1, \ldots , k$,
choose an arbitrary point $\zeta_i \in \partial E_i$ and set 
\[
  \delta 
 = \min_{i,j =1, \ldots , k} 
 d^*( \zeta_i, \gamma_j([0,1]) ) > 0 .
\]
Since 
$\mathbb{D}^*(\zeta_i, \delta)$ is  connected and satisfies 
$\zeta_i \in \mathbb{D}^*(\zeta_i, \delta) \cap V_i$, and
\[
 \mathbb{D}^*(\zeta_i, \delta) \cap \partial V_i 
 \subset 
 \mathbb{D}^*(\zeta_i, \delta) \cap (\gamma_1 ([0,1]) 
 \cup \cdots \cup \gamma_k ([0,1])) 
 = \emptyset ,
\]
it follows that $\mathbb{D}^*(\zeta_i, \delta) \subset V_i$
for each  $i$.

By Definition~\ref{def:kernel-convergence-without-reference-point}~(a)
there exists $N \in \mathbb{N}$ such that
for all $n \geq N$ and $i =1,\ldots, k$, we have 
\[
   \gamma_1 ([0,1]) \cup \cdots \cup \gamma_k ([0,1]) \subset \Omega_n 
  \quad \text{and} \quad 
  d^*(\zeta_i, \partial \Omega_n ) < \delta . 
\]
Therefore, for each $i =1, \ldots , k$ and $n \geq N$ 
there exists a point
$\zeta_i^{(n)} \in 
\partial \Omega_n \cap \mathbb{D}^*(\zeta_i, \delta)$.

\begin{tikzpicture}[scale=0.8]
\draw plot [smooth cycle, tension={0.6}]
coordinates { (3,2.2) (4,1.8) (5,1.7) (6,1.6) (7,1.8) (8,2.2) (8.8,3) 
(9.3,5) (9,6) (8,6.3) (7,3.8) (6,3) (5,2.8) (4.2,3) (3.2,4) (3,5)
(3.8,6) (5,6) (6,5) (6,6.4) (5,7) (3,6.4) (2,5) (2.3,3) };
\draw[rotate=30] (8,4) circle (4 and 2);
\draw[rotate=-30] (3,10) circle (4.5 and 1);
\begin{scope}[xshift=3cm,yshift=7.2cm]
\draw[fill=gray,rotate=60] (0,0) ellipse (0.3cm and 0.2cm) 
node[left=0.1cm] {$E_1$};   
\end{scope}
\begin{scope}[xshift=7.5cm,yshift=3cm]
\draw[fill=gray,rotate=120] (0,0) ellipse (0.3cm and 0.2cm) 
node[above=0.1cm] {$E_2$};   
\end{scope}
\begin{scope}[xshift=9.8cm,yshift=6cm]
\draw[fill=gray,rotate=20] (0,0) ellipse (0.3cm and 0.1cm) 
node[above=0.1cm] {$E_3$};   
\end{scope}

\node[left] at (2,8) {$\gamma_1([0,1])$};
\node at (3.5,8) {$V_1$};

\node[left] at (2,3) {$\gamma_2([0,1])$};
\node at (6,2.3) {$V_2$};

\node[above] at (9.6,7.4) {$\gamma_3([0,1])$};
\node at (8.8,7) {$V_3$};
\end{tikzpicture}

Let $E_i^{(n)}$ be the unique component 
of $\hat{\mathbb{C}} \setminus \Omega_n$ that contains $\zeta_i^{(n)}$.
We claim that $E_i^{(n)} \subset V_i$ for each $i=1, \ldots , k$.
Indeed, this follows from the fact that $E_i^{(n)}$ is a connected set
satisfying
\[
  E_i^{(n)} \cap \partial V_i 
 \subset E_i^{(n)} \cap (\gamma_1([0,1]) \cup \cdots \cup \gamma_k([0,1]) )
 \subset E_i^{(n)} \cap \Omega_n = \emptyset
\]
and $\zeta_i^{(n)} \in E_i^{(n)} \cap V_i$.

Since $V_1, \ldots , V_k$ are disjoint open sets,
it follows that
the components $E_1^{(n)}, \ldots , E_k^{(n)}$ 
are mutually distinct components of 
$\hat{\mathbb{C}} \setminus \Omega_n$.
Therefore, $C(\Omega_n ) \geq k$ for $n \geq N$.

If $C(\Omega )$ is finite, this implies
$\liminf_{n \rightarrow \infty} C(\Omega_n) \geq k = C(\Omega )$.
If $C(\Omega ) = \infty$, then for any $k \in \mathbb{N}$ 
we have  $\liminf_{n \rightarrow \infty} C(\Omega_n) \geq k$.
Hence,
$\liminf_{n \rightarrow \infty} C(\Omega_n ) = \infty = C(\Omega )$.
\end{proof}

Example \ref{example:strip-minus-sgments} shows that
the inequality in~(\ref{ineq:lower-semi-continuity})
cannot, in general, be replaced by an equality.

\begin{corollary}
\label{cor:Omega_t=hat-mathbb-C}
Let $\{ \Omega_t \}_{t \in I}$ 
be a continuous family of domains in $\hat{\mathbb{C}}$.
If $C(\Omega_{t_0}) = 0$ for some $t_0 \in I$, i.e.,
$\Omega_{t_0} = \hat{\mathbb{C}}$,
then $C(\Omega_t)=0$ for all $t \in I$.
\end{corollary}
\begin{proof}
Let $I_0 = \{ t \in I : C(\Omega_t) = 0 \}$.
Then $I_0$ is nonempty,  
and closed by Theorem~\ref{thm:lower-semi-continuity}.
Suppose $t_1 \in I_0$.
Then $\Omega_{t_1} = \hat{\mathbb{C}}$.
Since $\hat{\mathbb{C}}$ is compact and contained in $\Omega_{t_1}$,
the continuity of $\{ \Omega_t \}_{t\in I}$ at $t_1$
implies that there exists $\delta > 0$ such that
$\hat{\mathbb{C}} = \Omega_t$ 
for all $t \in I \cap (t_1 - \delta, t_1 + \delta )$. 
Thus $t_1$ is an interior point of $I_0$, and hence 
$I_0$ is open. 
Therefore, 
$I_0$ is a nonempty open and closed subset of $I$.
Since $I$ is connected by definition,
we conclude $I_0 = I$ as required.
\end{proof}

\section{Persistence of Complementary Components}
Let $\{ \Omega_t \}_{0 \leq t \leq \infty}$ be 
a family of domains, 
and set $E_t = \hat{\mathbb{C}} \setminus \Omega_t$ 
for $0 \leq t \leq \infty$, 
as in Example~\ref{eg:CantorSet}.
Then the corresponding family of universal covering maps 
$\{ f_t \}_{0 \leq t \leq \infty}$, normalized by 
$f_t(0)=0$ and $f_t'(0)>0$
forms a strictly increasing and continuous Loewner chain.
Fix $t_0 \in I$ and let $C_0$ be a connected component of $E_{t_0}$.
Since the family $\{ E_t \}_{t \in I}$ is nonincreasing in $t$, 
the sets $\{ C_0 \cap E_t \}_{t \in I}$ form a nonincreasing family 
as well.
Although the intersection $C_0 \cap E_t$ may shrink or split into 
multiple components as $t$ increases,
some portion of $C_0$ always survives in $E_t$.
In other words, $C_0$ never disappears entirely from the complement.

This persistence property holds more generally.
For instance, if $F$ is a closed subset of $E_{t_0}$ 
that is contained in a domain bounded by a Jordan curve 
in $\Omega_{t_0}$, 
then it is not difficult to show that 
$F \cap E_t \neq \emptyset$ for all $t \in I$.
We first prove a slightly more general result,
and then proceed to the general case.

\begin{proposition}
\label{prop:isolated_component_never_vansh}
Let $\{ \Omega_t \}_{t \in I}$ be a nondecreasing and continuous 
family of domains in $\hat{\mathbb{C}}$, and let
$E_t = \hat{\mathbb{C}} \setminus \Omega_t$ for $t \in I$. 
Suppose that $t_0 \in I$, and let $F$ be a nonempty 
closed subset of $E_{t_0}$. 
If there exists a domain $V$ in $\hat{\mathbb{C}}$ 
such that $F = E_{t_0} \cap V$, 
then for any $t \in I$,
\begin{equation}
\label{eq:F-nonvanishing}
    F \cap  E_t \neq \emptyset .
\end{equation}
\end{proposition}
\begin{proof}
It suffices to prove (\ref{eq:F-nonvanishing}) for 
$t \in I \cap (t_0, \infty ]$.
Suppose, for contradiction, that $F \cap E_t = \emptyset$
for some $t \in I \cap (t_0, \infty ]$.
Since $\{ E_t \}_{t \in I}$ is nonincreasing in $t$,
there exists $t_1 \in I \cap (t_0, \infty )$ such that
\begin{align}
\label{eq:before-t_1}
  &  F \cap E_t \neq \emptyset 
  \quad \text{for} \quad t \in I \cap [-\infty, t_1) 
\\
  &  F \cap E_t = \emptyset 
 \quad \text{for} \quad t \in I \cap (t_1, \infty ] .
\nonumber
\end{align}
In particular, for $t>t_1$ we have 
that $V \cap E_t = V \cap (E_{t_0} \cap E_t ) = F \cap E_t = \emptyset$.
Thus,
\begin{equation}
\label{eq:after-t_1}
  V \subset \Omega_t  
  \quad \text{for} \quad t \in I \cap (t_1, \infty ] .  
\end{equation}

We divide the argument into two cases. 
First, consider the case where 
$F \cap E_{t_1} = \emptyset$.
In this case, we have $F \subset \Omega_{t_1}$.
Since $F$ is compact,
it follows from condition~(a) that 
there exists $\delta > 0$ such that
$F \subset \Omega_t = \hat{\mathbb{C}} \setminus E_t$
for all $|t-t_1| < \delta$ with $t \in I$. 
This contradicts~(\ref{eq:before-t_1}).
Next, consider the case 
where $F \cap E_{t_1} \neq \emptyset$.
We will show that 
$V \cap \partial \Omega_{t_1} \neq \emptyset$.
Since $F \subset V$, we have 
 \begin{equation}
\label{eq:VintersectsE_t_1}
   V \setminus \Omega_{t_1} 
  = V \cap E_{t_1}
    = V \cap (E_{t_0} \cap E_{t_1})
   =  F \cap E_{t_1} \neq \emptyset .  
 \end{equation}
Here, we have $V \cap \Omega_{t_0} \neq \emptyset$.  
Indeed, suppose this is not the case; that is, assume 
$V \cap \Omega_{t_0} = \emptyset$.  
Then it would follow that $V \subset E_{t_0}$.
Thus we have $F = E_{t_0} \cap V = V$,
which implies that $F$ is both open and closed. 
Since $F \neq \emptyset$ and $F \neq \hat{\mathbb{C}}$, 
this contradicts the connectedness of $\hat{\mathbb{C}}$.
In particular, we obtain
\begin{equation}
\label{eq:VintersectsOmega_t_1}
  V \cap \Omega_{t_1} \supset V \cap \Omega_{t_0} \neq \emptyset . 
\end{equation}
Since $V$ is connected, it follows from
\eqref{eq:VintersectsE_t_1} and \eqref{eq:VintersectsOmega_t_1}
that
$V \cap \partial \Omega_{t_1} \neq \emptyset$.

Choose a point $c \in V \cap \partial \Omega_{t_1}$ arbitrarily.
Then, by condition~(b')
\begin{equation}
\label{eq:distance_between_c_and_E_t_tends_to_0}
   d^*(c, \partial \Omega_t ) \rightarrow 0
    \quad \text{as} \quad t \rightarrow t_1.
\end{equation}
On the other hand, one can choose $\varepsilon > 0$ such that
$\mathbb{D}^*(c, \varepsilon ) \subset V$.
By \eqref{eq:after-t_1}, we have 
for all $t \in I \cap (t_1,\infty ]$
that 
$d^*(c, \partial \Omega_t) \geq \varepsilon$, 
which contradicts~\eqref{eq:distance_between_c_and_E_t_tends_to_0}.

We have thus obtained a contradiction in both cases.  
It follows that \eqref{eq:F-nonvanishing} 
holds.
\end{proof}

Note that the set $F$ in the above proposition 
is a clopen subset of $E_{t_0}$,
and that the decomposition 
$E_{t_0} = F \cup (E_{t_0} \setminus F)$
constitutes a partition of $E_{t_0}$.
However, in the proof of the proposition,
it is  essential not only that this holds,
but also that the set $V$ is connected.

Let $X$ be a set, 
and let $\mathcal{F}$ be a nonempty family of subsets of $X$.
Then $\mathcal{F}$  is said to have
the finite intersection property if
for every finite collection 
$F_1, \ldots , F_n \in \mathcal{F}$, 
we have $F_1 \cap \cdots \cap F_n \neq \emptyset$.

The following fact will be used repeatedly in the proof of the 
next theorem. 
For convenience, we present it here as a lemma. 
For a proof, see for example Munkres~\cite{Munkres}.
\begin{lemma}
\label{lemma:finite_intersection_property}
Let $X$ be a compact topological space,
and let $\mathcal{F}$ be a nonempty family of closed subsets of $X$.
Suppose that $\mathcal{F}$ has the finite intersection
property.
Then $\mathcal{F}$ has nonempty intersection, that is,
$\bigcap_{F \in \mathcal{F}} F \neq \emptyset$. 
\end{lemma}

\begin{theorem}
\label{thm:nonvanishing-of-components}
Let $I \subset [-\infty, \infty ]$ be an interval, 
and let $\{ \Omega_t \}_{t\in I}$ 
be a nondecreasing and continuous family of domains 
in $\hat{\mathbb{C}}$. 
Set $E_t = \hat{\mathbb{C}} \setminus \Omega_t$ for each $t \in I$.
Suppose that
$t_0 \in I$, and that $C_0$ is a connected component of $E_{t_0}$.
Then 
\begin{equation}
\label{eq:nonvanishing}
   C_0 \cap \bigcap_{t \in I } E_t \neq \emptyset .
\end{equation}  
\end{theorem}
\begin{proof}
Note that the component $C_0$ is nonempty by definition. 
Hence, $E_{t_0} \neq \emptyset$,
and it follows from Corollary~\ref{cor:Omega_t=hat-mathbb-C}
that 
\begin{equation}
\label{eq:E_t_never_vanish}
  E_t \neq \emptyset \quad \text{for all } t \in I .
\end{equation}

To begin with, consider 
the special case where $E_{t_0}$ is connected.
In this case, 
since $C_0$ is the only component of $E_{t_0}$,  
we clearly have $C_0 =E_{t_0}$.
This implies
\[
  C_0 \cap E_t 
 =
 \begin{cases}
  C_0 , \quad & t \in I \cap [-\infty, t_0] ,\\ 
  E_t , \quad & t \in I \cap [t_0, \infty ] ,
 \end{cases}
\]
and therefore $C_0 \cap E_t \neq \emptyset$ 
for all $t \in I$.
Moreover, since $\{ E_t \}_{t \in I}$ is nonincreasing,
for any finite collection
$t_1, \ldots, t_n \in I$,
we have by \eqref{eq:E_t_never_vanish}
\begin{equation}
\label{eq:decreasing_sequence_of_nonempty_closed_subsets}
C_0 \cap (E_{t_1} \cap \cdots \cap E_{t_n} ) 
 = 
 C_0 \cap E_{\max \{ t_1, \ldots ,t_n \} } 
 \neq \emptyset . 
\end{equation}
In other words, the family $\{ C_0 \cap E_t \}_{t\in I}$
of closed subsets of the compact set $C_0$  
has the finite intersection property. 
By Lemma~\ref{lemma:finite_intersection_property} 
we obtain \eqref{eq:nonvanishing} in this case.

Next, we consider the case 
where $E_{t_0}$ has a component other than $C_0$.
Let $E_{t_0} = C_0 \cup \bigcup_{\lambda \in \Lambda} C_\lambda$
be the decomposition of $E_{t_0}$ into its connected components.
For each $\lambda \in \Lambda$, by Lemma \ref{lemma:separation}, 
there exists a simple closed curve
$\gamma_\lambda:[0,1] \rightarrow \Omega_{t_0}$ 
that separates $C_0$ and $C_\lambda$.
Let $V_\lambda$ be the component of 
$\hat{\mathbb{C}} \setminus \gamma_\lambda ([0,1])$
that contains $C_0$, and define $F_\lambda = V_\lambda \cap E_{t_0}$.
Since  
$\partial V_\lambda \cap E_{t_0} 
= \gamma_\lambda ([0,1]) \cap (\hat{\mathbb{C}} \setminus \Omega_{t_0}) =
 \emptyset$, 
we have $F_\lambda = \overline{V_\lambda} \cap E_{t_0}$, 
so $F_\lambda$ is closed.
Moreover, since $C_0 \subset F_\lambda$, the set $F_\lambda$ is nonempty.
Therefore 
by Proposition~\ref{prop:isolated_component_never_vansh},
we conclude that
$F_\lambda \cap E_t \neq \emptyset$ for all $t \in I$.

\noindent\textbf{Claim}. 
For any fixed $t \in I$, the family 
$\{ F_\lambda \cap E_t \}_{\lambda \in \Lambda}$ 
of closed subsets of the compact space $E_t$
has the finite intersection property.

We now prove the claim.
We may assume that $t \geq t_0$.
Let $V$ be the component of the open set 
$V_{\lambda_1} \cap \cdots \cap V_{\lambda_n}$ that contains $C_0$,
and define $F = V \cap E_{t_0}$.
Since 
\[
  \partial V \cap E_{t_0} 
 \subset  
 (\gamma_{\lambda_1}([0,1]) \cup \cdots \cup \gamma_{\lambda_n}([0,1]))
  \cap E_{t_0}
 \subset \Omega_{t_0} \cap E_{t_0}  = \emptyset,
\]
it follows that $F = \overline{V} \cap E_{t_0}$ is closed.
Moreover, since $C_0 \subset F$, we have $F \neq \emptyset$.
Therefore, by Proposition \ref{prop:isolated_component_never_vansh}
it follows that $F \cap E_t \neq \emptyset $ for all $t \in I$.
Consequently,
\begin{align*}
 \emptyset 
 \neq 
 F \cap E_t 
 = & (V \cap E_{t_0} )\cap E_t  
\\
 \subset & 
 \{ (V_{\lambda_1} \cap \cdots \cap V_{\lambda_n}) \cap E_{t_0} \}
 \cap E_t 
 = (F_{\lambda_1} \cap \cdots \cap F_{\lambda_n} ) \cap E_{t}, 
\end{align*}
as required.

From the claim and Lemma~\ref{lemma:finite_intersection_property},  
it follows that
\begin{equation}
\label{eq:intersection_property_for_F_lambda}
 \bigcap_{\lambda \in \Lambda} (F_\lambda \cap E_t )
 =
 \left( \bigcap_{\lambda \in \Lambda} F_\lambda \right) \cap E_t 
 \neq \emptyset 
\end{equation}
for all $t \in I$.

Now we show 
\begin{equation}
\label{eq:the_intersection_of_all_F_lambda_is_C_0}
  \bigcap_{\lambda \in \Lambda} F_\lambda = C_0. 
\end{equation}
Since $C_0 \subset F_\lambda $ for any $\lambda \in \Lambda$, 
we have $C_0 \subset \bigcap_{\lambda \in \Lambda} F_\lambda$.
On the other hand, since  
$\bigcap_{\lambda \in \Lambda} F_\lambda 
\subset E_{t_0} = C_0 \cup \bigcup_{\mu \in \Lambda } C_\mu$
and $F_\mu \cap C_\mu = \emptyset$ for all $\mu \in \Lambda$,
we obtain
\begin{align*}
 C_0 
 \subset
 \bigcap_{\lambda \in \Lambda} F_\lambda
 = & \,   
 \left( \bigcap_{\lambda \in \Lambda} F_\lambda \right) 
 \cap  \left( C_0 \cup \bigcup_{\mu \in \Lambda } C_\mu  \right)
\\
 = & \,   
 C_0 \cup 
 \left\{ \left( \bigcap_{\lambda \in \Lambda} F_\lambda \right)
 \cap \left( \bigcup_{\mu \in \Lambda } C_\mu  \right) \right\}
\\
 = & \,   
 C_0 \cup 
 \left\{ \bigcup_{\mu \in \Lambda } 
 \left( \bigcap_{\lambda \in \Lambda} F_\lambda \right) \cap C_\mu  \right\}
\\
 \subset & \,
 C_0 \cup \bigcup_{\mu \in \Lambda } 
 \left( F_\mu  \cap C_\mu  \right)
 = C_0 .
\end{align*}
Therefore, \eqref{eq:the_intersection_of_all_F_lambda_is_C_0} holds.
Combining this with \eqref{eq:intersection_property_for_F_lambda},
we conclude
\[
 C_0 \cap E_t \neq \emptyset \quad \text{for all } t \in I. 
\]

From this, the finite intersection property of 
the family $\{ C_0 \cap E_t \}_{t \in I}$ easily follows.
Indeed, we have
$(C_0 \cap E_{t_1}) \cap \cdots \cap (C_0 \cap E_{t_n}) 
= C_0 \cap E_{\max \{ t_1, \ldots , t_n \}} \neq \emptyset$
for every finite collection $t_1, \ldots , t_n \in I$.

Finally, by Lemma~\ref{lemma:finite_intersection_property} again,
we obtain $\bigcap_{t \in I} (C_0 \cap E_t ) 
= C_0 \cap \bigcap_{t \in I} E_t \neq \emptyset$,
as required.
\end{proof}

\begin{corollary}
\label{corollary:connectivity-is-nondecreasing}
Let $\{ \Omega_t \}_{t\in I}$ 
be a nondecreasing and continuous family of domains 
in $\hat{\mathbb{C}}$, 
and let $E_t= \hat{\mathbb{C}} \setminus \Omega_t$
for $t \in I$.
Then the function $C(\Omega_t)$ 
is nondecreasing and left-continuous on $I$.
Moreover, for any $t_1, t_2 \in I$ with $t_1 < t_2$,
the cardinality of the set of all components of $E_{t_1}$
does not exceed that of $E_{t_2}$. 
\end{corollary}
\begin{proof}
If $C( \Omega_t ) = 0$ for some $t \in I$, 
then by Corollary \ref{cor:Omega_t=hat-mathbb-C},
we have $C(\Omega_t) = 0$ for all $t \in I$. 
In this case, the function $C(\Omega_t)$ is clearly 
nondecreasing and left continuous on $I$.

Assume now that $C(\Omega_t ) \geq 1$, that is,
$E_t \neq \emptyset$ for all $t \in I$.
Let $t_1, t_2 \in I$ with $t_1 < t_2$.
For each component $C$ of $E_{t_1}$,
since $C \cap \bigcap_{t \in I} E_t  \neq \empty$,
we can choose a point $w \in C \cap E_{t_2}$.
(Here, we use the axiom of choice.)
Let $C'$ be the component of $E_{t_2}$ which contains $w$.
We denote the correspondence by $C' = m(C)$.
It is easy to see that
$m$ defines a map of the set of all components of $E_{t_1}$
into the set of all components of $E_{t_2}$,
and that this map is injective.
Therefore, the cardinality of the set of all components of $E_{t_1}$
does not exceed that of $E_{t_2}$.
In particular, we have $C(\Omega_{t_1}) \leq C(\Omega_{t_2})$.
Hence, the function $C(\Omega_t)$, $t \in I$, is nondecreasing.

Since $C(\Omega_t)$ is nondecreasing, we have
$\limsup_{t \nearrow t_0} C(\Omega_t) \leq C(\Omega_{t_0})$ 
for any $t_0 \in I$.
On the other hand, by Theorem \ref{thm:lower-semi-continuity},
$\liminf_{t \nearrow t_0} C(\Omega_t) \geq C(\Omega_{t_0})$.
Therefore, the limit exists, and we conclude 
\[
  \lim_{t \nearrow t_0} C(\Omega_t) = C(\Omega_{t_0}).
\]
\end{proof}

\section{Maximal Domains}
\begin{definition}
Let $\Omega$ be a domain in $\hat{\mathbb{C}}$. 
We say that 
$\Omega$ is \emph{maximal in the sense of kernel} 
if there exists no nondecreasing and continuous 
family of domains $\{ \Omega_t \}_{0 \leq t \leq \varepsilon}$ 
in  $\hat{\mathbb{C}}$ with $\varepsilon > 0$ 
such that $\Omega = \Omega_0 \subsetneq \Omega_\varepsilon$. 
\end{definition}

Assume that $\Omega$ is a hyperbolic domain in $\mathbb{C}$  
with $0 \in \Omega$.  
Let $f: \mathbb{D} \rightarrow \Omega$ be the unique universal  
covering map normalized by $f(0) = 0$ and $f'(0) > 0$.  
If $\Omega$ is maximal in the sense of kernel,  
then $f$ is maximal in the sense of continuous Loewner chains  
of covering maps.  
This can be easily verified by proving the contrapositive.  
The converse also holds and can be shown in a similar way  
by proving the contrapositive,  
but it requires some knowledge concerning continuous  
and nondecreasing families of domains that connect  
a hyperbolic domain and a parabolic domain.  
See Theorem~\ref{theorem:a_classification_L-chain_of_covering_map}  
for details.

\begin{theorem}
Let $\Omega$ be a domain in $\hat{\mathbb{C}}$ 
with nonempty complementary set 
$E = \hat{\mathbb{C}} \setminus \Omega$.
If $E$ is totally disconnected,
then $\Omega$ is maximal in the sense of kernel. 
\end{theorem}
\begin{proof}
Suppose, on the contrary, that there exists a nondecreasing 
and continuous family 
$\{ \Omega_t \}_{0 \leq t \leq \varepsilon}$ 
of domains in $\hat{\mathbb{C}}$
with $\Omega_0 = \Omega \subsetneq \Omega_\varepsilon$. 
Since  $E$ is totally disconnected, for any $w \in E$,  
the component of $E$ containing $w$ is $\{ w \}$.
Then, by Theorem~\ref{thm:nonvanishing-of-components}, 
we have 
$w \in E_\varepsilon := \hat{\mathbb{C}} \setminus \Omega_\varepsilon$,
and hence $E \subset E_\varepsilon$.
This contradicts $\Omega = \Omega_0 \subsetneq \Omega_\varepsilon$.
\end{proof}

It follows from the above corollary that
the $\Omega_{\infty}$ in Example~\ref{eg:CantorSet} is maximal.
We now present an example of a maximal domain 
whose complement is not totally disconnected.
\begin{example}
For $n \in \mathbb{N}$ and $k=0, \ldots , n$,
let
\[
   a_{n,k} = 1+  \frac{k}{n} + \frac{1}{n} i .
\]
Define
\[
  E 
 =  [1,2] \cup \{ \infty \} 
 \cup \bigcup_{n=1}^\infty \bigcup_{k=0}^n \{ a_{n,k} \}. 
\]
Then the line segment $[1,2]$ is a component of $E$,
and hence $E$ is not totally disconnected. 
The complementary set  
$\Omega := \hat{\mathbb{C}} \setminus E$ is a maximal domain.
\end{example}
\begin{proof}
Since $[1,2]$ is a connected subset of $E$,
there exists a unique connected component $A$ of $E$ 
containing $[1,2]$.
We claim that  $A =[1,2]$; that is, 
$[1,2]$ is itself a component of $E$.

Suppose, on the contrary,  
that $a_{n,k} \in A$ for some $n \in \mathbb{N}$ and 
$k \in \{ 0, \ldots, n \}$.
Then, we have the following decomposition
\begin{align*}
  a_{n,k} \in  & \, 
 \left\{ w \in \mathbb{C} : 
 \Imaginary w 
 > 
 \frac{1}{2} \left( \frac{1}{n} + \frac{1}{n+1} \right) \right\} \cap A,
\\ 
  [1,2] \subset  & \, 
 \left\{ w \in \mathbb{C} : 
 \Imaginary w 
 < 
 \frac{1}{2} \left( \frac{1}{n} + \frac{1}{n+1} \right) 
 \right\} \cap A .
\end{align*}
This gives a partition of $A$ into
two relatively open, disjoint subsets,
which contradicts the connectedness of $A$.
Thus, $a_{n,k} \not\in A$ for all $n \in \mathbb{N}$ and 
$k \in \{ 0, \ldots, n \}$.
Similarly, we have $\infty \not\in A$.
Therefore, $A=[1,2]$, and hence 
$[1,2]$ is a component of $E$.
In particular, $E$ is not totally disconnected.

We now show that $\Omega$  is maximal.
Suppose, on the contrary,  that
$\{ \Omega_t \}_{0 \leq t \leq \varepsilon}$ 
is a nondecreasing and continuous
family of domains in $\hat{\mathbb{C}}$ for some 
$\varepsilon > 0$, 
with $\Omega_0 = \Omega \subsetneq \Omega_\varepsilon$.
Let $E_t = \hat{\mathbb{C}} \setminus \Omega_t$ for  
$0 \leq t \leq \varepsilon$.
Since for each $n \in \mathbb{N}$ and $k=0, \ldots , n$,
the singleton set $\{ a_{n,k} \}$ is a component of $E$, 
it follows from Theorem~\ref{thm:nonvanishing-of-components}
that $ a_{n,k} \in E_\varepsilon$.
Moreover, every point in $[1,2]$ is an accumulation point
of $\{ a_{n,k} : n \in \mathbb{N} \text{ and } k=0, \ldots , n \}$.
Since $E_\varepsilon$ is closed, 
we conclude $[1,2] \subset E_\varepsilon$. 
Also, $\infty \in E_\varepsilon$.
Therefore, we have $E \subset E_\varepsilon$, which contradicts
the assumption $\Omega \subsetneq \Omega_\varepsilon$.
Hence, $\Omega$ is maximal.
\end{proof}

\section{Continuous Connection between Domains}
Next, we consider conditions under which 
two given domains can be connected by 
a nondecreasing and continuous one-parameter family of domains.

\begin{definition}
Let $D_0$ and $D_1$ be domains in $\hat{\mathbb{C}}$
with $D_0 \subset D_1 \subsetneq \hat{\mathbb{C}}$. 
We say that $D_0$ is \emph{continuously connected} to $D_1$
if there exists a nondecreasing and continuous family 
$\{ \Omega_t \}_{0\leq t \leq 1}$ of domains in $\hat{\mathbb{C}}$
such that $\Omega_0 = D_0$ and $\Omega_1 =D_1$. 
\end{definition}

When $0 \in D_0$
and $D_1$ is hyperbolic,
$D_0$ is continuously 
connected to $D_1$ if and only if 
there exists a continuous Loewner chain 
$\{ f_t  \}_{0 \leq t \leq 1}$ of covering maps 
such that $f_0, f_1 \in \mathcal{H}_0 (\mathbb{D})$ 
are the unique universal covering maps 
of $\mathbb{D}$ onto $D_0$ and $D_1$, respectively.

In the case where $C(D_1)$ is finite, i.e.,
when $D_1$ is finitely connected, 
we provide a necessary and sufficient condition 
for $D_0$ to be continuously connected to $D_1$.
To this end, we first present an elementary topological lemma 
and a weaker result for the case where 
both $D_0$ and $D_1$ are simply connected.

\begin{lemma}
\label{lemma:component_of_complementary_set_of_domain}
Let $\Omega $ be a domain in $\hat{\mathbb{C}}$, and 
let $C$ be a connected component of the complement 
$E = \hat{\mathbb{C}} \setminus \Omega$.
Then the set $\hat{\mathbb{C}} \setminus C$ is connected.
Moreover, $\hat{\mathbb{C}} \setminus C$
is a simply connected domain. 
\end{lemma}
\begin{proof}
Since $\Omega \subset \hat{\mathbb{C}} \setminus C$, 
there exists a unique component $\Omega'$ 
of $\hat{\mathbb{C}} \setminus C$ that contains $\Omega$.
It suffices to show that $\hat{\mathbb{C}} \setminus C$
has no other components besides $\Omega'$.

Assume that $D$ is another component of $\hat{\mathbb{C}} \setminus C$.
We will show that the union $A: = D \cup C$ is connected.
Since 
\[
  \partial D 
 \subset 
 \partial (\hat{\mathbb{C}} \setminus C) 
 = \partial C \subset C,
\]
it follows that $A = D \cup C = \overline{D} \cup C$.
Both $\overline{D}$ and $C$ are connected 
and $\overline{D} \cap C \supset \partial D \neq \emptyset$.
Thus, $A$ is connected.
Moreover, 
since $A \cap \Omega \subset (D \cap \Omega') \cup (C \cap \Omega) = \emptyset $,
we obtain $ C \subsetneq A \subset E$,
which contradicts the maximality of the component $C$ of $E$.
Therefore, $\Omega'$
is the only component of 
$\hat{\mathbb{C}} \setminus C$,
that is, $\Omega' = \hat{\mathbb{C}} \setminus C$.
Moreover, 
the domain $\Omega' = \hat{\mathbb{C}} \setminus C$ 
is simply connected, 
since its complement $C = \hat{\mathbb{C}} \setminus \Omega'$ 
is connected.
\end{proof}

We say that a set $E \subset \hat{\mathbb{C}}$ is 
a \emph{continuum}
if $E$ is a nonempty, compact, and connected subset of $\hat{\mathbb{C}}$.
A continuum is said to be \emph{nondegenerate}
if it contains at least two points, 
and \emph{degenerate} if it consists of a single point.

If $\Omega$ is a simply connected domain in $\hat{\mathbb{C}}$,
then the complement $\hat{\mathbb{C}} \setminus \Omega$
is either empty, a singleton or a nondegenerate continuum.

\begin{proposition}
\label{prop:continuous-connection-between-simply-connected-domains}
Let $D_0$ and $D_1$ be simply connected domains in $\hat{\mathbb{C}}$
with $D_0 \subset D_1 \subsetneq \hat{\mathbb{C}}$. 
Then $D_0$ is continuously 
connected to $D_1$. 
\end{proposition}
\begin{proof}
We may assume that $D_0 \subsetneq D_1$.
Moreover, after applying a linear fractional transformation if necessary,
we may also assume that $0 \in D_0 \subsetneq D_1 \subset \mathbb{C}$. 

First, suppose that
$\hat{\mathbb{C}} \setminus D_1$ is a nondegenerate continuum
containing $\infty$.
Then $\hat{\mathbb{C}} \setminus D_0$ is also 
a nondegenerate continuum containing $\infty$.
By the Riemann mapping theorem, for $j=0,1$,
there exists a unique conformal mapping 
$g_j \in \mathcal{H}_0(\mathbb{D})$ of $\mathbb{D}$ onto $D_j$.
It then follows from 
Theorem~\ref{theorem:schlicht-subordination-and-continuous-chain}
that there exists a continuous Loewner chain 
$\{ f_t \}_{0 \leq t \leq 1}$
of univalent functions such that $f_j = g_j$ for $j = 0, 1$.
Let $D_t = f_t (\mathbb{D})$ for $0 < t < 1$.
Then $\{ D_t \}_{0 \leq t \leq 1}$ 
is a nondecreasing and continuous family of domains 
in $\mathbb{C}$ that connects $D_0$ to $D_1$.

Next, suppose that 
$\hat{\mathbb{C}} \setminus D_1 = \{  \infty \}$, i.e.,
$D_1 = \mathbb{C}$.
Since  $D_0 \subsetneq D_1 = \mathbb{C}$
and $D_0$ is simply connected,
its complement $\hat{\mathbb{C}} \setminus D_0$ 
is a nondegenerate continuum containing $\infty$.
Hence, by the Riemann mapping theorem, 
there exists a unique conformal map 
$g_0 \in \mathcal{H}_0(\mathbb{D})$
of $\mathbb{D}$ onto $D_0$.

By applying 
Theorem~\ref{theorem:univalent_function_can_be_emmbedded_in_L-chain} 
or \cite[Theorem 6.4]{Pommerenke1}, 
we see that there exists a continuous Loewner chain 
$\{ f_t \}_{0 \leq t < \infty}$ of univalent functions 
such that $f_0 = g_0$ and $\lim_{t \to \infty} f_t'(0) = \infty$.
Then, Koebe's theorem 
implies that 
$\mathbb{D}(0, \frac{1}{4} f_t'(0)) \subset f_t(\mathbb{D})$, 
and hence $f_t(\mathbb{D}) \to \mathbb{C}$ as $t \to \infty$.
Thus, by setting 
$D_t = f_{\frac{t}{1 - t}}(\mathbb{D})$ for $0 < t < 1$, 
we obtain a nondecreasing and continuous family 
of domains $\{ D_t \}_{0 \leq t \leq 1}$
connecting $D_0$ and $D_1 = \mathbb{C}$.
\end{proof}

\begin{theorem}
Let $D_0$ and $D_1$ be domains in $\hat{\mathbb{C}}$
such that $D_0 \subset D_1 \subsetneq \hat{\mathbb{C}}$,
and suppose that $D_1$ is finitely connected.
Then $D_0$ is continuously connected to $D_1$ 
if and only if 
for every component $C$ of $\hat{\mathbb{C}} \setminus D_0$,
there exists at least a component $C'$ 
of $\hat{\mathbb{C}} \setminus D_1$ 
such that $C' \subset C$.
\end{theorem}
\begin{proof}
The necessity follows immediately from 
Theorem~\ref{thm:nonvanishing-of-components}.

After applying a linear fractional transformation if necessary,
we may assume that $0 \in D_0 \subsetneq D_1 \subset \mathbb{C}$. 
Let $E_i = \hat{\mathbb{C}}\setminus D_i$, $i=0,1$.
Suppose that
for every component $C$ of $\hat{\mathbb{C}} \setminus D_0$,
there exists a component $C'$ 
of $\hat{\mathbb{C}} \setminus D_1$ 
such that $C' \subset C$.
Note that this assumption implies 
$C(D_0) \leq C(D_1)$.

Since $E_1 \subset E_0$, 
for each component $C'$ of $E_1$, there exists 
a unique component $C$ of $E_0$ such that $C' \subset C$.
Combining this with the assumption of the theorem,
we can decompose $E_0$ and $E_1$ into their connected components 
as follows:
\begin{align}
 E_0 = \bigcup_{j = 1}^n C_j 
 \quad \text{and} \quad 
  E_1 
  = \bigcup_{j = 1}^n \bigcup_{k = 1}^{p_j} C_{j,k}' ,
\end{align}
where  
$\bigcup_{k = 1}^{p_j} C_{j,k}' \subset C_j$ for $j = 1, \ldots, n$,
with $n=C(D_0)$, $p_j \in \mathbb{N}$,
and $\sum_{j=1}^n p_j = C(D_1)$.

For each $j=1, \ldots, n$ and $k=1, \ldots, p_j$ let
\begin{equation}
   \Omega_j = \hat{\mathbb{C}} \setminus C_j
    \quad \text{and} \quad 
    \Omega_{j,k}' 
    = \hat{\mathbb{C}} \setminus C_{j,k}' .
\end{equation}
By Lemma~\ref{lemma:component_of_complementary_set_of_domain},
both $\Omega_j$ and $\Omega_{j,k}'$
are simply connected domains 
and we have $\Omega_j \subset \Omega_{j,k}'$. 
Then for each $j=1, \ldots, n$ and $k=1, \ldots, p_j$, 
Proposition~\ref{prop:%
continuous-connection-between-simply-connected-domains}
guarantees the existence of a nondecreasing and continuous
family $\{ \Omega_t^{j,k} \}_{0 \leq t \leq 1 }$ of
domains such that $\Omega_0^{j,k} = \Omega_j$
and $\Omega_1^{j,k} = \Omega_{j,k}'$.

For $t \in [0,1]$, let 
$\Omega_t$ denote 
the connected component of
\[
 \bigcap_{j = 1}^n \bigcap_{k = 1}^{p_j} \Omega_t^{j,k}
\]
that contains $0$.
Clearly, ${ \Omega_t }_{0 \leq t \leq 1}$ is nondecreasing.
Moreover, by Theorem \ref{thm:continuity-preseved-under-intersection},
the family $\{ \Omega_t \}_{0 \leq t \leq 1}$ is continuous.

We observe that
\begin{align*} 
 & \bigcap_{j = 1}^n \bigcap_{k = 1}^{p_j} \Omega_0^{j,k}
 = \bigcap_{j = 1}^n \Omega_j
 = \hat{\mathbb{C}} \setminus  
 \left( \bigcup_{j = 1}^n C_j \right)
 = \hat{\mathbb{C}} \setminus E_0 = D_0 ,
\\ 
 & 
 \bigcap_{j = 1}^n \bigcap_{k = 1}^{p_j} \Omega_1^{j,k}
 =  \bigcap_{j = 1}^n \bigcap_{k = 1}^{p_j} \Omega_{j,k}'
 = \hat{\mathbb{C}} \setminus  
 \left( \bigcup_{j = 1}^n \bigcup_{k = 1}^{p_j} C_{j,k}' \right)
 = \hat{\mathbb{C}} \setminus E_1  = D_1 .
\end{align*}
Therefore, $\Omega_0 = D_0$ and $\Omega_1 = D_1$, as desired.
\end{proof}

\section{A Classification of Loewner Chains of Universal Coverings}
Now we present a classification theorem 
for continuous Loewner chains of universal covering maps
defined on right-open intervals.

\begin{theorem}
\label{theorem:a_classification_L-chain_of_covering_map}
Let $I$ be a right-open interval and $\beta = \sup I \not\in I$,
and let $\{ f_t \}_{t \in I}$ be a continuous Loewner chain  of
universal covering maps of $\mathbb{D}$.
Define 
$\Omega_t = f_t (\mathbb{D})$
and $E_t = \hat{\mathbb{C}} \setminus \Omega_t$ 
for $t \in I$. 
Let $\Omega_\beta$ be the Loewner range of 
$\{ f_t \}_{t \in I}$, that is,
$\Omega_\beta = \bigcup_{t \in I} \Omega_t $,
and 
$E_\beta = \hat{\mathbb{C}} \setminus \Omega_\beta 
= \bigcap_{t \in I} E_t$.
Then $E_\beta$ is a closed set satisfying 
$\infty \in E_\beta $ and $0 \not\in E_\beta$.
\begin{itemize}
 \item[{\rm (i)}] 
If $E_\beta = \{ \infty \}$, 
then each $ f_t$ is univalent, and 
$\lim_{t \nearrow \beta} f_t '(0) = \infty$.
 \item[{\rm (ii)}] 
If $E_\beta = \{ \infty , w_0 \}$ for some 
$w_0 \in \mathbb{C} \setminus \{ 0 \}$, 
then
there exists  $t_0 \in I$ 
such that for $t_0 < t < \beta$,
$E_t$ consists of exactly two components:
$E_t^\infty$, containing $\infty$,
and $E_t^0$ containing $w_0$,
with 
\begin{equation}
\label{eq:components_containning_w_0_and_infty}
  \max_{w \in E_t^0} |w-w_0| \rightarrow 0 
 \quad \text{and} \quad
  \min_{w \in E_t^\infty} |w| \rightarrow \infty .
\end{equation}
Furthermore, $\lim_{t \nearrow \beta} f_t'(0) = \infty$.
\item[{\rm (iii)}] 
If $E_\beta$ contains more than two points (i.e., 
$\Omega_\beta$ is hyperbolic),
then $\{ f_t \}_{t \in I}$ 
converges locally uniformly on $\mathbb{D}$ as $t \nearrow \beta$  
to the unique universal covering map 
$f_\beta \in \mathcal{H}_0 (\mathbb{D})$ of $\mathbb{D}$ 
onto $\Omega_\beta$.
Moreover, the extended family $\{ f_t \}_{t \in I \cup \{ \beta \}}$ 
is a continuous Loewner chain  
of universal covering maps.
\end{itemize}
\end{theorem}
\begin{proof}
We first observe that by
Theorem~\ref{theorem:monotone-covergence-thm-for-domains}, 
the family $\{\Omega_t \}_{t \in I}$ converges to 
$\Omega_\beta$ as $t \nearrow \beta$ in the sense of kernel.

(i) Assume first that $E_\beta = \{ \infty \}$.
Then $\infty \in E_t$ for all $t \in I$,
which implies $C(\Omega_t) \geq 1$. 
By Theorem~\ref{thm:nonvanishing-of-components},
the function $C(\Omega_t)$ defined on $t \in I \cup \{ \beta \}$ 
is nondecreasing, 
and satisfies $C(\Omega_t) \leq  C(\Omega_\beta ) = 1$.
Therefore, we must have $C(\Omega_t) = 1$ for all $t \in I$,
i.e., each $\Omega_t$ is a simply connected domain in $\mathbb{C}$.

From this it follows that each $f_t $ is univalent on $\mathbb{D}$.
Indeed, for any $w \in \Omega_t$, choose
a path $\alpha$ in $\Omega_t$ from $0$ to $w$.
Let $\tilde{\alpha}$ be the lift of $\alpha$ starting at $0$.
Since $\Omega_t$ is simply connected, 
the endpoint $z$ of $\tilde{\alpha}$ is independent of
the choice of $\alpha$.
We define a map $g_t : \Omega_t \to \mathbb{D}$ by $g_t (w) = z$.
Then, clearly, $f_t \circ g_t = \id_{\Omega_t}$,
and hence $g_t$ is injective.

Moreover, $g_t: \Omega_t \to \mathbb{D}$ is surjective.
Indeed, for $z \in \mathbb{D}$, choose a path $\tilde{\alpha}$
from $0$ to $z$, and set $\alpha = f_t \circ \tilde{\alpha}$
and $w = \alpha (1)$.
Then by the path lifting Lemma,
the lift of $\alpha$ starting $0$
coincides with $\tilde{\alpha}$.
This implies that $g_t(w) = z$.

Since $g_t : \Omega_t \to \mathbb{D}$ is bijection,
the inverse $f_t = g_t^{-1}: \mathbb{D} \to \Omega_t$ is univalent.
In particular,
it follows from 
Proposition~\ref{proposition:Omega_beta_coincides_withC}
that $\lim_{t \nearrow \beta} f_t'(0) = \infty$.

(ii) 
Assume that there exists $w_0 \in \mathbb{C} \setminus \{ 0 \}$ 
such that $E_\beta = \{w_0, \infty \}$, $t \in E$.
Note that in this case, $C(\Omega_\beta ) = 2$.
As before, since the function $C(\Omega_t)$ is nondecreasing,
left-continuous, and integer-valued,
there exists $t_0 \in I$ such that 
$C(\Omega_t) \equiv 2$ on $(t_0, \beta )$.
Therefore, for $t_0 < t < \beta$, 
$E_t $ consists of exactly two components: 
$E_t^0$, containing $w_0$, and 
$E_t^\infty$ containing $\infty$.

Since 
$\Omega_t \rightarrow \Omega_\beta = \mathbb{C} \setminus \{ w_0 \}$
as $t \nearrow \beta$, by condition~(a'),
for any $r$ and $R$ satisfying 
$0 < r+|w_0| < R < \infty$, 
there exists $t_1$ with $t_0 \leq t_1 < \beta$
such that 
$\overline{\mathbb{D}}(0,R) \setminus \mathbb{D}(w_0,r) \subset \Omega_t$
for all $t \in (t_1, \beta )$.
This implies
$\max_{w \in E_t^0} |w-w_0| \leq r$
and $\min_{w \in E_t^\infty} |w| \geq R$
for all $t \in (t_1, \beta )$.
Therefore,
\[
  \max_{w \in E_t^0} |w-w_0| \rightarrow 0 
 \quad \text{and} \quad
  \min_{w \in E_t^\infty} |w| \rightarrow \infty
\]
as $t \nearrow \beta$, and thus
equation \eqref{eq:components_containning_w_0_and_infty} holds.

We now show that 
$\lim_{t \nearrow \beta} f_t'(0) = \infty$.
Without loss of generality, we may assume
$w_0 = -c $ for some $c >0$.
For $R > 1$, define the annulus
\[
  A_R = \left\{ w \in \mathbb{C} : \frac{c}{R} < |w+c| < Rc \right\} ,
\]
and let $g_R: \mathbb{D} \rightarrow A_R$ 
be the universal covering map normalized by $g_R(0)=0$ and $g_R'(0) > 0$.
Then 
\[
   g_R(z) 
 = \exp 
 \left[ \left( \frac{2}{i \pi} \log R \right) 
 \log \frac{1+iz}{1-iz} + \log c 
\right]
 - c .
\]
Since 
$\overline{A_R} \subset \Omega_\beta$ 
and $\Omega_t \rightarrow \Omega_\beta$ as $t \nearrow \beta$,
by condition (a'), there exists $t_2 \in I$ such that 
$\overline{A_R} \subset \Omega_t$ for all $t_2 \leq t< \beta$.
Therefore, $g_R$ is subordinate to $f_t$, and hence 
\[
   f_t'(0) \geq g_R'(0) = \frac{4 c \log R}{\pi} , 
 \quad t_2 \leq t < \beta .
\]
Since $R > 1$ is arbitrary, we conclude 
$\lim_{t \nearrow \beta} f_t'(0) = \infty$.

(iii)
Finally, assume that $E_\beta$ contains 
more than two points; that is,
$\Omega_\beta$ is hyperbolic.
Since $\Omega_t \to \Omega_\beta$,
Theorem~\ref{thm:3rd-part} implies that
$f_t \rightarrow f_\beta$ 
locally uniformly on $\mathbb{D}$ as $t \nearrow \beta$, 
where $f_\beta \in \mathcal{H}_0 ( \mathbb{D})$ 
is the unique universal covering map of $\mathbb{D}$
onto $\Omega_\beta$.
In particular, we have 
$\lim_{t \nearrow \beta} f_t'(0) = f_\beta'(0) < \infty$.

Furthermore, by Theorem \ref{thm:extendability-of-transiton-family},
the transition family $\{ \Omega_{s,t} \}$
admits a locally uniform limit 
$\omega_{s,\beta} =
\lim_{t \nearrow \beta} \omega_{s,t} \in \mathfrak{B}$.
Then for each $s \in I$, we have 
$f_s = \lim_{} f_t \circ \omega_{s,t} = f_\beta \circ \omega_{s,\beta}$,
and hence the extended family $\{ f_t \}_{t \in I \cup \{ \beta \}}$ 
is a Loewner chain of universal covering maps.
\end{proof}

Let $I_0$ be a bounded interval and $r \in (0,1)$.
In Chapter~\ref{chapter:SchlichtSubordination},
we observed that the class of all normalized Loewner chains 
of univalent functions 
is uniformly bounded on $I_0 \times \overline{\mathbb{D}}(0,r)$.
We now present an example showing that 
the corresponding class of all normalized Loewner chains 
of covering maps does not possess this boundedness property.

\begin{example}
\label{example:universal-covering-onto-outside-of-a-disc}
Let $a > 0$ and define
\[
   f_a(z,t) = a \left( e^{\frac{t}{2a}\frac{1+z}{1-z} -\frac{t}{2a} } - 1 \right) ,
 \quad (z,t) \in \mathbb{D} \times (0 , \infty ) .
\] 
Then $f_a( \cdot ,t)$ is the unique universal covering map 
of $\mathbb{D}$ onto 
$\mathbb{C} \setminus \overline{\mathbb{D}}(-a, ae^{-\frac{t}{2a}})$
normalized by $f_a(0,t)=0$ and $f_a'(0,t) = t$.
Since all the Maclaurin coefficients of $f_a(\cdot,t)$ are positive, 
it is easy to see that for fixed $t>0$ and $r \in (0,1)$,
\[
   \max_{|z|=r} |f_a(z,t)| = f_a(r,t) = a(e^{\frac{t}{a} \frac{r}{1-r}} - 1) 
  \nearrow \infty \quad \text{as} \quad a \searrow 0 .
\]
Therefore, 
the collection of normalized Loewner chains 
$\{ f_a ( \cdot , e^t) \}_{t \in \mathbb{R}}$, with $a > 0$,  
is not uniformly bounded 
on $\overline{\mathbb{D}}(0,r) \times [-T,T]$ for any 
fixed $r \in (0,1)$ and $T>0$.
\end{example}

\section{Lifting Loewner Chains to Universal Coverings}
Let $\{ f_t \}_{t \in I}$ be a Loewner chain 
such that each $\Omega_t := f_t(\mathbb{D})$, $t \in I$, is hyperbolic.
For each $t \in I$, let 
$\tilde{f}_t$ denote the unique universal covering map of
$\mathbb{D}$ onto $\Omega_t$ normalized by $\tilde{f}_t(0) = 0$ and 
$\tilde{f}_t '(0) > 0$.
Since the family $\{ \Omega_t \}_{t \in I}$ is nondecreasing,
the family $\{ \tilde{f}_t \}_{t \in I}$ forms 
a Loewner chain of universal covering maps.
For each $t \in I$, take $\tilde{\omega}_t \in \mathfrak{B}$
such that $f_t = \tilde{f}_t \circ \tilde{\omega}_t $.
Similarly, 
for each $(s,t)  \in I_+^2$, 
choose $\tilde{\omega}_{s,t} \in \mathfrak{B}$
such that
$\tilde{f}_{s} = \tilde{f}_{t} \circ \tilde{\omega}_{s,t}$.
Then we obtain the following diagram. 
\begin{center}
\begin{tikzpicture}
\node at (-4,5) (TSD) {$\mathbb{D}$};
\node at (4,5) (TTD) {$\mathbb{D}$};
\node at (-1.5,3.5) (SD) {$\mathbb{D}$};
\node at (1.5,3.5) (TD) {$\mathbb{D}$};
\node at (-1.5,1) (OS) {$\Omega_{s}$};
\node at (1.5,1) (OT) {$\Omega_{t}$};

\draw[thick,->] (TSD)--(TTD) node[midway,above] {$\tilde{\omega}_{s,t}$}; 
\draw[thick,<-] (TSD)--(SD) node[midway,right=3pt] {$\tilde{\omega}_{s}$}; 

\draw[thick,<-] (TTD)--(TD) node[midway,left=7pt] {$\tilde{\omega}_{t}$}; 
\draw[thick,->] (SD)--(TD) node[midway,above] {$\omega_{s,t}$}; 

\draw[thick,->] (TSD)--(OS) node[midway,left] {$\tilde{f}_{s}$}; 
\draw[thick,->] (TTD)--(OT) node[midway,right] {$\tilde{f}_{t}$}; 

\draw[thick,->] (SD)--(OS) node[midway,right] {$f_{s}$}; 
\draw[thick,->] (TD)--(OT) node[midway,left] {$f_{t}$}; 
\draw[thick,->] (OS)--(OT) node[midway,above] {$\inc$}; 
\end{tikzpicture} 
\end{center}

We now verify that the above diagram is commutative.
To this end, it suffices to show that 
$\tilde{\omega}_{s,t} \circ \tilde{\omega}_s 
= \tilde{\omega}_t \circ \omega_{s,t}$.
Indeed, from the identity
$\tilde{f}_t \circ \tilde{\omega}_{s,t} = \tilde{f}_s$,
we obtain
\begin{align*}
  \tilde{f}_t \circ \tilde{\omega}_{s,t} \circ \tilde{\omega}_s
  = \tilde{f}_s \circ \tilde{\omega}_s 
  = f_s
  = f_t \circ \omega_{s,t}
  =\tilde{f}_t \circ \tilde{\omega}_t \circ \omega_{s,t} .
\end{align*}
Since $\tilde{f}_t$ is locally univalent,
the identity $\tilde{\omega}_{s,t} \circ \tilde{\omega}_s (z) 
= \tilde{\omega}_t \circ \omega_{s,t} (z)$ holds 
in a neighborhood
of $0$.
Thus, by the identity theorem for analytic functions,
it holds on all of $\mathbb{D}$.

It is clear that the Loewner chain
$\{ \tilde{f}_t \}_{t\in I}$ is expanding
if and only if the original Loewner chain $\{ f_t \}_{t\in I}$
is expanding.
However, note that  
the continuity and strict monotonicity of $\{ f_t \}_{t \in I}$
are not necessarily preserved by $\{ \tilde{f}_t \}_{t \in I}$.
By suitably modifying Example~\ref{eg:discontinuity}
one can easily construct counterexamples to illustrate this.

\chapter{Loewner Theory on Fuchsian Groups}
\label{chapter:DeckTransformation}
\section{Constructing Transition Maps via Path Lifting}
Let $\{ f_t \}_{t \in I}$ be 
a Loewner chain of covering maps
with $\Omega_t = f_t(\mathbb{D})$ for $t \in I$.
For later use, 
we briefly recall the construction of 
$\omega_{s,t}$ for $(s,t) \in I_+^2$.

For $z \in \mathbb{D}$, choose a path 
$\tilde{\alpha}_s : [0,1] \rightarrow \mathbb{D}$ 
from $0$ to $z$.
Here, the subscript $s$ indicates that we are considering 
the covering map $f_s: \mathbb{D} \rightarrow \Omega$ 
and the tilde $\tilde{}$ signifies 
that $\tilde{\alpha}_s$ 
is a path in the covering surface $\mathbb{D}$.

Let $\alpha := f_s \circ \tilde{\alpha}_s $.
Then $\alpha$ is a path in $\Omega_s$ 
from $f_s (0) =0$ to $f_s (z)$.
Since $\Omega_s  \subset \Omega_t$,
there exists a unique path 
$\tilde{\alpha}_t : [0,1] \rightarrow \mathbb{D}$ 
called the lift of
$\alpha$ from $0$ 
with respect to the covering map
$f_t : \mathbb{D} \rightarrow f_t(\mathbb{D})$.
That is, $\tilde{\alpha}_t$ is a path satisfying 
$\tilde{\alpha}_t(0) = 0$ and 
$f_t \circ \tilde{\alpha}_t = \alpha (= f_s \circ \tilde{\alpha}_s)$.

Since $\mathbb{D}$ is simply connected,
the endpoint $\tilde{\alpha}_t (1)$
depends only on $z$ and not on the choice of $\tilde{\alpha}_s$.
We define a map $\omega_{s,t} : \mathbb{D} \to \mathbb{D}$ 
by $\omega_{s,t}(z) = \tilde{\alpha}_t (1) \in \mathbb{D}$.
It is easy to see that $\omega_{s,t}(0) =0$,
$f_t \circ \omega_{s,t} = f_s$,
and $\omega_{s,t}$ is analytic on $\mathbb{D}$.
\begin{center}
\begin{tikzpicture}[scale=0.8]
\draw[fill] (0,0) circle (1pt) node[below] {\small $0$};
\draw (0,0) circle (2.5cm);
\node at (-2,2) {\small $\mathbb{D}$};
\draw[->] plot [smooth] 
coordinates {(0,0) (-0.8,0.3) (-1,0.6) };
\node[below] at (-0.8,0.3) {\small $\tilde{\alpha}_s$};
\node[left] at (-1,0.6) {\small $z$};

\begin{scope}[xshift=7.5cm]
\draw[fill] (0,0) circle (1pt) node[below] {\small $0$};
\draw (0,0) circle (2.5cm);
\node at (-2,2) {\small $\mathbb{D}$};
\draw[->] plot [smooth] 
coordinates {(0,0) (-0.6,0.3) (-0.8,0.5) };
\node[below] at (-0.7,0.3) {\small $\tilde{\alpha}_t$};
\node[above] at (-0.8,0.4) {\small $\omega_{s,t}(z)$};
\end{scope}

\begin{scope}[xshift=-1cm,yshift=-8cm]
\node at (7,3.5) {\small $\Omega_{s}$};
\draw[fill] (3,1.5) circle (1pt) node[below] {\small $0$};
 \draw plot[smooth cycle]
coordinates{(0,0) (4,0.5) (6,2) (8,4) (4,3.5) (0.5,3) (1,2)};
\draw[fill] (3,1.5) circle (1pt) node[below] {\small $0$};
 \node at (6.4,1.8) {\small $\Omega_{t}$};
\draw plot[smooth cycle]
coordinates{(-0.2,-0.2) (4,0.3) (7.2,2) (8.2,4.2) (4,3.7) (0.4,3.1) (0.8,2)};

\draw[fill,gray] (3.5,2.0) ellipse (5pt and 2pt);
\draw[fill,gray] (4.5,2.5) ellipse (2pt and 5pt);
 \draw[->] plot[smooth]
coordinates{(3,1.5) (2.3,0.8) (1.8,1)};
\node at (2,0.6) 
{\footnotesize $f_{s}\! \circ \! \tilde{\alpha}_s \! = \! \alpha \! = \! f_{t} \! \circ \! \tilde{\alpha}_t$};
\end{scope}

\draw[->] (0.5,-2.7)-- (1.5,-4) node[midway,right] {\small $f_{s}$}; 
\draw[->] (6.6,-2.5)-- (5.6,-3.7) node[midway,left] {\small $f_{t}$}; 
\draw[->] (3,0)-- (4.5,0) node[midway,above] {\small $\omega_{s,t}$}; 
\end{tikzpicture}
\end{center}

The following theorem is a direct consequence of 
Theorems~\ref{thm:omega-is-univalent-if-it-is-continuous}
and~\ref{thm:omega-is-univalent-if-ft-is-continuous}.
Nevertheless, we provide here a purely topological proof 
based on the construction of $\omega_{s,t}$ given above.

\begin{theorem}
Let $\{ f_t \}_{t \in I}$ be a Loewner chain
of covering maps with
$\Omega_t = f_t ( \mathbb{D})$ for $t \in I$,
and let
$\{ \omega_{s,t}\}_{(s,t) \in I_+^2}$ 
be the associated transition family.
If $\{ f_t \}_{t\in I}$ is continuous, 
then for every $(s,t) \in I_+^2$, 
the map $\omega_{s,t}$ is univalent in $\mathbb{D}$.
\end{theorem}
\begin{proof}
Suppose that $\omega_{t_0, t_1^* }$ is not univalent 
for some $t_0 , t_1^*  \in I$ with $t_0 < t_1^*$.
Then there exist distinct points $z_1, z_2 \in \mathbb{D}$
such that $\omega_{t_0,t_1^*} (z_1) = \omega_{t_0,t_1^*} (z_2)$.
Since the function 
$t \mapsto \omega_{t_0, t} (z_1) - \omega_{t_0, t} (z_2)$ 
is continuous on $I \cap [t_0, \infty)$
and satisfies 
$\omega_{t_0,t_0}(z_1) - \omega_{t_0,t_0}(z_2) = z_1 - z_2 \neq 0$, 
there exists 
$t_1 \in I \cap (t_0, t_1^*]$
such that
\begin{align}
\label{eq:before_t_1}
 &  \omega_{t_0,t}(z_1) \not= \omega_{t_0,t}(z_2) 
 \quad \text{for all} \quad 
 t_0 \leq t <  t_1
\\   
 & \omega_{t_0,t_1}(z_1) = \omega_{t_0,t_1}(z_2) .
\label{eq:at_t_1}
\end{align}

Let 
$\tilde{\alpha}_{t_0}, \tilde{\beta}_{t_0} : [0,1] \rightarrow \mathbb{D}$ 
be paths in $\mathbb{D}$ from $0$ to $z_1$ and $z_2$, respectively.
Define 
$\alpha = f_{t_0} \circ \tilde{\alpha}_{t_0}$ and 
$\beta = f_{t_0} \circ \tilde{\beta}_{t_0}$.
Let 
$\tilde{\alpha}_{t_1} , \tilde{\beta}_{t_1}: [0,1] \rightarrow \mathbb{D}$
be the lifts of $\alpha $ and $\beta$, respectively,
starting at $0$ with respect to $f_{t_1}$.
Then, by \eqref{eq:at_t_1},
the paths $\tilde{\alpha}_{t_1}$ and $\tilde{\beta}_{t_1}$
have the same endpoint, since
\[
  \tilde{\alpha}_{t_1} (1) 
 = \omega_{t_0,t_1}(z_1) 
 = \omega_{t_0,t_1}(z_2)
 = \tilde{\beta}_{t_1} (1) .
\]

Since $\mathbb{D}$ is simply connected, there exists
a path homotopy 
$\tilde{F}_{t_1}:[0,1] \times [0,1] \rightarrow \mathbb{D}$ 
between $\tilde{\alpha}_{t_1}$ and $\tilde{\beta}_{t_1}$.
That is,
$\tilde{F}_{t_1}$ is a continuous map satisfying
\begin{align*}
  & \tilde{F}_{t_1}(u,0) = \tilde{\alpha}_{t_1} (u), \quad 
 \tilde{F}_{t_1}(u,1) = \tilde{\beta}_{t_1} (u) 
\\ 
  & \tilde{F}_{t_1} (0,v) = 0 
 = \tilde{\alpha}_{t_1} (0) = \tilde{\beta}_{t_1}(0), 
 \quad \tilde{F}_{t_1} (1,v) = \tilde{\alpha}_{t_1} (1) 
 = \tilde{\beta}_{t_1} (1)
\end{align*}
for all $u, v \in [0,1]$.

Clearly, the composition  $F := f_{t_1} \circ \tilde{F}_{t_1}$ 
defines a path homotopy between $\alpha$ and  $\beta$.
Since the family $\{ \Omega_t \}_{t \in I}$ 
is continuous at $t_1$ in the sense of kernel
and the compact set $F([0,1]\times [0,1])$
is contained in $\Omega_{t_1} = f_{t_1} (\mathbb{D})$, 
there exists $\delta > 0$
such that, for $t \in I$ with $|t-t_1| < \delta$,
we have $F ([0,1] \times [0,1]) \subset \Omega_t$.

Therefore, for each $t \in I$ with $0<|t-t_1| < \delta$, 
the path homotopy $F$ admits a unique lift 
$\tilde{F}_t :[0,1] \times [0,1] \rightarrow \mathbb{D}$ 
with respect to the covering map 
$f_t : \mathbb{D} \rightarrow \Omega_t$,
satisfying $\tilde{F}_t(0, v) = 0$ for all $v \in [0,1]$. 
Define $\tilde{\alpha}_t (u )= \tilde{F}_t (u,0)$
and $\tilde{\beta}_t (u )= \tilde{F}_t (u,1)$ for $0 \leq u \leq 1$.
Then $\tilde{\alpha}_t$ and  $\tilde{\beta}_t$
are the lifts of $\alpha$ and $\beta$, respectively, from $0$ 
with respect to $f_t$.

Since $\tilde{F}_t$ is a path homotopy, 
it follows that for all 
$t \in  (t_0 \vee (t_1 - \delta) , t_1) \cap I$,
\[
  \omega_{t_0,t}(z_1) = \tilde{\alpha}_t (1) 
 = \tilde{\beta}_t (1) = \omega_{t_0,t}(z_2) ,
\]
which contradicts \eqref{eq:before_t_1}.
\end{proof}

\section{Fundamental Group and Covering Transformations}
Let $f:\mathbb{D} \rightarrow \Omega$
be an analytic covering map of $\mathbb{D}$ onto 
a domain $\Omega$ in $\hat{\mathbb{C}}$.
We collect some known results concerning
the group of covering transformations associated with $f$.

A homeomorphism 
$\varphi : \mathbb{D} \rightarrow \mathbb{D}$
is called a covering transformation with respect to $f$
if $f \circ \varphi = f$. 
Clearly, each covering transformation is analytic,
and hence it is a conformal map of $\mathbb{D}$ onto itself.
The set of all covering transformations
with respect to $f$ 
forms a group under composition of maps,
denoted by $\Aut(\mathbb{D},f)$.
This group is a subgroup 
of the group $\Aut (\mathbb{D})$, 
the group of of all conformal automorphism of $\mathbb{D}$.
Each $\psi \in \Aut (\mathbb{D})$ is of the form 
\[
    \psi (z) 
    = e^{i \theta} \frac{z-z_0}{1-\overline{z_0}z}, 
    \quad z \in \mathbb{D},
\]
for some $\theta \in \mathbb{R}$ and $z_0 \in \mathbb{D}$.

Let $\varphi \in \Aut (\mathbb{D},f)$.
Then for every $w \in \Omega$ we have
$\varphi (f^{-1}(\{ w \})) = f^{-1}(\{ w \})$,
and the restriction $\varphi |_{f^{-1}(\{ w \})}$ 
defines a bijection of the fiber $f^{-1}(\{ w \})$ onto itself.
Moreover, the following holds.
\begin{lemma}
\label{lemma:existence_and_uniqueness_of_covering_transformations}
For any
$z_1$, $z_2 \in \mathbb{D}$ with
$f(z_1)=f(z_2)$,
there exists a unique $\varphi \in \Aut (\mathbb{D},f)$ 
such that $\varphi (z_1) =z_2$. 
\end{lemma}
\begin{proof}
For $z \in \mathbb{D}$, choose a path $\tilde{\alpha}_1$ 
from $z_1$ to $z$. 
Define $\alpha = f \circ \tilde{\alpha}_1$, 
and let $\tilde{\alpha}_2$ be the lift of 
$\alpha$ starting at $z_2$.
Then the endpoint $z'$ depend only on $z$, 
and not on the choice of $\tilde{\alpha}_1$.
Define $\varphi : \mathbb{D} \to \mathbb{D}$
by $\varphi (z) = z'$.
Clearly, we have that $\varphi (z_1) = z_2$ 
and $f \circ \varphi = f$.
In particular, 
this implies $\varphi$ is analytic in $\mathbb{D}$.

Similarly, if we construct a mapping $\psi$ 
by interchanging $z_1$ and $z_2$, 
it is easy to see that both $\psi \circ \varphi$ 
and $\varphi \circ \psi$ are the identity map 
on $\mathbb{D}$. 
Hence, $\varphi$ and $\psi$ are automorphisms of $\mathbb{D}$ 
and are inverses of each other. 
We conclude that 
$\varphi, \psi \in \Aut(\mathbb{D}, f)$.
\end{proof}

In the above proof, since $\tilde{\alpha}_2$ and 
$\varphi \circ \tilde{\alpha}_1$ share the same initial point and 
are both lifts of $\alpha$, they coincide. 
Replacing $\tilde{\alpha}_1$ by $\tilde{\alpha}$
and choosing a path $\tilde{\gamma}$ from $z_1$ to $z_2$,
we obtain the following figure.
Note that 
\[
 \varphi (z) 
 = \text{
 the endpoint of the lift of $(\alpha^{-1} * \gamma )* \alpha$
 starting at $z$}, 
\]
where 
$\gamma := f \circ \tilde{\gamma}$, and 
$\alpha^{-1}$ denotes the reverse path of $\alpha$ 
defined by
$\alpha^{-1} (t) = \alpha (1-t)$, $0 \leq t \leq 1$.

\begin{center}
\begin{tikzpicture}[scale=0.7]
\draw[fill] (0,-1) circle (1pt) node[below] {\tiny $z_1$};
\draw (0,0) circle (3cm);
\node at (-1.8,2) {$\mathbb{D}$};
\draw[->] plot [smooth] 
coordinates {(0,-1) (0.4,-0.5) (1,0) };
\node[right] at (1,0.1) {\tiny $z_2 = \varphi(z_1)$};
\node[right] at (0.35,-0.6) {\tiny $\tilde{\gamma}$};
\draw[->] plot [smooth] 
coordinates {(0,-1) (-0.8,-0.7) (-1,-0.4) };
\node[below] at (-0.8,-0.8) {\tiny $\tilde{\alpha}$};
\node[left] at (-1,-0.4) {\tiny $z$};
\draw[->] plot [smooth] 
coordinates {(1,0) (0.2,0.3) (0,0.6) };
\node[above right] at (-0.4,-0.35) 
{\tiny $\varphi \circ \tilde{\alpha}$};
\node[above] at (0,0.7) {\tiny 
$\varphi (z) = \varphi(\tilde{\alpha} (1))$ 
};

\begin{scope}[xshift=6.5cm,yshift=-3cm]
\node at (7,3.5) {\small $\Omega$};
\draw[fill] (3,1.5) circle (1pt);
\node at (3.8,1.1)
{\begin{minipage}{10ex}
  \tiny $\quad f(z_1)$ \\ $= f(z_2)$
   \end{minipage}} ;

 \draw plot[->,smooth cycle]
coordinates{(0,0) (4,0.5) (6,2) (8,4) (4,3.5) (0.5,3) (1,2)};

\draw plot[smooth cycle]
coordinates{(3,1.5) (4,2) (3.7,2.8) (3.2,2)};
\node at (3.9,3) {\tiny $\gamma$};
\draw[fill,gray] (3.5,2.0) ellipse (3pt and 5pt);
 \draw[->] plot[smooth]
coordinates{(3,1.5) (2.3,0.8) (1.8,1)};
\node at (2,0.6) {\tiny $\alpha$};
\end{scope}

\draw[->] (3.5,0) to [out=0, in=150] (6,-1);
\node at (4.75,0) {\small $f$}; 
\end{tikzpicture}
\end{center}

It follows easily from the lemma  that
every non-identity element 
$\varphi \in \Aut (\mathbb{D},f)$ has no fixed points.
Moreover,
we have the following uniqueness property.
For $\varphi, \psi \in \Aut(\mathbb{D},f)$,
\begin{equation}
\label{eq:uniqueness_for_deck_transformation}
  \varphi (z_0) = \psi (z_0) 
   \text{ for some $z_0 \in \mathbb{D}$ if and only if }
   \quad \varphi = \psi .
\end{equation}

Let $\gamma$ be a loop in $\Omega$ based at $f(0)$,
and let $[\gamma]$ denote 
the homotopy class of $\gamma$, that is,
the set of all loops path-homotopic to $\gamma$.
We denote the set of all such homotopy classes by 
$\pi_1 (\Omega , f(0))$. 
For $[\gamma]$, $[\delta] \in \pi_1 (\Omega , f(0))$ 
we define the product by 
$[\gamma]*[\delta] = [\gamma * \delta]$,
where $\gamma * \delta$ is the concatenation of the loops 
$\gamma$ followed by $\delta$, defined by
\begin{equation}
   \gamma * \delta (t)
    =
    \begin{cases}
    \gamma (2t) \quad & \text{if} \quad 0 \leq t \leq \frac{1}{2} ,
\\
     \delta (2t-1) \quad & \text{if} \quad \frac{1}{2} < t \leq 1 .
    \end{cases}
\end{equation}
One readily checks that the product is associative. 
The constant loop at $f(0)$ serves as the identity element, 
and each loop $\gamma$ has an inverse given by the reverse path 
$\gamma^{-1}(t) := \gamma(1-t)$ for $0 \leq t \leq 1$. 
Hence $\pi_1(\Omega, f(0))$ forms a group under this operation.
This group is called 
the fundamental group of $\Omega$ based at $f(0)$.

\begin{lemma}
The fundamental group $\pi_1(\Omega, f(0))$ is
isomorphic to $\Aut (\mathbb{D},f)$. 
\end{lemma}
For later applications, and in order to introduce
an explicit isomorphism 
$\tau : \pi_1(\Omega, f(0)) \rightarrow \Aut (\mathbb{D},f)$,
we now give a proof of the lemma.
\begin{proof}
Let $[\gamma] \in \pi_1(\Omega , f(0))$
and let $\tilde{\gamma}$ be the lift of $\gamma$ 
starting $0$ with respect to $f$.
Then the endpoint $\tilde{\gamma}(1)$ 
depends only on the equivalence class $[\gamma]$,
and not on the particular choice of the representative $\gamma$.
Since $\tilde{\gamma}(1)$ belongs to
$f^{-1}( \{ f(0) \})$, 
there exists a unique 
$\varphi \in \Aut(\mathbb{D},f)$ such that 
$\varphi (0) = \tilde{\gamma}(1)$.
We thus define a mapping 
$\tau : \pi_1(\Omega, f(0)) \rightarrow \Aut (\mathbb{D},f)$
by $\tau ([\gamma]) = \varphi$.
\begin{center}
\begin{tikzpicture}[scale=0.8]
\draw[fill] (0,0) circle (1pt) node[below] {\tiny $0$};
\draw (0,0) circle (2.5cm);
\node at (-2,2) {\small $\mathbb{D}$};
\draw[->] plot [smooth] 
coordinates {(0,0) (-0.8,0.3) (-1,0.6) };
\node[below] at (-0.8,0.3) {\tiny $\tilde{\gamma}$};
\node[above] at (-1.5,0) {\tiny 
 \begin{minipage}{8ex} $\varphi(0)$ \\ $=\tilde{\gamma}(1)$ 
 \end{minipage}};

\draw[->] plot [smooth] 
coordinates {(0,0) (0.8,0.3) (1,0.6) };
\node[below] at (0.8,0.3) {\tiny $\tilde{\delta}$};
\node[right] at (1,0.6) {\tiny $\psi(0)=\tilde{\delta}(1)$};

\draw[->] plot [smooth] 
coordinates {(-1,0.6) (-0.2,0.9) (0,1.2) };
\node[above] at (-0.6,0.7) {\tiny $\varphi \circ \tilde{\delta}$};

\begin{scope}[xshift=-3cm,yshift=-7cm]
\draw[fill] (3,1.5) circle (1pt) node[below] {\tiny $0$};
 \draw plot[smooth cycle]
coordinates{(0,0) (4,0.5) (6,2) (8,4) (4,3.5) (0.5,3) (1,2)};
 \node at (6.4,1.8) {\small $\Omega$};

\draw[fill,gray] (2.5,1.5) ellipse (5pt and 2pt);
 \draw[->] plot[smooth cycle]
coordinates{(3,1.5) (2.5,2) (2,1.5) (2.5,1) };
\node[below] at (2.5,1) {\tiny $\gamma$};

\draw[fill,gray] (3.5,2.5) ellipse (2pt and 5pt);
 \draw[->] plot[smooth cycle]
coordinates{(3,1.5) (4,2) (4.5,3) (3.5,3) };
\node[right] at (4.5,3) {\tiny $\delta$};

\end{scope}

\draw[->] (0,-2.7)-- (0,-3.5) node[midway,right] {\small $f$}; 
\end{tikzpicture}
\end{center}
We show that the map 
$\tau : \pi_1(\Omega, f(0)) \to \Aut (\mathbb{D},f)$
is a homomorphism.
Let $[\gamma], [\delta] \in \pi_1(\Omega, f(0))$, 
and $\tilde{\gamma}$ and $\tilde{\delta}$ be the lifted paths
of $\gamma$ and $\delta$ starting at $0$, respectively.
Then there exist unique $\varphi, \psi \in \Aut(\mathbb{D}, f)$
with $\varphi (0) = \tilde{\gamma}(1)$ and 
$\psi(0) = \tilde{\delta}(1)$.
Since 
$\tilde{\gamma}(1) = \varphi (0) = \varphi \circ \tilde{\delta}(0)$,
the product path $\tilde{\gamma} * (\varphi \circ \tilde{\delta})$
is well defined.
Moreover,
\[
  f(\tilde{\gamma} * (\varphi \circ \tilde{\delta})) 
  = f(\tilde{\gamma}) * f \circ ( \varphi \circ \tilde{\delta})
  = \gamma * \delta
\]
so it is a lift of $\gamma * \delta$ from $0$.
This implies 
\begin{align*}
 \tau ([\gamma * \delta])(0) 
 = & \, (\tilde{\gamma} * (\varphi \circ \tilde{\delta})) (1) \\
 = & \, \varphi \circ \tilde{\delta} (1) 
 = \varphi ( \tilde{\delta} (1))
 = \varphi(\psi (0))
 = \tau ([\gamma]) \circ \tau ([\delta]) (0) .   
\end{align*}
By \eqref{eq:uniqueness_for_deck_transformation},
we conclude that
$\tau ([\gamma * \delta]) = \tau ([\gamma]) \circ \tau ([\delta])$.

Next we show that $\tau$ is injective.
To this end, suppose $\tau ([\gamma]) = \id_{\mathbb{D}}$ 
for some $[\gamma ] \in \pi_1 (\Omega, f(0))$.
Let $\tilde{\gamma}$ be the lift of $\gamma$ starting at $0$.
Then, since
$\tilde{\gamma}(1) = \tau([\gamma])(0) = 0 $,
the path $\tilde{\gamma}$ is a loop in $\mathbb{D}$ based at $0$.
As $\mathbb{D}$ is simply connected,
we obtain $\tilde{\gamma} \sim e_{0}$ in $\mathbb{D}$, and hence
$\gamma \sim e_{f(0)}$ in $\Omega$.
Therefore $[\gamma] = [e_{f(0)}]$, i.e.,
$[\gamma]$ coincides with the identity element 
in $\pi_1(\Omega, f(0))$.
Thus the kernel of the homomorphism $\tau$ is trivial, and
$\tau$ is injective.

Finally, we show that $\tau$ is surjective.
For any $\varphi \in \Aut (\mathbb{D},f)$,
choose a path $\tilde{\gamma}$ in $\mathbb{D}$ from $0$
to $\varphi (0)$, and set $\gamma = f \circ \tilde{\gamma}$.
Then it is clear that $\tau ([\gamma ]) = \varphi$.
\end{proof}

\section{Induced Homomorphisms between Covering Transformation Groups}
For $t \in I$, let $\Gamma_t$ denote 
the group of covering transformations 
of the universal covering map
$f_t : \mathbb{D} \to \Omega_t$, i.e.,
$\Gamma_t = \Aut (\mathbb{D}, f_t )$.
Now we introduce 
a map $\sigma_{s,t} : \Gamma_s \rightarrow \Gamma_t$ 
for $(s,t) \in I_+^2$  as follows.
The remainder of this chapter is devoted to studying 
relations among
$\{f_t \}_{t \in I}$, $\{ \omega_{s,t}\}_{(s,t)\in I_+^2}$, 
$\{ \Gamma_t \}_{t \in I}$ and $\{ \sigma_{s,t}\}_{(s,t)\in I_+^2} $.

For $\varphi_s \in \Gamma_s$, choose a path 
$\tilde{\gamma}_s :[0,1] \rightarrow \mathbb{D}$ arbitrarily
from $0$ to $\varphi_s (0)$.
Since 
\[
 f_s(\tilde{\gamma}_s(1)) = f_s \circ \varphi_s (0) = f_s (0) = 0, 
\]
the path $\gamma := f_s \circ \tilde{\gamma}_s $ is 
a loop in $\Omega_s (\subset \Omega_t)$ based at $f_s(0) =0$.
Let $\tilde{\gamma}_t :[0,1] \rightarrow \mathbb{D}$ be the 
unique lift of $\gamma$ starting at $0$. 
Then there exists a unique $\varphi_t \in \Gamma_t$ such that
$\varphi_t (0) = \tilde{\gamma}_t (1)$.
Since $\mathbb{D}$ is simply connected,
the endpoint $\tilde{\gamma}_t(1)$ 
does not depend on the choice of 
$\tilde{\gamma}_s$, and hence 
$\varphi_t \in \Gamma_t$ is uniquely determined by 
$\varphi_s \in \Gamma_s$.
We thus define $\sigma_{s,t} : \Gamma_s \rightarrow \Gamma_t$
by $\sigma_{s,t} (\varphi_s) = \varphi_t$.
By definition, 
$\sigma_{t,t}$ is the identity mapping of $\Gamma_t$, 
and it is easy to see that the semigroup relation
\begin{equation}
  \sigma_{t_1,t_2} \circ \sigma_{t_0,t_1} = \sigma_{t_0,t_2} 
\end{equation}
holds for $t_0,t_1,t_2 \in I$ with $t_0 \leq t_1 \leq t_2$.

We now prove $\sigma_{s,t}$
is an injective homomorphism and satisfies
$\omega_{s,t} \circ \varphi_s 
= \sigma_{s,t}(\varphi_s)\circ \omega_{s,t}$.
\begin{proof}[Proof of Theorem~\ref{thm:int_the_injective_%
homomorphsim}]
Let $\varphi_s \in \Gamma_s$ and $z \in \mathbb{D}$, 
and let 
$\tilde{\gamma}_s, \tilde{\alpha}_s :[0,1] \rightarrow \mathbb{D}$ 
be paths from $0$ to $\varphi_s (0)$ 
and from $0$ to $z$, respectively.
Since  
$\tilde{\gamma}_s(1) = \varphi_s (0) 
= \varphi_s \circ \tilde{\alpha}_s (0)$,
the product path 
$\tilde{\beta}_s 
:= \tilde{\gamma}_s * (\varphi \circ \tilde{\alpha}_s)$ 
is well defined and
is a path from $0$ to $\varphi_s(z)$. 
Set $\gamma = f_s \circ \tilde{\gamma}_s$
and $\alpha = f_s \circ \tilde{\alpha}_s$.
Then $\gamma$ is a loop in $\Omega_s$ based at $0$ and
$\alpha$ is a path in $\Omega_s$ from $0$ to $f_s(z)$.
Clearly $\beta := f_s \circ \beta_s = \gamma * \alpha$.
Let $\tilde{\gamma}_t$
and $\tilde{\alpha}_t$ be the lifts of 
$\gamma$ and $\alpha$, respectively, from $0$ with respect to $f_t$.
As before, the product path 
$\tilde{\beta}_t := 
\tilde{\gamma}_t * (\varphi_t \circ \tilde{\alpha}_t)$ is
well defined and is a path from $0$ 
to $\varphi_t (\tilde{\alpha}_t(1))= \varphi_t(\omega_{s,t}(z))$.
Furthermore, 
since $\tilde{\beta}_t$ is the lifted path of 
$f_s \circ \tilde{\beta}_s = \beta$,
by definition the endpoint of $\tilde{\beta}_t$
coincides with $\omega_{s,t}(\varphi_s(z))$.
Thus we have
$\varphi_t(\omega_{s,t}(z)) = \omega_{s,t}(\varphi_s(z))$,
which is equivalent to
\begin{equation}
\label{eq:relation_between_sigma_omega}
  \sigma_{s,t}(\varphi_s)(\omega_{s,t}(z)) 
   = \omega_{s,t}(\varphi_s(z)). 
\end{equation}

\begin{center}
\begin{tikzpicture}
\draw[fill] (0,0) circle (1pt) node[below] {\tiny $0$};
\draw (0,0) circle (2.5cm);
\node at (-2,2) {$\mathbb{D}$};
\draw[->] plot [smooth] 
coordinates {(0,0) (0.4,0.5) (1,1) };
\node[right] at (1,0.9) {\tiny $\varphi_s (0)$};
\node[right] at (0.4,0.5) {\tiny $\tilde{\gamma}_s$};
\draw[->] plot [smooth] 
coordinates {(0,0) (-0.8,0.3) (-1,0.6) };
\node[below] at (-0.8,0.3) {\tiny $\tilde{\alpha}_s$};
\node[left] at (-1,0.6) {\tiny $z$};
\draw[->] plot [smooth] 
coordinates {(1,1) (0.2,1.3) (0,1.6) };
\node[above right] at (0.15,1.2) 
{\tiny $\varphi_s \circ \tilde{\alpha}_s$};
\node[above,left] at (0,1.7) {\tiny 
\begin{minipage}{9ex}
$\varphi_s(\tilde{\alpha}_s (1))$ \\
$=\varphi_s (z)$ 
\end{minipage}};

\begin{scope}[xshift=7.5cm]
\draw[fill] (0,0) circle (1pt) node[below] {\tiny $0$};
\draw (0,0) circle (2.5cm);
\node at (-2,2) {$\mathbb{D}$};
\draw[->] plot [smooth] 
coordinates {(0,0) (0.4,0.5) (1,1) };
\node[right] at (0.4,0.5) {\tiny $\tilde{\gamma}_t$};
\node[right] at (1,1) {\tiny $\varphi_t(0)$};
\draw[->] plot [smooth] 
coordinates {(0,0) (-0.6,0.3) (-0.8,0.5) };
\node[below] at (-0.8,0.3) {\tiny $\tilde{\alpha}_t$};
\node[left] at (-0.8,0.4) {\tiny $\omega_{s,t}(z)$};
\draw[->] plot [smooth] 
coordinates {(1,1) (0.5,1.3) (0.2,1.45) };
\node[above right] at (0.4,1.2) {\tiny $\varphi_t \circ \tilde{\alpha}_t$};
\node[above,left] at (0.4,1.4) {\tiny 
\begin{minipage}{14ex}
$\varphi_t(\omega_{s,t}(z)) $ \\
$= \varphi_t(\tilde{\alpha}_t (1))$
\end{minipage}};
\end{scope}

\begin{scope}[xshift=-1cm,yshift=-8cm]
\node at (7,3.5) {\small $\Omega_{s}$};
\draw[fill] (3,1.5) circle (1pt) node[below] {\tiny $0$};
 \draw plot[smooth cycle]
coordinates{(0,0) (4,0.5) (6,2) (8,4) (4,3.5) (0.5,3) (1,2)};
 \node at (6.4,1.8) {\small $\Omega_{t}$};
\draw plot[smooth cycle]
coordinates{(-0.2,-0.2) (4,0.3) (7.2,2) (8.2,4.2) (4,3.7) (0.4,3.1) (0.8,2)};

\draw plot[smooth cycle]
coordinates{(3,1.5) (4,2) (3.7,2.8) (3.2,2)};
\node at (3.9,3) {\tiny $\gamma$};
\draw[fill,gray] (3.5,2.0) ellipse (3pt and 5pt);
 \draw[->] plot[smooth]
coordinates{(3,1.5) (2.3,0.8) (1.8,1)};
\node at (2,0.6) {\tiny $\alpha$};
\end{scope}

\draw[->] (0.5,-2.7)-- (1.5,-4) node[midway,right] {\small $f_{s}$}; 
\draw[->] (6.6,-2.5)-- (5.6,-3.7) node[midway,left] {\small $f_{t}$}; 
\draw[->] (3,0)-- (4.5,0) node[midway,above] {\small $\omega_{s,t}$}; 
\end{tikzpicture}
\end{center}
Next we show that
$\sigma_{s,t}$ is a homomorphism.
Let $\varphi_s$, $\psi_s \in \Gamma_s$.
By (\ref{eq:relation_between_sigma_omega})
and $\omega_{s,t}(0) = 0$
we have
\begin{align*}
 \sigma_{s,t}(\varphi_s \circ \psi_s)(0)
 = & \,
  \sigma_{s,t}(\varphi_s \circ \psi_s)(\omega_{s,t}(0))
\\
 = & \,
  \omega_{s,t} (\varphi_s \circ \psi_s (0))
\\
 = & \,
  \omega_{s,t} (\varphi_s (\psi_s (0))
\\
 = & \,
  \sigma_{s,t} (\varphi_s) ( \omega_{s,t}(\psi_s (0))
\\
 = & \,
  \sigma_{s,t} (\varphi_s) ( \sigma_{s,t}(\psi_s) (\omega_{s,t}(0)))
\\
 = & \,
  \sigma_{s,t} (\varphi_s) ( \sigma_{s,t}(\psi_s) (0))
 = \sigma_{s,t} (\varphi_s) \circ \sigma_{s,t}(\psi_s) (0) .
\end{align*}
From \eqref{eq:uniqueness_for_deck_transformation} 
it follows that 
$\sigma_{s,t}(\varphi_s \circ \psi_s) 
= \sigma_{s,t} (\varphi_s) \circ \sigma_{s,t}(\psi_s)$.
Therefore, we conclude that $\sigma_{s,t}$ is a group homomorphism.

Finally, we show $\sigma_{s,t} $ is injective.
Suppose that $\sigma_{s,t}(\varphi_s) = \sigma_{s,t}(\psi_s)$
for some $\varphi_s,\psi_s \in \Gamma_s$.
Then, by (\ref{eq:relation_between_sigma_omega}),
we have
\[
 \omega_{s,t} ( \varphi_s (0) )
 = \sigma_{s,t}(\varphi_s) (\omega_{s,t}(0))
 = \sigma_{s,t}(\psi_s) (\omega_{s,t}(0))
 = \omega_{s,t} ( \psi_s (0) ) .
\]
Since $\omega_{s,t}$ is univalent, 
it follows that $\varphi_s (0) = \psi_s (0)$.
By  \eqref{eq:uniqueness_for_deck_transformation},
we conclude that $\varphi_s = \psi_s$, as required.
\end{proof}

\begin{corollary}
For $(s,t) \in I_+^2$ and 
$\varphi_s \in \Gamma_s$, the image domain 
$\omega_{s,t}(\mathbb{D})$ 
is $\sigma_{s,t}(\varphi_s )$ invariant under
$\sigma_{s,t}(\varphi_s )$, i.e.,
\[
  \sigma_{s,t}(\varphi_s ) ( \omega_{s,t}(\mathbb{D})) 
= \omega_{s,t}(\mathbb{D}).
\]
\end{corollary}
\begin{proof}
This follows immediately from 
the identity 
$\omega_{s,t} \circ \varphi_s = \sigma_{s,t}(\varphi_s) \circ \omega_{s,t}$
together with $\varphi_s (\mathbb{D}) = \mathbb{D}$. 
\end{proof}

Let $t_0 \in I$ and fix $\varphi \in \Gamma_{t_0}$, 
and let $\varphi_t  = \sigma_{t_0,t}(\varphi) \in \Gamma_t$
for $I \cap [t_0, \infty ]$. 
Since $\sigma_{t_0,t} : \Gamma_{t_0} \rightarrow \Gamma_t$ 
is a homomorphism, 
we have 
\[
    (\varphi^{-1})_t = \sigma_{t_0,t}(\varphi^{-1})
  = (\sigma_{t_0,t}(\varphi))^{-1} 
  = (\varphi_t)^{-1}.
\]
Thus, without ambiguity, we may simply write 
$\varphi_t^{-1}$.

\section{Loewner PDE for Covering Transformations}
Now we show that $\varphi_t$ 
satisfies a partial differential equation. 
\begin{proof}%
[Proof of Theorem~\ref{theorem:int_LoewnerPDEforDeckTransformation}]
It suffices to prove the theorem in the case $I = [t_0, t_0^*]$
with $-\infty < t_0 <t_0^* < \infty$.

If $\varphi_{t_0} = \id_{\mathbb{D}}$, then, 
since $\sigma_{t_0,t}$ is a homomorphism,
we have 
$\varphi_t = \sigma_{t_0,t}(\varphi_{t_0}) = \id_{\mathbb{D}}$,
and hence \eqref{eq:int_LoewnerPDEforDeckTransformation} 
holds trivially.

Suppose that 
$\varphi \in \Gamma_{t_0} \setminus \{ \id_{\mathbb{D}} \}$.
Then, by the fixed point free property of $\Gamma_t$,
it follows that $\varphi (0) \neq 0$.
Combining this with the fact that 
$\omega_{t_0,t}$ is univalent and $\omega_{t_0,t}(0)=0$,
we obtain $\varphi_t(0) = \omega_{t_0,t}(\varphi (0)) \not= 0$.
Since the mapping
$[t_0,t_0^*] \ni t \mapsto \varphi_t(0) \in \mathbb{D}$ 
is continuous, it follows that
\[
   0 <  
   m := \min_{t \in [t_0,t_0^*]} |\varphi_t(0)| 
  \leq 
  M  := \max_{t \in [t_0,t_0^*]} |\varphi_t(0)| 
  < 1 .
\]
Similarly, the mapping 
$[t_0,t_0^*] \ni t \mapsto \varphi_t^{-1}(0) \in \mathbb{D}$ 
is also continuous, and since $|\varphi_t^{-1}(0)| = |\varphi_t(0)|$,
we have 
\[
  0< m 
  = 
  \min_{t \in [t_0,t_0^*]} |\varphi_t^{-1}(0)| 
 \leq
  M 
  = 
 \max_{t \in [t_0,t_0^*]} |\varphi_t^{-1}(0)| < 1 .
\]
For $t \in [t_0,t_0^*]$, 
let 
$\zeta_t = \varphi_t^{-1}(0)$
and choose a real-valued continuous function
$\theta_t$
such that 
$e^{i \theta_t} = - \varphi_t(0)/\varphi_t^{-1}(0)$.
Then, $\varphi_t$ admits the representation
\[
   \varphi_t(z) = e^{i \theta_t}\frac{z-\zeta_t}{1-\overline{\zeta_t}z},
   \quad z \in \mathbb{D} .
\]
From these properties
it follows that the family 
$\{ \varphi_t \}_{t \in [t_0,t_0^*]}$ 
is continuous in the sense of locally uniform convergence 
on $\mathbb{D}$.

Let $t_1,t_2 \in [t_0,t_0^*]$ with $t_1 \leq t_2$.
Then 
\[
  \varphi_{t_2} = \sigma_{t_0,t_2}(\varphi ) 
= \sigma_{t_1,t_2} \circ \sigma_{t_0,t_1}(\varphi ) 
= \sigma_{t_1,t_2} (\varphi_{t_1}) .
\]
Combining this with
Theorem \ref{thm:int_the_injective_homomorphsim} 
we obtain
\[
   \varphi_{t_2} (\omega_{t_1,t_2}(z))
 = \sigma_{t_1,t_2} (\varphi_{t_1})(\omega_{t_1,t_2}(z))
 = \omega_{t_1,t_2}(\varphi_{t_1} (z)) .
\]
Hence
\begin{align}
\label{eq:pre_derivative_of_varphi_t}
   & \frac{\varphi_{t_2}(z) - \varphi_{t_1}(z)}{a(t_2)-a(t_1)}
\\
  =& \,
 \frac{\varphi_{t_2}(z) - \varphi_{t_2}(\omega_{t_1,t_2}(z))}{a(t_2)-a(t_1)} 
 + \frac{\varphi_{t_2}(\omega_{t_1,t_2}(z)) - \varphi_{t_1}(z)}{a(t_2)-a(t_1)}
\nonumber
\\
    =& \,
   - \frac{\omega_{t_1,t_2}(z)-z}{a(t_2)-a(t_1)} 
 \int_0^1   \varphi_{t_2}'((1-\lambda)z+
 \lambda\omega_{t_1,t_2}(z)) \, d \lambda 
  + \frac{\omega_{t_1,t_2}(\varphi_{t_1}(z)) - \varphi_{t_1}(z)}{a(t_2)-a(t_1)} .
\nonumber
\end{align}
Now fix $t \in [t_0,t_0^*] \setminus N$ arbitrarily.
Since $\varphi_{t_2}'(z) \rightarrow \varphi_t'$ and 
$\omega_{t_1,t_2}(z) \rightarrow z$
locally uniformly 
on $\mathbb{D}$ as $t_2 - t_1 \searrow 0$ with $t_1 \leq t \leq t_2$,
we obtain
\[
   \int_0^1   \varphi_{t_2}'((1-\lambda)z+
 \lambda\omega_{t_1,t_2}(z)) \, d \lambda \rightarrow \varphi_t'(z) .
\]
Furthermore, by Theorem \ref{theorem:a-derivative-of-omega},
we have 
\[
   \2sidelim \frac{\omega_{t_1,t_2}(z)-z}{a(t_2)-a(t_1)} 
 = \frac{\partial \omega}{\partial a(t)} (z,t)
 = - \frac{z}{a(t)} P(z,t) , \quad z \in \mathbb{D}
\]
with convergence locally uniform on $\mathbb{D}$.
Since $\varphi_{t_1}(z) \rightarrow \varphi_t(z)$, 
it follows that
\[
   \2sidelim \frac{\omega_{t_1,t_2}(\varphi_{t_1}(z))-
 \varphi_{t_1}(z)}{a(t_2)-a(t_1)} 
 = \frac{\partial \omega}{\partial a(t)} (\varphi_t(z),t)
 = - \frac{\varphi_t(z)}{a(t)} P(\varphi_t(z),t).
\]
Combining these equalities 
with \eqref{eq:pre_derivative_of_varphi_t}
we obtain \eqref{eq:int_LoewnerPDEforDeckTransformation}.

Suppose that  $a(t)$ is absolutely continuous 
and that $\dot{a}(t) := \frac{\partial a}{\partial t}  > 0$ a.e.
Let $E_0 $ be the set of all $t \in I$ 
at which $a$ is not differentiable.
Let $E_1 (\subset [\alpha, \beta ])$ be 
the set of all $t \in I$ at which 
$a$ is differentiable and $\dot{a} (t) =  0$.
Then $E_0 \cup E_1$ is a set of Lebesgue  measure $0$,
and for $t \in [t_0,t_0^*] \setminus (N \cup E_0 \cup E_1)$
we have
\[
   \frac{\partial \varphi}{\partial t} (z,t) 
 = \frac{\partial \varphi}{\partial a(t)} (z,t) \dot{a}(t)
\] 
Combining this with 
\eqref{eq:int_LoewnerPDEforDeckTransformation} 
we obtain 
\eqref{eq:int_LoewnerPDEforDeckTransformation_by_time}.
\end{proof}

\begin{example}
\label{example:covering_map_on_plane_minus_disks}
For $t > 0$ let 
\[
 f_t(z) = e^{\frac{t}{2} \left(\frac{1+z}{1-z} -1 \right)} - 1 , \quad
 z \in \mathbb{D}. 
\]
Then 
$f_t$ is the unique universal covering map of $\mathbb{D}$ 
onto 
$\mathbb{C} \setminus \overline{\mathbb{D}}(-1,e^{-\frac{t}{2}})$ 
with $f_t(0)=0$ and $a(t) := f_t'(0)=t$
and $\{ f_t \}_{t>0}$ forms a strictly increasing and continuous 
Loewner chain of universal covering maps.
\end{example}
We next compute 
$\omega_{s,t}$, $P(z,t)$, and related objects for 
Example~\ref{example:covering_map_on_plane_minus_disks}.

By Theorem \ref{theorem:FirstDecompositionTheorem}
the family $\{ f_t \}_{t>0}$ can be uniquely 
decomposed as 
$f_t = F \circ g_t$,
where $F$ is an entire function with $F(0)=F'(0)-1=0$, and
$\{ g_t \}_{t>0}$ is a Loewner chain of univalent functions 
with $g_t(0)=0$ and $g_t'(0)= f_t'(0)$, $t > 0$.  
In this case, it is easy to verify that 
$F(w) = e^w - 1$ and 
$g_t(z) = \frac{t}{2} \left(\frac{1+z}{1-z} -1 \right)$.
A straightforward  computation  
shows that the associated Herglotz and transition families
are given by
\begin{align*}
   P_t(z) \equiv & \,  1-z, \quad t>0 \text{ and } z \in \mathbb{D} \\
   \omega_{s,t}(z) = & \, 
 \frac{\frac{s}{t}\frac{z}{1-z}}{1+ \frac{s}{t}\frac{z}{1-z}}, \quad
 0<s<t \text{ and } z \in \mathbb{D} .
\end{align*}
Moreover, for each $t>0$, the group of covering transformations 
$\Gamma_t$ of
the covering map $f_t$ is generated by a single linear fractional 
transformation $\varphi_t \in \Aut (\mathbb{D})$ given by
\[
    \varphi_t(z) 
 = \frac{t-2\pi i}{t+2\pi i}
 \frac{z + \frac{2\pi i}{t-2\pi i}}{1- \frac{2\pi i}{t+2\pi i} z}. 
\]
The associated homomorphism 
$\sigma_{s,t} : \Gamma_s \rightarrow \Gamma_t$ is 
determined by $\sigma_{s,t}(\varphi_s) = \varphi_t$, 
for $0<s \leq t < 0$.

\chapter{Loewner Theory on Hyperbolic Metrics}
\label{chapter:HyperbolicMetrics}
\section{Hyperbolic Metrics}
Let $\Omega$ be a hyperbolic domain in $\mathbb{C}$,
and let $f: \mathbb{D} \rightarrow \Omega$ be 
a universal covering map.
For any simply connected subdomain 
$D \subset \Omega$, 
since $D$ is an evenly covered neighborhood of each $w \in D$,
there exists a single-valued branch 
$g :  D \to  \mathbb{D}$ 
of the inverse $f^{-1}$.  
The hyperbolic metric $\lambda (w) |dw|$ 
on $\Omega$ is then defined by
\[
  \lambda (w) = \frac{|g'(w)|}{1-|g(w)|^2} , \quad w \in \Omega .
\]
This definition does not depend on the choice of 
$D$ and $g$.
Indeed, if $g^*$ is another branch of $f^{-1}$ 
on $D^*$ with $D \cap D^* \neq \emptyset$,
we can write
$g^* = \varphi \circ g$ 
on $D \cap D^*$ 
for some 
$\varphi \in \Aut(\mathbb{D}, f)$.
Hence
\[
  \frac{|{g^*} '(w)|}{1-|g^*(w)|^2}
  = \frac{|\varphi'(g(w))| |g '(w)|}{1-|\varphi(g(w))|^2}
  = \frac{|g '(w)|}{1-|g(w)|^2} .
\]

Therefore, if $\{ f_t \}_{t \in I}$ is a continuous Loewner chain 
of covering maps with $\Omega_t = f_t(\mathbb{D})$, 
then the hyperbolic density $\lambda_t$ on $\Omega_t$  
is given by
\begin{align}
    \lambda_t (w) = \frac{|g_t'(w)|}{1-|g_t(w)|^2} ,
\end{align}
where $g_t$ denotes a single-valued branch of 
$f_t^{-1}$ on a simply connected subdomain of $\Omega_t$.

\section{Mixed Partials and Loewner Equation for Inverse Functions}
In this chapter, we derive the differential 
equation satisfied by $\lambda_t(w)$. 
To this end, we require a result concerning 
the equality of mixed partial derivatives
of $g_t(w)$ 
and the ordinary differential equation satisfied by $g_t(w)$.

\begin{lemma}
\label{lemma:CauchyIntegralFormula_for_inverse_functions}
Let $D$ and $G$ be domains in $\mathbb{C}$,
and let $f: D \to G$ be a conformal map
with inverse $g = f^{-1}$.
Suppose that $\gamma$ is a rectifiable 
and positively oriented Jordan curve 
in  $D$ 
such that the interior domain 
$D_i(\gamma)$ enclosed by $\gamma$
is contained in $D$. 
Then, for any $w \in D_i(f (\gamma ))$ and $m=0,1, \ldots $,
we have
\begin{equation}
   g^{(m)}(w) 
    = \frac{m!}{2 \pi i}
   \int_{\gamma} 
   \frac{\zeta f'(\zeta)}{(f(\zeta ) -w)^{m+1}} \, d \zeta .
\end{equation}
\end{lemma}
\begin{proof}
Let $w_0 \in D_i (f (\gamma ))$ and 
set $z_0 = g(w_0) \in D_i(\gamma )$.
The function 
\[
  \frac{zf'(z)}{f(z)-w_0},  \quad z \in D,	      
\]
is analytic in $D \setminus \{ z_ 0\}$ 
and has a simple pole at
$z_0$ with 
\[
  \frac{zf'(z)}{f(z)-w_0}
 =
 \frac{(z_0 + z-z_0)\{f'(z_0) + f''(z_0)(z-z_0)+ \cdots \}}
 {f'(z_0)(z-z_0)+ \cdots }
 = \frac{z_0}{z-z_0} + \cdots 
\]
near $z_0$.
Therefore, by the residue theorem, we obtain
\[
   g(w_0) 
 = z_0 
 = \frac{1}{2 \pi i} \int_{\gamma}
    \frac{\zeta f'(\zeta)}{f(\zeta ) -w_0} \, d \zeta . 
\]
The general case follows by induction on $m$ 
and differentiation.
\end{proof}

\begin{theorem}
\label{theorem:LoewnerODEforInverse}
Let $I = (\alpha, \beta )$ with 
$- \infty < \alpha < \beta < \infty$, 
and let $\{f_t \}_{t \in I}$ be a strictly increasing
and continuous Loewner chain of coverings maps 
with $a(t)=f_t'(0)$ for $t \in I$.
Let $N$ and $\{ P(\cdot,t) \}_{t \in I}$ denote
the associated $G_\delta$-subset
of $I$ with $\mu_a$-measure $0$ and Herglotz family 
of $\{ f_t \}_{t \in I}$ as in Theorem~\ref{thm:Loewner-DEs}.
Let $a_0 \in \mathbb{D}$, and let $D$ be a simply connected 
domain in $\mathbb{C}$ satisfying 
$f_t( a_0) \in D \subset \Omega_t$ for all $t \in I$.
For $t \in I$, 
let $g_t (w) = g(w,t)$
be the single-valued branch of $f_t^{-1}$ 
on $D$ such that $a_0 = g_t ( f_t( a_0) )$.
Then, 
for each $m \in \mathbb{N} \cup \{ 0 \}$
and $w \in D$, 
the function $\frac{\partial^m g}{\partial w^m} (w,t)$ 
is differentiable with respect to 
$a(t):= f_t'(0)> 0$ 
at $t_0 \in I \setminus N$,
and we have
\begin{equation}
\label{eq:the_equality_of_mixed_partial_derivatives}
  \frac{\partial^m}{\partial w^m} 
   \left( \frac{\partial g}{\partial a }\right) (w,t)
  =
  \frac{\partial}{\partial a } 
   \left( \frac{\partial^m g}{\partial w^m} \right) (w,t) .
\end{equation}
In particular,
\begin{align}
\label{eq:ode_satisfied_by_inverse_function}
    \frac{\partial g}{\partial a} (w,t)  
  = - \frac{g(w,t)P(g(w,t),t) }{a(t)} . 
\end{align}
\end{theorem}

\begin{proof}
Fix $t_0 \in I \setminus N$ and $w_0 \in D$ arbitrarily.

\noindent\textbf{Step 1.}
First we show that $\frac{\partial^m g}{\partial w^m}(w,t)$
is differentiable with respect to $a(t)$ at $t_0$,
and the convergence of the limit 
\[
   \sidelimv
 \frac{\frac{\partial^m g}{\partial w^m}(w,t_2) 
 - \frac{\partial^m g}{\partial w^m}(w,t_1)}{a(t_2)-a(t_1)}
\]
is locally uniform on $D$.
To this end, fix $r \in (0, d(w_0, \partial D))$ arbitrarily
and choose $\rho$, $\rho_1$ and $\rho_2$
such that $r< \rho_1 < \rho < \rho_2 < d(w_0, \partial D)$.
Define a Jordan curve $\gamma$ in $\mathbb{D}$ by
\[
  \gamma ( \theta ) = g_{t_0}(w_0+\rho e^{2 \pi \theta i}),
 \quad 0 \leq \theta \leq 1 .
\]
Since $f_t \rightarrow f_{t_0}$ locally uniformly 
on $\mathbb{D}$ as $t \to t_0$,
there exist $\delta > 0$ such that
for $|t-t_0| < \delta$,
\[
  |f_t (\gamma (\theta)) -f_{t_0} (\gamma (\theta) ) |
 = |f_t (\gamma (\theta)) -(w_0+\rho e^{2 \pi \theta i} ) |
 < \min \{ \rho_2-\rho, \rho - \rho_1 \},
\]
for all $0 \leq \theta \leq 1$.
Hence, for each $t \in (t_0-\delta, t_0+\delta)$, 
$f_t \circ \gamma$ is a Jordan curve contained in the 
closed annulus 
$\{ w \in \mathbb{C} : \rho_1 \leq |w-w_0| \leq \rho_2 \}$. 
Therefore, for $w \in \mathbb{D}(w_0, \rho_1)$,
Lemma~\ref{lemma:CauchyIntegralFormula_for_inverse_functions}
gives 
\[
    g_t^{(m)}(w) 
    = \frac{m!}{2 \pi i}
   \int_{\gamma} 
   \frac{z f_t'(z)}{(f_t(z ) -w)^{m+1}} \, d z .
\] 
From this it follows that, 
for $t_0 - \delta < t_1 < t_2 < t_0+\delta$,
\begin{align*}
 & \frac{g_{t_2}^{(m)}(w)- g_{t_1}^{(m)}(w) }{a(t_2)-a(t_1)}
\\
 = & 
 \frac{m!}{2 \pi i}
   \int_{\gamma} 
  \frac{1}{a(t_2)-a(t_1)}
 \left\{ \frac{z f_{t_2}'(z)}{(f_{t_2}(z ) -w)^{m+1}}
  - \frac{z f_{t_1}'(z)}{(f_{t_1}(z ) -w)^{m+1}}
 \right\} dz 
\\
 = & 
 \frac{m!}{2 \pi i}
   \int_{\gamma} 
  \frac{z \left( f_{t_2}'(z)-f_{t_1}'(z) \right) \, dz }
  {(a(t_2)-a(t_1))(f_{t_2}-w)^{m+1}}
\\
  & 
  -  \frac{m!}{2 \pi i}
   \int_{\gamma} 
  \frac{z f_{t_1}'(z) \left( f_{t_2}(z)-f_{t_1}(z) \right) }
  {a(t_2)-a(t_1)}
  \sum_{j=0}^m (f_{t_2}(z)-w)^{-j-1} (f_{t_1}(z)-w)^{j-m-1}
  \, dz .
\end{align*}
By Corollary~\ref{corollary:the_convergence_of_a_derivative}
we conclude that
\begin{align*}
 &
 \sidelimv
 \frac{g_{t_2}^{(m)}(w)- g_{t_1}^{(m)}(w) }{a(t_2)-a(t_1)}
\\
 = \, & 
 \frac{m!}{2 \pi i}
   \int_{\gamma} 
  \left\{ 
 \frac{z \frac{\partial^2 f}{\partial a \partial z}(z,t_0)}
  {(f_{t_0}-w)^{m+1}}
  -
 \frac{(m+1) z \frac{\partial f}{\partial z} (z,t_0) \frac{\partial f}{\partial a} (z,t_0) }
  {(f_{t_0}(z)-w)^{m+2}}
  \right\} dz , 
\end{align*}
and that the convergence is 
uniform on $\mathbb{D}(w_0, r)$.

\noindent\textbf{Step 2.}
From Step~1 it follows that
\begin{equation}
\label{eq:a_derivative_of_inverse}
  \frac{g_{t_2} (w) - g_{t_1} (w)}{a(t_2)-a(t_1)}
 \rightarrow \frac{\partial g}{\partial a } (w,t)
 \quad
 \text{as} \quad
 t_2-t_1 \searrow 0 \quad \text{with} \; t_1 \leq t_0 \leq t_2   
\end{equation}
uniformly on $\overline{\mathbb{D}}(w_0,r)$.
Combining this with
\[
   \frac{\partial^m g}{\partial w^m} (w,t)
 =
 \frac{m!}{2\pi i} 
 \int_{|\zeta-w_0|= r} \frac{g(\zeta,t)}{(\zeta -w)^{m+1}}
 \, d \zeta ,
 \; t_0 - \delta < t< t_0+ \delta
 \text{ and } w \in \mathbb{D}(w_0,r )
\]
we obtain
\[
   \frac{\partial^{m+1} g}{\partial a \partial w^m} (w,t_0)
 =
 \frac{m!}{2\pi i} 
 \int_{|\zeta-w_0|= r} 
 \frac{ \frac{\partial g}{\partial a}(\zeta, t_0) }{(\zeta -w)^{m+1}}
 \, d \zeta ,
 \quad  w \in \mathbb{D}(w_0,r ) .
\]
On the other hand, by \eqref{eq:a_derivative_of_inverse},
the function $\frac{\partial g}{\partial a } (w,t_0) $
is analytic in $\mathbb{D}(w_0,r)$
and continuous on $\overline{\mathbb{D}}(w_0,r)$.
Therefore,
\[
   \frac{\partial^{m+1} g}{\partial w^m \partial a } (w,t_0)
 =
 \frac{m!}{2\pi i} 
 \int_{|\zeta-w_0|= r} 
 \frac{ \frac{\partial g}{\partial a}(\zeta, t_0) }{(\zeta -w)^{m+1}}
 \, d \zeta ,
 \quad  w \in \mathbb{D}(w_0,r ) .
\]
Thus,
\[
 \frac{\partial^{m+1} g}{\partial a \partial w^m} (w,t_0) 
 = \frac{\partial^{m+1} g}{\partial w^m \partial a } (w,t_0),
 \quad  w \in \mathbb{D}(w_0,r ) . 
\]

\noindent\textbf{Step 3.}
Finally, we derive 
(\ref{eq:ode_satisfied_by_inverse_function}).
Since $g_t$ is a single-valued branch of $f_t^{-1}$,
we have
\[
   f(g(w,t),t) = w ,
 \quad t \in I
 \text{ and } w \in \mathbb{D}(w_0,r ) .
\]
Hence, for $t_1,t_2 \in (\alpha , \beta) $ 
with
$t_1 \leq t_0 \leq t_2$ and $t_1<t_2$,
we obtain
\begin{align}
\label{eq:f-of-g-is-identity}
  0 
 = & \, \frac{f(g(w,t_2),t_2) - f(g(w,t_1),t_1)}{a(t_2)-a(t_1)}
\\
 = & \, \frac{f(g(w,t_2),t_2) - f(\omega (g(w,t_1),t_1,t_2),t_2)}
 {a(t_2)-a(t_1)}
\nonumber
\\
 = & \, \frac{g(w,t_2) - \omega (g(w,t_1),t_1,t_2)}{a(t_2)-a(t_1)}
  \int_0^1 \frac{\partial f}{\partial z} (\alpha (\theta ) ,t_2) \, 
 d \theta  ,
\nonumber
\end{align}
where $\{ \omega (z,s,t) \}$
is the associated transition family and
$\alpha$ is the path defined by
\[
   \alpha (\theta )
 =
 (1-\theta) \omega (g(w,t_1),t_1,t_2) + \theta g(w,t_2),
 \quad 0 \leq \theta \leq 1 .
\]
By Theorem~\ref{theorem:continuity-of-transition-family}
and the fact $g(w,t)$ is continuous in $t$,
$\frac{\partial f}{\partial z} (\alpha (\theta ) ,t_2)$
converges uniformly to 
$\frac{\partial f}{\partial z} (g(w,t_0) ,t_0)$
on $[0,1]$ as
$t_2 - t_1 \searrow 0$ with $t_1 \leq t_0 \leq t_2$.

Furthermore, 
we have 
\begin{align*}
  & \, 
 \frac{g(w,t_2) - \omega (g(w,t_1),t_1,t_2)}{a(t_2)-a(t_1)}
\\
 = & \,
  \frac{g(w,t_2) - g(w,t_1) 
   - \{ \omega (g(w,t_0),t_1,t_2) - g(w,t_0) \} }{a(t_2)-a(t_1)}
\\
 & \, + \frac{-g(w,t_0) +  g(w,t_1) + \omega (g(w,t_0),t_1,t_2) 
 - \omega (g(w,t_1),t_1,t_2)}{a(t_2)-a(t_1)}
\\
 = & \,
  \frac{g(w,t_2) - g(w,t_1)}{a(t_2)-a(t_1)} 
   - \frac{ \omega (g(w,t_0),t_1,t_2) - g(w,t_0) }{a(t_2)-a(t_1)}
\\
 & \, + \frac{g(w,t_0) -  g(w,t_1)}{a(t_2)-a(t_1)} 
   \int_0^1 
 \left\{ 
 \frac{\omega }{\partial z} ( \beta ( \theta ),t_1,t_2) -1 \right\} \, 
 d \theta ,
\end{align*}
where $\beta $ is the path defined by
\[
 \beta (\theta) 
 = 
 (1-\theta) g(w,t_1) + \theta g(w,t_0), \quad 0 \leq \theta \leq 1 . 
\]

By Proposition~\ref{proposition:fundamental_inequalities}
\[
  \int_0^1 
 \left\{ 
 \frac{\partial \omega }{\partial z} ( \beta ( \theta ),t_1,t_2) -1 \right\} \, 
 d \theta
 \longrightarrow 0
\]
as $t_2-t_1 \searrow 0$ with $t_1 \leq t_0 \leq t_2$.

Now, by \eqref{eq:a_derivative_of_inverse},
the limit
\[
  \lim_{t_1 \nearrow t_0} 
 \frac{g_{t_0} (w) - g_{t_1} (w)}{a(t_0)-a(t_1)}
\]
exists.
Moreover, there exists $M>0$ and $\delta_1 > 0$
such that $|g_{t_0} (w) - g_{t_1} (w)| \leq M |t_0-t_1|$
for all $t_1 \in (t_0-\delta_1, t_0)$.
Thus, we obtain
\[
  \left| \frac{g(w,t_0) -  g(w,t_1)}{a(t_2)-a(t_1)} \right|
 \leq
  M \quad \text{for } t_0 - \delta_1 < t_1 \leq t_0 \leq t_2 
 \text{ with } t_1 < t_2 .
\]

Therefore, passing to the limit in (\ref{eq:f-of-g-is-identity})
as $t_2-t_1 \searrow 0$ with $t_1 \leq t_0 \leq t_2$,
we obtain 
\begin{align*}
 0 
 = & \,
 \left\{ 
 \frac{\partial g}{\partial a}(w,t_0) 
 -  \frac{\partial \omega}{\partial a}(g(w,t_0)) 
 \right\} \frac{\partial f}{\partial z} (g(w,t_0) ,t_0)
\\
  = & \, 
 \left\{ 
 \frac{\partial g}{\partial a}(w,t_0) 
 + \frac{g(w,t_0)}{a(t_0)} P(g(w,t_0),t_0) 
 \right\} \frac{\partial f}{\partial z} (g(w,t_0) ,t_0) .
\end{align*}
Since $f_t$ is locally univalent,
$\frac{\partial f}{\partial z} (z ,t) \not= 0 $ 
for all $z \in \mathbb{D}$.
Hence (\ref{eq:ode_satisfied_by_inverse_function}) holds.
\end{proof}

\section{Loewner Equation for the Hyperbolic Density}
We now derive an ordinary differential equation 
satisfied by $\lambda_t$.
\begin{theorem}
Under the same notation as in 
Theorem~\ref{theorem:LoewnerODEforInverse},
the density $\lambda_t(w) = \lambda(w,t)$ 
of the hyperbolic metric on $\Omega_t$ satisfies
\begin{align}
  \frac{\frac{\partial \lambda}{\partial a} (w,t)}{\lambda(w,t)}
 =  - \frac{1}{a(t)}  \Real 
 \left\{ \frac{1+|g_t(w)|^2}{1-|g_t(w)|^2}P(g_t(w),t) 
 +  g_t(w) \frac{\partial P}{\partial z}(g_t(w),t) \right\} .
\end{align}
\end{theorem}
\begin{proof}
By differentiating 
\eqref{eq:ode_satisfied_by_inverse_function}
with respect to $a(t)$ 
and applying 
\eqref{eq:the_equality_of_mixed_partial_derivatives},
we obtain
\begin{equation}
\label{eq:mixed-partial-derevatives}
  \frac{\partial^2 g }{\partial a \partial w}(w,t)
   = 
  \frac{\partial^2 g }{\partial w \partial a}(w,t)
   =
   - \frac{\frac{\partial g }{\partial w}(w,t)}{a(t)}
   \left\{ P(g(w,t))+ g(w,t) P(g(w,t),t) \right\}
\end{equation}
Now let $\lambda (w,t) = \lambda_t(w)$.
Differentiating
\[
   \log \lambda (w,t) 
   = 
   \frac{1}{2} \log \left\{ \frac{\partial g}{\partial w} (w,t) 
   \overline{\frac{\partial g}{\partial w} (w,t)} \right\} 
   - \log \left\{ 1- g(w,t)\overline{g(w,t)}\right\}
\]
with respect to $a(t)$,
and using 
\eqref{eq:the_equality_of_mixed_partial_derivatives}
and \eqref{eq:mixed-partial-derevatives},
we obtain 
\begin{align*}
  & \frac{\frac{\partial \lambda}{\partial a} (w,t)}{\lambda(w,t)}
\\
 = & \,
 \Real 
 \left\{ \frac{\frac{\partial^2 g}{\partial a \partial w} (w,t)}
 {\frac{\partial g}{\partial w} (w,t)} \right\}
 + 2 
 \frac{\Real 
 \left\{ \frac{\partial g}{\partial a} (w,t) 
 \overline{g(w,t)} \right\}}
 {1-|g(w,t)|^2}
\\
 = & \,
 - \frac{1}{a(t)} \Real 
 \left\{ 
  P(g(w,t),t)+ g(w,t) 
 \frac{\partial P}{\partial z}(g(w,t),t)
 \right\}
  - 2 \frac{|g(w,t)|^2 \Real \{ P(g(w,t),t)\}}{a(t)(1-|(g(w,t))|^2)}
\\
 = & \,  - \frac{1}{a(t)}  \Real 
 \left\{ \frac{1+|g_t(w)|^2}{1-|g_t(w)|^2}P(g_t(w),t) 
 +  g_t(w) \frac{\partial P}{\partial z}(g_t(w),t) \right\} .
\end{align*}
\end{proof}

Since
$P(\cdot ,t )$ is analytic in $\mathbb{D}$ with $\Real P (\cdot, t)> 0$ 
and $P(0,t) = 1$,
then there exists a family of Borel probability measures
$\{ \mu_t \}_{t \in I}$ on $\partial \mathbb{D}$ such that
\begin{align*}
 P(z,t) = \int_{\partial \mathbb{D}} 
 \frac{\zeta + z}{\zeta -z } \, d \mu_t (\zeta ) . 
\end{align*}
From this we obtain
\begin{align*}
   z \frac{\partial P(z,t)}{\partial z} 
  = & \, 
 \int_{\partial \mathbb{D}} \frac{2 z \zeta }{(\zeta -z)^2 } \, d \mu_t (\zeta ) ,
\\
   \frac{1+|z|^2}{1-|z|^2} P(z,t) 
 +  z \frac{\partial P(z,t)}{\partial z} 
  = & \, 
 \int_{\partial \mathbb{D}} 
 K(z,\zeta ) \, d \mu_t (\zeta ), 
\end{align*}
where
\begin{equation}
 K(z, \zeta ) 
  = 
  \frac{1+|z|^2}{1-|z|^2} \cdot \frac{\zeta + z}{\zeta -z } 
  +  \frac{2 z \zeta }{(\zeta -z)^2 } ,
  \quad z\in \mathbb{D} \text{ and }  \zeta \in \partial \mathbb{D}.
\end{equation}
It then follows that 
\begin{align*}
  \frac{\frac{\partial \lambda}{\partial a} (w,t)}{\lambda(w,t)}
 =
 - \frac{1}{a(t)}
  \int_{\partial \mathbb{D}} 
  \Real   K(g_t(w), \zeta )  \, d \mu_t (\zeta ) .
\end{align*}
Taking into account the condition $|\zeta|=1$ 
and carrying out the computation, we obtain
\begin{equation}
  \Real K(z,\zeta)
   = \frac{(1-|z|^2)^2}{|\zeta -z|^4} ,
\end{equation}
i.e., $\Real   K(z, \zeta )$ coincides with
the square of the Poisson kernel.
Therefore, 
\begin{equation}
  \frac{\frac{\partial \lambda}{\partial a} (w,t)}{\lambda(w,t)}
 =
 - \frac{1}{a(t)}
  \int_{\partial \mathbb{D}} 
  \frac{(1-|g_t(w)|^2)^2}{|\zeta - g_t(w)|^4}  \, d \mu_t (\zeta ) . 
\end{equation}
Moreover,
by the simple estimate
\[
  \left( \frac{1-|z|}{1+|z|} \right)^2
  \leq
  \Real   K(z, \zeta )
  \leq
  \left( \frac{1+|z|}{1-|z|} \right)^2 ,
\]
we have
\begin{equation}
 - \frac{1}{a(t)}
   \left( \frac{1+|g_t(w)|}{1-|g_t(w)|} \right)^2 
   \leq
  \frac{\frac{\partial \lambda}{\partial a} (w,t)}{\lambda(w,t)}
   \leq
 - \frac{1}{a(t)}
  \left( \frac{1-|g_t(w)|}{1+|g_t(w)|} \right)^2 .
\end{equation}
In particular, if $a(t)$ is locally  absolutely continuous 
in $I$
and $\frac{da}{dt}(t) > 0$ a.e., then we have
\begin{equation}
 - \frac{\frac{da}{dt}(t)}{a(t)}
   \left( \frac{1+|g_t(w)|}{1-|g_t(w)|} \right)^2 
   \leq
  \frac{\frac{\partial \lambda}{\partial t} (w,t)}{\lambda(w,t)}
   \leq
 - \frac{\frac{da}{dt}(t)}{a(t)}
  \left( \frac{1-|g_t(w)|}{1+|g_t(w)|} \right)^2 
\end{equation}
holds a.e. in $I$.

\appendix
\chapter{Separation Theorem}
\label{chapter:separation}
The separation lemma (Lemma \ref{lemma:separation}) 
can be proved by using results from combinatorial topological 
lemmas and theorems (see Newman \cite{Newman}). 
However, in this appendix we give a simpler proof 
by employing the Riemann mapping theorem together with 
several necessary and sufficient conditions for a domain 
in $\mathbb{C}$ to be simply connected.

We have repeatedly used the following criterion: 
a domain in $\hat{\mathbb{C}}$ is simply connected 
if and only if its complement is connected or empty.
We also rely on the following classical result.
\begin{lemma} 
Let $\Omega$ be a domain in $\mathbb{C}$.
Then $\Omega$ is simply connected if and only if 
the interior domain of every simple closed curve 
in $\Omega$ is contained in $\Omega$.
\end{lemma}
For a purely topological proof 
see 
Newman \cite[Chapter~VI]{Newman}. 
From the lemma we immediately obtain the following.
\begin{lemma}
\label{lemma:each_component_of_complementary_set_ofconnected_set}
Let $E$ be a nonempty compact 
connected set in $\hat{\mathbb{C}}$.
Then each component of $\hat{\mathbb{C}} \backslash E$  
is simply connected.
\end{lemma}

By a partition of a set $E$ in a topological space
we mean two nonempty subsets $H_1$ and $H_2$ of $E$ 
such that $H_1$ and $H_2$ are closed 
in the subspace topology of $E$, 
with $H_1 \cup H_2 = E$ and $H_1 \cap H_2 = \emptyset$. 
For a proof of the following lemma, 
see Newman \cite[Theorem~5.6]{Newman}. 
\begin{lemma}
\label{lemma:existence_of_partition}
Let $E$ be a compact set in a metric space $X$, 
$F_1$, $F_2$ be nonempty closed subsets of $E$ such that 
for any component $B$ of $E$, 
$F_1 \cap B = \emptyset$ or $F_2 \cap B = \emptyset$.
Then there exists a partition $H_1$, $H_2$ of $E$ with 
$F_1 \subset H_1$ and $F_2 \subset H_2$.
\end{lemma}

Now we prove the separation lemma.
\begin{proof}[Proof of Lemma \ref{lemma:separation}]
After a linear fractional transformation if necessary,
we may assume that $\infty \in F$.
By Lemma \ref{lemma:existence_of_partition} 
there exist closed sets $H_1$ and  $H_2$ 
with $H_1 \cap H_2 = \emptyset$ and 
$H_1 \cup H_2 = \hat{\mathbb{C}} \backslash \Omega$,
satisfying $C \subset H_1$ and $F \subset H_2$.
Since $\infty \not\in H_1$,
$H_1$ is a compact subset of $\mathbb{C}$.
We may assume that $H_2$ contains 
at least one point other than $\infty$;
otherwise the lemma is obvious.
It the follows that
we have 
$0 < d(H_1, H_2) := \{ |z-w| : z \in H_1, \; w \in H_2 \} < \infty$.

Let $S$ be a square 
with $H_1 \subset \Int S$, whosesides are 
parallel to the coordinate axes.
By a square we mean a closed solid square
consisting of both its boundary and interior, 
and we denote the set of interior points of $S$ by $\Int S$.
Let $\ell$ be the side length of $S$ and    
choose $n \in \mathbb{N}$ such that
\[
    \frac{\sqrt{3} \ell}{n} 
 < \min\{ d(H_1, H_2) , d (H_1 , \partial S ) \}.
\]
We divide $S$ into nonoverlapping small squares
of side length $\ell /n $ by equally spaced 
horizontal and vertical lines.
We call $\hat{\mathbb{C}} \backslash \Int S$ the unbounded square.
Let $K$ be the union of 
the unbounded square and those small squares that intersect $H_2$.

We show that $\partial K \cap (H_1 \cup H_2) = \emptyset$, i.e.,
$\partial K \subset \Omega$.
Note that $\partial K$ consists of edges $A$ of the small squares.
If $A \subset \partial S$, there exists a unique 
small square $T$ having $A$ as on of its edge.
Since $\diam T = \frac{\sqrt{2} \ell}{n} < d (H_1, \partial S)$,
we have $T \cap H_1 = \emptyset$.
Moreover, $T \cap H_2 = \emptyset$.
Indeed, if $T \cap H_2 \not= \emptyset$, then 
$A \subset T \cup (\hat{\mathbb{C}} \backslash \Int S) \subset K$.
This would imply $A^\circ \subset \Int K$, contradicting 
$A \subset \partial K$.
Here $A^\circ$ is the open segment obtained from $A$ by 
removing two vertices of $A$.

Next we consider the case 
$A \backslash \partial S \not= \emptyset$.
In this case there exist exactly two adjacent small 
squares $T_1$ and $T_2$ with $A \subset T_1 \cap T_2$, 
and we may assume $T_1 \cap H_2 = \emptyset$ and 
$T_2 \cap H_2 \not= \emptyset$.
Then $A \cap H_2 \subset T_1 \cap H_2 = \emptyset$.
Moreover, $A \cap H_1 = \emptyset$ 
since $T_2 \cap H_2 \not= \emptyset$ and 
$\diam (T_1 \cup T_2) = \frac{\sqrt{3} \ell}{n} < d (H_1, H_2)$.
Thus $\partial K \cap (H_1 \cup H_2) = \emptyset$.

\begin{tikzpicture}
\draw (0,0)--(6,0) node[midway,above] {$\partial S$};
\draw[very thick] (3,0)--(3,-1) ;
\node at (3.15,-0.7) {$A$};
\draw (2,0)--(2,-1)--(3,-1);
\node at (2.5,-0.5) {$T_1$};
\draw (4,0)--(4,-1)--(3,-1);
\node at (3.5,-0.5) {$T_2$};
\end{tikzpicture}
\hfill
\begin{tikzpicture}
\draw (0,-0.5)--(0,2) node[midway,left] {$\partial S$};
\draw[very thick] (2,0)--(2,1) ;
\node at (2.15,0.3) {$A$};
\draw (2,0)--(1,0)--(1,1)--(2,1);
\node at (1.5,0.5) {$T_1$};
\draw (2,0)--(3,0)--(3,1)--(2,1);
\node at (2.5,0.5) {$T_2$};
\end{tikzpicture}

Let $K_0$ be the component of $K$ containing
the unbounded square $\hat{\mathbb{C}} \backslash \Int S$, 
and write $K = K_0 \cup K_1 \cup \cdots \cup K_m$ be 
the decomposition of $K$ into connected components. 
Since 
$\partial K_j \subset \partial K \subset \Omega$ for $j=0,\ldots , m$,
there exists a path $\alpha_j : [0,1] \rightarrow \Omega$ with
$\alpha_j (0) \in \partial K_0$ and $\alpha_j (1) \in \partial K_j$ 
for $j=1,\ldots , m$.
Then 
$\tilde{F} = K \cup \bigcup_{j=1}^m \alpha_j ([0,1])$
is connected and closed in $\hat{\mathbb{C}}$, and  satisfies
\[
  \infty \in \tilde{F}, \quad F \subset H_2 \subset \tilde{F}
  \quad \text{and} \quad  H_1 \cap \tilde{F} = \emptyset .
\]
Let $\Omega_0$ be the component of 
$\hat{\mathbb{C}} \backslash \tilde{F}$
containing $C$.
Since
\[
  \partial \Omega_0 \subset \partial \tilde{F} 
  \subset \partial K \cup \bigcup_{j=1}^m \alpha_j ([0,1]) \subset \Omega  ,
\]
we have $\partial \Omega_0 \cap H_1 = \emptyset$,
and hence $\Omega_0 \cap H_1 (= \overline{\Omega_0} \cap H_1)$ 
is compact.
By Lemma \ref{lemma:each_component_of_complementary_set_ofconnected_set}  
$\Omega_0$ is simply connected. 
Thus, by the Riemann mapping theorem,
there exists a conformal map $h: \mathbb{D} \rightarrow \Omega_0$.
Since $\Omega_0 \cap H_1$ is compact,
there exists $r \in (0,1)$ with 
\[
  \Omega_0 \cap H_1 \subset h(\mathbb{D}(0,r)) .
\] 
Define a simple closed curve 
$\alpha : \partial \mathbb{D} \rightarrow \Omega_0$ 
by 
\[
    \alpha ( \zeta ) = h(r \zeta ), \quad \zeta \in \partial \mathbb{D} .
\]
Clearly, $\alpha(\partial \mathbb{D}) \cap H_1 = \emptyset$,
and $\alpha(\partial \mathbb{D}) \cap H_2 
\subset \Omega_0 \cap \tilde{F} = \emptyset$.
Therefore $\alpha$ is a simple closed curve in 
$\Omega = \hat{\mathbb{C}} \backslash (H_1 \cup H_2)$.
Furthermore,
the interior domain of $\alpha $ coincides with $h(\mathbb{D}(0,r))$
and contains $C ( \subset \Omega_0 \cap H_1)$.
Since $\alpha ( \partial \mathbb{D}) \cap \tilde{F} = \emptyset$,
the connected set $\tilde{F}$ is contained in either
the interior or exterior of $\alpha$.  
As $\infty \in \tilde{F}$, we conclude that 
$\tilde{F}$ lies in  the exterior of $\alpha$.  
Therefore $\alpha$ separates $C$ and $F (\subset \tilde{F})$.

\begin{center}
\begin{resizebox}{8cm}{8cm}
{
\begin{tikzpicture}
\draw [fill=gray,gray] (-2,-2) rectangle (10,10);
\draw[fill=white, very thick] (0,0)--(8,0)--(8,7)--(6,7)--(6,8)--(0,8)--cycle;
\node at (7,8) {$K_0$} ;

\draw[fill=gray,very thick] (2,3) rectangle (5,6);
\draw[fill=white,very thick] (2.5,3.5) rectangle (4.5,5.5);
\node at (4,5.7) {$K_1$} ;

\draw[fill=gray,very thick] (3,4) rectangle (4,5);
\node at (3.5,4.5) {$K_2$} ;

\draw[very thick] (1.5,1.5)--node[above] {$C$} (3,1.5);

\node at (6,1) {$\Omega_0$}; 

\draw[very thick,->] (3,8) to 
[out=30,in=90, relative] node[right] {$\alpha_1$}(4,6);

\draw[very thick,->] (8,4.5) to 
[out=-30,in=180, relative] node[below] {$\alpha_2$}(4,4.5);

\end{tikzpicture}
}
\end{resizebox}
\end{center}

\end{proof}

\backmatter
\bibliographystyle{amsalpha}

\printindex

\end{document}